\documentclass[11pt,a4paper]{amsart}
\usepackage{spinal}

\title[On symplectic fillings of spinal open book decompositions~II]{On 
symplectic fillings of spinal open book decompositions~II:\\
Holomorphic curves and classification}
\author{Samuel Lisi}
\address{University of Mississippi\\
Department of Mathematics\\
USA }
\email{stlisi@olemiss.edu}
\author{Jeremy Van Horn-Morris}
\address{Department of Mathematical Sciences \\
The University of Arkansas\\
USA}
\email{jvhm@uark.edu}
\author{Chris Wendl}
\address{Institut f\"ur Mathematik \\
Humboldt-Universit\"at zu Berlin \\
Germany}
\email{wendl@math.hu-berlin.de}
\thanks{S.L.~was partially supported by NSF award \# 1929176, 
a SEC travel grant,
by a UM CLASRG,
F.~Bourgeois's ERC Starting Grant StG-239781-ContactMath, and by V.~Colin's ERC Grant geodycon. 
C.W.~was partially supported during this project by a 
Humboldt Foundation Postdoctoral Fellowship, a Royal Society University
Research Fellowship, and EPSRC grant EP/K011588/1.
}

\begin{document}

\begin{abstract}

In this second paper of a two-part series, we prove that whenever
a contact $3$-manifold admits a uniform spinal open book decomposition
with planar pages, its (weak, strong and/or exact) symplectic and
Stein fillings can be classified up to deformation equivalence in
terms of diffeomorphism classes of Lefschetz fibrations.  This
extends previous results of the third author \cite{Wendl:fillable}
to a much wider class of contact manifolds, which we illustrate
here by classifying the strong and Stein fillings of all oriented
circle bundles with non-tangential $S^1$-invariant contact structures.
Further results include new vanishing criteria for the ECH contact
invariant and algebraic torsion in SFT, classification of fillings
for certain non-orientable circle bundles, and a general ``symplectic
quasiflexibility'' result about deformation classes of Stein
structures in real dimension four.

\end{abstract}

\maketitle

\tableofcontents

\section{Introduction}
\label{sec:intro}

This paper is the sequel to \cite{LisiVanhornWendl1}, which introduced the 
notion of a spinal open book decomposition 
$$
\boldsymbol{\pi} := \Big(\pi\spine : M\spine \to \Sigma,
\pi\paper : M\paper \to S^1, \{m_T\}_{T \subset \p M}\Big)
$$
of a $3$-manifold~$M$, a structure
that arises naturally whenever $M$ is the boundary of the total space of a
bordered Lefschetz fibration $\Pi : E \to \Sigma$
over a compact oriented surface $\Sigma$ with boundary.
A spinal open book splits $M$ into two (not
necessarily connected) pieces
$M\spine \cup M\paper$, the \defin{spine} and \defin{paper} respectively.
In the Lefschetz case, when $M = \p E$, 
$M\spine$ is the \emph{horizontal} boundary (the boundaries of all the fibers) and 
$M\paper$ is the \emph{vertical} boundary (the union of all fibers over the boundary of the base).
This means in particular that the ``corner''
$\p M\spine = M\spine \cap M\paper = \p M\paper$ is a disjoint union of
$2$-tori, and the two pieces are endowed with smooth fibrations
$\pi\spine : M\spine \to \Sigma$ and $\pi\paper : M\paper \to S^1$,
where the connected components of the fibers $\pi\paper^{-1}(*) \subset M\paper$
are surfaces with boundary called \defin{pages}, the fibers
of $\pi\spine$ are~$S^1$, and the connected components of its base are
surfaces with boundary known as \defin{vertebrae}.

Just as the Lefschetz fibration $\Pi$ naturally
determines a symplectic structure $\omega$ on $E$ up to deformation, 
$\boldsymbol{\pi}$ determines a contact structure $\xi$ on $M$ up to isotopy
such that the relationship ``$\p \Pi \cong \boldsymbol{\pi}$'' between the two 
decompositions makes $(E,\omega)$ a symplectic filling of~$(M,\xi)$.
One of our main goals in the present paper is to invert this relationship and prove
a far-reaching generalization of the main result of
\cite{Wendl:fillable}: for a particular class of spinal open books 
$\boldsymbol{\pi}$ on a closed contact
$3$-manifold $(M,\xi)$, the deformation classes of (weak, strong, exact
or Stein/Weinstein) symplectic fillings 
of $(M,\xi)$ are in one-to-one correspondence
with the diffeomorphism classes of Lefschetz fibrations 
filling~$\boldsymbol{\pi}$.  In addition to providing a powerful new tool
for classifying symplectic fillings, this result reveals the existence of
a special class of Stein surfaces, which are \emph{quasiflexible} in the
sense that their Stein homotopy types are determined by the (not necessarily
exact) deformation classes of their symplectic structures.

\subsubsection*{Main ideas and difficulties}

While the proofs in this paper tend to involve a lot of moving parts that
take many pages to pin down, the underlying ideas are easy to summarize.
In the background are two fundamental geometric phenomena that are quite
well known:
\begin{enumerate}
\item In the tradition of Thurston-Winkelnkemper \cites{Thurston:symplectic,ThurstonWinkelnkemper}
and Gompf \cites{GompfStipsicz,Gompf:hyperpencils,Gompf:locallyHolomorphic},
spinal open books and Lefschetz fibrations uniquely determine contact and symplectic
structures respectively up to deformation;
\item In the tradition of Gromov and McDuff \cites{Gromov,McDuff:rationalRuled},
certain types of symplectic submanifolds give rise to foliations by
$J$-holomorphic curves that determine the global structure of a symplectic
manifold.
\end{enumerate}
In our setting, the symplectic submanifolds feeding into McDuff's technique
are the pages of a planar spinal open book on the boundary of a
symplectic filling, and the resulting $J$-holomorphic foliation produces
(in favorable cases) a classification of the possible fillings.  This idea has appeared before
in \cites{Wendl:fillable,NiederkruegerWendl,Wendl:openbook2}, whose main results
are all special cases of the results of the present paper.  However, the level
of generality considered here introduces several new difficulties requiring
novel solutions, which have contributed substantially to the length 
of this paper.  

One difficulty is that due to the variety of topologies possible in a spinal
open book, the moduli spaces of holomorphic curves arising from their pages
do not consist exclusively of $0$- and $2$-dimensional families of embedded
curves.  They generally also include $1$-dimensional ``walls'' that must be
crossed, as well as multiply covered curves for which transversality is of
course a thorny issue.
Considerable effort is required in the compactness arguments of
\S\ref{sec:uniqueness} and \S\ref{sec:compactnessProofs} to either rule out
the appearance of such multiple covers or show that when they do arise
(which sometimes they must), the necessary transversality results hold anyway.
It is a minor miracle that the results work out as nicely as one would hope,
and we interpret this as convincing evidence for the naturality of our
approach to the filling problem.

A second difficulty concerns the precise type of symplectic fillings that one
is attempting to classify: unlike all previous papers on this problem
(see \S\ref{sec:context} below),
our approach produces a unified framework in which to classify the full
spectrum of weak, strong, exact and Weinstein/Stein fillings, each up to the corresponding notion of
deformation equivalence.  The inclusion of both weak and 
Weinstein deformation equivalence
in this list is one of the most novel details of the present work,
and it requires a quite intricate construction (carried out in
\S\ref{sec:model} and \S\ref{sec:holOpenBook}) of geometric data on 
the symplectization of a contact manifold supported by a spinal open book.

\subsubsection*{Outline of the paper}

The remainder of \S\ref{sec:intro} consists of a quick review of the salient
definitions and of statements of the main theorems to be proved in later
sections. In particular, the general results on
classification of fillings are stated in 
\S\ref{sec:classification}--\ref{sec:Stein}, followed in
\S\ref{sec:invariantsSketch} by results on computations of contact invariants.
The latter computations parallel the non-fillability results already proved in
\cite{LisiVanhornWendl1}.  Section~\ref{subsec:circleBundles} then gives a
sample application, explaining how the main results imply a classification of
the fillings of all partitioned contact $S^1$-bundles over oriented surfaces,
and \S\ref{sec:exotic} discusses some slightly subtler examples for which
our main classification theorems do not apply but the techniques of this
paper still have something to say.  Since we will need to make extensive use of
$J$-holomorphic curves, \S\ref{sec:SHSgeneral} gives a 
review of some of the essential technical results---its contents are
mostly standard, but it also includes (in \S\ref{sec:automatic}) some useful 
lemmas about coherent orientations that may not have appeared in
writing before.  Section~\ref{sec:model} continues the development
(begun in \cite{LisiVanhornWendl1}*{\S 3}) of a precise model for the
half-symplectization of a contact manifold supported by a spinal open book,
including an explicit foliation by $J$-holomorphic curves with specific 
properties that are needed for the 
proofs of the main results.  The analytical properties of this holomorphic
foliation are then studied in \S\ref{sec:holOpenBook}, including existence 
and uniqueness results that are needed for the computations of algebraic
contact invariants carried out in \S\ref{sec:invariants}.  
Finally, \S\ref{sec:impressivePart} carries out
the holomorphic curve arguments needed to complete the proof that
planar spinal open books can be extended to Lefschetz fibrations on
fillings.

\subsubsection*{Acknowledgments}

This project has taken several years to come to fruition, and we are
grateful to many people for valuable conversations along the way,
including Denis Auroux, \.{I}nan\c{c} Baykur, Michael Hutchings, 
Tom Mark, Patrick Massot, Richard Siefring and Otto van Koert.
We would also like to thank the American Institute of Mathematics for bringing
the three of us together at key junctures in this project.

\subsection{Some remarks on the context}
\label{sec:context}

Spinal open books provide a unifying perspective on many of the known
classification and obstruction results for symplectic fillings in dimension 
four.  The first such results were those of Gromov \cite{Gromov} and
Eliashberg \cite{Eliashberg:diskFilling}, which classified the fillings
of $S^3$ and established overtwistedness as a filling obstruction.
A short time later, McDuff \cite{McDuff:rationalRuled} classified the fillings
of the universally tight lens spaces $L(p,1)$ up to diffeomorphism, a
result that was later improved by Hind \cite{Hind:Lens} to a classification
up to Stein deformation equivalence.  
In \cites{Giroux:plusOuMoins,Eliashberg:fillableTorus},
Giroux and Eliashberg found the first examples of contact $3$-manifolds that 
are weakly but not strongly fillable, which are now understood more generally
in terms of Giroux torsion
\cites{Gay:GirouxTorsion,GhigginiHondaVanhorn,GhigginiHonda:twisted,NiederkruegerWendl}.  Further
generalizations of this filling obstruction were introduced by the third
author \cite{Wendl:openbook2} and Latschev \cite{LatschevWendl}.
Classification results for weak fillings have mostly been limited to 
cases where the contact manifold is a rational homology sphere (so that weak
fillings can be deformed to strong fillings)---the major exception is the
planar case, for which \cite{NiederkruegerWendl} showed that weak fillings
are always deformable to strong fillings without the need for any topological
assumption.  The techniques of the present paper can be used to provide new 
and in many cases conceptually simpler proofs of all of the results just 
mentioned, and we suspect that this is the case for most other previous
filling results obtained via holomorphic curve methods, 
e.g.~\cites{OhtaOno:simpleSingularities,Lisca:fillingsLens,Starkston:Seifert}.
In fact, proofs via spinal open books usually lead to strengthened versions 
of these results, e.g.~they always apply to (a subclass of) weak symplectic fillings
with potentially nontrivial cohomology at the boundary, in addition to strong and Stein
fillings.  Moreover, where many of the previous results have achieved
classification of fillings up to diffeomorphism, ours also determine
the symplectic and Stein deformation classes of fillings---in particular,
the Lefschetz fibration approach provides the first systematic technique 
beyond the isolated results in \cites{Eliashberg:diskFilling,CieliebakEliashberg,
Hind:RP3,Hind:Lens} for classifying Stein fillings up to Stein
deformation equivalence.

We are also able to recover certain
results that were not previously accessible via holomorphic curves, 
including some of the vanishing results for the Ozsv\'ath-Szab\'o contact 
invariant\footnote{Results
about the Ozsv\'ath-Szab\'o contact invariant follow from our results on the
ECH contact invariant via the isomorphism \cites{ColinGhigginiHonda:1,ColinGhigginiHonda:2,ColinGhigginiHonda:3}
between Heegaard Floer homology and embedded contact homology.} due to
Honda-Kazez-Matic \cite{HondaKazezMatic} and Massot \cite{Massot:vanishing}.
We will see some examples in \S\ref{subsec:circleBundles} of classification
problems that are easily solved using spinal open books but were not previously
accessible via any known techniques.

There are still some important results in this subject about which
our techniques probably have nothing to say.  Prominent examples include
Lisca's filling obstruction in terms of positive scalar
curvature \cite{Lisca:curvature}, Ghiggini's examples of strongly but
not exactly fillable contact manifolds \cite{Ghiggini:strongNotStein},
and the results of Li-Mak-Yasui \cite{LiMakYasui:CYcaps} and
Sivek with the second author \cite{SivekVHM:unitCotangent} on exact and
Stein fillings of unit cotangent bundles over higher genus surfaces.
These results all rely in some form on gauge theory, and they seem to
represent fundamentally different phenomena from those that are studied
in this paper.

\subsection{Classification of fillings}
\label{sec:classification}

We will assume in the following that the main definitions 
from \cite{LisiVanhornWendl1}*{\S 1.1} concerning types of symplectic
fillings, symplectic and Stein deformation equivalence, spinal open books, 
bordered Lefschetz fibrations, and the contact
and symplectic structures supported by them are already understood.
If $\Pi : E \to \Sigma$ is a bordered Lefschetz fibration and
$\boldsymbol{\pi}$ is the induced spinal open book at $\p E$, we will
continue to indicate this relationship via the notation
$$
\p \Pi \cong \boldsymbol{\pi}.
$$
Let us recall briefly what this means: breaking up $\p E$ into its horizontal
and vertical boundary $\p_h E \cup \p_v E$, the paper
$\pi\paper : M\paper \to S^1$ is defined from the fibration
$\Pi|_{\p_v E} : \p_v E \to \p \Sigma$ by identifying each connected component
of $\p\Sigma$ with $S^1$ and allowing the fibers of $\pi\paper$ to be
disjoint unions of fibers of~$\Pi$, while the spine $\pi\spine : M\spine \to
\widetilde{\Sigma}$ is defined by factoring $\Pi|_{\p_h E} : \p_h E \to \Sigma$
through a suitable covering map $\widetilde{\Sigma} \to \Sigma$ 
to make the fibers of $\pi\spine$ connected.
We repeat the following slightly technical definitions, since
they will play a substantial role in this paper.

\begin{defn}
A $3$-dimensional 
spinal open book will be called
\defin{partially planar} if its interior contains a page of genus zero.
A compact contact
$3$-manifold $(M,\xi)$, possibly with boundary, will be called a
\defin{partially planar domain} if $\xi$ is supported by a partially
planar spinal open book.  We then refer to any interior connected component of
the paper containing planar pages as a \defin{planar piece}.
\end{defn}

\begin{defn}
\label{defn:multiplicity}
Given a spinal open book $\boldsymbol{\pi}$ with paper
$\pi\paper : M\paper \to S^1$ and spine $\pi\spine : M\spine \to \Sigma$,
we define 
the \defin{multiplicity} of $\pi\paper$ at a boundary component
$\gamma \subset \p\Sigma$ as the degree
$$
m_\gamma \in \NN
$$
of the map
\begin{equation}
\label{eqn:multiplicity}
\gamma \to S^1 : \phi \mapsto \pi\paper(\pi\spine^{-1}(\phi)).
\end{equation}
\end{defn}
Recall that the map \eqref{eqn:multiplicity} is well defined due to the
condition that boundary components of fibers of $\pi\paper$ are always
fibers of~$\pi\spine$.  In general, this map is a finite cover, and 
$m_\gamma$ can also be understood as the number of distinct boundary components
of a page that lie in the torus~$\pi\spine^{-1}(\gamma)$.

\begin{remark}
\label{remark:rationalOB}
In this language, 
an ordinary open book can be defined as any spinal open book such that
the base $\Sigma$ of the spine is a finite disjoint union of disks and
$m_\gamma=1$ for every component $\gamma \subset \p\Sigma$.  Spinal open books
that satisfy the first condition but not the second are examples of
\emph{rational} open books in the sense of \cite{BakerEtnyreVanhorn}.

\end{remark}

\begin{defn}
\label{defn:simple}
A spinal open book $\boldsymbol{\pi}$ on a $3$-manifold 
$M$ will be called \defin{symmetric} if 
\begin{enumerate}[label=(\roman{enumi})]
\item $\p M = \emptyset$;
\item All pages are diffeomorphic;
\item For each of the vertebrae $\Sigma_1,\ldots,\Sigma_r \subset \Sigma$,
there are corresponding numbers $k_1,\ldots,k_r \in \NN$ such that every
page has exactly $k_i$ boundary components in $\pi\spine^{-1}(\p\Sigma_i)$ for
$i=1,\ldots,r$.
\end{enumerate}
We shall say that $\boldsymbol{\pi}$ is \defin{uniform} if, in addition
to the above conditions, there exists a fixed compact
oriented surface $\Sigma_0$ whose boundary components correspond bijectively
with the connected components of $M\paper$ such that for each $i=1,\ldots,r$ 
there exists a $k_i$-fold branched cover 
$$
\Sigma_i \to \Sigma_0
$$
for which the restriction to each connected boundary component $\gamma \subset \p\Sigma_i$
is an $m_\gamma$-fold cover of the component of $\p\Sigma_0$ corresponding
to the component of~$M\paper$ touching $\pi\spine^{-1}(\gamma)$,
where $m_\gamma$ denotes the multiplicity of $\pi\paper$ at~$\gamma$
(see Definition~\ref{defn:multiplicity}).

Finally, $\boldsymbol{\pi}$ is \defin{Lefschetz-amenable} if it is uniform
and all branched covers satisfying the above conditions have no branch
points.
\end{defn}

The symmetry condition played a large role in \cite{LisiVanhornWendl1} via
its presence in the definition of \emph{planar torsion}: essentially,
the non-symmetric spinal open books with a planar page are those that can
be shown to obstruct symplectic filling by an argument using spine
removal surgery and holomorphic spheres.

The significance of the uniformity condition is that every spinal open book
that arises as the boundary of a bordered Lefschetz fibration 
$\Pi : E \to \Sigma$ clearly
satisfies it; in fact, in this case the required branched covers 
$\Sigma_i \to \Sigma_0$ are honest covering maps, defined as mentioned above
by factoring $\Pi|_{\p_h E} : \p_h E \to \Sigma$ so that the fibers of the
spine become connected.  It is not true that spinal open books with this
property must always be Lefschetz-amenable, but there are many interesting
cases (e.g.~the oriented circle bundles in \S\ref{subsec:circleBundles})
where amenability is either obvious or can be checked using the Riemann-Hurwitz
formula, and Theorem~\ref{thm:classification} below then classifies 
all fillings in
terms of Lefschetz fibrations.  In \S\ref{sec:exotic}, we will
also discuss some interesting examples that are not Lefschetz-amenable,
and say what we can about the implications; for a more systematic study
of the non-amenable case, see the separate paper by
Min-Roy-Wang \cite{MinRoyWang:exotic}.

For a given closed contact $3$-manifold $(M,\xi)$, define the sets

\begin{equation*}
\begin{split}
\strong(M,\xi) &= \{ \text{strong fillings of $(M,\xi)$} \} \Big/ \sim, \\
\Liouville(M,\xi) &= \{ \text{exact fillings of $(M,\xi)$} \} \Big/ \sim, \\
\Stein(M,\xi) &= \{ \text{Stein fillings of $(M,\xi)$} \} \Big/ \sim,
\end{split}
\end{equation*}
where the equivalence relation is defined via
strong, Liouville or Stein deformation equivalence respectively.
Since minimality is preserved under symplectic deformation, 
we can also define the subset
$$
\strong\minimal(M,\xi) = \{ [(W,\omega)] \in \strong(M,\xi) \ |\ 
\text{$(W,\omega)$ is minimal} \},
$$
and observe that since every Stein filling is exact and every exact filling
is minimal, there are canonical maps
\begin{equation}
\label{eqn:SteinMinimal}
\Stein(M,\xi) \to \Liouville(M,\xi) \to \strong\minimal(M,\xi).
\end{equation}

Likewise, for a spinal open book $\boldsymbol{\pi}$ we define
$$
\Ll(\boldsymbol{\pi}) = \{ \text{bordered Lefschetz fibrations $\Pi : E
\to \Sigma$ with $\p\Pi \cong \boldsymbol{\pi}$} \} \Big/ \sim,
$$
where $\Pi : E \to \Sigma$ and $\Pi' : E' \to \Sigma'$ are considered 
equivalent if there exist orientation preserving diffeomorphisms 
$\varphi : \Sigma \to \Sigma'$ and $\Phi : E \to E'$, the latter restricting
to diffeomorphisms $\p_h E \to \p_h E'$ and $\p_v E \to \p_v E'$, such that
$\Pi' \circ \Phi = \varphi \circ \Pi$.  We also define the subset
$$
\Ll_A(\boldsymbol{\pi}) = \{ [\Pi] \in \Ll(\boldsymbol{\pi}) \ |\ 
\text{$\Pi$ is allowable} \},
$$
where we recall that $\Pi$ is called \defin{allowable} if
all the irreducible components of its fibers have
nonempty boundary.  Whenever $(M,\xi)$ is supported by a uniform
spinal open book $\boldsymbol{\pi}$, the results of \cite{LisiVanhornWendl1}*{\S 3}
yield canonical maps
\begin{equation}
\label{eqn:canonical}
\begin{split}
\Ll({\boldsymbol{\pi}}) &\to \strong(M,\xi), \\
\Ll_A({\boldsymbol{\pi}}) &\to \Stein(M,\xi).
\end{split}
\end{equation}
In \S\ref{sec:impressivePart},
we will use holomorphic curve technology to prove that the above maps can
sometimes be inverted:

\begin{thm}
\label{thm:classification}
Suppose $(M,\xi)$ is a closed contact $3$-manifold that is strongly
fillable and contains a compact domain $M_0 \subset M$, possibly with
boundary, on which $\xi$ is supported by a partially planar spinal open book 
$\boldsymbol{\pi}$.
Then $M = M_0$ and $\boldsymbol{\pi}$ is uniform.  Moreover, if
$\boldsymbol{\pi}$ is also Lefschetz-amenable, then the canonical 
maps of \eqref{eqn:SteinMinimal} and \eqref{eqn:canonical} are all bijections.
\end{thm}

In addition to a classification result, the above implies many
non-fillability results, since most spinal open books are not uniform:

\begin{cor}
\label{cor:nonuniform}
If $(M,\xi)$ is a closed contact $3$-manifold containing a partially
planar domain that is not uniform, then $(M,\xi)$ is not strongly fillable.
\qed
\end{cor}
\begin{remark}
As shown in \cite{LisiVanhornWendl1}, partially planar domains never admit
non-separating contact embeddings into closed symplectic $4$-manifolds,
thus the manifolds in Corollary~\ref{cor:nonuniform} can never appear at all
as contact-type hypersurfaces in closed symplectic manifolds.
\end{remark}

Even in the uniform case, it may happen that a given spinal
open book cannot be the boundary of a Lefschetz fibration because of 
restrictions imposed on its monodromy.  This phenomenon is familiar in the
case of ordinary open books and has been exploited in
\cites{PlamenevskayaVanHorn,Plamenevskaya:surgeries,Wand:planar,
KalotiLi:Stein,Kaloti:Stein}.  We will not consider the factorization
problem in much detail here, but will examine the simplest nontrivial case
in \S\ref{subsec:circleBundles}, namely when the pages are annuli, so
that their mapping class group has a single free generator.

\begin{remark}
Theorem~\ref{thm:classification} does not give a classification of fillings
for planar spinal open books that are uniform but not Lefschetz-amenable.
Under suitable conditions on the monodromy, one can construct a bordered
Lefschetz fibration filling a spinal open book of this type whenever there is
a choice of surface $\Sigma_0$ such that the vertebrae admit
\emph{unbranched} $k_i$-fold covers $\Sigma_i \to \Sigma_0$, but there may 
be additional fillings not obtained from this construction,
corresponding to additional branched covers.
Our proof of Theorem~\ref{thm:classification} will in fact produce on any such
filling a singular foliation by $J$-holomorphic curves which deforms
smoothly under symplectic deformations, but in
general it will have singularities that cannot be understood purely in terms
of Lefschetz fibrations, including a phenomenon that we refer to as
\emph{exotic fibers}.
See \S\ref{sec:exotic} for more discussion and some examples.
\end{remark}

\subsection{Weak fillings deform to strong fillings}
\label{sec:weak}

Under appropriate cohomological conditions, 
Theorem~\ref{thm:classification} 
can also be extended to a classification of weak fillings.  We recall first
the following definition from \cite{LisiVanhornWendl1}.

\begin{defn}
\label{defn:OmegaSeparating}
Suppose $(M,\xi)$ is a closed contact $3$-manifold and~$\Omega$ is a
closed $2$-form on~$M$.  A partially planar domain $M_0$ embedded 
in $(M,\xi)$ is called \defin{$\Omega$-separating} if it has a planar
piece $M_0^P \subset {\mathring M}_0$ such that $\Omega$ is exact on
every spinal component touching~$M_0^P$.
It is called \defin{fully separating} 
if this is true for all closed $2$-forms~$\Omega$ on~$M$.
\end{defn}

The condition here depends only on the cohomology class
$[\Omega] \in H^2_\dR(M)$ and is vacuous if $\Omega$ is exact.
Recall that every weak
filling $(W,\omega)$ for which $\omega$ is exact near the boundary can
be deformed to a strong filling,
cf.~\cite{Eliashberg:contactProperties}*{Proposition~3.1}.
In the special case of disk-like vertebrae, \emph{any} closed $2$-form
is exact on the spine, thus the following generalizes the  
theorem of Niederkr\"uger and the third author
\cite{NiederkruegerWendl} that weak fillings of planar contact manifolds
can (after blowing down) always be deformed to Stein fillings.

\begin{thm}
\label{thm:weak}
Suppose $(M,\xi)$ is a closed contact $3$-manifold, $\Omega$ is a closed
$2$-form on~$M$ and $(M,\xi)$ contains an $\Omega$-separating partially
planar domain.  Then every weak filling $(W,\omega)$ of $(M,\xi)$ for
which $[\omega|_{TM}] = [\Omega] \in H^2_\dR(M)$ is weakly symplectically
deformation equivalent to a strong filling of $(M,\xi)$.  In particular,
if the domain is fully separating then this is true for all weak fillings.
\end{thm}

One general application of this result concerns \emph{rational open books},
which were defined on contact $3$-manifolds in \cite{BakerEtnyreVanhorn}: 
like an open book, it gives a fibration
$M \setminus B \to S^1$ in the complement of some oriented link $B \subset M$,
but unlike an ordinary open book, the closures of the pages may be multiply 
covered at their boundaries.
We say that a closed contact $3$-manifold $(M,\xi)$ is \defin{rationally
planar} if it is supported by a rational open book with pages of genus zero.
The following extends one of the main results of \cite{NiederkruegerWendl}
from planar to rationally planar contact manifolds.

\begin{cor}
\label{cor:rationallyPlanar}
If $(M,\xi)$ is a rationally planar contact $3$-manifold, then all weak
symplectic fillings of $(M,\xi)$ are symplectically deformation equivalent
to strong fillings.
\end{cor}
\begin{proof}
We note first that using methods from \cite{Vanhorn:thesis},
any rational open book can be 
modified---without changing the
contact structure or the page genus---to one with the property that
every boundary component of the closure of a page covers the respective
binding component once (though there still may be multiple boundary components
covering the same component of the binding).  One can see this by presenting
a neighborhood of any $k$-fold covered binding component as
$\DD \times S^1$ with contact form $d\phi + \rho^2 \, d\theta$ in coordinates
$\rho e^{i\theta} \in \DD \subset \CC$ and $\phi \in S^1 = \RR/\ZZ$,
such that the pages in this region are parametrized by the punctured disks
$$
\DD \setminus \{0\} \hookrightarrow \DD \times S^1 : z \mapsto (z,\arg z^k + \phi)
$$
for different choices of constants $\phi \in S^1$.  (Note that here it is
possible for different choices of $\phi \in S^1$ to produce distinct
subsets of the same page.)  One can then choose any function
$f : \DD \to \CC$ that is $C^\infty$-close to $z \mapsto z^k$ and matches
it precisely for $\rho \ge 1/2$ but has exactly $k$ simple and positive zeroes,
and replace the pages above with
$$
\DD \setminus f^{-1}(0) \hookrightarrow \DD \times S^1 : z \mapsto (z , \arg f(z) + \phi).
$$
Each zero of $f$ now gives rise to a new page boundary component that
covers the corresponding component of $f^{-1}(0) \times S^1 \subset \DD \times S^1$ exactly once, and these modified
pages are also transverse to the Reeb vector field for the contact form
$d\phi + \rho^2 \, d\theta$.

With this modification in place, the resulting rational open book is still planar
but can also be interpreted as a spinal open book 
(see Remark~\ref{remark:rationalOB}).  The result then follows from
Theorem~\ref{thm:weak} since the spine is a union of solid tori, on which
all closed $2$-forms are exact.
\end{proof}

\begin{remark}
It is not known whether there exist rationally planar contact manifolds which
are not planar, though Corollary~\ref{cor:rationallyPlanar} may be interpreted
as providing some evidence against this.
\end{remark}

\subsection{Symplectic deformation implies Stein deformation}
\label{sec:Stein}

Another result closely related to Theorem~\ref{thm:classification} concerns the
question of to what extent the \emph{symplectic} geometry of a Stein
manifold determines its \emph{Stein} geometry.  The following two
questions express this more precisely.

\begin{question}
\label{question:1}
Do there exist two Stein domains that are symplectic deformation equivalent
but not Stein deformation equivalent?
\end{question}
\begin{question}
\label{question:2}
Is there a natural class of Stein domains with the property that
any two domains in the class are
Stein deformation equivalent if and only if they are
symplectically deformation equivalent?
\end{question}

The first question is completely open.  For the second, it is known
in higher dimensions that two \emph{flexible} Stein structures on a given 
manifold will always
be Stein homotopic whenever they can be related by a symplectic deformation;
this follows from an $h$-principle for flexible Weinstein structures,
\cite{CieliebakEliashberg}*{Chapter~14}, and it suffices in this case to
know that their underlying almost complex structures are homotopic.
We will show however that in real dimension four, there exists a larger class of Stein 
structures answering Question~\ref{question:2} than what
might be suggested by known flexibility results (e.g.~for the subcritical case).
In the following statement,
we say that a Stein domain $(W,J)$ is \defin{supported by} a certain
Lefschetz fibration $\Pi : E \to \Sigma$ if $\Pi$ admits a supported
almost Stein structure that is (after smoothing corners) almost Stein
deformation equivalent to $(W,J)$.\footnote{For a brief review of
almost Stein structures, see Definition~\ref{defn:almostStein}.}

\begin{thm}
\label{thm:SteinDeformation}
Suppose $(W,J_0)$ is a Stein domain of real dimension~$4$, supported by 
a bordered Lefschetz fibration $\Pi \colon E \to \Sigma$ with fibers of genus~$0$.
Suppose $J_1$ is another Stein structure on~$W$, and denote by
$\omega_0$ and $\omega_1$ the symplectic structures induced by choices
of plurisubharmonic functions for $J_0$ and $J_1$ respectively.  Then
$J_0$ and $J_1$ are Stein homotopic if and only if $\omega_0$ and $\omega_1$
are homotopic through symplectic structures convex at the boundary.

Moreover, if $\Sigma = \DD^2$, then $J_0$ and $J_1$ are Stein homotopic if
and only if there exist smooth homotopies $\{\omega_\tau\}_{\tau \in [0,1]}$
of symplectic forms on $W$ and contact structures
$\{\xi_\tau\}_{\tau \in [0,1]}$ on $\p W$ such that
$(W,\omega_\tau)$ is a weak filling of $(\p W,\xi_\tau)$ for all
$\tau \in [0,1]$.
\end{thm}

This result would be a corollary of Theorems~\ref{thm:classification} 
and~\ref{thm:weak} if one
could always assume that the spinal open book induced on the boundary of a
bordered Lefschetz fibration is Lefschetz-amenable, but the latter is false
in general (see Example~\ref{ex:bizarre} for a counterexample).  
We will prove the theorem in \S\ref{sec:quasi} by combining the
holomorphic curve arguments behind Theorems~\ref{thm:classification} 
and~\ref{thm:weak} with the
criterion established in \cite{LisiVanhornWendl1}*{\S 2.4} for the canonical
Stein structure supported by a Lefschetz fibration.

\begin{example}
By the main result of \cite{Wendl:fillable},
Theorem~\ref{thm:SteinDeformation} applies to all Stein fillings of
planar contact $3$-manifolds, which includes all subcritical fillings,
but also many critical examples such as the unit disk bundle in~$T^*S^2$.
A further class of non-subcritical examples comes from products 
$\Sigma_0 \times \Sigma_1$ of two Riemann surfaces with boundary
such that at least one of them has genus zero but neither is a disk;
this includes e.g.~the unit disk bundle in $T^*\TT^2$, which (after rounding corners)
is a product of two annuli.
\end{example}

\begin{remark}
\label{remark:probably}
Theorem~\ref{thm:SteinDeformation} probably also holds under the slightly
more general hypothesis that the contact boundary of $(W,J_0)$ is supported
by a planar spinal open book---the latter need not be the boundary of a
Lefschetz fibration since it might not be Lefschetz-amenable.  Proving
the theorem in this generality would require a better geometric
understanding of the so-called \emph{exotic} fibers that are possible
in non-amenable cases (cf.~\S\ref{sec:exotic}).
\end{remark}

\begin{remark}
If one wants to find examples of Stein surfaces that are symplectically but not
Stein deformation equivalent, then Theorem~\ref{thm:SteinDeformation} and Remark~\ref{remark:probably}
suggest searching among Stein surfaces $(W,J)$ whose contact boundaries $(M,\xi)$ do not admit supporting
spinal open books with planar pages.  The main results of this paper and
\cite{LisiVanhornWendl1} provide several mechanisms for recognizing
contact $3$-manifolds with the latter property, e.g.~by \cite{LisiVanhornWendl1}*{Corollary~1.30},
$(M,\xi)$ cannot contain a partially planar domain if it arises as a component
of a strong symplectic filling with disconnected boundary.
Popular examples include the unit disk bundles in $T^*\Sigma$ for
$\Sigma$ any oriented surface with genus at least two; the resulting
unit cotangent bundle is one component of an exact filling with disconnected boundary
that was famously constructed by McDuff \cite{McDuff:boundaries}.
\end{remark}

\subsection{Filling obstructions and contact invariants}
\label{sec:invariantsSketch}

Many special cases of the non-fillability statement in Corollary~\ref{cor:nonuniform}
follow already from the results on planar torsion in \cite{LisiVanhornWendl1},
but they can also be derived from computations of contact
invariants in embedded contact homology 
(cf.~\cites{Hutchings:ICM,Wendl:openbook2})
or symplectic field theory (cf.~\cites{SFT,LatschevWendl}).
The main invariant we have in mind is the \emph{order of algebraic
torsion}, as defined in \cite{LatschevWendl}.  This is a nonnegative
(or possibly infinite) integer extracted from the full symplectic field
theory algebra of a contact manifold; it equals zero if and only if the 
manifold is \emph{algebraically overtwisted} in the sense of
\cite{BourgeoisNiederkrueger:algebraically}, while positive values can be 
interpreted as measuring the manifold's ``degree of tightness''.
The following result, which provides the main motivation behind the 
terminology ``planar $k$-torsion,'' is a generalization 
of \cite{LatschevWendl}*{Theorem~6}.\footnote{See \cite{LisiVanhornWendl1}*{\S 1.3}
for the main definitions concerning planar torsion domains.}

\begin{thm}
\label{thm:algTorsion}
If $(M,\xi)$ has $\Omega$-separating planar $k$-torsion for some $k \ge 0$,
then it also has $\Omega$-twisted algebraic $k$-torsion.
\end{thm}

Since our contact manifolds $(M,\xi)$ in this paper are always
$3$-dimensional, we can also consider the closely related 
filling obstruction furnished by the \emph{ECH contact 
invariant}, i.e.~the distinguished class in the embedded contact homology
of $(M,\xi)$, defined by Hutchings (see e.g.~\cite{Hutchings:ICM}).
The next theorem is a direct generalization of the vanishing results proved
in \cite{Wendl:openbook2}:

\begin{thm}
\label{thm:ECHvanish}
If $(M,\xi)$ has $\Omega$-separating planar $k$-torsion for any $k \ge 0$,
then its ECH contact invariant with twisted coefficients in
$\ZZ[H_2(M) / \ker\Omega]$ vanishes.
\end{thm}

There is also an algebraic counterpart for the
theorem from \cite{LisiVanhornWendl1} that partially planar domains
obstruct semifillings with disconnected boundary: it involves the
so-called \emph{$U$-map} in ECH, which is defined by counting
index~$2$ holomorphic curves through a generic point in the symplectization.
This result generalizes the ECH version of a
planarity obstruction first established by Ozsv\'ath-Stipsicz-Szab\'o 
\cite{OzsvathSzaboStipsicz:planar} in Heegaard Floer homology and
extended to ECH in \cite{Wendl:openbook2}:

\begin{thm}
\label{thm:Umap}
If $(M,\xi)$ contains an $\Omega$-separating partially planar domain,
then for all $k \in \NN$, the contact invariant in
ECH with twisted coefficients in $\ZZ[H_2(M) / \ker\Omega]$ is
in the image of~$U^k$.
\end{thm}

See \S\ref{sec:S1bundleVanishing} for some sample applications of
these theorems, where they are used in particular to prove new vanishing 
results for contact invariants on 
circle bundles.

\subsection{Fillings of circle and torus bundles}
\label{subsec:circleBundles}

In \cite{LisiVanhornWendl1}*{\S 1.4}, we exhibited a large class of
$S^1$-invariant contact structures on circle bundles which are supported
by spinal open books with annular pages.  We now extend the
non-fillability results from that paper to a more comprehensive classification
of fillings.

Assume throughout this section that
$$
\pi : M \to B
$$
is a smooth fiber bundle
whose fibers are diffeomorphic to $S^1$ and whose 
total space is a closed, connected and
oriented $3$-manifold, while the base $B$ is a closed connected surface
that need not necessarily be orientable.  Reducing the structure group of
the bundle to $\Ortho(2)$ then defines the notion of \defin{$S^1$-invariant}
contact structures $\xi$ on~$M$, each of which determines a multicurve
$\Gamma \subset B$ by the condition that fibers over $\Gamma$ are tangent
to~$\xi$.  We say in this case that $\xi$ is \defin{partitioned by}~$\Gamma$,
and it follows that $B \setminus \Gamma$ must be orientable and
$\Gamma$ satisfies a further technical condition (it ``inverts orientations'').
Conversely, for any multicurve $\Gamma \subset B$ satisfying these two
conditions, there is a unique isotopy class of $S^1$-invariant contact
structures partitioned by~$\Gamma$.  We shall denote contact manifolds
of this type always by
$$
(M,\xi_\Gamma).
$$
The existence and uniqueness of $\xi_\Gamma$ is a famous result of Lutz
in the case where $B$ is orientable \cite{Lutz:77}, and in the general
case it was deduced in \cite{LisiVanhornWendl1} from the existence and
uniqueness of contact structures supported by spinal open books.
In particular, $(M,\xi_\Gamma)$ is supported by a spinal open book whose
paper is a tubular neighborhood of $\pi^{-1}(\Gamma)$, with annular pages,
while the vertebrae correspond to the connected components of
$B \setminus \Gamma$.

\begin{remark} \label{rem:seifert_fibered}
    In the case when $B$ is non-orientable, the total space is 
    nevertheless oriented, and there is still a well-defined Euler number.
    As it turns out, over a given base $B$, 
    these bundles are characterized by this Euler number.
    Indeed, this can be seen by viewing such fiber bundles $\pi \colon M \to B$ 
    as Seifert fibered spaces with no exceptional fibers. 
    For more details, see the discussion in \cite{Scott:geometries_3_manifolds}*{page 434}.
\end{remark}


\subsubsection{Classification of fillings}
\label{sec:S1bundleClassification}

Whenever $\pi : M \to B$
corresponds to a spinal open book that is Lefschetz-amenable,
Theorem~\ref{thm:classification} classifies the strong fillings of 
$(M,\xi_\Gamma)$ as bordered Lefschetz fibrations with annulus fibers.
The amenability condition is trivial to verify when $B$ is orientable.

\begin{thm}
\label{thm:circleBundles}
Suppose $\xi_\Gamma$ is an $S^1$-invariant contact structure on a circle
bundle $\pi : M \to B$, partitioned by a nonempty multicurve $\Gamma$,
where $B$ is orientable.  Then $(M,\xi_\Gamma)$ is strongly fillable if
and only if $B \setminus \Gamma$ has two connected components, both of them diffeomorphic
to a single surface~$\Sigma$, and the Euler number $e(\pi)$ of the bundle 
satisfies
\[ e(\pi) \ge 0. \]
Moreover, the Stein, Liouville and minimal strong
fillings of $(M,\xi_\Gamma)$ are all unique up to deformation
equivalence and can be characterized via supporting allowable Lefschetz
fibrations over $\Sigma$ with fiber $[-1,1] \times S^1$, which restrict to
trivial fibrations on the horizontal boundary and have 
$e(\pi)$ singular fibers.
\end{thm}
\begin{proof}
When $B$ is orientable, $\Gamma$ necessarily divides $B$ into two (each
possibly disconnected) components $B_+$ and~$B_-$, thus determining similar
labels $M\spine^\pm$ for corresponding components of the spine~$M\spine$.
Every page of $\pi\paper : M\paper \to S^1$ thus has one boundary component
touching $M\spine^+$ and the other touching $M\spine^-$, so symmetry of
$\boldsymbol{\pi}$ implies that $M\spine$ must have exactly two connected 
components, each touching one boundary component of every page.
This implies that each boundary component of the spine has multiplicity~$1$
in the sense of Definition~\ref{defn:multiplicity}.
If $\boldsymbol{\pi}$ is also uniform, then the vertebrae of the two spinal
components must also be diffeomorphic, and the Lefschetz-amenability
condition is trivially satisfied.
It follows that $B \setminus \Gamma$ has exactly two components and they are 
diffeomorphic to a fixed surface~$\Sigma$, and minimal fillings of $(M,\xi_\Gamma)$
correspond to Lefschetz fibrations over $\Sigma$ with annulus fibers.

Observe now that any two allowable Lefschetz fibrations over $\Sigma$ with annulus
fibers and with the same number of critical points are symplectic
deformation equivalent. 
Let $\Pi \colon E \to \Sigma$ be such an allowable Lefschetz fibration. 
Fix a
basepoint $z_0 \in \Sigma$ and choose an orientation-preserving identification 
of $\Pi^{-1}(z_0)$ with $[-1,1] \times S^1$. 
Trivialize the two components of $\p E \to \Sigma$ consistently with this.
After choosing a collection of paths in $\Sigma$ that connect $z_0$ with 
the points in $\p\Sigma$ that correspond to $\pi\paper^{-1}(1)$, we obtain a
well-defined monodromy map $[-1,1] \times S^1 \to [-1,1] \times S^1$ for each
boundary component of $\p \Sigma$. 
Notice that by changing the trivialization of $\p E \to \Sigma$, we may change
these monodromies, but their composition remains invariant, and will be
isotopic to a $k$-fold Dehn twist where $k$ is the number of singular fibers.
In particular, by a suitable choice of trivialization of $\p E \to
\Sigma$, we arrange for the monodromy about each boundary component of
$\p \Sigma$ to be trivial, except for one, where we have a $k$-fold Dehn
twist. A computation verifies that $\p E$ is then a circle bundle over the
doubled surface $\Sigma \cup_{\p \Sigma} (-\Sigma)$ with Euler class given by $k$.

\end{proof}

\begin{remark}
\label{remark:prequant}
It is possible for the partitioning multicurve 
$\Gamma \subset B$ of an $S^1$-invariant contact structure to be empty
when $B$ is orientable:
this means that $(M,\xi_\Gamma)$ is a prequantization bundle with its canonical
contact structure.  In this case Theorem~\ref{thm:circleBundles} does not 
apply, and in fact, the problem of classifying strong fillings of 
prequantization bundles is not generally tractable: e.g.~whenever $B$
has genus $g \ge 2$, there exists a prequantization bundle $(M,\xi_\Gamma)$
over $B$ admitting exact semifillings with disconnected boundary
(see \cite{McDuff:boundaries}), from which one can construct an unmanageable
multitude of topologically unrelated fillings of $(M,\xi_\Gamma)$ by attaching 
concave fillings from \cite{EtnyreHonda:cobordisms} to the other boundary 
component.
\end{remark}

\begin{thm}
\label{thm:circleBundles2}
Suppose $\xi_\Gamma$ is an $S^1$-invariant contact structure on a circle
bundle $\pi \colon M \to B$, partitioned by a nonempty multicurve~$\Gamma$,
where $B$ is not orientable and $\Gamma$ has $k \ge 0$ connected components
that are not co-orientable.  

If $(M,\xi_\Gamma)$ is strongly fillable, then
$B \setminus \Gamma$ is connected and
$$
k \le 2(g+1),
$$
where $g$ is the genus of $B \setminus \Gamma$.  

Assuming additionally that $B \setminus \Gamma$
is connected and $k = 2(g+1)$, 
$(M,\xi_\Gamma)$ is strongly fillable
if and only if
its Euler number (see Remark \ref{rem:seifert_fibered}) is non-negative.
In that case, its Stein,
Liouville and minimal strong fillings are unique up to deformation equivalence.
\end{thm}
\begin{proof}
Assume $B$ is non-orientable and $\Gamma$ consists of $k$ components with
nontrivial normal bundle and $\ell$ components with trivial normal bundle.
If $\boldsymbol{\pi}$ is symmetric, then the spine can have at most two
connected components, and it has exactly two only if every page has its two
boundary components touching different spinal components, which means
$k=0$ and the $\ell$ components of $\Gamma$ divide $B$ into two connected
components $B_+$ and $B_-$.  But since $\Gamma$ inverts orientations,
this would imply that $B$ is
orientable and thus contradicts our assumptions.  We conclude that
$B \setminus \Gamma$ is connected and has the homotopy type of a
compact oriented surface $\Sigma$ with some genus $g \ge 0$ and $k + 2\ell$ 
boundary components.  The multiplicity of $\pi\paper$ is $2$ at the
$k$ boundary components of $M\spine$ corresponding to curves that are not
co-orientable, and $1$ at its other $2\ell$ boundary components.
Uniformity of $\boldsymbol{\pi}$ then means that there exists a double
branched cover of $\Sigma$ over some surface $\Sigma_0$ of arbitrary
genus $h \ge 0$ with $k+\ell$ boundary components.  
By the Riemann-Hurwitz formula, the algebraic count of branch points is
$$
-\chi(\Sigma) + 2\chi(\Sigma_0) = - (2 - k - 2\ell - 2g) + 2(2 - k - \ell - 2h) =
2 - k + 2 g - 4 h \ge 0,
$$
hence the required branched cover is possible for any genus $h \ge 0$ satisfying
$$
4 h \le 2(g+1) - k.
$$
If equality is achieved, then the resulting branched cover has no branch
points.
In particular, this will always be the case if $k = 2(g+1)$, so the case of $k =
2(g+1)$ is Lefschetz-amenable. 

By Theorem \ref{thm:classification}, it follows that, if $k=2(g+1)$,  any
filling of 
$(M,\xi_\Gamma)$ is (up to deformation) realized as a Lefschetz fibration $\Pi \colon E
    \to \Sigma_0$ with annular fibers where $\Sigma_0$ has genus zero and $k+\ell$ boundary components.
    Furthermore, $\p \Pi$ gives the spinal open book decompositions described
    by $B$, $\Gamma$.

Notice now that $k = 2(g+1)$ is even. We may thus decompose $\Sigma_0$ into a
collection of pairs of pants, of annuli and of disks with the property that
each subsurface has an even number of boundary components among the $k$ 
boundary components of $\Sigma_0$ that correspond to the non-co-orientable
curves in $\Sigma = B \setminus \Gamma$, and all Lefschetz critical values are
contained in the disks. From this, the restrictions of the fibration to the pairs
of pants and annuli are smooth symplectic fibrations with annulus fibers.
Furthermore, if the base is the annulus, they will either be a trivial fibration
or the fattened mapping torus of the ``flip'' (map of the annulus by $(r,
\theta) \mapsto (-r, -\theta)$). 
If the base is a pair-of-pants, the fibration will be one of these two models
with a fiber deleted. 

Now, choose a framing of the spinal open book decomposition $\boldsymbol{\pi}$,
i.e.~a trivialization of the circle bundle $\pi\spine \colon M\spine \to
\Sigma$.
This then allows us to define the monodromy of each component of the paper. Notice
that the ``flip'' map and a Dehn twist commute (up to homotopy). From this, we
observe that a change in framing has no effect on the composition of all the
monodromies. A computation now shows this composition must have the number of
Dehn twists given by the Euler number of $\pi \colon M \to B$. 

Extending the framing of the spinal open book to the Lefschetz fibration $\Pi$, 
we obtain that the net monodromy around the vertical boundary is some number of
Dehn twists, given by precisely the number of critical fibers. 

\end{proof}

\subsubsection{Vanishing results for contact invariants}
\label{sec:S1bundleVanishing}

In \cite{LisiVanhornWendl1} we gave a characterization of which partioned 
$S^1$-invariant contact circle bundles have planar $1$-torsion.
Combining that result with Theorems~\ref{thm:algTorsion}
and~\ref{thm:ECHvanish} gives the following statement, generalizing a result
for trivial circle bundles that was proved in \cite{LatschevWendl}:

\begin{cor}
\label{cor:S1planarTorsion}
Suppose $\xi_\Gamma$ is an $S^1$-invariant contact structure on a circle
bundle $\pi : M \to B$, partitioned by a nonempty multicurve $\Gamma$, and
that either of the following holds:
\begin{enumerate}[label=(\roman{enumi})]
\item $B \setminus \Gamma$ has at least three connected components;
\item $B \setminus \Gamma$ is disconnected and $B$ is non-orientable.
\end{enumerate}
Then $(M,\xi_\Gamma)$ has (untwisted) algebraic $1$-torsion and vanishing
(untwisted) ECH contact invariant.
\qed
\end{cor}

When the bundle is trivial, we can use some input from Seiberg-Witten
theory to obtain a stronger result for the ECH contact invariant:

\begin{thm}
\label{thm:S1productVanishing}
Suppose $\pi : M \to B$ is a trivial circle bundle and $\xi_\Gamma$ is an
$S^1$-invariant contact structure partitioned by a multicurve
$\Gamma \subset B$.  Then $(M,\xi_\Gamma)$ has nonzero (untwisted)
ECH contact invariant if and only if $\Gamma$ divides $B$ into exactly
two connected components that are diffeomorphic to each other.
\end{thm}
\begin{proof}
When $B = \Sigma_+ \cup_\Gamma \Sigma_-$ for a connected surface
$\Sigma_+ \cong \Sigma_- \cong \Sigma$, the
ECH contact invariant of $(B \times S^1,\xi_\Gamma)$ is nonzero
because it has a strong filling, namely the trivial annulus fibration
over~$\Sigma$.  Excluding the cases covered by Corollary~\ref{cor:S1planarTorsion},
it then remains to prove that the ECH contact invariant vanishes whenever
$B \setminus \Gamma$ has two connected components
$\Sigma_+$ and $\Sigma_-$ with differing genus.

This follows from \cite{Wendl:openbook2} if either component has genus zero,
but if both have positive genus, then we must instead appeal to
Seiberg-Witten theory.  Denote the contact
invariant by $[\boldsymbol{\emptyset}]$; it is an element of 
$\ECH_*(B \times S^1,\xi_\Gamma, 0)$, the embedded contact homology
of $(B \times S^1, \xi_\Gamma)$ generated by
orbit sets with total homology class $0 \in H_1(B \times S^1)$.
By Theorem~\ref{thm:Umap}, there exists for every $k \in \NN$ an element
$\boldsymbol{\gamma}_k \in \ECH_*(B \times S^1,\xi_\Gamma, 0)$ such that $U^k \boldsymbol{\gamma}_k =
[\boldsymbol{\emptyset}]$.  Now observe that if $\Sigma_+ \ncong \Sigma_-$,
then $c_1(\xi_\Gamma) \in H^2(B \times S^1)$ is not torsion;
indeed, 
$$
c_1(\xi_\Gamma) = \left( \chi(\Sigma_+) - \chi(\Sigma_-) \right) 
\PD\left[ \{*\} \times S^1 \right].
$$ 
By the work of Taubes \cite{Taubes:ECH=SWF1}, 
$\ECH_*(B \times S^1,\xi_\Gamma , 0)$
is isomorphic to a certain version of the monopole Floer homology of 
Kronheimer and Mrowka \cite{KronheimerMrowka:SWF} for the $\Spinc$-structure 
determined by the homotopy class of~$\xi_\Gamma$.
The first Chern class of this $\Spinc$-structure
is precisely $c_1(\xi_\Gamma)$ and is thus not torsion, so by results
of Kronheimer and Mrowka \cite{KronheimerMrowka:SWF}, the monopole Floer 
homology is finitely generated.  
Observe now that if $[\boldsymbol{\emptyset}] \ne 0$, then
$\ECH_*(B \times S^1,\xi_\Gamma , 0)$ cannot be finitely generated, as the
generators $\boldsymbol{\gamma}_1,\boldsymbol{\gamma}_2,\ldots$ will be 
linearly independent, so this is a contradiction.
\end{proof}

\begin{remark}
Since the proof of Theorem~\ref{thm:S1productVanishing} relies on gauge
theory in addition to holomorphic curves,
we do not know whether $(B \times S^1,\xi_\Gamma)$ has a finite order
of algebraic torsion when $\Gamma$ divides $B$ into two connected components
with differing positive genus, and
there is no apparent reason to believe that it should.  It would interesting
to resolve this question, as it is not known thus far whether the filling 
obstructions furnished by SFT and the ECH contact invariant in dimension
three are independent.
\end{remark}

\subsubsection{Parabolic torus bundles}
\label{sec:parabolic}

A specific subclass of the contact circle bundles covered by the results above
can also be described as torus bundles with parabolic monodromy.  
All such bundles can be presented in the form
$$
T_\pm(k) := (\RR \times \TT^2) \Big/ (\rho,z) \sim (\rho+1, \pm A_k z)
$$
for some $k \in \ZZ$, where $A_k = \begin{pmatrix} 1 & 0 \\ k & 1 \end{pmatrix}$.
We will denote the coordinates on $\TT^2 = S^1 \times S^1$ by $(\phi,\theta)$.

Given an integer $m \ge 0$, we define a \emph{rotational} contact structure
$\zeta_m$ whose lift to $\RR \times \TT^2$ can be written as
$$
\zeta_m = \ker\left[ f(\rho)\, d\theta + g(\rho)\, d\phi \right]
$$
for some path $(f,g) : \RR \to \RR^2 \setminus \{0\}$ that rotates about the
origin by an angle of greater than $2\pi m$ but at most $2\pi (m+1)$
as $\rho$ varies along a closed unit interval in~$\RR$.  Also, for
$m \in \NN$ we define
$$
\eta_m = \ker\big( |k+1| \cos(2\pi m \phi) \, d\rho + \sin(2\pi m \phi)\, d\theta
- k \rho \sin(2\pi m \phi) \, d\phi \big).
$$
By results of Giroux \cites{Giroux:infiniteTendues,Giroux:bifurcations},
every universally tight contact structure on $T_\pm(k)$ is diffeomorphic to
at least one of these models. 

Defining a circle bundle $\pi : T_+(k) \to \TT^2 = (\RR \times S^1) / \ZZ : 
[(\rho,\phi,\theta)] \mapsto [(\rho,\phi)]$, all of the contact structures 
$\zeta_m$ and $\eta_m$ are $S^1$-invariant and partitioned by multicurves
$\Gamma \subset \TT^2$, where:
\begin{itemize}
\item For $\zeta_m$ with $k \ge 0$, $\Gamma$ consists of $2(m+1)$ curves of
the form $\{\rho = \text{const}\}$;
\item For $\zeta_m$ with $k < 0$, $\Gamma$ consists of $2m$ curves of the
form $\{\rho = \text{const}\}$;
\item For $\eta_m$, $\Gamma$ consists of $2m$ curves of the form
$\{\phi = \text{const}\}$.
\end{itemize}
The Euler number of the bundle $\pi \colon T_+(k) \to \TT^2$ is $k$.

Note that for each $m \in \NN$, \cite{Giroux:infiniteTendues}*{Th\'eor\`eme~6} proves
that $(T_+(k),\eta_m)$ and $(T_+(k),\zeta_{m-1})$ are contactomorphic when 
$k \ge 0$, while $(T_+(k),\eta_m)$ and $(T_+(k),\zeta_m)$ are contactomorphic 
when $k < 0$.  Thus the following corollary of Theorem~\ref{thm:circleBundles}
covers all universally tight contact structures on $T_+(k)$ with the exception 
of $\zeta_0$ for $k < 0$:

\begin{cor}
Given $k \in \ZZ$ and $m \in \NN$, $(T_+(k),\eta_m)$ is strongly fillable
if and only if $m=1$ and 
$k \ge 0$, 
and its
strong fillings are all Lefschetz fibrations over the annulus
with annular fibers and monodromy maps that fix the boundary.
\qed
\end{cor}

Similarly, both families of contact structures on $T_-(k)$ are $S^1$-invariant for the
non-orientable circle bundle $\pi : T_-(k) \to \KK^2 : [(\rho,\phi,\theta)]
\mapsto [(\rho,\phi)]$ over the Klein bottle
$$
\KK^2 = (\RR \times S^1) \Big/ (\rho,\phi) \sim (\rho+1,-\phi).
$$
The multicurves $\Gamma \subset \KK^2$ can now be described as follows:
\begin{itemize}
\item For $\zeta_m$, $\Gamma$ consists of $2m+1$ curves of the form
$\{\rho = \text{const}\}$ with trivial normal bundle;
\item For $\eta_m$, $\Gamma$ includes the two curves $\{ \phi=0 \}$ and
$\{ \phi=1/2 \}$ with nontrivial normal bundles and $m-1$ additional curves
of the form $\{ \rho=\text{const} \}$ with trivial normal bundles.
\end{itemize}
According to \cite{Giroux:infiniteTendues}*{Th\'eor\`eme~6}, all contact structures
in this list on each individual manifold $T_-(k)$ are pairwise non-diffeomorphic.
For $m \ge 1$, $(T_-(k),\zeta_m)$ has positive Giroux torsion and is thus
known to be not fillable.  For $(T_-(k),\zeta_0)$, $\KK^2 \setminus \Gamma$ is
homotopy equivalent to an annulus and the condition $k \le 2(g+1)$ in
Theorem~\ref{thm:circleBundles2} is
satisfied but with strict inequality, so the Lefschetz-amenability condition
fails and we cannot classify fillings (but see \S\ref{sec:exotic} for
more on this example).  It is also not hard to check that $(T_-(k),\eta_m)$
has positive Giroux torsion for every $m \ge 3$. 
We do not know if it has Giroux torsion for $m=2$.
Nevertheless, Corollary~\ref{cor:S1planarTorsion} implies
that $(T_-(k),\eta_2)$ does have planar $1$-torsion, and is thus
non-fillable. 
Finally, for
$\eta_1$ we can apply Theorem~\ref{thm:circleBundles2} to deduce
uniqueness of fillings. Notice that the Euler number of $\pi \colon T_-(k) \to
K^2$ is $-k$. This yields:
\begin{cor}
Given $k \in \ZZ$ and $m \in \NN$, $(T_-(k),\eta_m)$ is strongly fillable
if and only if $m=1$ and $k \le 0$, and its
strong fillings are all Lefschetz fibrations over the annulus with
annular fibers and monodromy maps that interchange boundary components.
\qed
\end{cor}

\begin{example}
\label{ex:eta1}
The unique Stein filling of $(T_-(0),\eta_1)$ is presentable as the smooth
annulus fibration over the annulus $[-1,1] \times S^1$ such that the 
monodromy around $\{*\} \times S^1$ is $[-1,1] \times S^1 \to
[-1,1] \times S^1 : (s,t) \mapsto (-s,-t)$ (i.e.~the ``flip'' map
    appearing in the proof of Theorem~\ref{thm:circleBundles2}).
\end{example}

\subsection{The non-amenable case and exotic fibers}
\label{sec:exotic}

The most important part of Theorem~\ref{thm:classification} does not hold 
for spinal open books that are not Lefschetz-amenable, but our arguments
will still provide something that can be used to achieve a
classification of fillings in the general case.  
The following is a summary 
of some more technical results proved in \S\ref{sec:impressivePart};
for this discussion we permit ourselves the luxury of a slightly imprecise 
statement since we do not intend to prove anything with it.

\begin{remark}
We elected not to attempt any full classification of fillings for non-amenable
cases in this paper, but since its original appearance in preprint form,
such classification results have in fact been achieved 
in a followup paper by Min, Roy and Wang \cite{MinRoyWang:exotic},
based in part on the results stated below.
\end{remark}

\begin{prop}
\label{prop:generalFoliation}
Suppose $(W,\omega)$ is a weak symplectic filling of a contact $3$-manifold
$(M,\xi)$ supported by a partially planar spinal open book $\boldsymbol{\pi}$ 
such that $\omega$ is exact on the spine~$M\spine$.
Then $(W,\omega)$ admits a symplectic completion $\widehat{W}$ with a
compatible almost complex structure $J$ and a smooth surjective map
$$
\Pi : \widehat{W} \to \mM,
$$
where $\mM$ is an oriented surface with cylindrical ends that are in
bijective correspondence to the connected components of~$M\paper$,
and every fiber
$\Pi^{-1}(*)$ is a (possibly nodal) $J$-holomorphic curve with cylindrical
ends asymptotic to closed Reeb orbits in~$(M,\xi)$.  More precisely,
$\mM$ admits a partition
$$
\mM = \mM_\reg \cup \mM_\sing \cup \mM_\exot,
$$
where $\mM_\sing$ and $\mM_\exot$ are each finite sets, and
\begin{itemize}
\item Fibers in $\Pi^{-1}(\mM_\reg)$ are embedded $J$-holomorphic curves
asymptotic to simply covered Reeb orbits;
\item Fibers in $\Pi^{-1}(\mM_\sing)$ are nodal $J$-holomorphic curves
asymptotic to simply covered Reeb orbits,
each formed as the union of two embedded curves that intersect each other
exactly once, transversely;
\item Fibers in $\Pi^{-1}(\mM_\exot)$ are embedded $J$-holomorphic curves
with one end asymptotic to a doubly covered Reeb orbit, and all other ends
asymptotic to simply covered orbits.
\end{itemize}
For each vertebra $\Sigma_i$, there is also a properly embedded 
$J$-holomorphic curve $S_i \subset \widehat{W}$ such that
$$
\Pi|_{S_i} : S_i \to \mM
$$
is a proper branched cover with simple branch points and is
$m_\gamma$-to-$1$ on the cylindrical end corresponding to each boundary 
component of $\gamma \subset \p\Sigma_i$, where $m_\gamma \in \NN$ is the
corresponding multiplicity (see Definition~\ref{defn:multiplicity}).
Moreover, $\Pi|_{S_i}$ is an honest covering map (i.e.~without branch points)
if and only if $\mM_\exot = \emptyset$.
Finally, all of this data deforms smoothly under generic deformations of
$J$ compatible with deformations of the symplectic structure.
\end{prop}

The distinguishing feature of the Lefschetz-amenable case is that the
set $\mM_\exot$ is guaranteed to be empty, in which case we will show in \S\ref{sec:Lefschetz}
that $\Pi : \widehat{W} \to \mM$ gives rise to a Lefschetz fibration
filling~$\boldsymbol{\pi}$, with singular fibers corresponding to the
finite set~$\mM_\sing$.  When this condition fails and 
$\Pi|_{S_i} : S_i \to \mM$ has branch points, the proposition yields a
more general type of decomposition of the filling, including the so-called
\defin{exotic fibers} $\Pi^{-1}(u)$ for $u \in \mM_\exot$.  These are
singular in the sense that they have different topology from the nearby 
regular fibers, but their singularities occur ``at infinity'' and resemble 
the multiple fibers of a Seifert fibration on a $3$-manifold.  
A more precise topological description of exotic fibers is
given in \cite{MinRoyWang:exotic}, where it is used to prove
classification results for fillings without the Lefschetz-amenability
assumption.  

We now give three examples where one can see that exotic
fibers must appear.

\begin{example}
\label{ex:zeta0}
The parabolic torus bundles $(T_-(k),\zeta_0)$ discussed in
\S\ref{sec:parabolic} can be presented as $S^1$-invariant contact structures
on circle bundles over the Klein bottle~$\KK^2$, partitioned along a single
co-orientable curve $\Gamma \subset \KK^2$ such that $\KK^2 \setminus \Gamma$
is a cylinder.  It follows that $(T_-(k),\zeta_0)$ is supported by a
spinal open book $\boldsymbol{\pi}$ with one spine component fibering 
over the annulus, and one family of annular pages whose two boundary 
components meet the spine at separate boundary components, each with 
multiplicity~$1$.  The uniformity condition is satisfied because there 
exists a double branched cover of $[-1,1] \times S^1$ over the disk whose 
restriction to each boundary component has degree~$1$, but since every such 
branched cover has (algebraically) two branch points, $\boldsymbol{\pi}$ is not
Lefschetz-amenable.  Proposition~\ref{prop:generalFoliation} now endows the
completion $\widehat{W}$ of any filling of $(T_-(k),\zeta_0)$ with a
$J$-holomorphic foliation that includes exotic fibers.
\end{example}

\begin{remark}
\label{remark:resolve}
Note that while fillings of $(T_-(k),\zeta_0)$ cannot be presented as
Lefschetz fibrations filling~$\boldsymbol{\pi}$, they do sometimes exist:
e.g.~$(T_-(0),\zeta_0)$ can be presented as a quotient of the standard
contact~$\TT^3$ by a free contact $\ZZ_2$-action that extends over the filling
$T^*\TT^2$ of $\TT^3$ as a symplectic $\ZZ_2$-action with four fixed points
on the zero-section.  The resulting symplectic orbifold has four singular 
points with neighborhoods bounded by the standard contact~$\RR P^3$, so the
singularities can be resolved by replacing these neighborhoods with 
neighborhoods of the zero-section in $T^*S^2$.  If we choose a 
$\ZZ_2$-invariant plurisubharmonic function on $T^*\TT^2$ with local minima
at the four fixed points, then this desingularization results in a Stein
filling $W$ of $(T_-(0),\zeta_0)$.  Note that $H_2(W) \ne 0$, whereas
the unique Stein filling of $(T_-(0),\eta_1)$ that we saw in Example~\ref{ex:eta1}
has trivial second homology, so this furnishes a new proof of
Giroux's theorem \cite{Giroux:infiniteTendues}
that $\zeta_0$ and $\eta_1$ are non-isomorphic contact structures on~$T_-(0)$.
\end{remark}

\begin{example}
\label{ex:S1S2/Z2}
The standard contact structure $\xi\std$ on $S^1 \times S^2$ can be written
in the form $\ker \left[ f(\theta)\, dt + g(\theta)\, d\phi \right]$ where $t \in S^1 =
\RR / \ZZ$ is the standard coodinate, $(\theta,\phi)$ are spherical polar
coordinates on~$S^2$, and $(f,g) : [0,\pi] \to \RR^2$ traces a path that
winds counterclockwise from the positive to the negative $x$-axis.
Choosing $f$ and $g$ to be odd and even functions respectively, we can define
the quotient
$$
(M,\xi) = (S^1 \times S^2,\xi\std) \Big/ (t,\theta,\phi) \sim (-t,\pi-\theta,\phi+\pi),
$$
which is a non-orientable circle bundle over~$\RR P^2$ with orientable total space.
The open book $(S^1 \times S^2) \setminus \{\theta = 0,\pi\} \to S^1 :
(t,\theta,\phi) \mapsto \phi$ then projects to a \emph{rational} open book on $M$
supporting~$\xi$, with one binding component and annular pages such that
the monodromy is an involution exchanging boundary components.  This can also be
interpreted as a spinal open book~$\boldsymbol{\pi}$, where the spine is a 
single solid torus and the paper is a single $S^1$-family of annuli touching 
it with multiplicity~$2$; in fact, this is the same construction that arises
naturally if we view $(M,\xi)$ as a circle bundle.
Since the only vertebra is a disk, uniformity 
demands a branched double cover of $\DD^2$ over itself,
and such a cover will always have one branch point, so $\boldsymbol{\pi}$
is not Lefschetz-amenable.  Any completed filling of $(M,\xi)$ will then carry a
foliation whose generic leaves are $J$-holomorphic cylinders, but that also 
includes exotic fibers in the form of
$J$-holomorphic planes asymptotic to a doubly covered Reeb orbit.
\end{example}

\begin{example}
\label{ex:bizarre}
We now exhibit a planar spinal open book that is not Lefschetz-amenable for
which some but not all fillings can be described as Lefschetz fibrations.

Let $\Sigma_g$ denote the compact connected and oriented surface with
genus~$g$, and denote by $\Sigma_{g,m}$ the compact surface with boundary
obtained by punching $m$ holes in~$\Sigma_g$.  The surface $\Sigma_{2,2}$
admits two double branched covers
$$
\Sigma_{2,2} \stackrel{\varphi_1}{\longrightarrow} \Sigma_{1,2}, \qquad
\Sigma_{2,2} \stackrel{\varphi_0}{\longrightarrow} \Sigma_{0,2},
$$
where both are $2$-to-$1$ maps on each boundary component, and the
Riemann-Hurwitz formula implies that $\varphi_1$ is unbranched, while
$\varphi_0$ has four simple branch points.  The resulting deck transformations
define a pair of orientation-preserving involutions
$$
\psi_1, \psi_0 : \Sigma_{2,2} \to \Sigma_{2,2},
$$
which we can assume are symplectic for suitable choices of area forms
on~$\Sigma_{2,2}$.  Now consider a Weinstein domain defined via the trivial 
annulus fibration $\widetilde{E} = \Sigma_{2,2} \times \Sigma_{0,2}$;
using the natural correspondence between annular
spinal open books and circle bundles, we can view the contact boundary
$(\widetilde{M},\widetilde{\xi})$ of $\widetilde{E}$ 
as a trivial circle bundle over the symmetric
double $\Sigma_5$ formed by gluing together two copies of $\Sigma_{2,2}$
along an orientation-reversing map of their boundaries, and $\widetilde{\xi}$ is
an $S^1$-invariant contact structure partitioned by $\p\Sigma_{2,2} \subset
\Sigma_5$.  The contact manifold we're actually interested in is a
$\ZZ_2$-quotient of this: define the Weinstein domain
$$
E = (\Sigma_{2,2} \times \Sigma_{0,2}) \Big/ (z,w) \sim (\psi_1(z),\sigma(w)),
$$
where $\sigma$ is the involution $(s,t) \mapsto (-s,-t)$ on
$\Sigma_{0,2} = [-1,1] \times S^1$.  This is obviously a symplectic manifold
(for suitable choices of area forms on $\Sigma_{2,2}$ and $\Sigma_{0,2}$) since
the involution $\psi_1 \times \sigma$ is symplectic and without fixed points, 
and one can see
its Weinstein structure in terms of the natural annulus fibration over
$\Sigma_{2,2} / \ZZ_2 = \Sigma_{1,2}$ that it inherits from the trivial
annulus fibration on~$\widetilde{E}$.  The induced spinal open book 
$\boldsymbol{\pi}$ on the
boundary $(M,\xi)$ of $E$ has two paper components with monodromy exchanging
the boundary components of the annulus, and these are attached to separate
boundary components of a single spine component of the form 
$S^1 \times \Sigma_{2,2}$.  Viewing $(M,\xi)$ as an $S^1$-invariant circle
bundle, it fibers over the union of
$\Sigma_{2,2}$ with two M\"obius bands, i.e.~$\Sigma_2 \# 2 \RR P^2$,
with $\xi$ partitioned by a multicurve $\Gamma \subset \Sigma_2 \# 2\RR P^2$
with two components, both not co-orientable, and
$(\Sigma_2 \# 2 \RR P^2) \setminus \Gamma$ is thus a genus $2$ surface with
two cylindrical ends.  As a consequence,
the condition $k \le 2(g+1)$ in Theorem~\ref{thm:circleBundles2} is satisfied, 
but with strict inequality, so
$\boldsymbol{\pi}$ is not Lefschetz-amenable.  

In the context of
Proposition~\ref{prop:generalFoliation}, this means that there are multiple
possibilities for an unknown filling $W$ of $(M,\xi)$: it may indeed admit
a Lefschetz fibration over $\Sigma_{1,2}$ since there exist unbranched
double covers $\Sigma_{2,2} \to \Sigma_{1,2}$, and the filling $E$ described
above is an example of this.  But the moduli space $\mM$ in the proposition
could also have the topology of $\Sigma_{0,2}$, with the branch points of
$\varphi_0 : \Sigma_{2,2} \to \Sigma_{0,2}$ giving rise to exotic fibers.
To see that this also \emph{must} sometimes happen, notice that we can define an
alternative filling of $(M,\xi)$ by starting from the symplectic orbifold
$$
\widehat{E}' := (\Sigma_{2,2} \times \Sigma_{0,2}) \Big/ (z,w) \sim
(\psi_0(z),\sigma(w)),
$$
as the spinal open book on $\p\widetilde{E}$ induced by the trivial
fibration also descends to $\boldsymbol{\pi}$ on 
$\p\widehat{E}' = \p \widetilde{E} / \ZZ_2$.  
The singularities of $\widehat{E}'$ at fixed points of $\psi_0 \times \sigma$
(two for each branch point of~$\varphi_0$) can be resolved by replacing 
neighborhoods with copies of $T^*S^2$ 
(cf.~Remark~\ref{remark:resolve}).  Choosing a $\ZZ_2$-invariant 
plurisubharmonic function on $\Sigma_{2,2} \times \Sigma_{0,2}$ with local
minima at the fixed points, one produces in this way a new Stein filling
$E'$ of $(M,\xi)$, in which the eight orbifold singularities of $\widehat{E}'$
have been replaced by Lagrangian spheres.  

We now notice that the contact manifolds $\p E$ and $\p E'$ are both
circle bundles over the same non-orientable base, with invariant contact
structures, partitioned by the same multicurve. Furthermore, by 
constructing a section of $\p\widetilde{E} \to \Sigma_5$ that is
$\ZZ_2$-equivariant for either of the two $\ZZ_2$-actions, we deduce that these
two are the same bundle. 
By construction, $E'$ has eight Lagrangian spheres, and we claim that
$E$ has none, thus proving that $E$ and $E'$ are non-diffeomorphic Stein
fillings of $(M,\xi)$.  Indeed,
the map $\widetilde E \to E$ is an honest 2-to-1 covering map, so the
    preimage of any Lagrangian sphere in $E$ would be a pair of Lagrangian
    spheres in $\widetilde E$, in particular, having square $-2$. 
But all classes in $H_2(\widetilde E)$ have self-intersection $0$ by the
K\"unneth formula (or alternatively: none of them are represented by
spheres, since $\pi_2(\widetilde E)$ is trivial).
\end{example}

\section{Generalities on punctured holomorphic curves}
\label{sec:SHSgeneral}

The contents of this section are mostly standard, but a quick review
seems worthwhile in order to fix terminology and notation in preparation
for later holomorphic curve arguments.

\subsection{Stable Hamiltonian structures and symplectization ends}
\label{sec:stable}

Stable Hamiltonian structures
(or ``SHS'' for short) were first introduced in a dynamical context in
\cite{HoferZehnder} and 
reappeared in \cite{SFTcompactness} as the natural setting for the theory
of punctured holomorphic curves.  For our purposes, they provide a convenient
generalization of the notion of the \emph{symplectization} of a contact
manifold.  The particular SHS that arise in this paper can be thought of as
degenerate limits of certain contact forms in which explicit constructions 
of holomorphic curves become much easier.  For a more comprehensive
discussion of the topology of stable Hamiltonian structures,
see \cite{CieliebakVolkov}.

Given an oriented $(2n-1)$-dimensional manifold~$M$, a pair 
$$
\hH = (\Omega,\Lambda)
$$
consisting of a smooth $2$-form $\Omega$ and $1$-form $\Lambda$ is called a
\defin{stable Hamiltonian structure} if
\begin{enumerate}[label=(\roman{enumi})]
\item $\Lambda \wedge \Omega^{n-1} > 0$, 
\item $d\Omega = 0$,
\item $\ker\Omega \subset \ker d\Lambda$.
\end{enumerate}
Such a pair gives rise to two important objects: a co-oriented
hyperplane distribution $\Xi := \ker\Lambda$, and a positively transverse
vector field~$R_\hH$ determined by the conditions
$$
\Omega(R_\hH,\cdot) \equiv 0 \quad\text{ and }\quad
\Lambda(R_\hH) \equiv 1.
$$
By analogy with contact forms, we will refer to $R_\hH$ as the
\defin{Reeb vector field} of~$\hH$.  It reduces to the usual
contact notion of the Reeb vector field for $\Lambda$ whenever the latter
happens also to be a contact form; SHS with this property will be said to
be of \defin{contact type}.  Note that this definition does not require
$\Omega$ to exact, though $(d\Lambda,\Lambda)$ is always an example of
an SHS when $\Lambda$ is contact.
If $\dim M = 3$, we will say that $\hH = (\Omega,\Lambda)$ is of
\defin{confoliation type} whenever
$$
\Lambda \wedge d\Lambda \ge 0,
$$
which is equivalent to the condition $d\Lambda|_{\Xi} \ge 0$ and means that
$\Xi \subset TM$ is a \emph{confoliation} in the 
sense of \cite{EliashbergThurston}.

Stable Hamiltonian structures arise naturally in the context of
\defin{stable hypersurfaces} as defined in \cite{HoferZehnder}.
Given a symplectic manifold $(W,\omega)$, a compact hypersurface $M \subset W$
is called \defin{stable} if there exists a vector field $Z$ on a neighborhood
of $M$ in $W$ that is everywhere transverse to $M$ and determines a
$1$-parameter family of hypersurfaces with isomorphic characteristic line
fields: more precisely, this means that if $\Phi_Z^t$ denotes the flow
of~$Z$, then the real line bundle
$$
\ker \left( (\Phi_Z^t)^*\omega|_{TM} \right) \subset TM
$$
is independent of $t$ near $t=0$.  In this case we call $Z$ a
\defin{stabilizing vector field} for $M$, and the pair $(\Omega,\Lambda)$
defined by
$$
\Omega := \omega|_{TM}, \qquad \Lambda := \iota_Z \omega|_{TM}
$$
is an SHS on~$M$.  One can use the Moser deformation trick to show that
a neighborhood of $M$ in $(W,\omega)$ is then symplectomorphic to
a collar of the form
\begin{equation}
\label{eqn:SHScollarGeneral}
\big( (-\delta,\delta) \times M , d(t\Lambda) + \Omega \big)
\end{equation}
for sufficiently small $\delta > 0$, where $t$ denotes the coordinate
on~$(-\delta,\delta)$ and the symplectomorphism identifies $\{0\} \times M$ 
with $M \subset W$.  Conversely, $d(t\Lambda) + \Omega$ is 
symplectic on $(-\delta,\delta) \times M$ whenever
$(\Omega,\Lambda)$ is an SHS and $\delta > 0$ is sufficiently small.
The following variant of \eqref{eqn:SHScollarGeneral} is less commonly seen
in the literature but will be convenient for our purposes: defining the
alternative coordinate $r := \log(t+1)$ on the first factor and adjusting
the value of $\delta > 0$ accordingly, \eqref{eqn:SHScollarGeneral} becomes
\begin{equation}
\label{eqn:SHScollarGeneral2}
\big( (-\delta,\delta) \times M , d\left( (e^r - 1)\Lambda \right) + \Omega \big) .
\end{equation}
As an important special case, $Z$ is always stabilizing if it is a
\defin{Liouville} vector field transverse to~$M$, 
i.e.~$\Lie_Z \omega = \omega$.  In this case $\lambda := \iota_Z\omega$
satisfies $d\lambda = \omega$ and restricts to $M$ as a contact form
$\alpha := \lambda|_{TM}$, hence the resulting stable Hamiltonian structure
is $(d\alpha,\alpha)$ and the symplectic structure
in \eqref{eqn:SHScollarGeneral2} takes the form
$d(e^r \alpha)$, one of the standard formulas for the symplectization
$\RR \times M$ of the contact manifold $(M,\Xi = \ker\alpha)$.

By analogy with the contact case, one can define the \defin{symplectization
of $(M,\hH)$} for any stable Hamiltonian structure
$\hH = (\Omega,\Lambda)$ by choosing suitable
diffeomorphisms of \eqref{eqn:SHScollarGeneral2} with $\RR \times M$:
equivalently, this means we consider $\RR \times M$ with the family of
symplectic forms $\omega_\varphi$ defined by
\begin{equation}
\label{eqn:SHSsymplectization}
\omega_\varphi := 
d\Big( \big( e^{\varphi(r)} - 1 \big) \Lambda \Big) + \Omega,
\end{equation}
where $\varphi$ is chosen arbitrarily from the set
\begin{equation}
\label{eqn:tT}
\tT := \left\{ \varphi \in C^\infty(\RR,(-\delta,\delta))\ \big|\ 
\varphi' > 0 \right\}.
\end{equation}
More generally, suppose $(W,\omega)$ is a compact $2n$-dimensional 
symplectic manifold
with stable boundary $\p W = -M_- \coprod M_+$, equipped with a stabilizing
vector field $Z$ that points inward at $M_-$ and outward at~$M_+$.
Denote the induced SHS on $M_\pm$ by $\hH_\pm = (\Omega_\pm,\Lambda_\pm)$;
note that the orientation conventions here are chosen such that
$\Lambda_\pm \wedge \Omega_\pm^{n-1} > 0$ on~$M_\pm$.  We can now
identify neighborhoods of $M_\pm$ in $(W,\omega)$ symplectically with
collars of the form
\begin{equation*}
\begin{split}
& \left( [0,\delta) \times M_+ , d\big( (e^r - 1)\Lambda_+ \big) + \Omega_+ \right), \\
& \left( (-\delta,0] \times M_- , d\big( (e^r - 1)\Lambda_- \big) + \Omega_- \right).
\end{split}
\end{equation*}
Modifying \eqref{eqn:tT} to
\begin{equation}
\label{eqn:tTcob}
\tT := \left\{ \varphi \in C^\infty(\RR,(-\delta,\delta))\ \big|\ 
\varphi' > 0 \text{ and $\varphi(r) = r$ for $r$ near~$0$} \right\},
\end{equation}
we can use any $\varphi \in \tT$ to define a \defin{symplectic completion}
$(\widehat{W},\omega_\varphi)$ of $(W,\omega)$ by
$$
\widehat{W} := \big( (-\infty,0] \times M_- \big) \cup_{M_-} W
\cup_{M_+} \big( [0,\infty) \times M_+ \big),
$$
where the above collar neighborhoods are used to glue the pieces together
smoothly and the symplectic form is defined by
$$
\omega_\varphi := \begin{cases}
d\left( (e^{\varphi(r)} - 1)\Lambda_- \right) + \Omega_- & 
\text{ on $(-\infty,0] \times M_-$},\\
\omega & \text{ on $W$},\\
d\left( (e^{\varphi(r)} - 1)\Lambda_+ \right) + \Omega_+ & 
\text{ on $[0,\infty) \times M_+$}.
\end{cases}
$$

\subsection{Finite energy holomorphic curves}
\label{sec:energy}

Given a stable Hamiltonian structure 
$\hH = (\Omega,\Lambda)$ with induced hyperplane field $\Xi = \ker\Lambda$ and
Reeb vector field $R_\hH$, we denote by $\jJ(\hH)$ the space of 
$\RR$-invariant almost complex structures on the symplectization
$\RR \times M$ that are \defin{compatible} with $\hH$,
meaning that for $J \in \jJ(\hH)$,
\begin{enumerate}[label=(\roman{enumi})]
\item $J \p_r = R_\hH$, where $\p_r$ denotes the unit vector in the
$\RR$-direction;
\item $J(\Xi) = \Xi$ and $\Omega(\cdot,J\cdot)$
defines a bundle metric on~$\Xi$.
\end{enumerate}
In the special case $(\Omega,\Lambda) = (d\alpha,\alpha)$ with $\alpha$
a contact form, this reproduces the standard definition for almost
complex structures compatible with contact forms, and we shall 
in this case abbreviate
$$
\jJ(\alpha) := \jJ(\hH), \quad \text{ where } \quad
\hH := (d\alpha,\alpha).
$$
The following trivial observation will be helpful because it permits
the use of a slightly nonstandard stable Hamiltonian structure (in particular
with $\Omega$ non-exact) for computing holomorphic curve
invariants that are usually defined in terms of contact forms.
\begin{prop}
\label{prop:sameJ}
Suppose $\dim M = 3$, $\alpha$ is a contact form, and 
$\Omega$ is any closed $2$-form for which $\hH \coloneqq (\Omega,\alpha)$ is a 
stable Hamiltonian structure. 
Then $\jJ(\hH) = \jJ(\alpha)$.
\end{prop}
\begin{proof}
Since $\alpha$ is contact, the Reeb vector field $R_\hH$ is the same as the
contact Reeb vector field for~$\alpha$.  The only difference between the 
conditions defining $\jJ(\hH)$ and $\jJ(\alpha)$ is thus that
$J \colon \Xi \to \Xi$ must be compatible with 
$\Omega|_{\Xi}$ in the first case and compatible with 
$d\alpha|_{\Xi}$ in the second case. Since $\Xi$
is complex $1$-dimensional and $\Omega|_\Xi$ and $d\alpha|_\Xi$ induce the
same orientation, these conditions are identical.
\end{proof}

Any given $J \in \jJ(\hH)$ is tamed by all of the symplectic
forms $\omega_\varphi$ in \eqref{eqn:SHSsymplectization} on the symplectization $\RR \times M$
if the constant $\delta > 0$ in \eqref{eqn:tT} is chosen sufficiently small;
in the case $\dim M = 3$, which will be our primary interest, $J$ is also
$\omega_\varphi$-compatible for all $\varphi \in \tT$.
Given a Riemann surface
$(S,j)$ and $J$-holomorphic curve $u : (S,j) \to (\RR\times M,J)$,
we therefore define the \defin{energy} of~$u$ by
\begin{equation}
\label{eqn:energy1}
E(u) := \sup_{\varphi \in \tT} \int_S u^*\omega_\varphi.
\end{equation}
The same formula can be used to define the energy of a
$J$-holomorphic curve $u : (S,j) \to (\widehat{W},J)$, where 
$\widehat{W}$ denotes the completion of a cobordism $(W,\omega)$ with
stable boundary $-M_- \coprod M_+$ as in \S\ref{sec:stable}, and $J$
is chosen from the space
$$
\jJ(\omega,\hH_+,\hH_-)
$$
consisting of almost complex structures $J$ on $\widehat{W}$ such that
$J|_W$ is compatible with $\omega$ and
\begin{equation*}
\begin{split}
J_+ &:= J|_{[0,\infty) \times M_+} \in \jJ(\hH_+), \\
J_- &:= J|_{(-\infty,0] \times M_-} \in \jJ(\hH_-).
\end{split}
\end{equation*}
Any $J \in \jJ(\omega,\hH_+,\hH_-)$ is 
$\omega_\varphi$-tame
on
$\widehat{W}$ for every $\varphi \in \tT$, hence the energy
\eqref{eqn:energy1} is always nonnegative, and is positive unless the
curve is constant.

\begin{remark}
The notion of energy described here is slightly
different from the one defined in \cite{SFTcompactness}, but is equivalent
to it in the sense that uniform bounds on either imply uniform bounds on
the other.
\end{remark}

We will always take the domain of our holomorphic curves to be punctured 
Riemann surfaces $\dot{S} = S \setminus \Gamma$, i.e.~$(S,j)$ is
a closed Riemann surface and $\Gamma \subset S$ is a finite ordered set.
The surface $\dot{S}$ will also be assumed to be connected unless otherwise specified.
When this needs to be emphasized, we will call a curve 
$u : \dot{S} \to \widehat{W}$ \defin{connected} whenever its domain is
connected; if $\dot{S}$ is disconnected, then the \defin{connected components}
of $u$ are defined to be its restriction to the connected components
of~$\dot{S}$.
A punctured $J$-holomorphic curve 
$u : \dot{S} \to \widehat{W}$
with positive finite energy is either positively or
negatively asymptotic to periodic orbits of~$R_{\hH_+}$ or $R_{\hH_-}$
respectively at each of its nonremovable
punctures; in short, finite energy punctured
$J$-holomorphic curves are \defin{asymptotically cylindrical},
cf.~\cite{SFTcompactness}.  

\begin{remark}
The terms ``finite energy'' and
``asymptotically cylindrical'' are often used as synonyms when describing
$J$-holomorphic curves, and we shall generally consider these conditions to be
implied whenever we refer to ``punctured'' holomorphic curves.
The underlying presumption, unless stated otherwise, is always that the domain 
is the complement of a finite (sometimes empty) set of points in a closed
Riemann surface, and that all the punctures are non-removable.
\end{remark}

We consider two holomorphic curves equivalent if they are
related to each other by biholomorphic maps of their domains that
take punctures to punctures with the ordering of punctures preserved.
The resulting equivalence classes are
called \defin{unparametrized} $J$-holomorphic curves.  We will often
abuse notation and use a parametrized map $u : \dot{S} \to \widehat{W}$
to refer to the \emph{unparametrized} curve that it
represents.  When speaking of moduli spaces, we will always mean a
space of unparametrized $J$-holomorphic curves that are asymptotically 
cylindrical, with a topology such that a sequence is considered to converge 
if and only if one can find parametrizations with a fixed punctured domain 
$\dot{S} = S \setminus \Gamma$
such that the complex structures on $S$ converge in $C^\infty$
while the maps $\dot{S} \to \widehat{W}$ converge in $C^\infty$ on 
compact subsets and in
$C^0$ up to the cylindrical ends (measured via any choice of
translation-invariant metric on the ends).  For a given $J$, the 
corresponding moduli will typically be denoted by
$$
\mM(J).
$$

In the $\RR$-invariant case
$J \in \jJ(\hH)$, an important example of a finite
energy holomorphic curve is the \defin{trivial cylinder}
$$
u : \RR\times S^1 \to \RR\times M : (s,t) \mapsto (Ts,x(Tt))
$$
over any orbit $x : \RR \to M$ with $x(T) = x(0)$ for $T > 0$; this curve can
be parametrized as a punctured sphere with one positive and one negative
puncture, both approaching the same orbit.
We shall sometimes abbreviate the unparametrized curve
represented by the trivial cylinder described above as
$$
\RR\times \gamma,
$$
where~$\gamma : S^1 \to M : t \mapsto x(Tt)$ specifies the periodic orbit 
in question, which may in general be multiply covered.

If the asymptotic
orbits of a finite energy $J$-holomorphic curve $u$
are all nondegenerate or Morse-Bott, then the moduli space $\mM(J)$
near~$u$ can be described as the 
zero set of a Fredholm section whose index corresponds to the 
\defin{virtual dimension}
of the moduli space near~$u$.  We will call this virtual dimension the
\defin{index} of~$u$ and denote it by $\ind(u) \in \ZZ$.  By a punctured
version of the Riemann-Roch theorem (see \cite{Schwarz}), the index of a curve 
$u : \dot{S} \to \widehat{W}$ can be written as
\begin{equation}
\label{eqn:index}
\ind(u) = (n-3) \chi(\dot{S}) + 2 c_1^\Phi(u^*T\widehat{W}) +
\sum_{z \in \Gamma^+} \muCZ^\Phi(\gamma_z) - 
\sum_{z \in \Gamma^-} \muCZ^\Phi(\gamma_z),
\end{equation}
where $\dim_\RR \widehat{W} = 2n$, $\Gamma = \Gamma^+ \coprod \Gamma^-$ are
the positive and negative punctures with asymptotic orbits 
$\{\gamma_z\}_{z \in \Gamma}$, $\Phi$ is an arbitrary choice of
complex trivializations for the bundles $\Xi_\pm = \ker \Lambda_\pm$ along
these orbits, $\muCZ^\Phi(\gamma_z) \in \ZZ$ are the Conley-Zehnder
indices relative to these trivializations, and $c_1^\Phi(u^*T\widehat{W})$
is the relative first Chern number of $u^*T\widehat{W} \to \dot{S}$ with
respect to the asymptotic trivialization determined up to homotopy
by~$\Phi$.  The curve~$u$ is said to be \defin{Fredholm regular} if it 
represents a transverse intersection of the aforementioned Fredholm section 
with the zero section: in this case a neighborhood of $u$ in $\mM(J)$ is a 
smooth  orbifold (or manifold if~$u$ has no automorphisms) of dimension $\ind(u)$.
For further discussion of Fredholm regularity, 
see for example \cite{Wendl:automatic}.

Every asymptotically cylindrical holomorphic 
curve is either \defin{simple} (and thus
\defin{somewhere injective}) or \defin{multiply covered}, where the latter
means that it factors as the composition of another $J$-holomorphic curve
with a branched cover of closed Riemann surfaces with degree at least two.
By various standard transversality results (see for example
\cites{McDuffSalamon:Jhol,Dragnev,Wendl:SFT}),
the relevant spaces of compatible almost complex structures admit
comeager subsets for which all simple curves are Fredholm regular.
We will generally say that $J$ is \defin{generic} whenever it belongs to
the comeager subset for which the relevant transversality result of this
type holds.

It is sometimes useful to observe that
if $\dim M = 3$ and $J \in \jJ(\hH)$ where $\hH = (\Omega,\Lambda)$ 
is a confoliation-type SHS,
then every $J$-holomorphic curve $u : \dot{S} \to \RR \times M$ satisfies
$u^*d\Lambda \ge 0$.  Since the period of any closed orbit of $R_\hH$ 
parametrized by a loop $\gamma : S^1 \to M$ is given by
$\int_{S^1} \gamma^*\Lambda$, the following is an immediate consequence
of Stokes' theorem:
\begin{prop}
\label{prop:posPuncture}
Suppose $\dim M = 3$, $\hH = (\Omega,\Lambda)$ is a confoliation-type
stable Hamiltonian structure,
$J \in \jJ(\hH)$ and $u : \dot{S} \to \RR\times M$ is a nonconstant
finite energy $J$-holomorphic curve with positive and/or negative punctures
$\Gamma = \Gamma^+ \cup \Gamma^-$ asymptotic to the periodic orbits
$\{\gamma_z\}_{z \subset \Gamma}$.  Then $\#\Gamma^+ \ge 1$, and the periods
$T(\gamma_z) > 0$ of the orbits $\gamma_z$ satisfy
$$
\sum_{z \in \Gamma^+} T(\gamma_z) - \sum_{z \in \Gamma^-} T(\gamma_z) \ge 0.
$$
\qed
\end{prop}

Given $J \in \jJ(\omega,\hH_+,\hH_-)$ with the closed orbits of $R_{\hH_+}$
and $R_{\hH_-}$ assumed nondegenerate or Morse-Bott,
moduli spaces of punctured $J$-holomorphic curves in $(\widehat{W},J)$ with
uniform energy bounds satisfy a compactness theorem described in
\cite{SFTcompactness}.  The compactified moduli space
$$
\overline{\mM}(J)
$$
consist of so-called
(stable) \defin{holomorphic buildings}, which generalize the ``broken'' 
holomorphic curves familiar from Floer homology.
For our purposes,\footnote{In our description of holomorphic buildings we
ignore certain technical details such as \emph{decorations}, which
play no role in our arguments; these details are explained fully in
\cite{SFTcompactness}.}
the objects in this compactification can be described as follows.
A \defin{nodal $J$-holomorphic curve} in $\widehat{W}$, also sometimes 
called a \emph{holomorphic building of height~$1$}, is an equivalence class of 
tuples
$$
(S,j,\Gamma,u,\Delta)
$$
where $(S,j)$ is a closed but not necessarily connected Riemann surface,
$\Gamma \subset S$ is a finite ordered set defining the punctured
surface $\dot{S} := S \setminus \Gamma$, $u : (\dot{S},j) \to (\widehat{W},J)$
is an asymptotically cylindrical $J$-holomorphic curve and $\Delta$ is a
finite unordered set of unordered pairs $\{ z_+,z_- \}$ of distinct points
in $\dot{S}$ such that $u(z_+) = u(z_-)$.  Each pair $\{z_+,z_-\} \in \Delta$
is called a \defin{node}, and we sometimes also refer to the individual
points $z_\pm \in \dot{S}$ as \defin{nodal points}.  Two nodal curves are equivalent
if they are related by a biholomorphic identification of their domains that
preserves all the structure (including ordering of the punctures and pairing
of the nodal points).  Note that the asymptotically cylindrical behavior
of $u$ automatically partitions $\Gamma$ into sets of positive and negative
punctures $\Gamma = \Gamma^+ \coprod \Gamma^-$.  Nodal curves in the
symplectizations $\RR \times M_\pm$ can be defined in the same way, with the
additional feature that $\RR$-translations act on the space of nodal curves,
so that one can also strengthen the equivalence relation and consider
$\RR$-equivalence classes of nodal curves.

A \defin{holomorphic building} in $(\widehat{W},J)$ can now be regarded as
a finite ordered list of nodal curves $u = 
(u_1,\ldots,u_N)$ for some $N \in \NN$, which are called
\defin{levels} of $u$, and have the following properties and
additional data:
\begin{itemize}
\item Exactly one of the levels $u_M$ for some $M \in \{1,\ldots,N\}$
is a nodal curve in $\widehat{W}$; this is called the \defin{main level}
of the building.  It is also allowed to be \emph{empty}, meaning its
domain is the empty set.
\item Every level $u_\ell$ for $\ell \ne M$ is a nonempty 
$\RR$-equivalence class of nodal curves in one of the symplectizations
$\RR \times M_\pm$; $M_-$ for $\ell < M$ and $M_+$ for $\ell > M$.  These
are called \defin{lower} and \defin{upper levels} respectively.
\item For each $\ell \in \{1,\ldots,N-1\}$, $u$ is endowed with the
additional data of a bijection from the
positive punctures of $u_\ell$ to the negative punctures of
$u_{\ell+1}$ such that the asymptotic orbits of punctures that
correspond under this bijection are identical.  We will refer to corresponding
pairs of punctures as \defin{breaking punctures} and their asymptotic orbits
as \defin{breaking orbits}.
\end{itemize}
The positive and negative punctures of the building
$u = (u_1,\ldots,u_N)$ are defined as the
positive punctures of $u_N$ and the negative punctures of
$u_1$ respectively, and the \defin{connected components} of
$u$ are the connected components of its constituent levels.
One can define from $u$ a topological surface $\widehat{S}$
obtained from the disjoint union of the domains of all the levels by
performing connected sums along all the paired-up nodal points forming
nodes and all the corresponding breaking punctures.  The building is
then said to be \defin{connected} if and only if $\widehat{S}$ is connected,
and its \defin{arithmetic genus} is the genus of~$\widehat{S}$.
This punctured surface is diffeomorphic to the domain of any sequence of
smooth curves that converge to the building in the SFT-topology;
in particular, any such sequence admits a sequence of parametrizations
$u_k : \dot{S} \to \widehat{W}$ that can be transformed into a
$C^0$-convergent sequence of continuous maps $\dot{S} \to W$ by
projecting cylindrical ends and upper/lower levels to $M_\pm$ and gluing
the components of the limiting building together along nodes and breaking 
orbits.
The buildings that form $\overline{\mM}(J)$ are also always assumed
to be \defin{stable}, which means that none of the upper or lower levels is a 
disjoint union of trivial cylinders, and any connected component with
genus zero on which the map is constant (a so-called \defin{ghost bubble})
has at least three nodal points.  This condition guarantees that limits
in the SFT-topology are unique.
We shall generally describe a connected component of a
holomorphic building as \defin{nontrivial} if it is nonconstant and is
not a trivial cylinder.

For moduli spaces of curves in a
symplectization $(\RR \times M,J)$ with $J \in \jJ(\hH)$, 
the distinction between lower/main/upper 
levels is meaningless: instead, one compactifies $\mM(J) / \RR$
to obtain a space $\overline{\mM}(J)$ of buildings with
at least one and at most finitely many levels, all of them consisting
of $\RR$-equivalence classes of (possibly disconnected and nodal)
unparametrized curves in $\RR \times M$.

Within the space of holomorphic buildings, we shall sometimes make a 
distinction between \defin{nontrivial buildings} and \defin{smooth curves}:
the latter means buildings that have only one level and no nodes, hence
they are also elements of~$\mM(J)$, whereas by ``nontrivial buildings'' 
we mean everything in $\overline{\mM}(J) \setminus \mM(J)$.

Since the index of a holomorphic curve depends only on its
asymptotic ends and relative homology class, the \defin{index of a building}
can be defined formally by a natural generalization of \eqref{eqn:index}
replacing $\dot{S}$ with $\widehat{S}$, and in this way
the index extends to a continuous $\ZZ$-valued
function on~$\overline{\mM}(J)$.

For the purposes of the next
statement, observe that given any building $u$, deleting the nodes from
all levels changes $u$ into a disjoint union of some unique collection of
(not necessarily stable) connected holomorphic buildings
$u_1,\ldots,u_m$, each endowed with the extra structure of a finite set of
points in their domains (the former nodal points).
We shall in this case call $u_1,\ldots,u_m$ the \defin{maximal non-nodal 
subbuildings} of~$u$ (see Figure~\ref{fig:maxNonnodal}).  The relation in the following 
proposition is an immediate consequence of \eqref{eqn:index} via
the observation that if $\dot{S}$ is a surface obtained from a collection of 
surfaces $\dot{S}_1,\ldots,\dot{S}_m$ by performing connected sums at a set of
$N$ distinct pairs of distinct points $\{z^+_j,z^-_j\} \subset
\dot{S}_1 \coprod \ldots \coprod \dot{S}_m$ for $j=1,\ldots,N$, then
$\chi(\dot{S}) = \sum_{i=1}^m \chi(\dot{S}_i) - 2N$.

\begin{figure}
\begin{postscript}
\psfrag{What}{$\widehat{W}$}
\psfrag{RtimesM+}{$\RR \times M_+$}
\psfrag{RtimesM-}{$\RR \times M_-$}
\includegraphics[scale=0.8]{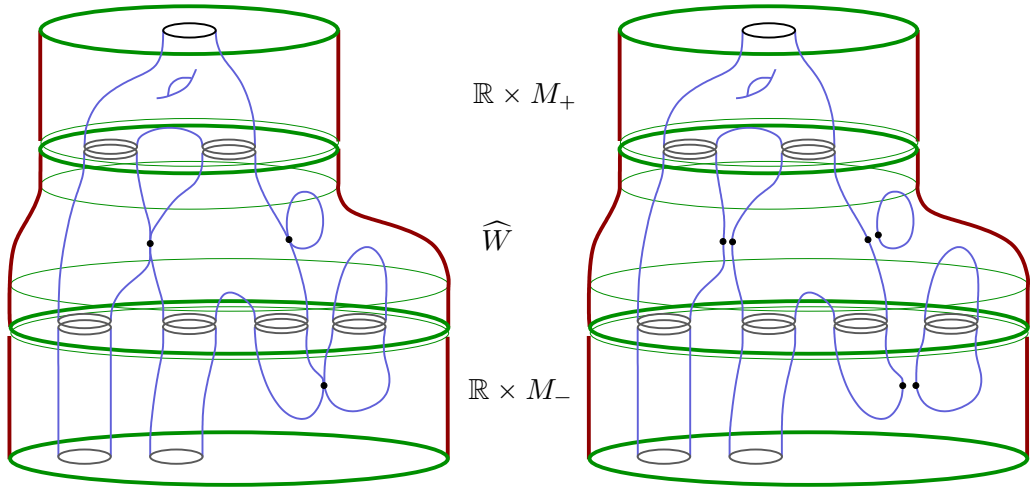}
\end{postscript}
\caption{\label{fig:maxNonnodal}
The picture at the left shows a holomorphic building with arithmetic genus two,
which is broken up at the right into three maximal non-nodal subbuildings,
one with arithmetic genus~$1$ and two with arithmetic genus~$0$.}
\end{figure}

\begin{prop}
\label{prop:indexNodes}
For any holomorphic building $u$ in $\widehat{W}$ with $N \ge 0$ nodes
and $N_i \ge 0$ nodal points on each of its maximal non-nodal subbuildings
$u_i$ for $i=1,\ldots,m$,
$$
\ind(u) = \sum_{i=1}^m \left[ \ind(u_i) - (n-3) N_i \right] = 
\sum_{i=1}^m \ind(u_i) - 2 (n-3) N,
$$
where $\dim \widehat{W} = 2n$.  \qed
\end{prop}

Let us specialize the above result to dimension four and consider the role
played by constant components.  These have no punctures but must have nodal
points; setting $n=2$, \eqref{eqn:index} implies that a constant component
$u_i$ with domain $S$ of genus~$g$ has index $-\chi(S) = 2g-2$, which is 
nonnegative except in the case of ghost bubbles.  Stability requires however 
that the Euler characteristic of $S$ should always become negative after
removing nodal points, thus
$$
\ind(u_i) + N_i = - \chi(S) + N_i > 0.
$$
This gives rise to the following corollary of
Proposition~\ref{prop:indexNodes}:

\begin{prop}
\label{prop:indexNodes4}
Assume $\dim_\RR \widehat{W} = 4$ and $u$ is a holomorphic building 
with $m$ nonconstant maximal non-nodal subbuildings $u_1,\ldots,u_m$,
each with $N_i \ge 0$ nodal points.  Then
$$
\ind(u) \ge \sum_{i=1}^m \left[ \ind(u_i) + N_i \right],
$$
with equality if and only if $u$ has no constant components.
In particular, if $u$ has arithmetic genus~$0$ and has at least one node, then
$$
\ind(u) \ge 2 + \sum_{i=1}^m \ind(u_i)
$$
with equality if and only if there is exactly one node and no ghost bubbles.
\end{prop}
\begin{proof}
The second statement follows because in the case of arithmetic genus zero,
every ghost bubble has at least three nodal points and this implies the
existence of at least three nodal points on nonconstant components as
well; any other scenario would lead to positive arithmetic genus.
\end{proof}

\subsection{Intersection theory}
\label{sec:Siefring}

In this section we summarize some useful facts from the intersection theory
of asymptotically cylindrical holomorphic curves, due to Siefring
\cite{Siefring:intersection}.  A more elementary introduction to this theory
can also be found in \cite{Wendl:Durham}. See also the summary given in
\cite{FishSiefring:connect}*{Section 3.3}.

Assume as in \S\ref{sec:stable} that $\widehat{W}$ is the completion of a
symplectic cobordism $(W,\omega)$ with stable boundary $\p W = - M_- \coprod
M_+$ carrying stable Hamiltonian structures $\hH_\pm = (\Omega_\pm,\Lambda_\pm)$, 
and $J \in \jJ(\omega,\hH_+,\hH_-)$.
Siefring's intersection theory associates
to any pair of asymptotically cylindrical (but not necessarily $J$-holomorphic)
maps $u$ and~$v$ into $\widehat{W}$ with nondegenerate asymptotic orbits 
an intersection number
$$
u * v \in \ZZ,
$$ 
which depends only on the asymptotic orbits of the two maps and their
relative homology classes.  It is nonnegative whenever $u$ and~$v$ are 
$J$-holomorphic
curves with non-identical images, and strictly positive whenever these have
nonempty intersection.  It also extends in a continuous way
to the \emph{compactified} moduli space of holomorphic curves as defined
in \cite{SFTcompactness}: one can define $u * v$ for two holomorphic
buildings, and it is additive across levels (with extra nonnegative
\defin{breaking contributions} for common breaking orbits between two levels)
and invariant under homotopies through the compactified moduli space,
including ``infinite $\RR$-translations'' which shove levels up or
down or insert or delete trivial cylinders.  When $u$ and $v$ are
holomorphic curves with non-identical images, $u * v$
counts their actual intersections (with multiplicity when they are 
non-transverse), in addition to a nonnegative count of 
\defin{asymptotic contributions},
i.e.~``hidden'' intersections that can emerge from infinity under perturbations.
The latter can be expressed in terms of asymptotic
winding numbers: fixing a choice of complex trivialization $\Phi$ for each of 
the bundles $\Xi_\pm = \ker \Lambda_\pm$ along closed Reeb orbits,
every nondegenerate Reeb orbit~$\gamma$ has certain
\defin{extremal winding numbers}
$$
\alpha_-^\Phi(\gamma) \le \alpha_+^\Phi(\gamma) \in \ZZ
$$
such that by the asymptotic formula of \cite{HWZ:props1},
the asymptotic winding of any holomorphic curve approaching~$\gamma$
at a positive end is bounded from above by $\alpha_-^\Phi(\gamma)$, and at
a negative end it is bounded from below by~$\alpha_+^\Phi(\gamma)$.  These
are the winding numbers relative to $\Phi$ of the so-called (positive and
negative) \defin{extremal eigenfunctions} that appear in asymptotic formulas,
and they are related to the Conley-Zehnder index by the formulas
\begin{equation}
\label{eqn:CZwinding}
\begin{split}
\muCZ^\Phi(\gamma) &= 2\alpha_-^\Phi(\gamma) + p(\gamma) =
2\alpha_+^\Phi(\gamma) - p(\gamma), \\
p(\gamma) &= \alpha_+^\Phi(\gamma) - \alpha_-^\Phi(\gamma) \in \{0,1\},
\end{split}
\end{equation}
proved in \cite{HWZ:props2}.  The general definition of $u * v$
expresses it in terms of the \defin{relative intersection number}
$$
u \bullet_\Phi v \in \ZZ,
$$
which is homotopy invariant but depends on the choice of 
trivialization~$\Phi$ whenever $u$ and $v$ have asymptotic orbits in
common: $u \bullet_\Phi v$ is the algebraic count of intersections between
$u$ and a generic small perturbation of $v$ that pushes it in the
direction determined by $\Phi$ at infinity.  Notice that this notion is
also well defined when $u=v$ and it extends in a natural way to the case
where $u$ and $v$ are holomorphic buildings, simply by adding relative
intersection numbers across levels.  The following is then a
direct consequence of the definition in \cite{Siefring:intersection} and
will suffice for computing $u*v$ in our applications:

\begin{lemma}
\label{lemma:u*v}
Suppose $u$ and $v$ are holomorphic buildings with only positive punctures, 
and that for each asymptotic orbit $\gamma$ of~$u$ or $v$, there exists a 
trivialization $\Phi$ along the underlying simple orbit covered by 
$\gamma$ such that
in the induced trivialization along $\gamma$, $\alpha_-^\Phi(\gamma) = 0$.
Then $u * v = u \bullet_\Phi v$.  \qed
\end{lemma}

The usual adjunction formula
for the closed case can now be generalized to somewhere injective punctured
holomorphic curves in the form
\begin{equation}
\label{eqn:adjunctionEq}
u * u = 2\left[ \delta(u) + \delta_\infty(u) \right] + c_N(u) + 
\left[ \bar{\sigma}(u) - \#\Gamma \right].
\end{equation}
Here $\delta(u)$ is the (nonnegative) algebraic count of double points and 
critical points, and $\delta_\infty(u)$ is an (also nonnegative) asymptotic 
contribution such that $\delta(u) + \delta_\infty(u)$ is homotopy invariant
and counts the double points of a generic perturbation of~$u$.  The
\defin{normal Chern number} $c_N(u) \in \ZZ$ is another homotopy invariant
quantity which, in the immersed case, equals the relative first Chern number
of the normal bundle of~$u$ with respect to trivializations determined by
the extremal eigenfunctions at the asymptotic orbits.  We denote by
$\Gamma$ the set of punctures of~$u$, and $\bar{\sigma}(u)$ denotes the
\defin{spectral covering number}, which is the sum over all $z \in \Gamma$
of the covering multiplicities of the relevant extremal eigenfunctions.
In many applications one does not need to compute $\bar{\sigma}(u)$, as it
is at least immediate from the definition that $\sigma(u) \ge \#\Gamma$,
hence \eqref{eqn:adjunctionEq} gives rise to an inequality
\begin{equation}
\label{eqn:adjunction}
u * u \ge 2\left[ \delta(u) + \delta_\infty(u) \right] + c_N(u).
\end{equation}
When more precise information is needed, the following 
will suffice for our purposes:
\begin{lemma}
\label{lemma:spectralCovering}
Suppose $u$ is a somewhere injective holomorphic curve with only positive
punctures which satisfy the hypothesis of Lemma~\ref{lemma:u*v}.  Then
$\bar{\sigma}(u)$ is the sum of the covering multiplicities of the
asymptotic orbits of~$u$.  In particular, $\bar{\sigma}(u) = \#\Gamma$
whenever all asymptotic orbits are simply covered.
\qed
\end{lemma}

The easiest way to compute $c_N(u)$ is usually via the Fredholm index,
as it satisfies
\begin{equation}
\label{eqn:2cN}
2 c_N(u) = \ind(u) - 2 + 2g + \#\Gamma_0,
\end{equation}
where $g \ge 0$ is the genus of the domain of $u$ and $\#\Gamma_0 \ge 0$ 
denotes the number of punctures at which the
Conley-Zehnder indices of the asymptotic orbits are even.  In the 
$\RR$-invariant case $\widehat{W} = \RR \times M$ with $J \in \jJ(\hH)$
for a fixed stable Hamiltonian structure $\hH = (\Omega,\Lambda)$,
the normal Chern number also appears in the important relation
\begin{equation}
\label{eqn:windpi}
0 \le \windpi(u) + \asympdef(u) = c_N(u),
\end{equation}
which was originally proved by Hofer-Wysocki-Zehnder \cite{HWZ:props2} and
applies to any curve $u$ that is not a cover of a trivial cylinder.
Here $\windpi(u) \ge 0$ is an integer which algebraically counts the 
non-immersed points of the projection of~$u$ to~$M$, and $\asympdef(u) \ge 0$ 
is an integer measuring the difference of the asymptotic winding at each end
from the relevant extremal value $\alpha_\pm^\Phi(\gamma)$.  We refer to
\cite{Wendl:compactnessRinvt}*{\S 4} for a fuller discussion of this
relation using the same notation used here (only some of which appeared
in \cite{HWZ:props2}).

\begin{remark}
The theory defined in \cite{Siefring:intersection} applies to any moduli
spaces of asymptotically cylindrical holomorphic curves with fixed
asymptotic orbits satisfying a nondegeneracy or Morse-Bott condition.
One can also define a more general theory allowing moduli spaces whose
asymptotic orbits move freely in Morse-Bott families---the main results
of this theory are outlined in \cite{Wendl:automatic}*{\S 4}, but we will
not need this here.
\end{remark}

In many applications, a special role is played by somewhere injective curves 
whose intersection-theoretic properties force them not only to be embedded but 
also to avoid intersecting their neighbors in the moduli space.
In particular, a $J$-holomorphic curve $u : \dot{S} \to \widehat{W}$
is called \defin{nicely embedded} if it is somewhere injective and satisfies
$$
\delta(u) = \delta_\infty(u) = 0 \quad \text{ and } \quad
u * u \le 0.
$$
This is a slight generalization of a definition that first appeared
(with an extra ``stability'' condition) in \cite{Wendl:automatic}.
It simplifies slightly in the $\RR$-invariant case since every
curve $u$ can then be perturbed via $\RR$-translation to a nearby curve,
which will be different unless $u$ is a cover of a trivial cylinder,
hence $u * u \ge 0$ always holds in such cases.  The nicely embedded
condition can thus be reduced in the $\RR$-invariant case to
$$
u * u = 0 \qquad \text{ or } \qquad \text{$u$ is a trivial cylinder},
$$
as \eqref{eqn:windpi} implies in this case that $c_N(u) \ge 0$,
so $\delta(u) = \delta_\infty(u) = 0$ then follows from the adjunction
inequality \eqref{eqn:adjunction}.  An additional consequence of
\eqref{eqn:windpi} in the $\RR$-invariant case is that nicely embedded
curves other than trivial cylinders satisfy $\windpi(u) = \asympdef(u)=0$,
and the homotopy invariance of $u*u=0$ implies that they never intersect
their own $\RR$-translations, hence their projections
to the $3$-manifold
$$
\dot{S} \stackrel{u}{\longrightarrow} \RR \times M 
\stackrel{\pr}{\longrightarrow} M
$$
are embedded.  Nicely embedded curves in the symplectization 
also satisfy a strong compactness
theorem proved in \cite{Wendl:compactnessRinvt}, which we will make use
of to prove uniqueness  in \S\ref{sec:uniqueness}.  


The following lemma can be
applied in the case $u=v$ to characterize nicely embedded curves:

\begin{lemma}
\label{lemma:nicelyEmbedded}
Assume $J \in \jJ(\hH)$ for a stable Hamiltonian structure $\hH$ on~$M$.
For any pair of (possibly identical) connected finite energy
$J$-holomorphic curves $u$ and~$v$ in $\RR\times M$ which are not covers of
trivial cylinders and have nondegenerate asymptotic orbits, 
we have $u * v = 0$ whenever the following conditions hold:
\begin{enumerate}
\item There is no simply covered orbit $\gamma$ with odd Conley-Zehnder
index such that covers of~$\gamma$ appear at negative ends of~$u$ and
positive ends of~$v$.
\item For every negative asymptotic orbit~$\gamma$ of~$u$, 
$v * (\RR\times\gamma) = 0$.
\item For every positive asymptotic orbit~$\gamma$ of~$v$, 
$u * (\RR\times\gamma) = 0$.
\end{enumerate}
\end{lemma}
\begin{proof}
    By the ``infinite $\RR$-translation'' used in
    \cite{Siefring:intersection}*{Lemma 5.7, Theorem 5.8}, we have $u * v = u^+ * v^-$,
where $u^+$ and $v^-$ are $2$-level holomorphic buildings defined as follows:
\begin{itemize}
\item $u^+$ has $u$ on the top level and the trivial cylinders over its
negative asymptotic orbits on the bottom level,
\item $v^-$ has $v$ on the bottom level and the trivial cylinders over
its positive asymptotic orbits on the top level.
\end{itemize}
To compute $u^+ * v^-$, we sum the respective intersection numbers of
corresponding levels, together with breaking contributions, all of
which are nonnegative.  The top level thus contributes
$u * (\RR\times \gamma)$ for every positive asymptotic orbit~$\gamma$
of~$v$, and the bottom level similarly contributes
$v * (\RR\times \gamma)$ for every negative asymptotic orbit~$\gamma$
of~$u$.  The breaking contributions come from
orbits which occur as breaking orbits of both buildings,
but these contributions are zero for orbits with even Conley-Zehnder
index since the the eigenvectors for the largest
negative eigenvalue and for the smallest positive eigenvalue of the
corresponding asymptotic operator have the same winding
(cf.~\cite{Wendl:Durham}*{Appendix~C.5}).
\end{proof}

Intersection numbers between trivial cylinders and their own covers are
usually tricky to deal with, as they need not be nonnegative in general,
but they are easy at least in the following special case:

\begin{lemma}
\label{lemma:evenOrbits}
Suppose $\gamma$ is a simply covered nondegenerate periodic orbit of
$R_\hH$ in $M$ with even
Conley-Zehnder index, and $u$ and $v$ denote any $J$-holomorphic 
covers of the trivial cylinder $\RR\times \gamma$.  Then
$u * v = 0$.
\end{lemma}
\begin{proof}
Note that all covers~$\gamma^m$ of~$\gamma$  
also have even Conley-Zehnder index, 
hence \eqref{eqn:CZwinding} gives 
$\alpha_+^\Phi(\gamma^m) - \alpha_-^\Phi(\gamma^m) = 0$.
The result now follows 
directly from the definition of $u*v$ 
\cite{Siefring:intersection}*{Equation (2.3)}.
\end{proof}

\subsection{Automatic transversality and coherent orientations}
\label{sec:automatic}

We continue under the assumption that $(W,\omega)$ is a compact symplectic
cobordism with stable boundary components $\p W = -M_- \coprod M_+$ carrying
stable Hamiltonian structures $\hH_\pm = (\Omega_\pm,\Lambda_\pm)$, and
$\widehat{W}$ denotes the completion obtained by attaching cylindrical ends.
The following special case of a transversality criterion from 
\cite{Wendl:automatic} is often useful because it requires neither genericity 
nor somewhere injectivity.

\begin{prop}
\label{prop:automatic}
Assume $J \in \jJ(\omega,\hH_+,\hH_-)$, $\dim_\RR W = 4$, and
$u : \dot{S} \to \widehat{W}$ is an
immersed finite-energy $J$-holomorphic curve asymptotic to nondegenerate
Reeb orbits and satisfying
$$
\ind(u) > c_N(u).
$$
Then $u$ is Fredholm regular.   \qed
\end{prop}

The phenomenon underlying this transversality criterion also has useful
consequences for orientations of moduli spaces, in particular for spaces
of dimension~$0$, where orienting the moduli space simply means associating
a sign to each element.  We shall use the coherent orientations framework
described by Bourgeois and Mohnke \cite{BourgeoisMohnke}, based on earlier
work of Floer and Hofer \cite{FloerHofer:orientations}.
Notice that by \eqref{eqn:2cN}, a curve $u$ with
index~$0$ can satisfy $\ind(u) > c_N(u)$ only if $c_N(u) = -1$, in which
case $u$ must have genus~$0$ and all its asymptotic orbits have odd
Conley-Zehnder index.  The following result will play a key role in
\S\ref{sec:impressivePart} for
proving stability of $J$-holomorphic foliations under homotopies.

\begin{prop}
\label{prop:orientIndex0}
In the $4$-dimensional setting of Proposition~\ref{prop:automatic}, suppose 
$u_0$ and $u_1$ are two immersed $J$-holomorphic curves with the same 
number of punctures and identical sets of positive and/or negative 
asymptotic orbits, and also satisfying
$$
\ind(u_i) = 0, \qquad c_N(u_i) = -1, \qquad \text{ for $i=0,1$}.
$$
Then any choice of coherent orientations provided by \cite{BourgeoisMohnke}
assigns to $u_0$ and~$u_1$ the same sign.
\end{prop}

Let us state a corresponding result for the
$\RR$-invariant setting $(\RR \times M,J)$
with $J \in \jJ(\hH)$ before discussing the proofs of both.  
Since we usually want to consider $\mM(J) / \RR$
rather than $\mM(J)$, the important rigid objects in this space are 
represented by curves of index~$1$ that are not covers of trivial cylinders.  
The relation \eqref{eqn:windpi} implies that such curves can
satisfy $\ind(u) > c_N(u)$ only if $c_N(u) = 0$, in which case
\eqref{eqn:2cN} implies that the genus is zero and exactly one asymptotic
orbit has even Conley-Zehnder index.
Regular index~$1$ curves in $(\RR \times M,J)$ come in $1$-dimensional
moduli spaces of curves related to each other by $\RR$-translation, and the 
$\RR$-action thus induces a
tautological orientation on these spaces.  If a global orientation of 
$\mM(J)$ is given, one then associates a positive sign to $[u] \in \mM(J) / \RR$
if the given orientation matches the tautological one induced by the
$\RR$-action, and a negative sign otherwise.

We must briefly recall some specifics about asymptotic eigenfunctions.
If $u : \dot{S} \to \widehat{W}$
has a positive/negative puncture $z \in \Gamma^\pm$ asymptotic to a 
nondegenerate
orbit $\gamma$, then the asymptotic formula of \cite{HWZ:props1} and later
refinements in \cites{Mora,Siefring:asymptotics} describe
the approach of $u$ to $\gamma$ in terms of eigenfunctions of the 
\defin{asymptotic operator}
$$
\mathbf{A}_\gamma : \Gamma(\gamma^*\Xi_\pm) \to \Gamma(\gamma^*\Xi_\pm),
$$
a symmetric first-order differential operator that depends only
on~$\gamma$, and whose eigenfunctions were mentioned already 
in \S\ref{sec:Siefring}.  Parametrizing the trivial
cylinder over $\gamma$ as $u_\gamma : \RR \times S^1 \to \RR \times M_\pm :
(s,t) \mapsto (Ts,\gamma(t))$ and choosing a translation-invariant metric
to define the exponential map on $\RR \times M_\pm$, one can find 
coordinates $(s,t) \in \RR_+ \times S^1$ for the cylindrical end approaching
$z \in \Gamma^\pm$, and a section $h_z$ of $u_\gamma^*\Xi_\pm$ such that
$$
u(s,t) = \exp_{u_\gamma(s,t)} h_z(s,t) \quad \text{ for $s$ close to $\pm\infty$},
$$
where
\begin{equation}
\label{eqn:asympFormula}
h_z(s,t) = e^{\lambda s} \left( e_z(t) + r_z(s,t) \right).
\end{equation}
Here $\lambda \in \RR$ is an eigenvalue of $\mathbf{A}_\gamma$ with
$\pm \lambda < 0$, $e_z \in \Gamma(\gamma^*\Xi_\pm)$ is a nontrivial section
belonging to the corresponding eigenspace, and $r_z \in \Gamma(u_\gamma^*\Xi_\pm)$
is a remainder term satisfying
$|r_z(s,t)| \to 0$ uniformly as $s \to \pm\infty$.  Let
$$
V_\gamma^\pm \subset \Gamma(\gamma^*\Xi_\pm)
$$
denote the eigenspace of $\mathbf{A}_\gamma$ with the largest negative 
eigenvalue if $z \in \Gamma^+$ or the smallest positive eigenvalue if 
$z \in \Gamma^-$.  We then use \eqref{eqn:asympFormula} to define the 
\defin{leading asymptotic eigenfunction} $\ev^\infty_z(u) \in V_\gamma^\pm$
of $u$ at $z$ by
$$
\ev^\infty_z(u) := \begin{cases}
e_z & \text{ if $e_z \in V_\gamma^\pm$},\\
0 & \text{ otherwise}.
\end{cases}
$$
The case $\ev^\infty_z(u) = 0$ occurs if and only if the exponential decay
rate in \eqref{eqn:asympFormula} is faster than the slowest rate allowed
by the spectrum of~$\mathbf{A}_\gamma$.  As implied by the notation,
$\ev^\infty_z$ can be thought of as an \defin{asymptotic evaluation map}
from the moduli space to a finite-dimensional space of eigenfunctions, and
we will treat is as such in \S\ref{sec:impressivePart},
cf.~Lemma~\ref{lemma:genAsymp}.

The following result can now be summarized by saying that for a pair of immersed
and automatically regular index~$1$ curves with the same asymptotic orbits
in the symplectization of a $3$-manifold,
their signs will match if and only if they each approach their unique
even orbit ``from the same side''.

\begin{prop}
\label{prop:orientIndex1}
Assume $M$ is a $3$-manifold with stable Hamiltonian structure 
$\hH = (\Omega,\Lambda)$, $J \in \jJ(\hH)$, and $u_0$ and $u_1$ are two
immersed $J$-holomorphic curves that are not covers of trivial cylinders,
have the same number of punctures and identical sets of positive and/or
negative asymptotic orbits, and satisfy
$$
\ind(u_i) = 1, \qquad c_N(u_i) = 0, \qquad \text{ for $i=0,1$}.
$$
Let $e_i \in V_\gamma^\pm$ for $i=0,1$ denote the leading asymptotic
eigenfunction of $u_i$ at its unique puncture asymptotic to an orbit $\gamma$
with even Conley-Zehnder index.  Then
$$
e_1 = \kappa e_0 \quad \text{ for some $\kappa \in \RR \setminus \{0\}$},
$$
and for any choice of coherent orientations provided by
\cite{BourgeoisMohnke}, the signs assigned to $[u_0]$ and $[u_1]$ as
elements of $\mM(J) / \RR$ match if and only if $\kappa > 0$.
\end{prop}

Both propositions will be proved by similar arguments.  To prepare for this,
we need to recall a few details from \cite{BourgeoisMohnke} and
\cite{Wendl:automatic}; we shall use notation consistent with the latter
reference.

To any finite-energy $J$-holomorphic curve $u : \dot{S} \to \widehat{W}$ with 
punctures $\Gamma = \Gamma^+ \cup \Gamma^-$ positively/negatively asymptotic to 
Reeb orbits $\{\gamma_z\}_{z \in \Gamma}$, one can associate a Fredholm
operator
$$
\mathbf{D}_u : W^{1,p,\delta}\big(u^*T\widehat{W}\big) \oplus V_\Gamma 
\to L^{p,\delta}\big(\overline{\Hom}_\CC(T\dot{S},u^*T\widehat{W})\big),
$$
called the \defin{linearized Cauchy-Riemann operator} at~$u$.  Here $p > 2$,
and $W^{1,p,\delta}(u^*T\widehat{W})$ denotes the Banach space of
Sobolev class $W^{1,p}$ sections $\eta$ of $u^*T\widehat{W}$ satisfying the 
exponential decay condtion 
$e^{\delta s} \eta \in W^{1,p}([0,\infty) \times S^1)$
in holomorphic cylindrical coordinates $(s,t) \in [0,\infty) \times S^1$
near each puncture, where $\delta > 0$ is a small constant.  The space
$V_\Gamma \subset \Gamma(u^*T\widehat{W})$ is of dimension $2\#\Gamma$ and
consists of smooth sections that are constant near infinity in a suitable
choice of trivialization.  Since $\mathbf{D}_u$ is Fredholm, it has a
determinant line
$$
\det(\mathbf{D}_u) = \Lambda^{\max} \ker \mathbf{D}_u \otimes
(\Lambda^{\max} \coker \mathbf{D}_u)^*.
$$
The asymptotic form of $\mathbf{D}_u$ near each
puncture $z \in \Gamma^\pm$ is determined by the 
asymptotic operator $\mathbf{A}_{\gamma_z}$.
The procedure of \cite{BourgeoisMohnke} for defining orientations is then
to orient the determinant line bundles over the topological spaces of
isomorphism classes of Cauchy-Riemann type Fredholm operators of the above
form with fixed asymptotic operators at the punctures.  This is done so as
to make the orientations compatible with certain linear gluing operations,
so that the resulting orientations are called \emph{coherent}.
They give rise to orientations for spaces of $J$-holomorphic curves in the
following way.  If $u : (\dot{S},j) \to (\widehat{W},J)$ is Fredholm regular, 
then the implicit function theorem
gives the moduli space $\mM(J)$ of unparametrized $J$-holomorphic curves the
structure of a smooth orbifold of dimension $\ind(u)$ near~$u$, with its
tangent space at $u$ idenified with
\begin{equation}
\label{eqn:tangentModuli}
T_u \mM(J) = \ker D\dbar_J(j,u) \big/ \aut(\dot{S},j).
\end{equation}
Here $D\dbar_J(j,u)$ denotes the linearization of the \defin{nonlinear
Cauchy-Riemann operator}
$$
\dbar : \tT \times \bB \to \eE : (j,u) \mapsto Tu + J(u) \circ Tu \circ j,
$$
where $\bB$ is a Banach manifold of $W^{1,p}$-smooth maps $\dot{S} \to \widehat{W}$
whose tangent space at $u$ is the domain of $\mathbf{D}_u$,
$\tT$ is a smooth family of complex structures on $\dot{S}$ parametrizing
a neighborhood of $[j]$ in Teichm\"uller space, and $\eE \to \tT \times \bB$
is a smooth Banach space bundle whose fiber over $(j,u)$ is the
target space of~$\mathbf{D}_u$.  The space $\aut(\dot{S},j)$ is the Lie
algebra of the automorphism group of $(\dot{S},j)$, which embeds into
$\ker D\dbar(j,u)$ via the map taking vector fields $X$ on $\dot{S}$ to
sections $Tu(X)$ of $u^*T\widehat{W}$.  Since $\aut(\dot{S},j)$ is naturally
a complex vector space, any orientation of $\ker D\dbar_J(j,u)$ gives rise
to an orientation of $T_u \mM(J)$, thus it suffices to orient the determinant
line of $D\dbar(j,u)$, which is equivalent to orienting its kernel since
$D\dbar(j,u)$ is assumed surjective in the Fredholm regular case.  This
operator takes the form
$$
\mathbf{L}_{j,u} := D\dbar(j,u) 
: T_j \tT \oplus T_u \bB \to \eE_{(j,u)} : (y,\eta) \mapsto
\mathbf{D}_u \eta + J(u) \circ Tu \circ y.
$$
Any continuous family of $J$-holomorphic curves gives rise to a continuous
family of Fredholm operators of this form, all of which can be retracted
through Fredholm operators to the corresponding 
linearized Cauchy-Riemann operators via the homotopy
\begin{equation}
\label{eqn:detRetraction}
\mathbf{L}^s_{j,u}(y,\eta) := 
\mathbf{D}_u \eta + s \, J(u) \circ Tu \circ y, \qquad s \in [0,1].
\end{equation}
Taking $s=0$, we have $\ker \mathbf{L}^0_{j,u} = T_j\tT \oplus \ker \mathbf{D}_u$
and $\coker \mathbf{L}^0_{j,u} = \coker \mathbf{D}_u$.  Since Teichm\"uller
space is also naturally complex, $T_j\tT$ has a canonical orientation, so
that the orientation of $\det(\mathbf{D}_u)$ defined in \cite{BourgeoisMohnke}
induces an orientation of $\det(\mathbf{L}_{j,u}^0)$, and we use the
homotopy $\{ \mathbf{L}_{j,u}^s \}_{s \in [0,1]}$ to define from this an
orientation of $\det(\mathbf{L}_{j,u}^1)$, therefore orienting
$T_u \mM(J)$.

Recall now from \cite{Wendl:automatic} that if $u : \dot{S} \to \widehat{W}$
is immersed and $N_u \to \dot{S}$ denotes its normal bundle, the natural 
complex bundle splitting $u^*T\widehat{W} = T\dot{S} \oplus N_u$ decomposes 
$\mathbf{D}_u$ in block form as
\begin{equation}
\label{eqn:Dudecomp}
\mathbf{D}_u = \begin{pmatrix}
\mathbf{D}_u^T & \mathbf{D}_u^{NT} \\
0 & \mathbf{D}_u^N
\end{pmatrix},
\end{equation}
where $\mathbf{D}_u^T$ and $\mathbf{D}_u^N$ are real-linear Cauchy-Riemann 
type operators on $T\dot{S}$ and $N_u$ respectively, and the latter is
called the \defin{normal Cauchy-Riemann operator} of~$u$.  We can extend
the splitting $u^*T\widehat{W} = T\dot{S} \oplus N_u$ over the circle
compactification of the end near each puncture $z \in \Gamma^\pm$
so that $N_u$ at the end is identified with $\gamma_z^*\Xi_\pm$ and
$T\dot{S}$ is identified with the trivial complex subbundle generated by $\RR$ 
and the Reeb vector field.  The space $V_\Gamma$ can then be chosen as a
space of sections of $T\dot{S}$, and $\mathbf{D}_u^T$ identified with the
\emph{standard} Cauchy-Riemann operator on $\dot{S}$,
$$
\mathbf{D}_{(\dot{S},j)} : W^{1,p,\delta}(T\dot{S}) \oplus V_\Gamma \to
L^{p,\delta}(\overline{\End}_\CC(T\dot{S})),
$$
i.e.~the linearization at the identity of the nonlinear Cauchy-Riemann
operator for asymptotically cylindrical holomorphic maps 
$(\dot{S},j) \to (\dot{S},j)$.  The cokernel of this operator has a natural
identification with the tangent space to Teichm\"uller space, so one can
always assume that the map
\begin{equation}
\label{eqn:TeichmuellerTransversality}
\begin{split}
T_j \tT \oplus \left( W^{1,p,\delta}(T\dot{S}) \oplus V_\Gamma \right) &\to
L^{p,\delta}(\overline{\End}_\CC(T\dot{S})) \\
(y,X) &\mapsto jy + \mathbf{D}_{(\dot{S},j)} X
\end{split}
\end{equation}
is surjective.  Writing a section of $u^*T\widehat{W} = T\dot{S} \oplus N_u$
as $(X,\eta)$, this gives a decomposition of 
$\mathbf{L}_{j,u} = D\dbar_J(j,u)$ as
\begin{equation}
\label{eqn:Ludecomp}
\mathbf{L}_{j,u}(y,X,\eta) = \left( jy + \mathbf{D}_{(\dot{S},j)} X +
\mathbf{D}_u^{NT} \eta , \mathbf{D}_u^N \eta \right),
\end{equation}
showing that $\mathbf{L}_{j,u}$ is surjective if and only if
$\mathbf{D}_u^N$ is surjective.  Note also that $\ker \mathbf{D}_{(\dot{S},j)}$
is naturally isomorphic to $\aut(\dot{S},j)$, so injecting the latter into
$\ker \mathbf{L}_{j,u}$ as the subspace $\{0\} \oplus 
\ker \mathbf{D}_{(\dot{S},j)} \oplus \{0\}$ and using \eqref{eqn:tangentModuli},
we obtain from this expression a natural isomorphism
\begin{equation}
\label{eqn:tangentModuliNormal}
T_u \mM(J) = \ker \mathbf{L}_{j,u} \big/ \ker \mathbf{D}_{(\dot{S},j)}
\to \ker \mathbf{D}_u^N : [(y,X,\eta)] \mapsto \eta.
\end{equation}

\begin{proof}[Proof of Proposition~\ref{prop:orientIndex0}]
Assume $u_0$ and $u_1$ are as stated in the proposition.  Since
\eqref{eqn:2cN} implies that both have genus zero, their domains are
diffeomorphic, and $\ind(u_0) = \ind(u_1) = 0$ implies that the complex
line bundles $N_{u_0}$ and $N_{u_1}$ also admit
a bundle isomorphism that is asymptotic to their canonical identification
at the ends.  Let us assume first for simplicity that $u_0$ and $u_1$ have 
isomorphic conformal structures on their domains, so we can represent them 
by the same complex structure $j$ on $\dot{S}$ and fix a single slice
$\tT$ parametrizing the Teichm\"uller space near~$[j]$.  Then after identifying
both $N_{u_0}$ and $N_{u_1}$ with some fixed complex line bundle
$E \to \dot{S}$, we can assume $\mathbf{D}_0^N := \mathbf{D}_{u_0}^N$ and
$\mathbf{D}_1^N := \mathbf{D}_{u_1}^N$ are Cauchy-Riemann type operators
on the same bundle over the same domain
$$
\mathbf{D}_i^N : W^{1,p,\delta}(E) \to L^{p,\delta}(\overline{\Hom}_\CC(
T\dot{S},E)), \qquad  \text{ $i = 0,1$},
$$
and these are related to $\mathbf{D}_i := \mathbf{D}_{u_i}$ and
$\mathbf{L}_i := \mathbf{L}_{j,u_i}$ for $i=0,1$ as in 
\eqref{eqn:Dudecomp} and \eqref{eqn:Ludecomp} respectively.
Now since the space of Cauchy-Riemann type operators with fixed asymptotic
orbits is affine, we can choose a homotopy $\{\mathbf{D}_\tau^N\}_{\tau \in [0,1]}$
from $\mathbf{D}_0^N$ to $\mathbf{D}_1^N$, which induces homotopies 
$\{\mathbf{D}_\tau\}$ from
$\mathbf{D}_0$ to $\mathbf{D}_1$ and $\{\mathbf{L}_\tau\}$ from
$\mathbf{L}_0$ to~$\mathbf{L}_1$.  By \cite{Wendl:automatic}*{Proposition~2.2},
the operators $\mathbf{D}_\tau^N$ are always surjective, hence so 
are~$\mathbf{L}_\tau$.  Since $\ind(u_0) = \ind(u_1) = 0$, the kernel
of $\mathbf{L}_\tau$ is therefore identical to the subspace
$\aut(\dot{S},j)$ for every $\tau \in [0,1]$.

Now suppose a choice of coherent orientations as constructed in
\cite{BourgeoisMohnke} is given.  This assigns a continuously varying
orientation to the determinant of $\mathbf{D}_\tau$ for each $\tau \in [0,1]$.
In order to determine whether $u_0$ and $u_1$ have the same sign, one must 
consider the determinant line bundle over a $1$-parameter family of Fredholm 
operators from $\mathbf{L}_0$ to $\mathbf{L}_1$ constructed in three parts:
\begin{enumerate}
\item Retract $\mathbf{L}_0$ to $0 + \mathbf{D}_0$ as in 
\eqref{eqn:detRetraction};
\item Follow the homotopy $\mathbf{D}_\tau$ 
from $\mathbf{D}_0$ to $\mathbf{D}_1$;
\item Unretract from $0 + \mathbf{D}_1$ to $\mathbf{L}_1$, again as in
\eqref{eqn:detRetraction}.
\end{enumerate}
The retractions for $i=0,1$ transfer the given orientations of
$\det(\mathbf{D}_i)$ to orientations of $\det{\mathbf{L}_i} =
\Lambda^{\max} \ker \mathbf{L}_i$, and the sign of $u_i$ depends on whether
the latter orientations match the canonical orientation of $\aut(\dot{S},j)$
as a complex vector space.  (Note that if $\aut(\dot{S},j)$ is trivial,
then it means the $\mathbf{L}_i$ are isomorphisms, so that $\det(\mathbf{L}_i)$ 
is tautologically equal to $\RR$ and the sign of $u_i$ depends on whether the
induced orientation of $\RR$ matches the tautological one.)  Since the
orientations of $\det(\mathbf{D}_\tau)$ must be continuous in~$\tau$,
following the three-part homotopy from $\mathbf{L}_0$ to $\mathbf{L}_1$
therefore determines the relationship between the signs of $u_0$ and~$u_1$.
We notice
however that the retraction from $\mathbf{L}_\tau$ to $0 + \mathbf{D}_\tau$
can also be performed for every $\tau \in [0,1]$, hence the three-part
homotopy can be deformed with fixed endpoints to 
$\{\mathbf{L}_\tau\}_{\tau \in [0,1]}$.  The latter is a homotopy through
surjective operators, and for any continuous family of orientations of
$\ker \mathbf{L}_\tau$, either all or none of them match the orientation of
$\aut(\dot{S},j)$.  This proves that the signs of $u_0$ and $u_1$ are
equal as claimed.

If $u_0$ and $u_1$ have inequivalent conformal structures $j_0$ and $j_1$ on 
their domains, then the above argument must be supplemented by an
initial step choosing a continuous deformation of Cauchy-Riemann type
operators to accompany a deformation from $j_0$ to $j_1$ in the space of
complex structures and a simultaneous deformation of the corresponding
Teichm\"uller slices.  This can always be done since the space of complex
structures on $\dot{S}$ compatible with its orientation is contractible.
The key point is that \cite{Wendl:automatic}*{Prop.~2.2} always guarantees
surjectivity for the restriction of the Cauchy-Riemann type operators to
a line bundle isomorphic to~$N_{u_0}$ with the same asymptotic operators.
\end{proof}

\begin{proof}[Proof of Proposition~\ref{prop:orientIndex1}]
Most steps are the same as for Prop.~\ref{prop:orientIndex0}, so let us
merely clarify the differences.  The normal operators
$\mathbf{D}_\tau^N$ now have index~$1$ and have $1$-dimensional kernels
since \cite{Wendl:automatic}*{Prop.~2.2} again implies that they are always
surjective.
The normal Chern number $c_N(u_0) = c_N(u_1) = 0$ can in this case
be interpreted as the relative first Chern number of $E \to \dot{S}$
with respect to the asymptotic trivializations that determine the
extremal winding for holomorphic sections, hence the nontrivial elements
$\eta \in \ker \mathbf{D}_\tau^N$ are guaranteed to be nowhere zero and
to have extremal winding at every end (cf.~\cite{Wendl:automatic}*{\S 2.2}).
Let $z \in \Gamma^\pm$ denote the unique puncture for both $u_0$ and $u_1$
at which the asymptotic orbit $\gamma$ has even Conley-Zehnder index.
The extremal eigenspace $V_\gamma^\pm$ is then $1$-dimensional as a
consequence of \eqref{eqn:CZwinding} since by \cite{HWZ:props2},
exactly two eigenvalues of $\mathbf{A}_\gamma$ counting multiplicity
have eigenfunctions with any given winding.
The extremal winding condition thus implies that any nontrivial
$\eta \in \ker \mathbf{D}_\tau^N$ has a nontrivial asymptotic eigenfunction
in $V_\gamma^\pm$ at the puncture~$z$.  A continuous family of such
sections for $\tau \in [0,1]$ therefore determines a continuous path
in $V_\gamma^\pm \setminus \{0\}$.

Since $0 \le \asympdef(u_i) \le c_N(u_i) = 0$ for $i=0,1$ by
\eqref{eqn:windpi}, the leading asymptotic
eigenfunctions $e_i \in V_\gamma^\pm$ of $u_i$ at~$z$ are also
nonzero.  Let $\eta_i \in \ker \mathbf{D}_i^N$ for $i=0,1$ denote the
canonical generators that are identified via
\eqref{eqn:tangentModuliNormal} with the infinitesimal generator of the
$\RR$-translation action on~$u_i$.  The asymptotic formula 
\eqref{eqn:asympFormula} implies that these have asymptotic eigenfunctions at 
$z$ of the form $\kappa_i e_i \in V_\gamma^\pm$ for some constants
$\kappa_i > 0$.  Hence there exists a continuous family 
$\{\eta_\tau\}_{\tau \in [0,1]}$ of nontrivial
generators of $\ker \mathbf{D}_\tau^N$ connecting $\eta_0$ to $\eta_1$
if and only if $e_0$ and $e_1$ are positive multiples of one another.

With this understood, the signs of $[u_i] \in \mM(J) / \RR$ for
$i=0,1$ are determined as follows.  The retractions as in 
\eqref{eqn:detRetraction} from $\mathbf{L}_i$ to $0 + \mathbf{D}_i$
transfer the given orientation of $\det(\mathbf{D}_i)$ to an orientation
of $\det(\mathbf{L}_i) = \Lambda^{\max} \ker \mathbf{L}_i$, hence
orienting $\ker\mathbf{L}_i$.  Dividing the latter by the canonically
oriented subspace $\aut(\dot{S},j)$ and using \eqref{eqn:tangentModuliNormal}
then induces an orientation of the $1$-dimensional space
$\ker \mathbf{D}_i^N$, for which $\eta_i$ is either positive or negative,
so this is the sign of~$[u_i]$.  To relate these signs to each other, 
one follows the
orientations along the ``three-part homotopy'' of Fredholm operators 
described in the proof of Prop.~\ref{prop:orientIndex0}, which is again
homotopic to a path $\mathbf{L}_\tau$ consisting of surjective operators
with kernels isomorphic to $\aut(\dot{S},j) \oplus \ker \mathbf{D}_\tau^N$.
One therefore obtains the same sign for $[u_0]$ and $[u_1]$ if and only if
there exists a path $\{\eta_\tau\}$ of generators of $\ker \mathbf{D}_\tau^N$
as described in the previous paragraph, which reduces to the question of
whether $e_0$ and $e_1$ lie in the same component of $V_\gamma^\pm 
\setminus \{0\}$.
\end{proof}

\section{A symplectic model of a cylindrical end}
\label{sec:model}

Throughout this section, assume $(M,\xi)$ is a
closed connected contact $3$-manifold on which $\xi$ is 
supported by a spinal open book
$$
\boldsymbol{\pi} := \Big(\pi\spine : M\spine \to \Sigma,
\pi\paper : M\paper \to S^1 \Big).
$$
The purpose of this section is to construct a symplectic and almost
complex model of the half-symplectization $[0,\infty) \times M$ 
of $(M,\xi)$, designed such that given any hypothetical symplectic filling 
$(W,\omega)$ of $(M,\xi)$, we can define a symplectic completion of 
$(W,\omega)$ that contains an abundance of pseudoholomorphic curves.
The construction is an extension of the model collar neighborhood 
described in \cite{LisiVanhornWendl1}*{\S 4}, which views
$(M,\xi)$ as a smoothing of the boundary
(with corners) of a noncompact $4$-manifold
$E$ whose boundary has two smooth faces 
$$
\p E = \p_v E \cup \p_h E,
$$
interpreted as the vertical and horizontal boundaries respectively of a
(locally defined) symplectic fibration.  Here it is not necessary to
assume $(M,\xi)$ admits a symplectic filling, as we can instead
identify it with the contact-type boundary of a collar neighborhood
$(-1,0] \times M$ in its own symplectization.
By attaching cylindrical ends to
the fibers of the aforementioned fibration and also extending it over cylindrical ends
attached to the base, one obtains the \emph{double completion}
$\widehat{E}$, which contains $E$ as a bounded subdomain.  We will 
endow $\widehat{E}$ with a Liouville structure $\lambda$
and compatible almost complex structure~$J_+$ having the following properties:
\begin{enumerate}
\item $J_+$ admits a suitable exhausting $J_+$-convex function and thus defines
almost Stein structures on suitable subdomains of~$\widehat{E}$, 
homotopic to the Liouville structure~$\lambda$;
\item The corner in $\p E$ can be smoothed to produce a contact hypersurface
contactomorphic to $(M,\xi)$;
\item A neighborhood of infinity in $\widehat{E}$ can be identified with the
half-symplectization $[0,\infty) \times M$ of a suitable stable Hamiltonian
structure $\hH$ on $M$, with $J_+ \in \jJ(\hH)$;
\item The symplectization $\RR \times M$ of the aforementioned stable
Hamiltonian structure admits a
foliation by embedded $J_+$-holomorphic curves that project to $M\paper$
as the pages of~$\boldsymbol{\pi}$;
\item $\widehat{E}$ also contains embedded $J_+$-holomorphic curves that 
intersect the holomorphic pages transversely and project to $M\spine$
as sections of $\pi\spine : M\spine \to \Sigma$, i.e.~``holomorphic vertebrae''.
\end{enumerate}

Since it is only semi-standard, we recall the following definition from
\cite{LisiVanhornWendl1}*{\S 1.1} of a geometric structure that is intermediate
between Stein and Weinstein structures.

\begin{defn}
\label{defn:almostStein}
An \defin{almost Stein structure} $(J,f)$
on a smooth compact oriented manifold $W$ with boundary and corners consists of an
almost complex structure $J : TW \to TW$ and smooth function $f : W \to \RR$ such that
$\lambda := - df \circ J$ is a Liouville form with $J$ tamed by~$d\lambda$,
and $\lambda$ restricts to a positive contact form on every smooth
face of~$\p W$.  (Note that we do not require $f|_{\p W}$ to be constant since $\p W$
may have corners, but it is automatic that the Liouville vector field dual to
$\lambda$ is gradient-like for $f$ and outwardly transverse to every face of~$\p W$.) 
\end{defn}

It will be immediate that the construction outlined above can also be modified by 
replacing the symplectic form $d\lambda$ with $C\, d\lambda + \eta$ for any 
closed $2$-form
$\eta$ and a sufficiently large constant $C > 0$; this makes the smoothing
of $\p E$ a \emph{weakly} contact hypersurface, and makes it possible to
attach the model on top of weak fillings of~$(M,\xi)$.  We will see that
this trick only interferes with the construction of the stable
Hamiltonian structure and resulting $J_+$-holomorphic curves if the symplectic 
structure is non-exact on the spine, so when this is not the case, we obtain 
nontrivial moduli spaces in completions of weak fillings and will use them in
\S\ref{sec:impressivePart} for the proofs of Theorems~\ref{thm:classification},
\ref{thm:weak} and \ref{thm:SteinDeformation}.  After explaining the
construction of the stable Hamiltonian structure and holomorphic curves
in \S\ref{sec:model}, it will be the purpose of
\S\ref{sec:holOpenBook} to extend the construction to the case
$\p M \ne \emptyset$, and then to show that the data on $M$ can be
perturbed to a contact structure isotopic to~$\xi$ and to explore the
consequences of this.
For genus zero pages, the perturbation results in a finite energy foliation
just as for planar open books (cf.~\cite{Wendl:openbook}), and this foliation
will be used in \S\ref{sec:invariants} to prove 
Theorems~\ref{thm:algTorsion}, \ref{thm:ECHvanish} and~\ref{thm:Umap}.

\begin{remark}
Some higher-dimensional analogues of the double completion model 
(inspired by an earlier draft of the present paper) appear
in \cites{Moreno:thesis,Moreno:algebraicGiroux,Moreno:SFTcomputations}.
\end{remark}

\subsection{Collar neighborhoods and smoothed hypersurfaces}
\label{sec:collars}

We will need to use the following notation originating in
\cite{LisiVanhornWendl1}*{\S 2.2}.  

We denote collar neighborhoods
of the boundaries in $\Sigma$, $M\spine$ and $M\paper$ by
$\nN(\p\Sigma) \cong (-1,0] \times \p\Sigma$, 
$\nN(\p M\spine) \cong (-1,0] \times \p M\spine$ and 
$\nN(\p M\paper) \cong (-1,0] \times \p M\paper$
respectively, where it is assumed that
$\pi\spine^{-1}(\nN(\p\Sigma)) = \nN(\p M\spine)$, and a trivialization
of $\pi\spine : M\spine \to \Sigma$ has been fixed so as to
identify $\nN(\p M\spine)$ with $\nN(\p\Sigma) \times S^1$.
Fixing an identification of each component of $\p\Sigma$
with $S^1$ then determines coordinates
\begin{equation*}
\begin{split}
(s,\phi) \in (-1,0] \times S^1 &\subset (-1,0] \times \p \Sigma = 
\nN(\p\Sigma) \subset \Sigma,\\
(s,\phi,\theta) \in (-1,0] \times S^1 \times S^1 &\subset
(-1,0] \times \p M\spine = \nN(\p M\spine) \subset M\spine,
\end{split}
\end{equation*}
which satisfy $\pi\spine(s,\phi,\theta) = (s,\phi) \in \nN(\p\Sigma)$
on $\nN(\p M\spine)$.  Making suitable choices of collar coordinates
$(t,\theta) \in (-1,0] \times S^1$ near the boundary of a fiber
of $\pi\paper : M\paper \to S^1$ and adjusting the monodromy $\mu$ by
an isotopy so that $\mu(t,\theta)=(t,\theta)$ in these coordinates
(but allowing a permutation of boundary components), we also 
identify each component of
$\nN(\p M\paper)$ with $S^1 \times (-1,0] \times S^1$ and thus
define coordinates
$$
(\phi,t,\theta) \in S^1 \times (-1,0] \times S^1 \subset
\nN(\p M\paper) \subset M\paper
$$
in which 
\begin{equation}
\label{eqn:localMult}
\pi\paper(\phi,t,\theta) = m\phi \in S^1 \quad\text{ on }\quad S^1 \times (-1,0] \times S^1 \subset \nN(\p M\paper)
\end{equation}
for some $m \in \NN$.  Here $m$ is the \emph{multiplicity} of 
$\pi\paper$ at the adjacent boundary component of the spine
(see Definition~\ref{defn:multiplicity}),
and it may have distinct values
on different connected components of $\nN(\p M\paper)$.
The coordinates defined on $\nN(\p M\spine)$ and $\nN(\p M\paper)$
should be assumed consistent with each other in the sense that
the respective $\phi$- and $\theta$-coordinates match each other on
$\p M\spine = \p M\paper$.

The model collar neighborhood of $M \subset (-1,0] \times M$ was defined
in \cite{LisiVanhornWendl1}*{\S 4.1} as
$$
E := \big( (-1,0] \times M\spine\big) \cup_\Phi \big( (-1,0] \times M\paper\big),
$$
where $(-1,0] \times \nN(\p M\spine)$ is glued to $(-1,0] \times \nN(\p M\paper)$
via the diffeomorphism
$$
(-1,0] \times \overbrace{(-1,0] \times \p M\spine}^{\nN(\p M\spine)} \stackrel{\Phi}{\longrightarrow} 
(-1,0] \times \overbrace{(-1,0] \times \p M\paper}^{\nN(\p M\paper)} :
(t,s,x) \mapsto (s,t,x)
$$
for $x \in \p M\spine = \p M\paper$.  This object is depicted as the more
darkly shaded region in Figure~\ref{fig:doubleCompletion}, along with the
vertical and horizontal boundaries
$$
\p_v E := \{0\} \times M\paper \subset E, \qquad
\p_h E := \{0\} \times M\spine \subset E,
$$
and their respective collar neighborhoods
$$
\nN(\p_v E) := (-1,0] \times M\paper \subset E, \qquad
\nN(\p_h E) := (-1,0] \times M\spine \subset E,
$$
whose intersection (a neighborhood of the corner $\p_v E \cap \p_h E$)
we sometimes denote by
$$
\nN(\p_v E \cap \p_h E) := \nN(\p_v E) \cap \nN(\p_h E) \subset E.
$$
The definition of the gluing map $\Phi$ gives rise to well-defined coordinates
$$
(s,\phi,t,\theta) \in (-1,0] \times S^1 \times (-1,0] \times S^1 \subset
\nN(\p_v E \cap \p_h E)
$$
on each component of $\nN(\p_v E \cap \p_h E)$.

On the collars $\nN(\p_h E)$ and $\nN(\p_v E)$ there are natural fibrations
\begin{equation*}
\begin{split}
\nN(\p_h E) = (-1,0] \times (\Sigma \times S^1) &\stackrel{\Pi_h}{\longrightarrow} \Sigma : 
\big(t,(z,\theta)\big) \mapsto \pi\spine(z,\theta) = z,\\
\nN(\p_v E) = (-1,0] \times M\paper 
&\stackrel{\Pi_v}{\longrightarrow} (-1,0] \times S^1 : (s,x) \mapsto (s,\pi\paper(x)),
\end{split}
\end{equation*}
which can be written in the coordinates $(s,\phi,t,\theta)$ on 
$\nN(\p_v E \cap \p_h E)$ as
\begin{equation}
\label{eqn:fibrationCorner}
\Pi_h(s,\phi,t,\theta) = (s,\phi) \in\nN(\p\Sigma), \qquad
\Pi_v(s,\phi,t,\theta) = (s,m\phi) \in (-1,0] \times S^1.
\end{equation}
Since the fibers of these two fibrations match on the region of overlap,
they give rise to
a well-defined \defin{vertical subbundle}
$$
V E := \ker T\Pi_h \text{ or } \ker T\Pi_v \subset T E,
$$
which on $\nN(\p_h E)$ is spanned by the vector fields
$\p_t$ and~$\p_\theta$.  Figure~\ref{fig:doubleCompletion} is drawn so that
the fibers would be represented as vertical lines in the picture.

Observe that in light of the canonical identifications
$\p_v E = M\paper$ and $\p_h E = M\spine$, the boundary $\p E = \p_v E \cup \p_h E$
has a canonical identification with $M = M\paper \cup M\spine$, though it
cannot be regarded as a smooth submanifold due to the corner at $\p_v E \cap \p_h E$.
We can however smooth the corner to define a smooth hypersurface diffeomorphic
to~$M$.  Specifically,
choose a pair of smooth functions
$F , G : (-1,1) \to (-1,0]$ that satisfy the following conditions:
\begin{itemize}
\item $(F(\rho),G(\rho)) = (\rho,0)$ for $\rho \le - 1 / 4$;
\item $(F(\rho),G(\rho)) = (0,-\rho)$ for $\rho \ge 1 / 4$;
\item $G'(\rho) < 0$ for $\rho > - 1 / 4$;
\item $F'(\rho) > 0$ for $\rho < 1 / 4$.
\end{itemize}
Now let
$$
M^0 \subset E
$$
denote the smooth hypersurface obtained from $\p E$ by replacing 
$\p E \cap \nN(\p_v E \cap \p_h E)$ in $(s,\phi,t,\theta)$-coordinates with
\begin{equation}
\label{eqn:smoothing}
\left\{ (F(\rho),\phi, G(\rho),\theta)\ \Big|\ 
\phi,\theta \in S^1,\ -1 < \rho < 1 \right\};
\end{equation}
see Figure~\ref{fig:doubleCompletion}.
We shall present $M^0$ as the union of three open subsets
$$
M^0 = \widecheck{M}^0\paper \cup \widecheck{M}^0\corner \cup \widecheck{M}^0\spine,
$$
defined as follows (see Figure~\ref{fig:M0cover}):
\begin{itemize}
\item $\widecheck{M}^0\paper$ is the complement of $\{ t \ge -1/2\} \subset 
\nN(\p M\paper)$ in~$\p_v E$.
\item $\widecheck{M}^0\spine$ is the complement of $\{ s \ge -1/2\}$
in~$\p_h E$.
\item $\widecheck{M}^0\corner$ is the smoothing region \eqref{eqn:smoothing}, 
each of whose connected components carry
a positively oriented coordinate system $(\rho,\phi,\theta)$ identifying
it with $(-1,1) \times S^1 \times S^1$.
\end{itemize}
The portion of $\widecheck{M}^0\paper$ intersecting $\widecheck{M}^0\corner$ carries coordinates
$(\phi,t,\theta)$ that identify its connected components with
$S^1 \times (-1,-1/2) \times S^1$ and are related to the coordinates
on $\widecheck{M}^0\corner$ by $\rho = -t$.  
Similarly, the overlap of $\widecheck{M}^0\spine$ with $\widecheck{M}^0\corner$
carries coordinates $(s,\phi,\theta)$ that identify its components with
$(-1,-1/2) \times S^1 \times S^1$ and satisfy $s = \rho$.

It will also be convenient to define a second hypersurface
$$
M^- \subset E
$$
by translating $M^0$ a distance of $-3/4$ in both the $s$- and $t$-coordinates;
see Figure~\ref{fig:doubleCompletion}.  This
contains portions of the two hypersurfaces
$\{-3/4\} \times M\spine \subset \nN(\p_h E)$ and
$\{-3/4\} \times M\paper \subset \nN(\p_v E)$ and a translated copy of
\eqref{eqn:smoothing} replacing the neighborhood of their intersection.

\begin{figure}
    \begin{postscript}
\psfrag{...}{$\ldots$}
\psfrag{phE}{$\p_h E$}
\psfrag{-1}{$-1$}
\psfrag{s}{$s$}
\psfrag{t}{$t$}
\psfrag{N(phE)}{$\nN(\p_h E)$}
\psfrag{Nhat(phE)}{$\widehat{\nN}(\p_h E)$}
\psfrag{Nhat(pvEintphE)}{$\widehat{\nN}(\p_v E \cap \p_h E)$}
\psfrag{N(pvE)}{$\nN(\p_v E)$}
\psfrag{Nhat(pvE)}{$\widehat{\nN}(\p_v E)$}
\psfrag{pvE}{$\p_v E$}
\psfrag{Sigma}{$\Sigma$}
\psfrag{Sigmahat}{$\widehat{\Sigma}$}
\psfrag{pSigma}{$\p\Sigma$}
\psfrag{N(pSigma)}{$\nN(\p\Sigma)$}
\psfrag{Nhat(pSigma)}{$\widehat{\nN}(\p\Sigma)$}
\psfrag{Pih}{$\Pi_h$}
\psfrag{M0}{$M^0$}
\psfrag{M-}{$M^-$}
\includegraphics{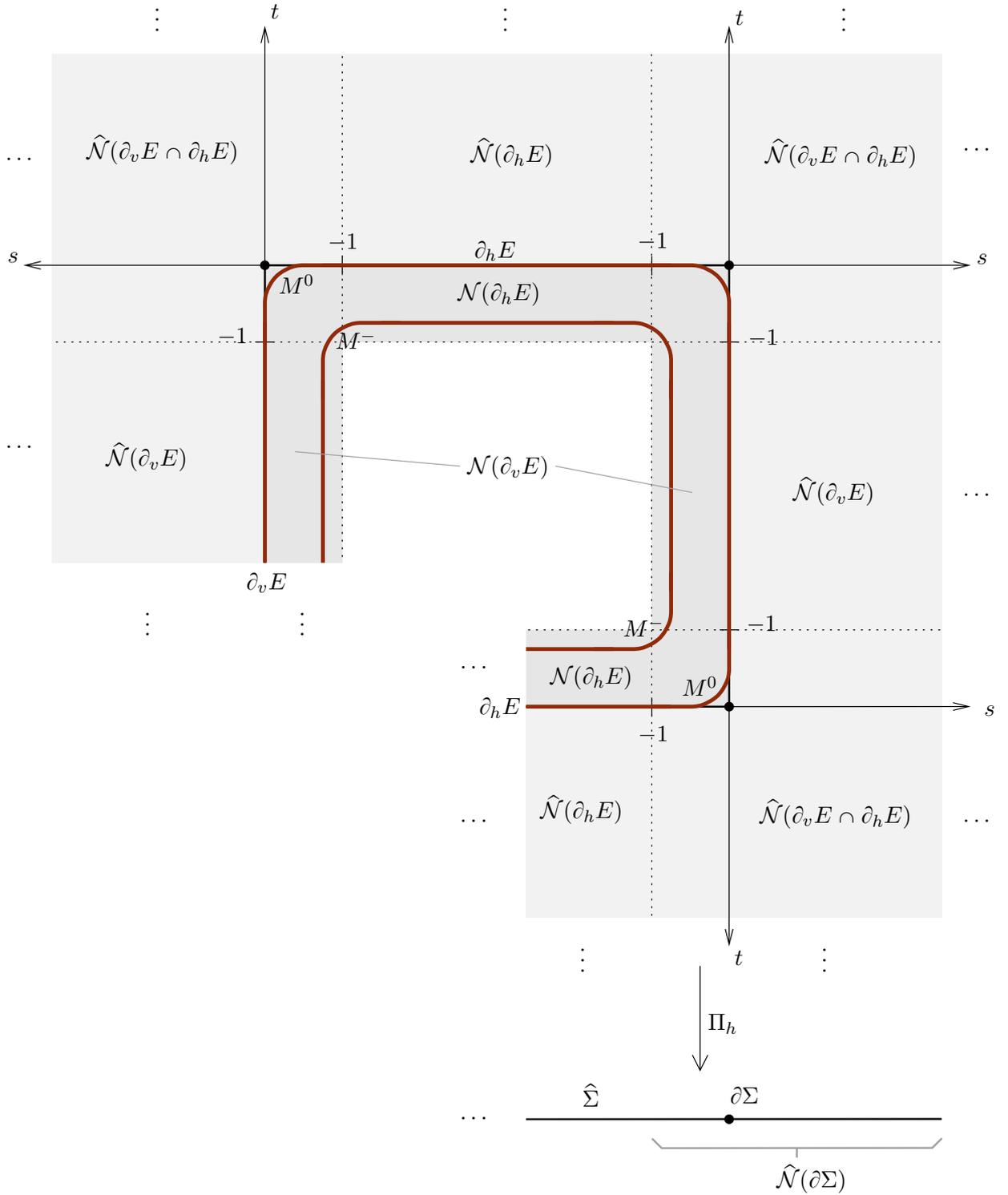}
\end{postscript}
\caption{\label{fig:doubleCompletion}
The darkly shaded region is the model collar neighborhood $E$ (with boundary
$\p E = \p_v E \cup \p_h E$ and corner $\p_v E \cap \p_h E$),
together with the smooth hypersurfaces $M^0,M^- \subset E$
defined in \S\ref{sec:collars}.  The lightly shaded region represents the
rest of the double completion $\widehat{E} \supset E$ as defined in 
\S\ref{sec:doubleCompletion}.}
\end{figure}

\begin{figure}
    \begin{postscript}
\psfrag{N(phE)}{$\nN(\p_h E)$}
\psfrag{N(pvE)}{$\nN(\p_v E)$}
\psfrag{-1/2}{$-1/2$}
\psfrag{-1}{$-1$}
\psfrag{t}{$t$}
\psfrag{s}{$s$}
\psfrag{pvE}{$\p_v E$}
\psfrag{phE}{$\p_h E$}
\psfrag{Mcheck0P}{$\widecheck{M}^0\paper$}
\psfrag{Mcheck0S}{$\widecheck{M}^0\spine$}
\psfrag{Mcheck0I}{$\widecheck{M}^0\corner$}
\psfrag{...}{$\ldots$}
\includegraphics{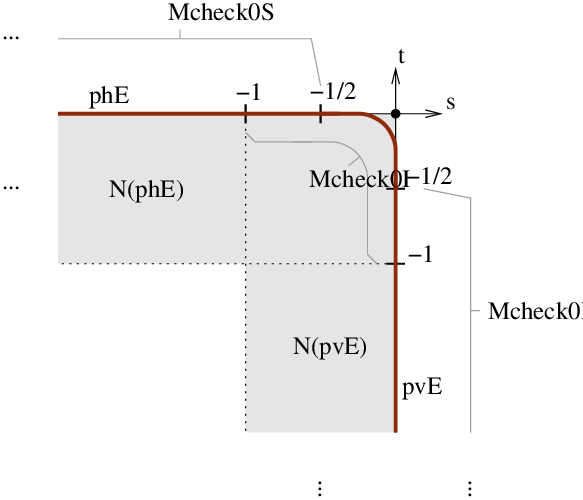}
\end{postscript}
\caption{\label{fig:M0cover}
The decomposition $M^0 = \widecheck{M}^0\paper \cup \widecheck{M}^0\corner \cup \widecheck{M}^0\spine$.}
\end{figure}

\subsection{The double completion}
\label{sec:doubleCompletion}

The collar neighborhoods $\nN(\p \Sigma) \subset \Sigma$,
$\nN(\p M\spine) \subset M\spine$ and $\nN(\p M\paper) \subset M\paper$
give rise to \emph{completions}, constructed in each case by attaching 
cylindrical ends
and extending the coordinate $s$ or $t$ to take values
in $(-1,\infty)$: we shall indicate each of these completions 
by placing hats over the relevant symbol, hence
\begin{align*}
\widehat{\Sigma} &:= \Sigma \cup_{\p \Sigma} \left( [0,\infty) \times \p \Sigma \right), &\\
\widehat{M}\spine &:= M\spine \cup_{\p M\spine} \left( [0,\infty) \times \p M\spine \right),& \\
\widehat{M}\paper &:= M\paper \cup_{\p M\paper} \left( [0,\infty) \times \p M\paper \right),
\end{align*}
and the collars become cylindrical ends whose components have coordinates
\begin{equation*}
\begin{split}
(s,\phi) \in (-1,\infty) \times S^1 \subset (-1,\infty) \times \p\Sigma
&=: \widehat{\nN}(\p\Sigma) \subset \widehat{\Sigma},\\
(s,\phi,\theta) \in (-1,\infty) \times S^1 \times S^1 \subset 
(-1,\infty) \times \p M\spine &=:
\widehat{\nN}(\p M\spine) \subset \widehat{M}\spine, \\
(\phi,t,\theta) \in S^1 \times (-1,\infty) \times S^1 \subset
(-1,\infty) \times \p M\paper &=:
\widehat{\nN}(\p M\paper) \subset \widehat{M}\paper.
\end{split}
\end{equation*}
We will continue to denote by $\pi\spine$ and $\pi\paper$ the natural 
extensions of the fibrations to $\widehat{M}\spine = \widehat{\Sigma} \times S^1
\to \widehat{\Sigma}$ and
$\widehat{M}\paper \to S^1$ respectively.  The \defin{double completion}
$\widehat{E}$ of $E$ is defined as
$$
\widehat{E} := \big( (-1,\infty) \times \widehat{M}\spine \big) \cup_{\widehat{\Phi}}
\big( (-1,\infty) \times \widehat{M}\paper \big),
$$
where $\widehat{\Phi}$ is the obvious extension of the previous gluing map to a
diffeomorphism
\begin{equation*}
\begin{split}
(-1,\infty) \times \overbrace{(-1,\infty) \times \p M\spine}^{\widehat{\nN}(\p M\spine)} &\stackrel{\widehat{\Phi}}{\to} 
(-1,\infty) \times \overbrace{(-1,\infty) \times \p M\paper}^{\widehat{\nN}(\p M\paper)}, \\
(t,s,x) &\mapsto (s,t,x).
\end{split}
\end{equation*}
This noncompact $4$-manifold without boundary contains $E$ as a bounded
subdomain, and the collars $\nN(\p_v E)$ and $\nN(\p_h E)$ are then
bounded subsets of the enlarged subsets
$$
\widehat{\nN}(\p_v E) := (-1,\infty) \times \widehat{M}\paper, \qquad
\widehat{\nN}(\p_h E) := (-1,\infty) \times \widehat{M}\spine,
$$
with the $s$- and $t$-coordinates now taking values in $(-1,\infty)$.
Their intersection is the so-called \defin{diagonal end}
$$
\widehat{\nN}(\p_v E \cap \p_h E) := \widehat{\nN}(\p_v E)
\cap \widehat{\nN}(\p_h E),
$$
whose connected components carry coordinates $(s,\phi,t,\theta)$ identifying
them with $(-1,\infty) \times S^1 \times (-1,\infty) \times S^1$.
The fibrations $\Pi_h$ and $\Pi_v$ have natural extensions
\begin{equation*}
\begin{split}
\widehat{\nN}(\p_h E) = (-1,\infty) \times (\widehat{\Sigma} \times S^1) &\stackrel{\Pi_h}{\longrightarrow} \widehat{\Sigma} : 
\big(t,(z,\theta)\big) \mapsto \pi\spine(z,\theta) = z,\\
\widehat{\nN}(\p_v E) = (-1,\infty) \times \widehat{M}\paper 
&\stackrel{\Pi_v}{\longrightarrow} (-1,\infty) \times S^1 : (s,x) \mapsto (s,\pi\paper(x)),
\end{split}
\end{equation*}
hence $\Pi_h(s,\phi,t,\theta) = (s,\phi)$ and $\Pi_v(s,\phi,t,\theta) = (s,m\phi)$
on the diagonal end.  We denote the resulting vertical subbundle by
$V\widehat{E} \subset T\widehat{E}$.
Figure~\ref{fig:doubleCompletion} shows the complement of $E$ in $\widehat{E}$
as the lightly shaded region.

\subsection{Symplectic structure}
\label{sec:Liouville}

The data on $\widehat{E}$ defined in this section constitute an 
enhancement and extension of the Liouville
structure already defined on $E \subset \widehat{E}$ in \cite{LisiVanhornWendl1}*{\S 4.1}.

Fix a complex structure $j$ on $\Sigma$ that takes the form
$$
j \p_s = \frac{1}{m} \p_\phi \quad \text{ on $\widehat{\nN}(\p\Sigma)$},
$$
where for each component of $\widehat{\nN}(\p\Sigma)$, $m \in \NN$ is the 
multiplicity that appeared in \eqref{eqn:localMult}; recall that this number may differ on
distinct connected components of $\widehat{\nN}(\p\Sigma)$.
Next, fix a $j$-convex function $\varphi : \widehat{\Sigma} \to \RR$ with
$$
\varphi(s,\phi) = e^s \quad\text{ on $\widehat{\nN}(\p\Sigma)$}.
$$
Such a function can always be found by starting from a 
Morse function on $\Sigma$ with critical points of index~$0$ and~$1$ and then
postcomposing it with a sufficiently convex function,
see e.g.~\cite{LatschevWendl}*{Lemma~4.1}.  This gives rise to a Liouville form
$$
\sigma := - d\varphi \circ j
$$
on $\widehat{\Sigma}$ with
$$
\sigma = m e^s \, d\phi \quad \text{ on $\widehat{\nN}(\p\Sigma)$}.
$$
We will also use $\sigma$ to denote the pullback of this Liouville form
under the trivial bundle projection 
$\Pi_h : \widehat{\nN}(\p_h E) \to \widehat{\Sigma}$, and since 
$\pi\paper(\phi,t,\theta) = m\phi$ on $\widehat{\nN}(\p M\paper)$, $\sigma$ extends 
globally to a $1$-form on $\widehat{E}$ satisfying
$$
\sigma = e^s \, d\pi\paper \quad \text{ on $\widehat{\nN}(\p_v E)$},
$$
where we are abusing notation slightly by using $\pi\paper : 
\widehat{\nN}(\p_v E) \to S^1$ to denote the composition of the 
fibration $\pi\paper : \widehat{M}\paper \to S^1$ with the obvious projection
$\widehat{\nN}(\p_v E) = (-1,0] \times \widehat{M}\paper \to \widehat{M}\paper$,
hence defining $d\pi\paper$ as a real-valued $1$-form on~$\widehat{\nN}(\p_v E)$.

Next, choose a $1$-form $\lambda$ on $\widehat{M}\paper$ such that
$d\lambda$ is positive on all fibers of $\pi\paper : \widehat{M}\paper \to S^1$
and
$$
\lambda = e^t \, d\theta \quad\text{ on } \widehat{\nN}(\p M\paper).
$$
Such a $1$-form can easily be found by first defining it on a single fiber
and then acting on it with the monodromy and interpolating
(see e.g.~\cite{Etnyre:lectures}*{Theorem~3.13}).  We will use the same symbol
to denote the pullback of $\lambda$ via the projection
$\widehat{\nN}(\p_v E) = (-1,\infty) \times \widehat{M}\paper \to \widehat{M}\paper$,
and it then extends to a global $1$-form on $\widehat{E}$ such that
$$
\lambda = e^t \, d\theta \quad\text{ on } \widehat{\nN}(\p_h E).
$$
Since $d\lambda|_{V\widehat{W}} > 0$ by construction, one can regard $\lambda$
as a \emph{fiberwise Liouville form} (cf.~\cite{LisiVanhornWendl1}*{\S 2})
on~$\widehat{E}$, and we observe also that since $\lambda|_{T(\p_h E)} = d\theta$,
its restriction to $\p E = \p_v E \cup \p_h E$ can also be regarded as a
\emph{fiberwise Giroux form}.

For applications in the almost Stein category, it will be convenient to
add another condition on the construction of~$\lambda$.
Pick a smooth family $J\fib$ of complex structures on the fibers of
$\pi\paper : \widehat{M}\paper \to S^1$ such that 
$$
J\fib \p_t = \p_\theta \quad \text{ on $\widehat{\nN}(\p M\paper)$}.
$$
On each individual fiber, the space of 
smooth $J\fib$-convex functions that match $e^t$ in the collar near 
the boundary is convex.  One can therefore use a partition of unity to
construct a smooth function $f\fib : \widehat{M}\paper \to \RR$ whose restriction
to each fiber has this property, and we are free to assume
\begin{equation}
\label{eqn:fibLiouvilleCondition}
\lambda\big|_{V\widehat{E}} = - d f\fib \circ J\fib\big|_{V\widehat{E}}.
\end{equation}

We can now apply the Thurston trick as described in \cite{LisiVanhornWendl1}*{\S 2.1}: 
for any constant $K \ge 0$, we define a
$1$-form $\lambda_K$ by
$$
\lambda_K := K\sigma + \lambda.
$$
Then there exists a constant $K_0 \ge 0$ such that
$d\lambda_K$ is symplectic everywhere on~$\widehat{E}$ for each $K \ge K_0$.  
Note that the unboundedness of $\widehat{E}$ does not pose any problem here
due to the precise formulas we have for $\lambda$ near infinity
(cf.~\cite{LisiVanhornWendl1}*{Remark~4.1}).  In particular, we have
\begin{equation}
\label{eqn:lambdaKboundary}
\lambda_K = K\, \sigma + e^t \, d\theta
\text{ on $\widehat{\nN}(\p_h E)$},\qquad\text{ and }\qquad
\lambda_K = K e^s\, d\pi\paper + \lambda
\text{ on $\widehat{\nN}(\p_v E)$},
\end{equation}
hence
\begin{equation}
\label{eqn:lambdaKcorner}
\lambda_K = K m e^s\, d\phi + e^t\, d\theta
\quad \text{ on $\widehat{\nN}(\p_v E \cup \p_h E)$}.
\end{equation}
We will assume that the condition $K \ge K_0$ holds from now on, and we will
occasionally also enlarge $K_0$ in order to satisfy extra conditions
(e.g.~for Lemma~\ref{lemma:LiouvilleTransverse} below).  Since
$d\lambda_K$ is now symplectic, there exists a Liouville vector field
$V_K$ on $(\widehat{E},d\lambda_K)$ defined via the condition
$$
d\lambda_K(V_K,\cdot) \equiv \lambda_K.
$$
From \eqref{eqn:lambdaKboundary} we find
\begin{equation}
\label{eqn:LiouvilleFormula}
V_K = V_\sigma + \p_t \quad \text{ on $\widehat{\nN}(\p_h E) = (-1,\infty) \times \widehat{\Sigma} \times S^1$},
\end{equation}
where $V_\sigma$ denotes the Liouville vector field on $\widehat{\Sigma}$
dual to~$\sigma$; in particular, $V_\sigma = \p_s$ on the cylindrical end
$\widehat{\nN}(\p\Sigma)$, hence
$$
V_K = \p_s + \p_t \quad \text{ on $\widehat{\nN}(\p_v E \cap \p_h E)$}.
$$

\begin{lemma}
\label{lemma:LiouvilleTransverse}
If $K_0 \ge 0$ is sufficiently large and $K \ge K_0$, then
$ds(V_K) > 0$ holds on $\widehat{\nN}(\p_v E)$.
\end{lemma}
\begin{proof}
It suffices to show that the restriction of $\lambda_K$ to 
$\{s\} \times \widehat{M}\paper$ for each $s \in (-1,\infty)$ is a positive 
contact form, or equivalently,
$$
ds \wedge \lambda_K \wedge d\lambda_K > 0 \quad \text{ on $\widehat{\nN}(\p_v E)$}.
$$
Since $\sigma = e^s\, d\pi\paper$ on $\widehat{\nN}(\p_v E)$, we compute
\begin{equation*}
\begin{split}
ds \wedge \lambda_K \wedge d\lambda_K &=
ds \wedge (K e^s \, d\pi\paper + \lambda) \wedge (K e^s \, ds \wedge d\pi\paper
+ d\lambda) \\
&= K e^s \left( ds \wedge d\pi\paper \wedge d\lambda +
\frac{1}{K e^s} \, ds \wedge \lambda \wedge d\lambda \right),
\end{split}
\end{equation*}
and see that the first term is a positive volume form since
$d\lambda\big|_{V\widehat{E}} > 0$, while the second is bounded with respect to any
$s$-invariant metric and thus uniformly small if $K$ is large.
\end{proof}

The lemma implies that for any pair of constants $s_0 , t_0 \in (-1,\infty)$,
the boundary of the region $\{ s \le s_0, \ t \le t_0 \} \subset \widehat{E}$
is a contact hypersurface in $(\widehat{E},d\lambda_K)$ after smoothing
the corner.

For applications to weak fillings, we will also need to allow certain
cohomological perturbations to the model $(\widehat{E},d\lambda_K)$.
Fix on $M$ a closed $2$-form $\eta$ that has support in the interior of
$M\paper \setminus \nN(\p M\paper)$, so pulling back via the projection
$\widehat{\nN}(\p_v E) = (-1,\infty) \times \widehat{M}\paper \to 
\widehat{M}\paper$ defines $\eta$ as a closed $2$-form on $\widehat{E}$
that vanishes in $\widehat{\nN}(\p_h E)$ and is uniformly bounded on
$\widehat{\nN}(\p_v E)$ for any choice of $s$-invariant metric.  In the
following, we say that an oriented hypersurface endowed with a co-oriented
contact structure in a symplectic $4$-manifold is \defin{weakly contact} if
the restriction of the symplectic form to the contact structure is positive.

\begin{lemma}
\label{lemma:etaPert}
Given $K \ge K_0$, there exists a constant $C_0 > 0$ such that for all 
$C \ge C_0$, the $2$-form $\omega'_E := C \, d\lambda_K + \eta$
is symplectic on $\widehat{E}$ and for each $s \in (-1,\infty)$,
the hypersurface
$$
\p_v^s \widehat{E} := \{s\} \times \widehat{M}\paper \subset \widehat{E}
$$
with contact structure 
$\xi_v^s := \ker \left( \lambda_K|_{T(\p_v^s\widehat{E})} \right)$ is
weakly contact in $(\widehat{E},\omega'_E)$.
\end{lemma}
\begin{proof}
This is mainly a matter of replacing $d\lambda$ with 
$d\lambda + \frac{1}{C} \eta$ and then
repeating the usual calculations carefully enough
to make sure that nothing goes wrong as $s \to \infty$. 
The nondegeneracy of $\omega'_E$ on $\widehat{\nN}(\p_v E)$, for instance, 
follows by computing
\begin{equation*}
\begin{split}
\frac{1}{C^2} & \omega'_E \wedge \omega'_E =
\left( d\lambda_K + \frac{1}{C}\eta \right) \wedge 
\left( d\lambda_K + \frac{1}{C}\eta \right) \\
&= K e^s \left[ 2 \, ds \wedge d\pi\paper \wedge \left( d\lambda +
\frac{1}{C}\eta \right) + 
\frac{1}{K e^s} \left(d\lambda + \frac{1}{C}\eta \right)
\wedge \left( d\lambda + \frac{1}{C} \eta \right) \right],
\end{split}
\end{equation*}
in which the first term in the brackets is uniformly positive for sufficiently large
$C > 0$ and the second is bounded with respect to any $s$-invariant metric.
The weak contact condition for $(\p_v^s \widehat{E},\xi_v^s)$ follows by a similar 
modification of the proof of Lemma~\ref{lemma:LiouvilleTransverse},
showing
$$
\frac{1}{C} \, ds \wedge \lambda_K \wedge \omega'_E = ds \wedge
\lambda_K \wedge \left( d\lambda_K + \frac{1}{C} \eta \right) > 0
\quad
\text{ on $\widehat{\nN}(\p_v E)$.}
$$
\end{proof}

For the remainder of \S\ref{sec:model}, we fix $K \ge K_0$,
$C \ge C_0$ as in the above lemma and consider the rescaled symplectic form
$$
\omega_E := \frac{1}{K C} \left( C\, d\lambda_K + \eta \right) =
d\sigma + \frac{1}{K} d\lambda + \frac{1}{KC} \eta
$$ 
on $\widehat{E}$.  The scaling by $1 / KC$ has no deep significance but will
be convenient for technical reasons when we talk about stable Hamiltonian
structures below.  Since $\eta$ vanishes in $\widehat{\nN}(\p_h E)$,
$V_K$ is also a Liouville vector field for $\omega_E$ in this region.
Then the fact that $V_K = \p_s + \p_t$ in the diagonal end implies that
the boundary
of any region of the form $\{ s \le s_0,\ t \le t_0 \} \subset \widehat{E}$
can be made into a \emph{weakly} contact hypersurface in 
$(\widehat{E},\omega_E)$ by smoothing the corner, with the contact structure
defined by restricting~$\lambda_K$.  Two specific examples of smooth hypersurfaces
were defined in this way at the end of \S\ref{sec:collars}: we define contact
forms and contact structures on $M^0$ and $M^-$ respectively by
\begin{equation*}
\begin{split}
\alpha^0 := \lambda_K\big|_{T M^0}, \qquad & \xi_0 := \ker\alpha^0 \subset T M^0,\\
\alpha^- := \lambda_K\big|_{T M^-}, \qquad & \xi_- := \ker\alpha^- \subset T M^-.
\end{split}
\end{equation*}
This makes $(M^-,\xi_-)$ and $(M^0,\xi_0)$ into weakly contact hypersurfaces
in $(\widehat{E},\omega_E)$, and the definition of $\lambda_K$ implies that
both are (after suitably identifying $M^0$ and $M^-$ with~$M$) supported by
the spinal open book~$\boldsymbol{\pi}$, hence both are isotopic to~$\xi$.

\subsection{Stable Hamiltonian structure}
\label{sec:SHS}

We now endow the hypersurface $M^0 \subset \widehat{E}$ 
with a stable Hamiltonian structure that is related to its
contact structure $\xi_0$ but will be
better suited for finding holomorphic pages in its symplectization.
Fix a smooth cutoff function $\beta : (-1,\infty) \to [0,1]$ satisfying
\begin{itemize}
\item $\beta \equiv 0$ on $(-1,-1/2]$;
\item $\beta \equiv 1$ on $[-1/4,\infty)$;
\item $\beta' \ge 0$ and $\supp(\beta') \subset (-1/2,-1/4)$.
\end{itemize}
Now consider the vector field on $\widehat{E}$ defined by (see Figure~\ref{fig:stabVec})
$$
Z := \begin{cases}
V_\sigma + \beta(t) \p_t & \text{ on $\widehat{\nN}(\p_h E)$},\\
\p_s & \text{ everywhere else}.
\end{cases}
$$
Here again $V_\sigma$ denotes the Liouville vector field dual to $\sigma$
on~$\widehat{\Sigma}$, so we observe that $Z \equiv V_K$ on the region
$\{ t \ge -1/4 \} \subset \widehat{\nN}(\p_h E)$.  Everywhere else in
$\widehat{\nN}(\p_v E \cap \p_h E)$, we can
plug in $V_\sigma = \p_s$ and thus write $Z = \p_s + \beta(t) \p_t$.
This vector field is obviously transverse to $M^0$; we will now show that
it is also a \emph{stabilizing} vector field in the sense of
\S\ref{sec:stable}, and thus makes $M^0$ into a stable hypersurface.

\begin{figure}
    \begin{postscript}
\psfrag{Nhat(phE)}{$\widehat{\nN}(\p_h E)$}
\psfrag{Nhat(pvEintphE)}{$\widehat{\nN}(\p_v E \cap \p_h E)$}
\psfrag{Nhat(pvE)}{$\widehat{\nN}(\p_v E)$}
\psfrag{-1/2}{$t=-1/2$}
\psfrag{-1/4}{$t=-1/4$}
\psfrag{t}{$t$}
\psfrag{s}{$s$}
\psfrag{pvE}{$\p_v E$}
\psfrag{phE}{$\p_h E$}
\psfrag{M0}{$M^0$}
\psfrag{...}{$\ldots$}
\includegraphics{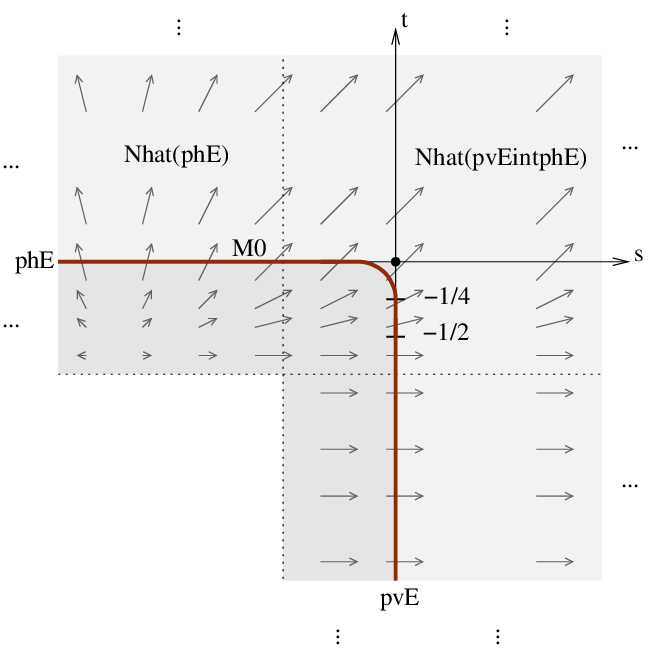}
\end{postscript}
\caption{\label{fig:stabVec}
The stabilizing vector field $Z$ transverse to $M^0 \subset \widehat{E}$.}
\end{figure}

\begin{lemma}
\label{lemma:stabilizing}
The vector field $Z$ is a \emph{stabilizing vector field} for $M^0$
in $(\widehat{E},\omega_E)$.
\end{lemma}
\begin{proof}
This is immediate in the region where $Z = V_K$, since $V_K$ is Liouville
for $\omega_E$ in that region and all Liouville vector fields have the
stabilizing property.  It is similarly immediate in the region where
$Z = \p_s$, as the hypersurfaces obtained by flowing $\p_v E$ along $\p_s$
each have constant $s$-coordinate and $\omega_E$ thus restricts to each of
them as $\frac{1}{K} \, d\lambda + \frac{1}{K C} \eta$, 
defining a characteristic line field that
does not depend on~$s$.  It remains to consider the region 
$-1/2 \le t \le -1/4$ in which $Z = \p_s + \beta(t) \p_t$, and the key
point here is that the characteristic line field is exceedingly simple:
we have $\omega_E = \frac{1}{K}\, d\lambda_K = m e^s\, ds \wedge d\phi +
\frac{1}{K} e^t\, dt \wedge d\theta$ in this region, while the hypersurfaces in
question still have constant $s$-coordinate, hence the characteristic
line field on each is spanned by $\p_\phi$, a vector field that is
preserved by the flow of~$Z$.
\end{proof}

The lemma implies that the pair $\hH_0 = (\Omega_0,\Lambda_0)$ defined by
$$
\Omega_0 := \omega_E\big|_{TM^0}, \qquad 
\Lambda_0 := \iota_Z \omega_E\big|_{TM^0}
$$
is a stable Hamiltonian structure on~$M^0$,
cf.~\S\ref{sec:stable}.  We will denote the corresponding
oriented hyperplane field by $\Xi_0 := \ker \Lambda_0 \subset TM^0$
and the Reeb vector field by $R_0$, where by definition
$$
\Omega_0(R_0,\cdot) \equiv 0, \qquad \Lambda_0(R_0) \equiv 1.
$$
One can compute the following explicit
formulas for $\Lambda_0$, $\Omega_0$ and $R_0$ in the regions
$\widecheck{M}^0\spine , \widecheck{M}^0\corner, \widecheck{M}^0\paper \subset M^0$ defined in
\S\ref{sec:collars} (cf.~Figure~\ref{fig:M0cover}).

On $\widecheck{M}^0\spine \subset \p_h E = \Sigma \times S^1$, $Z$ matches the Liouville
vector field $V_K$, which is $\omega_E$-dual to $\frac{1}{K} \lambda_K$, 
hence
$$
(\Omega_0,\Lambda_0) = \left(\frac{1}{K}\, d\alpha^0, \frac{1}{K} \alpha^0\right)
= \left(d\sigma, \frac{1}{K} \, d\theta + \sigma \right)
\quad
\text{ on $\widecheck{M}^0\spine$}.
$$
It follows that $R_0$ is a suitably rescaled version of the Reeb vector
field for $\alpha^0$, that is,
$$
R_0 = K \p_\theta \quad
\text{ on $\widecheck{M}^0\spine$}.
$$

On $\widecheck{M}^0\paper \subset \p_v E = M\paper$, $Z = \p_s$ and thus
$$
(\Omega_0,\Lambda_0) = \left(\frac{1}{K}\, d\lambda + \frac{1}{KC}\eta, d\pi\paper\right) \quad
\text{ on $\widecheck{M}^0\paper$},
$$
so $d\Lambda_0 = 0$ on this region and thus $\Xi_0$ is integrable; 
indeed, the integral submanifolds of $\Xi_0$ are simply the fibers of 
$\pi\paper : M\paper \to S^1$.  The Reeb vector field can be written as
\begin{equation}
\label{eqn:RHpaper}
R_0 = e_{S^1}^\# \quad \text{ on $\widecheck{M}^0\paper$},
\end{equation}
where we denote by $e_{S^1} \in T S^1$ the canonical unit vector field on
$S^1 = \RR / \ZZ$ and use the superscript ``$\#$'' to denote its
horizontal lift with respect to the
connection on $\pi\paper : M\paper \to S^1$ defined as the
$(C\, d\lambda + \eta)$-symplectic complement of the vertical subbundle.
In each collar $S^1 \times (-1,-1/2) \times S^1 \subset \nN(\p M\paper)
\cap \widecheck{M}^0\paper$, we can write $d\pi\paper = m\, d\phi$
for the appropriate multiplicity $m \in \NN$ and use 
$(\phi,t,\theta)$-coordinates to write
\begin{equation}
\label{eqn:RHpaperCollar}
R_0 = \frac{1}{m} \p_\phi
\quad \text{ on $\widecheck{M}^0\paper \cap \nN(\p M\paper)$}.
\end{equation}

Finally, using the coordinates $(\rho,\phi,\theta) \in (-1,1) \times S^1
\times S^1$ on connected components of $\widecheck{M}^0\corner$, $\Omega_0$ and $\Lambda_0$
are determined by the functions $F$ and $G$ that were chosen in
\S\ref{sec:collars} for smoothing
the corner, as well as the cutoff function $\beta$ in the definition
of~$Z$: we have
\begin{equation}
\label{eqn:OmegaLambda}
\begin{split}
\Omega_0 &=  m e^{F(\rho)} F'(\rho) \, d\rho \wedge d\phi +
\frac{1}{K} e^{G(\rho)} G'(\rho) \, d\rho \wedge d\theta , \\
\Lambda_0 &= m e^{F(\rho)} \, d\phi + \frac{1}{K} e^{G(\rho)} \beta(G(\rho))
\, d\theta ,
\end{split}
\end{equation}
which leads to
\begin{equation}
\label{eqn:RH}
R_0 = \frac{1}{\beta(G(\rho)) F'(\rho) - G'(\rho)}
\left( -\frac{1}{m} e^{-F(\rho)} G'(\rho) \, \p_\phi
+ K e^{-G(\rho)} F'(\rho) \, \p_\theta \right) .
\end{equation}

\subsection{Nondegenerate perturbation}
\label{sec:nondegPert}

The stable Hamiltonian structure $\hH_0 = (\Omega_0,\Lambda_0)$ defined above on
$M^0$ has the unfortunate property that the orbits of $R_0$ in $\widecheck{M}^0\spine$
are degenerate.  We will follow the standard procedure for perturbing
to get nondegenerate orbits, which depends on a choice of Morse function
on the space parametrizing the orbits, in this case~$\Sigma$.
Choose a smooth function
$$
H : \Sigma \to [0,\infty)
$$
such that
\begin{enumerate}
\item $H$ is Morse outside of $\nN(\p\Sigma)$;
\item On $\nN(\p\Sigma)$ in coordinates $(s,\phi)$, $H$ depends only on~$s$,
and it satisfies $\p_s H < 0$ except on a smaller neighborhood of the
boundary with closure in the region $\{ -1/4 < s \le 0 \}$, where $H$ vanishes.
\end{enumerate}
We shall denote by
$$
\CritMorse(H) \subset \Sigma
$$
the finite set of Morse critical points of $H$; this excludes the critical
points near~$\p \Sigma$ where $H$ vanishes.  Extend $H$ to a smooth function 
$$
\widehat{H} : M^0 \to [0,\infty)
$$ 
that vanishes on $\widecheck{M}^0\paper$ and satisfies $\widehat{H}(z,\theta) = H(z)$
on $\widecheck{M}^0\spine \subset \Sigma \times S^1$, and
$\widehat{H}(\rho,\phi,\theta) = H(F(\rho),\phi)$ on~$\widecheck{M}^0\corner$.  Now if
$\Phi^\tau_Z$ denotes the flow of $Z$ for time~$\tau$, we fix a small constant
$\nondegParam > 0$ and observe that the perturbed hypersurface (see Figure~\ref{fig:nondegenerate})
$$
M^+ := \left\{ \Phi_Z^{\nondegParam \widehat{H}(x)}(x) \in \widehat{E}
\ \Big|\ x \in M^0 \right\}
$$
is still stabilized by~$Z$; indeed, $M^+$ still matches $\p_v E$
in the region where $Z$ is not~$V_K$, and everywhere else $Z$ is Liouville
and manifestly transverse to~$M^+$.  The obvious diffeomorphism
of $M^0$ to $M^+$ defined by flowing along~$Z$ induces a
decomposition
$$
M^+ = \widecheck{M}^+\spine \cup \widecheck{M}^+\corner \cup \widecheck{M}^+\paper
$$
corresponding to the decomposition $M^0 = \widecheck{M}^0\spine \cup \widecheck{M}^0\corner \cup \widecheck{M}^0\paper$
that we defined in \S\ref{sec:collars}, and we will use the same coordinate
systems on these subsets that were used on $\widecheck{M}^0\spine$, $\widecheck{M}^0\corner$
and~$\widecheck{M}^0\paper$.  Let
$$
\hH_+ := (\Omega_+,\Lambda_+)
$$
denote the stable Hamiltonian structure induced by $Z$ on $M^+$,
with oriented hyperplane field $\Xi_+$ and Reeb vector
field $R_+$.
Since $\widecheck{M}^0\spine$ is a contact hypersurface and
$\widecheck{M}^+\spine$ is obtained
from it by flowing along a Liouville vector field,
$(\Omega_+,\Lambda_+)$ on $\widecheck{M}^+\spine$
takes the form $((1/K)\, d\alpha^+, (1/K) \alpha^+)$,
where $\alpha^+$ is a contact form given by
\begin{equation}
\label{eqn:alpha'}
\alpha^+ := \lambda_K|_{T \widecheck{M}^+\spine} = 
e^{\nondegParam \widehat{H}} \alpha^0 =
e^{\nondegParam \widehat{H}} (K \sigma + d\theta).
\end{equation}
The resulting perturbed Reeb vector field takes the form
\begin{equation}
\label{eqn:ReebSigma}
R_+ = e^{-\nondegParam \widehat{H}} \left(
\left(1 + \nondegParam \sigma(X_H)\right) K \p_\theta
- \nondegParam X_H \right) \quad \text{ on $\widecheck{M}^+\spine$,}
\end{equation}
where $X_H$ denotes the Hamiltonian vector field of $H$ on 
$(\Sigma,d\sigma)$, determined by
$$
d\sigma(X_H,\cdot) = -dH.
$$
Notice that for some large threshold $T > 0$ that goes to $\infty$ as
$\nondegParam \to 0$, all periodic orbits up to period $T$ in $\widecheck{M}^+\spine$
have image $\{z\} \times S^1 \subset \Sigma \times S^1$ for some
$z \in \CritMorse(H)$.  We will generally fix the value of $K$ and
assume $\nondegParam > 0$ is sufficiently small to arrange this whenever
convenient; we can then also assume without loss of generality that
$\p_\theta$ and $\Xi_+$ are transverse, so the projection
$T\pi\spine : T \widecheck{M}^+\spine \to T\Sigma$ restricts to 
$\Xi_+$ as a fiberwise orientation-preserving isomorphism
\begin{equation}
\label{eqn:fiberwiseIso}
\Xi_+|_{\widecheck{M}^+\spine} \stackrel{T\pi\spine}{\longrightarrow}
T\Sigma.
\end{equation}

\begin{figure}
\begin{postscript}
\psfrag{Nhat(phE)}{$\widehat{\nN}(\p_h E)$}
\psfrag{Nhat(pvEintphE)}{$\widehat{\nN}(\p_v E \cap \p_h E)$}
\psfrag{Nhat(pvE)}{$\widehat{\nN}(\p_v E)$}
\psfrag{t}{$t$}
\psfrag{s}{$s$}
\psfrag{pvE}{$\p_v E$}
\psfrag{phE}{$\p_h E$}
\psfrag{M+}{$M^+$}
\psfrag{...}{$\ldots$}
\includegraphics{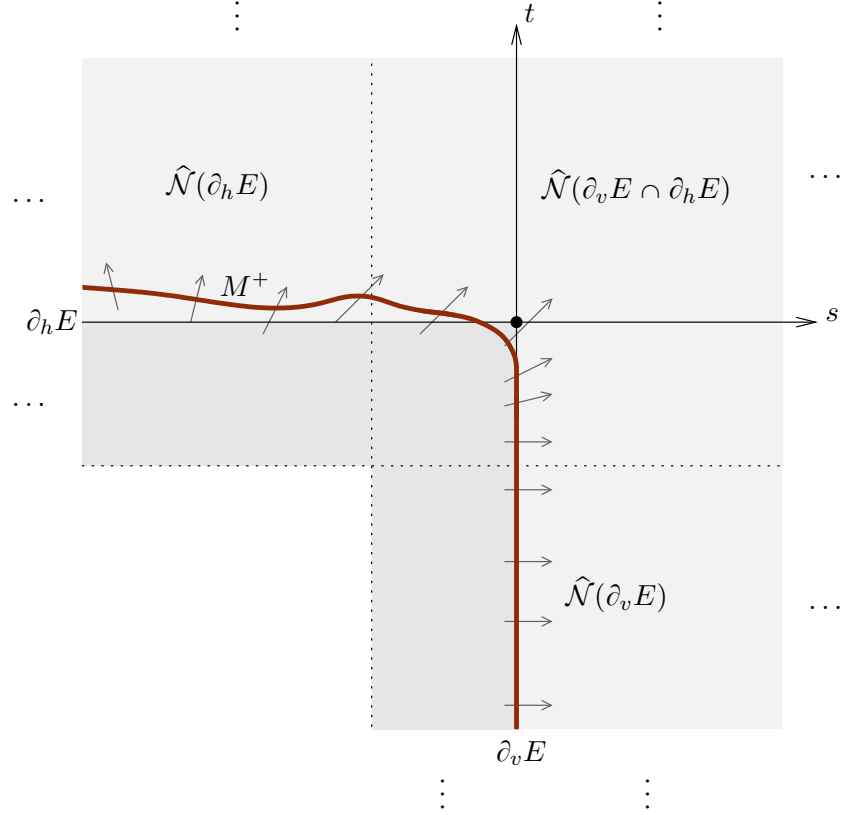}
\end{postscript}
\caption{\label{fig:nondegenerate}
The perturbed stable hypersurface $M^+ \subset \widehat{E}$.}
\end{figure}

On $\widecheck{M}^+\corner$, the formulas \eqref{eqn:OmegaLambda} and
\eqref{eqn:RH} can be adapted to write $\Omega_+$, $\Lambda_+$
and $R_+$ as
\begin{equation}
\label{eqn:OmegaLambdaEps}
\begin{split}
\Omega_+ &=  m e^{F_+(\rho)} F_+'(\rho) \, d\rho \wedge d\phi +
\frac{1}{K} e^{G_+(\rho)} G_+'(\rho) \, d\rho \wedge d\theta , \\
\Lambda_+ &= m e^{F_+(\rho)} \, d\phi + \frac{1}{K} e^{G_+(\rho)} \beta(G_+(\rho))
\, d\theta ,\\
R_+ &= \frac{1}{\beta(G_+(\rho)) F_+'(\rho) - G_+'(\rho)}
\left( -\frac{1}{m} e^{-F_+(\rho)} G_+'(\rho) \, \p_\phi
+ K e^{-G_+(\rho)} F_+'(\rho) \, \p_\theta \right) ,
\end{split}
\end{equation}
where the perturbed versions of the functions $F$ and $G$ are defined by
$$
F_+(\rho) := F(\rho) + \nondegParam H(F(\rho),\cdot), \qquad
G_+(\rho) := G(\rho) + \nondegParam H(F(\rho),\cdot).
$$
Let us file away for future use the following detail, which results from
the particular conditions we have imposed on $G$ and~$H$.
\begin{lemma}
\label{lemma:GnotFlat}
The function $G_+$ satisfies $G_+'(\rho) < 0$ for all 
$\rho \in (-1,1)$, hence by \eqref{eqn:OmegaLambdaEps}, 
$d\phi(R_+) \ne 0$ on~$\widecheck{M}^+\corner$.  \qed
\end{lemma}

The region $\widecheck{M}^+\paper$ is identical to $\widecheck{M}^0\paper$, and here 
$(\Omega_+,\Lambda_+) = (\Omega_0,\Lambda_0)$, with
$R_+$ also given by \eqref{eqn:RHpaper} and~\eqref{eqn:RHpaperCollar}.

\subsection{The cylindrical end}
\label{sec:end}

As was discussed in \S\ref{sec:stable}, one can use the Moser deformation
trick to show that $M^+$ has a
collar neighborhood in $(\widehat{E},\omega_E)$ that can be identified
symplectically with
\begin{equation}
\label{eqn:SHScollar}
\left( (-\delta,\delta) \times M^+ , 
d\big( (e^r - 1) \Lambda_+\big) + \Omega_+ \right),
\end{equation}
for sufficienlty small $\delta > 0$, with $r$ denoting the coordinate
on $(-\delta,\delta)$.  This is true for arbitrary stable Hamiltonian
structures, but it will be convenient to take advantage of a few properties
of our example $\hH_+ = (\Omega_+,\Lambda_+)$ that are
nicer than the general case.  To start with, $\Xi_+ = \ker
\Lambda_+$ is everywhere either a positive contact structure or a
foliation, hence it is a \emph{confoliation}; equivalently,
$\Lambda_+ \wedge d\Lambda_+ \ge 0$.  As a consequence,
the collar \eqref{eqn:SHScollar} remains symplectic if we replace
$(-\delta,\delta) \times M^+$ by an infinite half-cylinder:
\begin{equation}
\label{eqn:SHSend}
\left( [0,\infty) \times M^+ , d\big( (e^r - 1) \Lambda_+ \big)
+ \Omega_+ \right).
\end{equation}
Observe that this reduces in regions where $\Omega_+ = d\Lambda_+$ to the
usual half-symplectization of a contact form,
$\left( [0,\infty) \times M^+ , d(e^r \Lambda_+) \right)$.
As it turns out, a half-cylinder of the form \eqref{eqn:SHSend}
is already present in the model $(\widehat{E},\omega_E)$.  Denote
$$
\widehat{\nN}_+(\p E)
:= \left\{ \Phi_Z^\tau(x) \in \widehat{E}\ \Big|\ \tau \ge 0,\ 
x \in M^+ \right\} \subset \widehat{E},
$$
with $\Phi^\tau_Z$ again denoting the flow of~$Z$ for time~$\tau$.
This is the \emph{unbounded} closed subset of $\widehat{E}$ with
boundary~$M^+$; see Figure~\ref{fig:end}.

\begin{figure}

\begin{postscript}
\psfrag{t}{$t$}
\psfrag{s}{$s$}
\psfrag{pvE}{$\p_v E$}
\psfrag{phE}{$\p_h E$}
\psfrag{M+}{$M^+$}
\psfrag{M-}{$M^-$}
\psfrag{-1/4}{$\rho=1/4$}
\psfrag{-1/2}{$\rho=1/2$}
\psfrag{r}{$r$}
\psfrag{rho}{$\rho$}
\psfrag{Vhat}{$\widehat{\vV}$}
\psfrag{...}{$\ldots$}
\includegraphics{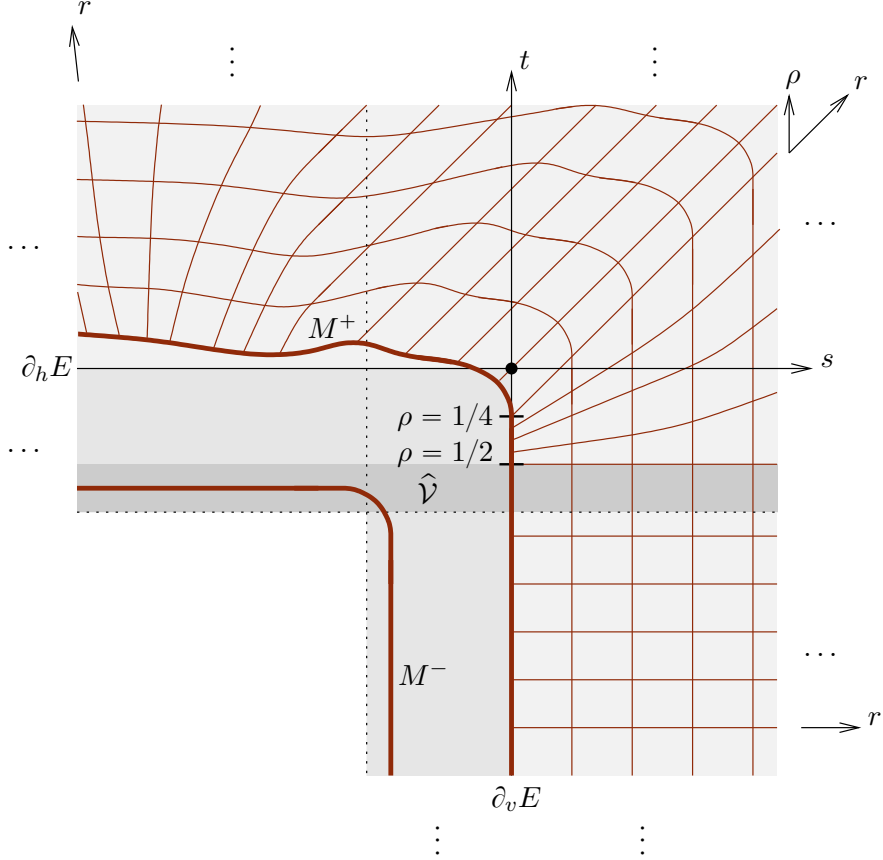}
\end{postscript}
\caption{\label{fig:end}
The grid represents the half-cylinder $[0,\infty) \times M^+ \cong \widehat{\nN}_+(\p E) \subset \widehat{E}$.
The darkly shaded region is the set $\widehat{V} \subset \widehat{E}$ defined in
\S\ref{sec:Jandf}, where the holomorphic vertebrae live.}
\end{figure}

\begin{lemma}
\label{lemma:theEnd}
The region $\widehat{\nN}_+(\p E)$ is the image of an embedding
$\Psi : [0,\infty) \times M^+ \hookrightarrow
\widehat{E}$ defined on $[0,\infty) \times (1/4,1/2) \times S^1 \times S^1
\subset [0,\infty) \times \widecheck{M}^+\corner$ by
\begin{equation}
\label{eqn:PsiInterface}
\begin{split}
\Psi(r,\rho,\phi,\theta) &= \left( r , \phi ,
-\rho + \log [ (e^r-1) \beta(-\rho) + 1] , \theta \right) \\
& \quad \in (-1,\infty) \times S^1 \times (-1,\infty) \times S^1 \subset
\widehat{\nN}(\p_v E \cap \p_h E),
\end{split}
\end{equation}
and everywhere else by
$$
\Psi(r,x) = \Phi_Z^r(x),
$$
Moreover, $\Psi^*\omega_E = d((e^r-1) \Lambda_+) + \Omega_+$.
\end{lemma}
\begin{proof}
It is straightforward to see that $\Psi$ is a smooth map whose image
is $\widehat{\nN}_+(\p E)$; indeed, since $Z = \p_s + \beta(t) \p_t$
on the diagonal end, the definition in \eqref{eqn:PsiInterface}
matches the flow of $Z$ for time $r$ on regions where
$\beta(-\rho)$ is $0$ or $1$, which excludes only a compact subset of
$\{ \rho \in (1/4,1/2)\}$.
The main thing to verify is thus the formula for~$\Psi^*\omega_E$.
Consider first $(r,\rho,\phi,\theta) \in [0,\infty) \times \widecheck{M}^+\corner$
with $\rho \in [1/4,1/2]$.  Then $F_+(\rho) = 0$ and
$G_+(\rho) = -\rho$, so we have $\Lambda_+ = K m \, d\phi
+ \frac{1}{K} e^{-\rho} \beta(-\rho) \, d\theta$ and
$\Omega_+ = - \frac{1}{K} e^{-\rho} \, d\rho \wedge d\theta$.
Meanwhile, since the image of $\Psi$ on this region lies
in $\widehat{\nN}(\p_v E \cap \p_h E)$, we have
$\omega_E = \frac{1}{K}\, d\lambda_K$ and 
$\lambda_K = K m e^s \, d\phi + e^t\, d\theta$, hence
$$
\Psi^*\lambda_K = K m e^r \, d\phi + e^{-\rho} [(e^r-1) \beta(-\rho) + 1]\, d\theta.
$$
From these formulas, a quick computation shows $\Psi^*\omega_E =
\frac{1}{K} \, d(\Psi^*\lambda_K) = d((e^r-1) \Lambda_+) + \Omega_+$.

For $x \in \widecheck{M}^+\paper \subset \p_v E$, we have $Z = \p_s$ and thus
$\Psi(r,x) = (r,x) \in \widehat{\nN}(\p_v E)$, so 
writing $\omega_E = \frac{1}{K}\, d\lambda_K + \frac{1}{KC} \eta$ with 
$\lambda_K = K e^s\, d\pi\paper + \lambda$
gives
\begin{equation*}
\begin{split}
\Psi^*\omega_E &= \frac{1}{K} \, d(\Psi^*\lambda_K) + \frac{1}{KC} \Psi^*\eta = 
d\left(e^r \, d\pi\paper + \frac{1}{K}\lambda\right) + \frac{1}{KC} \eta \\
&= d(e^r-1) \wedge d\pi\paper + \frac{1}{K}\, d\lambda + \frac{1}{KC}\eta.
\end{split}
\end{equation*}
This also reproduces $d((e^r-1) \Lambda_+) + \Omega_+$ when we plug in
$\Lambda_+ = \Lambda_0 = d\pi\paper$ and $\Omega_+ =
\Omega_0 = \frac{1}{K}\, d\lambda + \frac{1}{KC}\eta$.

On the remaining regions, $Z$ is the Liouville vector field $V_K$, and
$\Lambda_+$ and $\Omega_+$ match the restrictions of
$\frac{1}{K} \lambda_K$ and $\omega_E = \frac{1}{K}\, d\lambda_K$ 
respectively to~$TM^+$, thus 
$$
\Psi^*\left(\frac{1}{K} \lambda_K\right) = e^r \Lambda_+,
$$
implying $\Psi^*\omega_E = d(e^r \Lambda_+) =
d((e^r-1) \Lambda_+) + d\Lambda_+$.  The desired formula follows
since $\Omega_+ = d\Lambda_+$ on this region.
\end{proof}

Throughout the following, we will omit the embedding $\Psi$ from the
notation and simply identify $[0,\infty) \times M^+$ with the subdomain
$\widehat{\nN}_+(\p E) \subset \widehat{E}$.

\subsection{Almost complex and almost Stein structures}
\label{sec:Jandf}

We shall now choose an almost complex structure from the space
$\jJ(\hH_+)$ of $\RR$-invariant structures compatible with~$\hH_+$
(see \S\ref{sec:energy}).  Any $J_+ \in \jJ(\hH_+)$ determines an
$\omega_E$-compatible almost complex structure on
$[0,\infty) \times M^+ = \widehat{\nN}_+(\p E)$, and we will choose
$J_+ \in \jJ(\hH_+)$ to satisfy some
additional conditions that will be convenient for our main applications.
It suffices to specify an orientation-preserving complex structure on
the subbundle
$$
J_+ : \Xi_+ \to \Xi_+
$$
over each of the regions $\widecheck{M}^+\spine$, $\widecheck{M}^+\corner$ 
and~$\widecheck{M}^+\paper$, as the translation-invariance condition and
$J_+(\p_r) = R_+$ then determine $J_+ \in \jJ(\hH_+)$ uniquely.

Recall that in \S\ref{sec:Liouville} we endowed $\widehat{\Sigma}$ 
with a complex structure $j$ satisfying 
\begin{equation}
\label{eqn:littlej}
j\p_s = \frac{1}{m} \p_\phi
\end{equation}
on its cylindrical ends, where the multiplicity $m \in \NN$ may be different
on distinct ends.
Over $\widecheck{M}^+\spine \subset \Sigma \times S^1$, 
define $J_+ : \Xi_+ \to \Xi_+$ 
as the pullback of $j$ under the fiberwise isomorphism 
$\Xi_+ \to T\Sigma$ defined via~$T\pi\spine$
(see~\eqref{eqn:fiberwiseIso}).  This makes $J_+$ invariant
under the $S^1$-action defined by translating the $\theta$-coordinate.

On $\widecheck{M}^+\paper$, which is canonically identified with an open
subset of $M\paper \subset M = \p E$, choosing $J_+ \in \jJ(\hH_+)$ 
is equivalent to choosing smoothly varying complex structures on the fibers
of $\pi\paper : M\paper \to S^1$.  We already made such a choice
when $\lambda$ was defined in \S\ref{sec:Liouville}: set
$$
J_+|_{\Xi_+} := J\fib \quad \text{ on $\widecheck{M}^+\paper$},
$$
which has the property that
\begin{equation}
\label{eqn:JonVerticalCollar}
J_+ \p_t = \p_\theta \quad\text{ on $\nN(\p M\paper) \cap \widecheck{M}^+\paper$}.
\end{equation}

Now using the coordinates $(\rho,\phi,\theta) \in (-1,1)
\times S^1 \times S^1$ on each component of $\widecheck{M}^+\corner$, 
the formulas \eqref{eqn:OmegaLambdaEps} imply that
$\Xi_+$ is spanned by vector fields of the form
$$
v_1 := \p_\rho, \qquad v_2 := a(\rho) \p_\phi + b(\rho) \p_\theta
$$
for a unique choice of smooth functions $a, b : (-1,1) \to \RR$
such that $\Lambda_+(v_2) \equiv 0$ and 
$\Omega_+(v_1,v_2) \equiv 1$.
In $\widecheck{M}^+\corner \cap \widecheck{M}^+\spine$, we write 
$s = F_+(\rho) = \rho + \nondegParam H(\rho,\cdot)$, 
$t = G_+(\rho) = \nondegParam H(\rho,\cdot) \ge 0$
and $\beta(G_+(\rho)) = 1$,
so the fiberwise isomorphism $\Xi_+ \to T\Sigma$ takes
$v_1$ and $v_2$ to positive multiples of $\p_s$ and $\p_\phi$ respectively,
hence by \eqref{eqn:littlej}, we have
\begin{equation}
\label{eqn:JonCorner}
J_+ v_1 = h(\rho) v_2
\end{equation}
on $\widecheck{M}^+\corner \cap \widecheck{M}^+\spine$
for a suitable choice of smooth positive function~$h$.
Likewise, on $\widecheck{M}^+\corner \cap \widecheck{M}^+\paper$, writing $s = F_+(\rho) = 0$, 
$t = G_+(\rho) = -\rho$ and $\beta(G_+(\rho)) = 0$ gives 
$\Lambda_+ = m \, d\phi$ and 
$\Omega_+ = - \frac{1}{K} e^{-\rho} d\rho \wedge d\theta$, so $v_2$ reduces
to the form $b(\rho) \p_\theta$ with $b(\rho) < 0$.  It follows that $v_1$ 
and $v_2$ are negative multiples of $\p_t$ and $\p_\theta$ respectively, 
implying via \eqref{eqn:JonVerticalCollar} that
\eqref{eqn:JonCorner} is again valid for a suitable choice of positive
function~$h(\rho)$.
We can thus use \eqref{eqn:JonCorner} to extend 
$J_+ : \Xi_+ \to \Xi_+$ 
over the rest of $\widecheck{M}^+\corner$ by extending $h$ arbitrarily to a smooth 
positive function on $(-1,1)$.  For certain computations we will find it
convenient to impose one further condition on the function $h(\rho)$
for $\rho \in (1/4,1)$, in particular on the ``interpolation'' 
region $\{ 1/4 < \rho < 1/2 \}$.  Here we have
$\Lambda_+ = m \, d\phi + \frac{1}{K} e^{-\rho} \beta(-\rho)\, d\theta$
and thus $\Lambda_+(v_2) = 
m a(\rho) + \frac{1}{K} e^{-\rho} \beta(-\rho) b(\rho) = 0$,
implying $b(\rho) \ne 0$, or in other words,
$v_2$ must always have a nontrivial $\p_\theta$-component.  We are therefore
free to choose $h(\rho)$ to produce the obvious extension of
\eqref{eqn:JonVerticalCollar} into this region; since $\p_\rho = -\p_t$,
\eqref{eqn:JonCorner} then becomes
\begin{equation}
\label{eqn:Jinterp}
J_+ \p_\rho = -\p_\theta + \frac{1}{K m} e^{-\rho} \beta(-\rho) \p_\phi \quad 
\text{ on $\{ 1/4 < \rho < 1 \} \subset \widecheck{M}^+\corner$}.
\end{equation}

With the choices above in place, $J_+ \in \jJ(\hH_+)$ now determines
an $\omega_E$-compatible almost complex structure
on $\widehat{\nN}_+(\p E) \subset \widehat{E}$, which will next be
extended to the rest of~$\widehat{E}$.  In order to obtain 
holomorphic vertebrae, we start by making a careful choice of $J_+$
on the open subset (see Figure~\ref{fig:end})
$$
\widehat{\vV} := (-1,-1/2) \times \widehat{\Sigma} \times S^1 \subset 
\widehat{\nN}(\p_h E) \subset \widehat{E}.
$$
Writing tangent spaces at points $(t,z,\theta) \in \widehat{\vV}$ as
$T_{(t,z,\theta)}\widehat{E} = T_z \widehat{\Sigma} \oplus \Span(\p_t,\p_\theta)$
we set
$$
J_+(t,z,\theta)|_{T_z \widehat{\Sigma}} := j(z), \qquad
J_+(t,z,\theta) \p_t = \p_\theta.
$$
Note that since $\p_r = \p_s$ and $R_+ = R_0 = 
\frac{1}{m} \p_\phi^\# = \frac{1}{m} \p_\phi$ on
$\widehat{\vV} \cap \widehat{\nN}_+(\p E)$,
this is consistent with the existing definition of $J_+$ on 
$\widehat{\nN}_+(\p E)$.  The following observation is immediate.

\begin{prop}
\label{prop:vertebrae}
For each connected component $\dot{\Sigma}_0 \subset \widehat{\Sigma}$
and each $(t,\theta) \in (-1,-1/2) \times S^1$, the surface
$$
\{t\} \times \dot{\Sigma}_0 \times \{\theta\} \subset 
\widehat{\vV} \subset \widehat{E}
$$
is the image of a properly embedded $J_+$-holomorphic curve whose
intersection with the cylindrical end $\widehat{\nN}_+(\p E)$
is a union of positive trivial half-cylinders over simply covered closed
orbits of~$R_+$.  \qed
\end{prop}
The $J_+$-holomorphic curves in the above proposition will be referred to
henceforward as \defin{holomorphic vertebrae}.

A natural extension of $J_+$ into the region 
$$
(-1,0] \times \widecheck{M}^+\paper \subset \nN(\p_v E)
$$
is defined by requiring $J_+$ to be invariant under the flow of $\p_s$
on $(-1,\infty) \times \widecheck{M}^+\paper \subset \widehat{\nN}(\p_v E)$.
Note that this is also compatible with previous choices on the
intersection of this region with~$\widehat{\vV}$.

At this point, we have defined $J_+$ everywhere on $\widehat{E}$ except
in the region of $\widehat{\nN}(\p_h E)$ bounded between
$\{ t = -1/2 \}$ and~$M^+$; this is roughly the region between the dark and light
shading in Figure~\ref{fig:end}.
The purpose of the next two lemmas is to find an extension of $J_+$ to this
region that will also fit into an almost Stein structure.

\begin{lemma}
\label{lemma:almostStein}
On the region where $J_+$ has been defined so far, it satisfies
$$
- df_+ \circ J_+ = \lambda_+
$$
for a suitable smooth function $f_+ : \widehat{E} \to \RR$ and
$1$-form $\lambda_+$ on $\widehat{E}$ such that:
\begin{enumerate}
\item $d\lambda_+$ is symplectic and compatible with the orientation of~$\widehat{E}$;
\item $df_+(V') > 0$ everywhere, where $V'$ denotes the vector field dual to
$\lambda_+$, i.e.~defined by $d\lambda_+(V',\cdot) \equiv \lambda_+$;
\item $\lambda_+ = \frac{1}{K} \lambda_K$ on $\widehat{\nN}(\p_h E)$, and the
restrictions of $\lambda_+$ and $\frac{1}{K} \lambda_K$ to the vertical
subbundle $V\widehat{E}$ match everywhere;
\item $\lambda_+$ restricts to $M^-$ as a contact form inducing a contact
structure isotopic to~$\xi_-$.
\end{enumerate}
\end{lemma}
\begin{proof}
We begin by finding a function $f_+$ that satisfies $-df_+ \circ J_+ = 
\frac{1}{K} \lambda_K$, or equivalently $\frac{1}{K}\lambda_K \circ J_+ = df_+$,
on the regions where $J_+$ has been defined so far.

On $[0,\infty) \times \widecheck{M}^+\spine \subset \widehat{\nN}_+(\p E)$
this is easy because, as we saw in the proof of Lemma~\ref{lemma:theEnd},
$\frac{1}{K} \lambda_K = e^r \Lambda_+$.  Since any
$J_+ \in \jJ(\hH_+)$ automatically satisfies
$$
(e^r \Lambda_+) \circ J_+ = d(e^r),
$$
we are free to fix $f_+ := e^r$ on this region.

On $[0,\infty) \times \widecheck{M}^+\corner \subset \widehat{\nN}_+(\p E)$
the same computation is valid and produces $f_+ = e^r$ wherever 
$\frac{1}{K} \lambda_K = e^r \Lambda_+$, which is still true for 
$\rho \le 1/4$ but ceases to be true in 
$\{ 1/4 < \rho < 1 \}$, so here a more careful computation is required.
Since $F_+(\rho) = 0$ and $G_+(\rho) = -\rho$ in this region,
we have $\Omega_+ = -\frac{1}{K} e^{-\rho} \, d\rho \wedge d\theta$
and $\Lambda_+ = m \, d\phi + 
\frac{1}{K} e^{-\rho} \beta(-\rho)\, d\theta$, hence $R_+ = 
\frac{1}{m} \p_\phi$, and $J_+$ is determined by \eqref{eqn:Jinterp}.
In the mean time $\frac{1}{K} \lambda_K = m e^s \, d\phi + \frac{1}{K}
e^t \, d\theta$, with the $s$ and $t$ coordinates related to $r$ and
$\rho$ via the embedding defined in Lemma~\ref{lemma:theEnd}, namely
$$
s = r, \qquad t = -\rho + \log\left[ (e^r - 1) \beta(-\rho) + 1 \right].
$$
Evaluating $\frac{1}{K} \lambda_K \circ J_+$ on the unit vectors in
$(r,\rho,\phi,\theta)$-coordinates then leads to the formula
$$
\frac{1}{K} \lambda_K \circ J_+ = e^r\, dr - \frac{1}{K} e^{-\rho} 
\left[ 1 - \beta(-\rho) \right] \, d\rho = df_+,
$$
where $f_+ := e^r - \frac{1}{K} g(\rho)$, with
$$
g(\rho) := \int_0^\rho e^{-x} \left[ 1 - \beta(-x) \right] \, dx.
$$
The details of $g(\rho)$ are unimportant beyond the following two 
observations: first, it is nonnegative, and strictly positive for all
$\rho \ge 1/2$; second, its derivative for $\rho \ge 1/2$ is $e^{-\rho}$,
so using the alternative coordinates $s=r$ and $t = -\rho$ on this region,
we can rewrite $f_+$ as
$$
f_+ = e^s + \frac{1}{K} (e^t - c)
$$
for some constant $c \in \RR$ such that $e^t - c < 0$ for all
$t \in (-1,-1/2]$.

The above function can now be extended over $\widehat{\vV}$ using the
$j$-convex function $\varphi : \widehat{\Sigma} \to \RR$, which we recall 
satisfies $-d\varphi \circ j = \sigma$ and $\varphi(s,\phi) = e^s$ on
$\widehat{\nN}(\p \Sigma)$.  Indeed, the function
$$
f_+ := \varphi + \frac{1}{K} (e^t - c)
$$
on $\widehat{\vV}$ has the property $-df_+ \circ J_+ = \sigma 
+ \frac{1}{K} e^t\, d\theta = \frac{1}{K} \lambda_K$.  Moreover,
since $\varphi$ is subharmonic and equals~$1$ on~$\p\Sigma$, it is
strictly less than~$1$ on the interior of~$\Sigma$, implying that
$f_+ < 1$ on the portion of $\widehat{\vV}$ disjoint from
$\widehat{\nN}_+(\p E)$.  Since $V_K = V_\sigma + \p_t$, it follows
that $f_+$ can be extended over the rest of $\widehat{\nN}(\p_h E)$
satisfying $df_+(V_K) > 0$.

We will next extend $f_+$ as a $J_+$-convex function to
$\widehat{\nN}(\p_v E) \setminus \widehat{\nN}(\p_v E \cap \p_h E)$,
which includes the rest of $\widehat{\nN}_+(\p E)$.
For this we can use the Thurston trick for almost Stein structures,
as in \cite{LisiVanhornWendl1}*{\S 2.4} or the appendix of \cite{BaykurVanhorn:large}.
Recall that in \S\ref{sec:Liouville} we chose a 
function $f\fib : M\paper \to \RR$ that matches $e^t$ in $\nN(\p M\paper)$
and is fiberwise $J_+$-convex; composing this function with the projection
$\widehat{\nN}(\p_v E) \supset (-1,\infty) \times M\paper \to M\paper$
gives a function
$$
f\fib : (-1,\infty) \times M\paper \to \RR
$$
such that $\p_s f\fib \equiv 0$, $f\fib(s,\phi,t,\theta) = e^t$ for
$t \in (-1,-1/2]$ and the $1$-form
$$
\lambda\fib := -d f\fib \circ J_+
$$
is fiberwise Liouville (outside the region $\{t > -1/2\}$, which we are free
to ignore for this discussion).  Moreover, the restriction of $\lambda\fib$
to the vertical subspaces matches our construction of $\lambda$
from \S\ref{sec:Liouville}.
Now observe that since $\p_r = \p_s$ and
$J_+ \p_r = R_+$ is a horizontal lift of the unit vector field in $S^1$,
the projection $\Pi_v$ maps the region in question \emph{holomorphically} to
$(-1,\infty) \times S^1$ with its standard complex structure~$i$, thus
$$
- d (e^s \circ \Pi_v) \circ J_+ = \Pi_v^*( - d(e^s) \circ i) =
e^s \, d\Pi_v = \sigma.
$$
Extending $f_+$ by $f_+ = e^s + \frac{1}{K} f\fib$, it follows that
$$
\lambda_+ := - df_+ \circ J_+ = \sigma + \frac{1}{K} \lambda\fib,
$$
and by the usual application of the Thurston trick
as in \cites{LisiVanhornWendl1,BaykurVanhorn:large}, $d\lambda_+$ is symplectic if $K > 0$ is
sufficiently large.  We are free to assume this, since none of the
other data on the region in question depends on the value of~$K$.
Note that by the same argument that was used previously for $V_K$
(see Lemma~\ref{lemma:LiouvilleTransverse}), we can
also assume after increasing $K > 0$ that the dual Liouville vector
field $V'$ satisfies $ds(V') > 0$ everywhere on $\widehat{\nN}(\p_v E)$.
Since $\lambda_+ = \frac{1}{K} \lambda_K$ in $\widehat{\nN}(\p_h E)$
by construction, we also have $V' = V_K$ in this region, so that
$V'$ is manifestly transverse to~$M^-$, implying that $\lambda_+|_{TM^-}$
is contact.  Moreover, this contact form on $M^-$ has been constructed
in the same manner as a Giroux form for the spinal open book, implying
that the induced contact structure is isotopic to~$\xi_-$.
\end{proof}

The next result is of a much more general nature.

\begin{lemma}
\label{lemma:almostStein2}
Assume $(W,\omega)$ is a smooth symplectic manifold with a Liouville vector
field $V_\lambda$ that is nowhere zero, and denote the dual Liouville form 
by~$\lambda$ (i.e.~$d\lambda = \omega$
and $\omega(V_\lambda,\cdot) = \lambda$).  Suppose
$f : W \to \RR$ is a smooth function satisfying $df(V_\lambda) > 0$.  Then
$$
\xi := \ker df \cap \ker \lambda \subset TW
$$
is a smooth symplectic subbundle of codimension~$2$, and there is a natural
homeomorphism
$$
\jJ(\lambda,f) \to \jJ(\xi,\omega) : J \mapsto J|_{\xi},
$$
where $\jJ(\lambda,f)$ denotes the space of $\omega$-compatible almost
complex structures $J$ on $W$ satisfying $\lambda = -df \circ J$, and
$\jJ(\xi,\omega)$ is the space of compatible complex structures
on the symplectic vector bundle $(\xi,\omega)$.  In particular, 
it follows that $\jJ(\lambda,f)$ is nonempty and contractible.
\end{lemma}
\begin{proof}
The fact that $\lambda$ is Liouville and $df(V_\lambda) > 0$ implies that
$\lambda$ restricts as a contact form to each level set of~$f$; the
subbundle $\xi$ is then the union of all the resulting contact structures
in the level sets.
We claim first that any $J \in \jJ(\lambda,f)$ preserves~$\xi$ and thus has
a restriction $J|_\xi$ in $\jJ(\xi,\omega)$.  Indeed, if $\lambda = -df \circ J$
then $v \in \ker \lambda$ implies $Jv \in \ker df$ and $v \in \ker df$ implies
$Jv \in \ker \lambda$, so this proves the claim.
Now let $R_\lambda$ denote the unique vector field on $W$ satisfying
$$
df(R_\lambda) \equiv 0, \qquad \lambda(R_\lambda) \equiv 1, \qquad
d\lambda(R_\lambda,\cdot)|_\xi \equiv 0,
$$
i.e.~$R_\lambda$ restricts to each level set of $f$ as the Reeb vector field
determined by~$\lambda$.  Then the relation $\lambda = -df \circ J$ and the
fact that $\omega$ is $J$-invariant (since $J$ is $\omega$-compatible)
imply that the vector field $\frac{1}{df(V_\lambda)} J V_\lambda$ satisfies
the same conditions that define~$R_\lambda$, hence
$$
J V_\lambda = df(V_\lambda) R_\lambda.
$$
This relation defines the inverse of the map $\jJ(\lambda,f) \to
\jJ(\xi,\omega) : J \mapsto J|_\xi$.
\end{proof}

Using the Liouville form $\lambda_+$ and Lyapunov function $f_+$ on
$\widehat{E}$ supplied by
Lemma~\ref{lemma:almostStein}, we can now use Lemma~\ref{lemma:almostStein2}
to extend $J_+$ over the rest of $\widehat{E}$ such that 
$\lambda_+ = - df_+ \circ J_+$ and $J_+$ is $d\lambda_+$-compatible, hence
$f_+$ is $J_+$-convex.  Note that Lemma~\ref{lemma:almostStein2} allows for
considerable freedom in the choice of this extension, and we shall only
need to impose one further condition:
$$
\text{$J_+$ is $S^1$-invariant on $\widehat{\nN}(\p_h E) = (-1,\infty) \times
\widehat{\Sigma} \times S^1$}.
$$
Here $S^1$-invariance means invariance with respect to the coordinate
$\theta \in S^1$; the assumption is already satisfied on the portions of 
$\widehat{\nN}(\p_h E)$ where $J_+$ has been defined so far, and it is
possible on the rest because $\lambda_K$, $\widehat{\vV}$ and $M^+$ are 
all $S^1$-invariant objects, and so is $f_+$ without loss of generality.

For applications to almost Stein fillings, we will take $\eta = 0$ and
our symplectic form on $\widehat{E}$ is thus the exterior derivative of
the Liouville form $\frac{1}{K} \lambda_K$.  
We now have a minor headache however
because the model almost Stein structure $(J_+,f_+)$ arising from the
above construction produces another Liouville structure
$\lambda_+ = -df_+ \circ J_+$,
which is in general different from $\frac{1}{K} \lambda_K$, in particular
they differ on~$\widehat{\nN}(\p_v E)$.  In order to define a useful
notion of energy for holomorphic curves in this setting, we will need
the following interpolation.

\begin{lemma}
\label{lemma:LiouvilleInterp}
There exists a Liouville form $\Theta$ on $\widehat{E}$ with the following
properties:
\begin{enumerate}
\item $\Theta = \lambda_+$ on $E$;
\item $\Theta = \frac{1}{K} \lambda_K$ on $[T,\infty) \times M^+
\subset \widehat{\nN}_+(\p E)$ for $T > 0$ sufficiently large;
\item $d\Theta$ tames~$J_+$.
\end{enumerate}
\end{lemma}
\begin{proof}
We set $\Theta = \lambda_+$ on $\widehat{\nN}(\p_h E)$ since $\lambda_+$ and
$\frac{1}{K} \lambda_K$ already match on this region.  On $\widehat{\nN}(\p_v E)$,
choose $\Theta$ to be of the form
$$
\Theta = \left[ 1 - g(s) \right] \frac{1}{K} \lambda_K + g(s) \lambda_+
$$
for some smooth function $g : (-1,\infty) \to [0,1]$ with $g(s) = 1$ for
$s \le 0$ and $g(s) = 0$ for $s$ sufficiently large.  We then have
\begin{equation}
\label{eqn:dTheta}
d\Theta = \left[ 1 - g(s) \right] \frac{1}{K} \, d\lambda_K + g(s) d\lambda_+ +
g'(s) \, ds \wedge \left( \lambda_+ - \frac{1}{K} \lambda_K \right).
\end{equation}
Recall that $J_+$ is tamed by both $\frac{1}{K} d\lambda_K$ and
$d\lambda_+$; it is compatible with the former by construction, and it is tamed
by the latter because $\lambda_+ = - df_+ \circ J_+$ where $f_+$ is $J_+$-convex.
It follows that the interpolation forming the first two terms in
\eqref{eqn:dTheta} is also a nondegenerate $2$-form taming~$J_+$; moreover,
the construction of $\lambda_+$ and $\lambda_K$ guarantees that it tames
$J_+$ in a uniform way as $s \to \infty$.  It therefore suffices to choose
$g$ changing slowly enough so that the $g'(s)$ term in \eqref{eqn:dTheta}
does not ruin nondegeneracy, and this can be done at the cost of achieving
the condition $g(s) = 0$ only for $s \ge T$ with $T$ sufficiently large.
\end{proof}

\subsection{Holomorphic pages}
\label{sec:holPages}

The main advantage of choosing $J_+$ compatible with the stable Hamiltonian
structure $\hH_+ = (\Omega_+,\Lambda_+)$ instead of 
contact data is that the pages of $\boldsymbol{\pi}$ can be lifted to properly
embedded $J_+$-holomorphic curves in $\widehat{\nN}_+(\p E)$.
Since $J_+$ on $\widehat{\nN}_+(\p E) = [0,\infty) \times M^+$
belongs to $\jJ(\hH_+)$, we can equally well regard $J_+$
as an $\RR$-invariant almost complex structure on 
$\RR \times M^+$, and we will now use it to construct a $J_+$-holomorphic 
foliation on $\RR \times M^+$.

Denote by $T\fF_+$ the $2$-dimensional distribution on $\RR \times M^+$
defined by
$$
(T\fF_+)_{(r,x)} = \begin{cases}
\Xi_+ & \text{ for $x \in \widecheck{M}^+\paper$},\\
\Span\{\p_\theta,J_+ \p_\theta\} & \text{ for $x \in \widecheck{M}^+\corner \cup 
\widecheck{M}^+\spine$}.
\end{cases}
$$
This distribution is smooth and $J_+$-invariant;
indeed, $\Xi_+$ is necessarily $J_+$-invariant since
$J_+ \in \jJ(\hH_+)$, and since $\Xi_+$ matches the vertical 
subbundle of $\pi\paper : M\paper \to S^1$ on $\widecheck{M}^+\paper$,
it also is spanned by $\p_\theta$ and $J_+\p_\theta$ in the
collar where the $\theta$-coordinate is defined.
It is also easy to see that $T\fF_+$ is $\RR$-invariant, and it is
\emph{integrable}: the latter is obvious in $\RR \times \widecheck{M}^+\paper$,
and everywhere else it follows from the fact that $J_+$ is $S^1$-invariant,
as this implies
\begin{equation}
\label{eqn:thetaCommute}
[\p_\theta,J_+\p_\theta] \equiv 0.
\end{equation}
Denote by $\fF_+$ the set of leaves of the foliation on $\RR\times M^+$ 
tangent to~$T\fF_+$.  The next result shows that each of these leaves
is the image of an embedded asymptotically cylindrical
$J_+$-holomorphic curve as defined in \S\ref{sec:energy}, hence $\fF_+$ is a
\emph{finite energy foliation} in the sense of
Hofer-Wysocki-Zehnder \cite{HWZ:foliations}.
In the following, we use the Riemannian metric
$$
\langle\cdot,\cdot \rangle := d\sigma(\cdot,j\cdot)
$$
on $\Sigma$ in order to define the gradient vector field $\nabla H$ of
$H : \Sigma \to [0,\infty)$.  Observe that on the collar $\nN(\p\Sigma)$,
since $\sigma = m e^s\, d\phi$ and $j\p_s = \frac{1}{m} \p_\phi$ for the
appropriate multiplicity $m \in \NN$, we have
$d\sigma(\cdot,j\cdot) = e^s \left( ds \otimes ds + m^2 \, d\phi \otimes d\phi \right)$,
while $H(s,\phi)$ depends only on the $s$-coordinate, thus
$\nabla H$ points in the $s$~direction, orthogonal to~$\p\Sigma$.  The
Hamiltonian vector field $X_H$ determined on $(\Sigma,d\sigma)$ by~$H$
can now be written as
\begin{equation}
\label{eqn:XH}
X_H = j \nabla H.
\end{equation}

\begin{prop}
\label{prop:Jfoliation}
The leaves of the $\RR$-invariant foliation $\fF_+$ are the images of
asymptotically cylindrical $J_+$-holomorphic curves.  In fact,
each leaf of this foliation is one of the following:
\begin{enumerate}
\item A \defin{trivial cylinder} $\RR \times \gamma$, where
$\gamma \subset M^+$ is a closed Reeb orbit of the form
$\gamma = \{z\} \times S^1 \subset \widecheck{M}^+\spine \subset
\Sigma \times S^1$ for some $z \in \CritMorse(H)$.
\item A \defin{holomorphic gradient flow cylinder}, admitting a (not necessarily
holomorphic) parametrization 
$u : \RR \times S^1 \hookrightarrow \RR \times M^+$ of the form
$$
u(s,t) = (a(s),\ell(s),t) \in \RR \times \widecheck{M}^+\spine \subset
\RR \times \Sigma \times S^1,
$$
where $a : \RR \to \RR$ is a strictly increasing proper function and
$\ell : \RR \to \Sigma$ is a solution of the gradient flow
equation $\dot{\ell} = \nabla H(\ell)$ approaching two distinct critical
points of $H$ as $s \to \pm\infty$.
\item A \defin{holomorphic page}, which is a connected and properly embedded
submanifold formed as a union of subsets of the following type:
\begin{itemize}
\item $\{s\} \times P \subset \RR \times \widecheck{M}^+\paper$, where
$s \in \RR$ is a constant and $P \subset \widecheck{M}^+\paper$ is the portion of
a page of $\pi\paper : M\paper \to S^1$ lying in $\widecheck{M}^+\paper$;
\item Annuli admitting (not necessarily holomorphic) parametrizations
$u : (-1,1) \times S^1 \hookrightarrow \RR \times M^+$ of the form
$$
u(s,t) = (a(s),s,\phi,t) \in \RR \times (-1,1) \times S^1 \times S^1
\subset \RR \times \widecheck{M}^+\corner
$$
for some bounded functions $a : (-1,1) \to \RR$ and constants $\phi \in S^1$;
\item Half-cylinders admitting (not necessarily holomorphic) parametrizations
$u : [0,\infty) \times S^1 \hookrightarrow \RR \times M^+$ of the form
$$
u(s,t) = (a(s),\ell(s),t) \in \RR \times \widecheck{M}^+\spine \subset
\RR \times \Sigma \times S^1,
$$
where $a : [0,\infty) \to \RR$ is a strictly increasing proper function and
$\ell : [0,\infty) \to \Sigma$ is a solution of the gradient flow
equation $\dot{\ell} = \nabla H(\ell)$ that begins at time $s=0$ as a
trajectory in $\nN(\p\Sigma)$ orthogonal to $\p\Sigma$ and
approaches a critical point of $H$ as $s \to \infty$.
\end{itemize}
\end{enumerate}
In particular, each of the holomorphic gradient flow cylinders and pages
projects through $\RR \times M^+ \to M^+$ to an embedded
surface in $M^+$ whose closure is a compact embedded surface
bounded by Reeb orbits in $\CritMorse(H) \times S^1 \subset \widecheck{M}^+\spine$.
\end{prop}
\begin{proof}
At any point $(z,\theta) \in \CritMorse(H) \times S^1$, $\p_\theta$ is proportional
to~$R_+$, hence $J_+\p_\theta$ is proportional to $\p_r$ and the trivial 
cylinder $\RR\times \gamma \subset \RR \times M^+$ over the periodic 
orbit~$\gamma$ through~$(z,\theta)$ therefore forms an integral 
submanifold of the distribution.  Similarly, each integral submanifold in
$\RR \times \widecheck{M}^+\paper$ is contained
in a set of the form $\{s\} \times P \subset \RR \times M\paper$, with
$s \in \RR$ a constant and $P \subset M\paper$ a page of~$\boldsymbol{\pi}$.
To complete the proof, we mainly need to justify the following two claims:
\begin{itemize}
\item At any point $(z,\theta) \in \widecheck{M}^+\spine \subset \Sigma \times S^1$,
there exist $a,b,c \in \RR$ with $c \ne 0$ such that
\begin{equation}
\label{eqn:linCombNew}
J_+ \p_\theta = a\p_\theta + b \p_r + c\nabla H.
\end{equation}
\item At any point $(\rho,\phi,\theta) \in \widecheck{M}^+\corner$, there exist
$b,c \in \RR$ with $c \ne 0$ such that
\begin{equation}
\label{eqn:linCombNew2}
J_+ \p_\theta = b \p_r + c \p_\rho.
\end{equation}
\end{itemize}
To verify \eqref{eqn:linCombNew}, we first observe that since $\p_\theta$
is transverse to~$\Xi_+$ on $\widecheck{M}^+\spine \subset \Sigma \times S^1$,
there exist unique functions $P, Q : \Sigma \to \RR$ such that
$$
\nabla H + P\p_\theta \in \Xi_+ \quad\text{ and }\quad
j\nabla H + Q\p_\theta \in \Xi_+,
$$
and the definition of $J_+$ in terms of $j$ via the natural fiberwise 
isomorphism $\Xi_+ \to T\Sigma$ then implies
$J_+(\nabla H + P\p_\theta) = j\nabla H + Q\p_\theta$.  
By \eqref{eqn:ReebSigma} and \eqref{eqn:XH}, 
we then have
$$
J_+(\nabla H + P\p_\theta) = -\frac{1}{\nondegParam} e^{\nondegParam H} R_+ + 
\left( Q + \frac{1 + \nondegParam \sigma(X_H)}{\nondegParam} K \right) \p_\theta,
$$
and applying $-J_+$ to both sides yields
$$
\nabla H + P\p_\theta = - \frac{1}{\nondegParam} e^{\nondegParam H} \p_r -
\left( Q + \frac{1 + \nondegParam \sigma(X_H)}{\nondegParam} K \right) J_+ \p_\theta.
$$
The coefficient in front of $J_+\p_\theta$ cannot be zero since
$\nabla H$, $\p_\theta$ and $\p_r$ are not linearly dependent, so this
allows us to write $J_+\p_\theta$ in the form \eqref{eqn:linCombNew} as claimed.
In fact, we obtain the following precise formula for $T\fF_+$ in this region,
\begin{equation}
\label{eqn:TFspine}
T\fF_+ = \Span\left\{ \p_\theta , \nabla H + \frac{1}{\nondegParam} e^{\nondegParam H}
\p_r \right\} \quad \text{ on $\RR \times \widecheck{M}^+\spine$},
\end{equation}
which shows that the functions $a(s)$ appearing in 
parametrizations of leaves in $\RR \times \widecheck{M}^+\spine$ are strictly
increasing.

The proof of \eqref{eqn:linCombNew2} follows similarly from our
definition of~$J_+$ in~$\widecheck{M}^+\corner$.  Here we have $\p_\rho \in \Xi_+$,
with $R_+$ given by \eqref{eqn:OmegaLambdaEps}, and
Lemma~\ref{lemma:GnotFlat} implies that $R_+$ and $\p_\theta$ are 
always linearly independent, so they span the same subbundle as
$\p_\phi$ and~$\p_\theta$.  This implies
$$
J_+\p_\rho \in \Span(\p_\theta,\p_\phi) = \Span(\p_\theta,R_+),
$$
and thus $J_+\p_\rho = a \p_\theta + b R_+$ 
for some $a,b\in \RR$ with $a \ne 0$.  
Applying~$J_+$ to both sides of this gives the desired result.
\end{proof}

In the following, we shall often blur the distinction between 
leaves of $\fF_+$ and the corresponding unparametrized holomorphic curves,
referring to both via parametrizations $u : \dot{S} \to \RR\times M^+$.
We will examine the analytical properties of the curves in $\fF_+$
more closely in~\S\ref{sec:holOpenBook}.

Identifying $[0,\infty) \times M^+ \subset \RR \times M^+$ in the usual
way with $\widehat{\nN}_+(\p E) \subset \widehat{E}$, $\fF_+$ also
determines a foliation on $\widehat{\nN}_+(\p E)$, which we shall extend
into $\widehat{E}$ by setting 
$$
T\fF_+ := \begin{cases}
\Span\{\p_\theta,J_+ \p_\theta\} & \text{ in $\widehat{\nN}(\p_h E)$}, \\
V\widehat{E} & \text{ everywhere else}.
\end{cases}
$$
Indeed:
\begin{prop}
\label{prop:integrable}
The distribution $T\fF_+$ on $\widehat{E}$ is $J_+$-invariant and integrable,
and matches the vertical subbundle $VE$ on a neighborhood of~$M^-$.
Moreover, $T\fF_+$ is transverse to the hypersurfaces $\{ t = \text{const} \}$
in~$\widehat{\nN}(\p_h E)$.
\end{prop}
\begin{proof}
Integrability follows from \eqref{eqn:thetaCommute} since $J_+$ is
$S^1$-invariant in $\widehat{\nN}(\p_h E)$, and $J_+$-invariance is also
immediate because $J_+$ was defined to preserve the vertical subbundle
outside of $\widehat{\nN}(\p_h E)$.  The transversality claim follows
from the fact that $J_+$ is $\omega_E$-tame and 
$\omega_E = d\sigma + \frac{1}{K} e^t \, dt \wedge d\theta$ in
$\widehat{\nN}(\p_h E)$, thus
$$
0 < \omega_E(\p_\theta,J_+\p_\theta) = - \frac{1}{K} e^t \, dt(J_+\p_\theta).
$$
\end{proof}

Figure~\ref{fig:holfol} shows a picture of the foliation on~$\widehat{E}$,
plus a single holomorphic vertebra (see Prop.~\ref{prop:vertebrae}) 
that intersects every leaf positively.

\begin{figure}
\begin{postscript}
\psfrag{t}{$t$}
\psfrag{s}{$s$}
\psfrag{pvE}{$\p_v E$}
\psfrag{phE}{$\p_h E$}
\psfrag{M+}{$M^+$}
\psfrag{M-}{$M^-$}
\psfrag{r}{$r$}
\psfrag{rho}{$\rho$}
\psfrag{Vhat}{$\widehat{\vV}$}
\psfrag{...}{$\ldots$}
\includegraphics{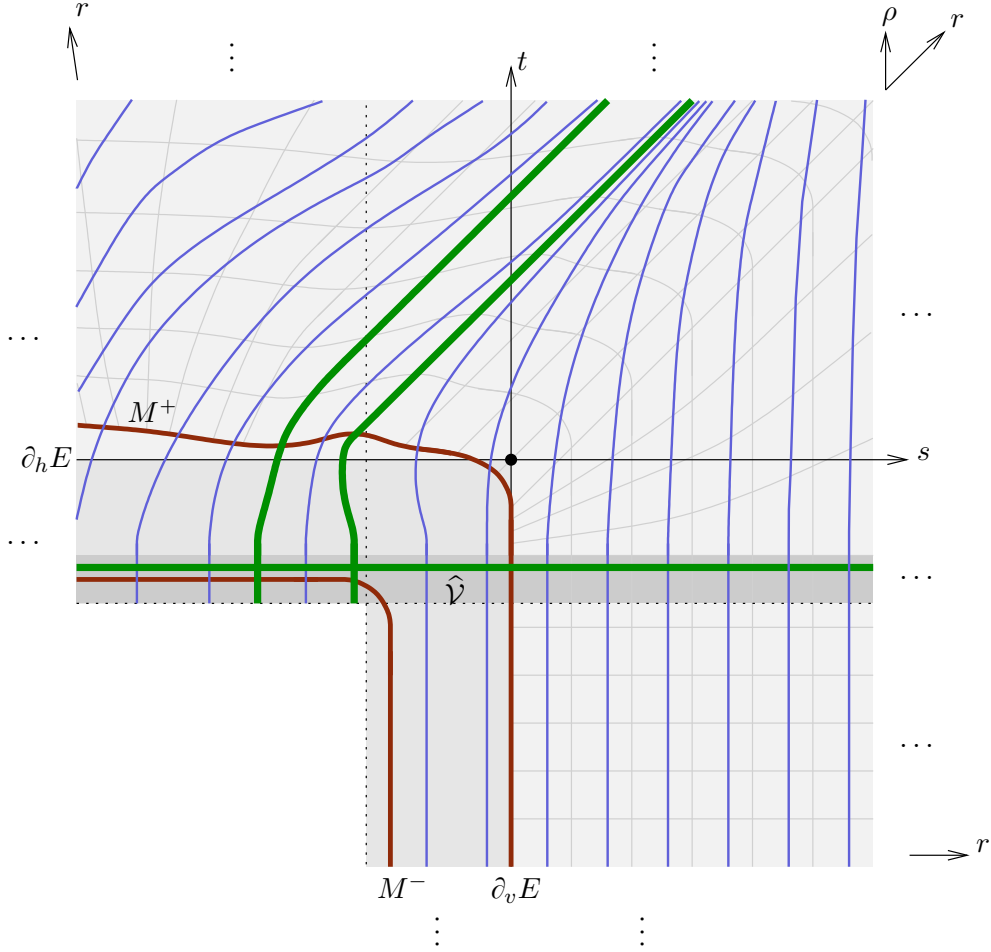}
\end{postscript}
\caption{\label{fig:holfol}
The $J_+$-holomorphic foliation $\fF_+$ in~$\widehat{E}$.  The picture includes
two special leaves whose intersections with the cylindrical end
$\widehat{\nN}_+(\p E) = [0,\infty) \times M^+$ are trivial cylinders over
Reeb orbits corresponding to critical points of $H : \Sigma \to \RR$,
and all other leaves approach these cylinders asymptotically at infinity.
All leaves are also intersected transversely by a holomorphic vertebra
in the region~$\widehat{\vV}$.}
\end{figure}

\subsection{Large subdomains with weakly contact boundary}
\label{sec:largeSubdomains}

The construction in the present subsection will be needed in the final step
of the proofs of Theorems~\ref{thm:classification} and~\ref{thm:weak},
in order to show that our $J$-holomorphic foliation obtained by analytical
methods gives rise to a bordered Lefschetz fibration with supported
symplectic structure in the sense of \cite{LisiVanhornWendl1}*{\S 2.3}.
The goal is to exhaust $\widehat{E}$ by bounded subdomains
$$
\widehat{E}_R \subset \widehat{E}, \qquad
\widehat{E} = \bigcup_{R > 0} \widehat{E}_R
$$
such that each $\p \widehat{E}_R$ is a weakly contact hypersurface 
(with corner) deformation equivalent to $(M^-,\xi_-)$ and a neighborhood
of $\p \widehat{E}_R$ in $\widehat{E}_R$ looks like the neighborhood of
the boundary in a bordered Lefschetz fibration with fibers given by
leaves of~$\fF_+$.

Fix a pair of numbers $c > 0$ and $\delta \in (1/4,1/2)$,
and define a smooth
hypersurface $M^c \subset \widehat{\nN}_+(\p E)$ with nonempty boundary via the
following conditions (see Figure~\ref{fig:largeSubd}):
\begin{enumerate}
\item $M^c$ contains $\{c\} \times \widecheck{M}^+\paper \subset [0,\infty)
\times M^+ = \widehat{\nN}_+(\p E)$;
\item $M^c$ is a union of subsets of leaves of~$\fF_+$;
\item $\p M^c \subset \{ \rho = \delta \} \subset [0,\infty) \times \widecheck{M}^+\corner$.
\end{enumerate}
It will be useful to note that $M^c$ is $\theta$-invariant in the region near
its boundary where the $\theta$-coordinate is defined.
By adjusting $\delta$ appropriately, one can also assume that
$\beta(-\delta) > 0$ and that $M^c$ is everywhere transverse to the Liouville
vector field~$V_K$; the latter follows from the formula $V_K = \p_s + \p_t$
in the diagonal end, as we are free to assume by moving $\delta$ closer to
the region where $\beta(-\rho)=0$ that the tangent spaces to $M^c$ are
always $C^0$-close to those of the ``vertical'' hypersurfaces 
$\{c\} \times \widehat{M}\paper$.  The key consequence of the condition
$\beta(-\delta) > 0$ is the following: by \eqref{eqn:PsiInterface}, we have
$$
\p_r = \p_s + \frac{e^r \beta(-\rho)}{(e^r - 1) \beta(-\rho) + 1} \p_t
$$
in the region $1/4 < \rho < 1/2$, so $\beta(-\rho) > 0$ implies that the flow
of $\p_r$ moves positively in the $t$-coordinate.  In particular, given
$R > 0$, we can find a number $r_1 > 0$ such that if $\Phi_{\p_r}^{t}$ denotes
the time~$t$ flow of $\p_r$,
$$
\p \left( \Phi_{\p_r}^{r_1}(M^c) \right) \subset \{R\} \times \widehat{M}\spine
\subset \widehat{\nN}(\p_h E).
$$
With this understood, define $\widehat{E}_R \subset \widehat{E}$ to be the
region in $\widehat{E}$ bounded by $\Phi_{\p_r}^{r_1}(M^c) \subset
\widehat{\nN}(\p_v E)$ and 
$\{R\} \times \widehat{M}\spine \subset \widehat{\nN}(\p_h E)$.  Its boundary has two
smooth faces $\p \widehat{E}_R = \p_v \widehat{E}_R \cup \p_h \widehat{E}_R$,
where
$$
\p_v \widehat{E}_R = \Phi_{\p_r}^{r_1}(M^c) \subset \widehat{\nN}(\p_v E),
$$
and
$$
\p_h \widehat{E}_R \subset \{R\} \times \widehat{M}\spine \subset
\widehat{\nN}(\p_h E).
$$
The $\RR$-invariance of the foliation $\fF_+$ implies that $\p_v \widehat{E}_R$
is a union of $1$-parameter families of compact subsets of leaves of~$\fF_+$.

\begin{figure}
    \begin{postscript}
\psfrag{t}{$t$}
\psfrag{s}{$s$}
\psfrag{pvE}{$\p_v E$}
\psfrag{phE}{$\p_h E$}
\psfrag{M+}{$M^+$}
\psfrag{M-}{$M^-$}
\psfrag{r}{$r$}
\psfrag{rho}{$\rho$}
\psfrag{Vhat}{$\widehat{\vV}$}
\psfrag{Mc}{$M_c$}
\psfrag{pMc}{$\p M_c$}
\psfrag{pvER}{$\p_v \widehat{E}_R$}
\psfrag{phER}{$\p_h \widehat{E}_R$}
\psfrag{pvERphER}{$\p_v\widehat{E}_R \cap \p_h\widehat{E}_R$}
\psfrag{rhodelta}{$\rho=\delta$}
\psfrag{NpvER}{$\nN(\p_v \widehat{E}_R)$}
\psfrag{NphER}{$\nN(\p_h \widehat{E}_R)$}
\psfrag{t=R}{$t=R$}
\psfrag{...}{$\ldots$}
\includegraphics[scale=1]{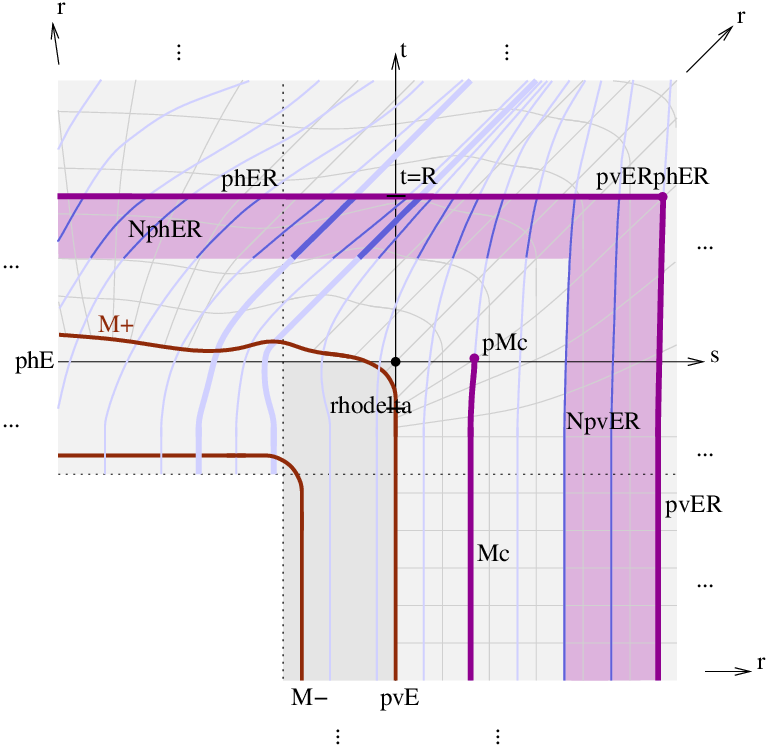}
\end{postscript}
\caption{\label{fig:largeSubd}
The construction of the region $\widehat{E}_R$ with boundary
$\p\widehat{E}_R = \p_v\widehat{E}_R \cup \p_h\widehat{E}_R$ and corner
$\p_v\widehat{E}_R \cap \p_h\widehat{E}_R$, together with the collar
neighborhoods $\nN(\p_v\widehat{E}_R)$, $\nN(\p_h\widehat{E}_R) \subset \widehat{E}_R$, which
carry fibrations whose fibers are leaves of the holomorphic foliation
from Figure~\ref{fig:holfol}.}
\end{figure}

\begin{lemma}
\label{lemma:bigContact}
For each $R > 0$, there exists a smooth isotopy of $\widehat{E}_R$ to $E$ 
through domains with the property that both smooth faces of their boundaries 
are weakly contact hypersurfaces in $(E,\omega_E)$ with the contact structure
induced by~$\lambda_K$, and the corner of each is contained in
$\widehat{\nN}(\p_v E \cap \p_h E)$.
\end{lemma}
\begin{proof}
Since $\omega_E$ is exact in $\widehat{\nN}(\p_h E)$ with Liouville vector
field~$V_K = V_\sigma + \p_t$, the contact-type property for
$\p_h \widehat{E}_R$ is immediate.  The weakly contact property for
$\p_v \widehat{E}_R$ follows mostly from Lemma~\ref{lemma:etaPert};
we only need to examine the ``bent'' region near the boundary of
$\p_v\widehat{E}_R$ slightly more closely.  Since this region also lies
in $\widehat{\nN}(\p_h E)$, it suffices to check that $\p_v\widehat{E}_R$
is transverse to $V_K = \p_s + \p_t$.  We have explicitly assumed this to
be true for $M^c$, so we need to show that it remains true after flowing
$M^c$ by~$\p_r$, particularly in the region $\{ 1/4 < \rho < 1/2 \}$,
where the flow is given by \eqref{eqn:PsiInterface}.  We can assume
each tangent space to $M^c$ in the relevant region is spanned by
$\p_\phi$, $\p_\theta$ and $\p_t + a \p_s$ for some $a \in \RR$ with
$|a|$ small.  The flow of $\p_r$ does not change the first two vectors
in this frame, and its change to the third one stretches the $t$-direction
but not the $s$-direction.  Thus as long as $\delta$ has been chosen to
make $|a|$ sufficiently small, flowing by $\p_r$ cannot make these tangent
spaces tangent to $\p_s + \p_t$; moreover, one sees from this discussion
that $\p_v \widehat{E}_R$ is isotopic to a subset of
$\{s_0\} \times \widehat{M}\paper \subset \widehat{\nN}(\p_v E)$ for some
constant $s_0 > 0$, through
a family of weakly contact hypersurfaces that are all transverse to $V_K$
and have fixed boundary.  One can then define a suitable isotopy to $E$
through domains bounded by hypersurfaces of the form
$\{\text{const}\} \times \widehat{M}\spine \subset \widehat{\nN}(\p_h E)$ and 
$\{\text{const}\} \times \widehat{M}\paper \subset \widehat{\nN}(\p_v E)$.
\end{proof}

Figure~\ref{fig:largeSubd} depicts the boundary faces of $\p \widehat{E}_R$
on the backdrop of the cylindrical end and holomorphic foliation from
Figures~\ref{fig:end} and~\ref{fig:holfol} respectively.  It also shows
the collar neighborhoods $\nN(\p_v\widehat{E}_R)$ and $\nN(\p_h\widehat{E}_R)$
as described in the following lemma, carrying fibrations whose fibers
are leaves of the foliation.

\begin{lemma}
\label{lemma:bigFibrations}
For each $R > 0$, the boundary faces of $\widehat{E}_R$ admit collar 
neighborhoods
\begin{equation*}
\begin{split}
\p_v\widehat{E}_R &\subset \nN(\p_v\widehat{E}_R) \cong (-1,0] \times
\p_v\widehat{E}_R \subset \widehat{E}_R,\\
\p_h\widehat{E}_R &\subset \nN(\p_h\widehat{E}_R) \cong (-1,0] \times
\p_h\widehat{E}_R \subset \widehat{E}_R,
\end{split}
\end{equation*}
with fibrations
$$
\Pi_v^R : \nN(\p_v\widehat{E}_R) \to (-1,0] \times S^1,\qquad
\Pi_h^R : \nN(\p_h\widehat{E}_R) \to \Sigma
$$
that satisfy the following conditions:
\begin{enumerate}
\item The vertical subbundle for both fibrations is defined by the
integrable distribution~$T\fF_+$.
\item $\Pi_v^R$ is pseudoholomorphic near $\p_v\widehat{E}_R$ with respect to
$J_+$ on $\widehat{E}$ and the standard complex structure 
on $(-1,0] \times S^1$.
\item On $\nN(\p_h\widehat{E}_R)$, $\omega_E$ has a primitive that restricts
to a contact form on $\p_h\widehat{E}_R$ for which the boundaries of the
fibers of $\Pi_h^R$ are closed Reeb orbits.
\end{enumerate}
\end{lemma}
\begin{proof}
This is based essentially on four observations.  First, the characteristic
line field of $\p_h\widehat{E}_R$ as a hypersurface in $(\widehat{E},\omega_E)$
is spanned by~$\p_\theta$.  Since $J_+$ is $\omega_E$-compatible, it follows 
that $J_+\p_\theta \in T\fF_+$ is transverse to $\p_h\widehat{E}_R$, hence
the leaves of $\fF_+$ intersect $\p_h\widehat{E}_R$ transversely in loops
tangent to~$\p_\theta$, and these are Reeb orbits for any contact form
given by a primitive of~$\omega_E$.

The second observation, which is clear already from the construction of
$\widehat{E}_R$, is that $\p_v\widehat{E}_R$ is a union of (compact subsets
of) leaves of~$\fF_+$, in particular it is a disjoint union of
$S^1$-parametrized families of leaves.  The collar
$\nN(\p_v\widehat{E}_R)$ with its fibration can therefore be obtained by 
extending these $S^1$-families to families parametrized by annuli.

Third, the foliation $\fF_+$ is invariant under flows in the $r$-direction,
and in a neighborhood of $\p_v\widehat{E}_R$, we have
$ds(\p_r) = 1$.

Finally, $J_+$ satisfies $J_+\p_r = R_+$ throughout 
$\widehat{\nN}_+(\p E)$, and in a neighborhood 
of~$\p_v\widehat{E}_R$, we can assume every point belongs to either 
$[0,\infty) \times \widecheck{M}^+\paper \subset
\widehat{\nN}_+(\p E)$ or the portion of $[0,\infty) \times
\widecheck{M}^+\corner \subset \widehat{\nN}_+(\p E)$ with $\rho > 1/4$,
so that \eqref{eqn:RHpaper} and
\eqref{eqn:OmegaLambdaEps} give 
$R_+ = e_{S^1}^\#$, a horizontal lift of the canonical unit
vector on $S^1$ under the fibration $\pi\paper : M\paper \to S^1$.
By the third observation above, it follows that 
$\Pi_v^R : \nN(\p_v\widehat{E}_R) \to (-1,0]\times S^1$ can be arranged
such that near the boundary, $\p_r$ and $J_+\p_r$ are horitontal lifts
of the two canonical basis vector fields on $(-1,0] \times S^1$, meaning
$\Pi_v^R$ is pseudoholomorphic.
\end{proof}

\section{Holomorphic curves in spinal open books}
\label{sec:holOpenBook}

In this section we study $J$-holomorphic curves in the symplectization
of a contact $3$-manifold carrying a spinal open book.  As we saw in
\S\ref{sec:model}, every spinal open book on a closed
$3$-manifold gives rise to a stable Hamiltonian
structure and a compatible almost complex structure on its symplectization,
for which the pages lift to a foliation by embedded $J$-holomorphic curves
with positive ends approaching nondegenerate Reeb orbits in the spine.
Our first task in this section will be write down the easy extension of
this construction to the case of manifolds with boundary, and then to show
that the stable Hamiltonian structure can be perturbed to a contact
structure supported by the spinal open book.  We then examine the
analytical properties of the curves in the foliation, and show in particular
that the \emph{planar} curves among them are stable and will survive the
perturbation from stable Hamiltonian to contact data; moreover, these will 
in fact be the \emph{only} holomorphic curves with certain asymptotic behavior
that exist after the perturbation.
The most important results are Propositions~\ref{prop:J0foliation}
(existence), \ref{prop:IFT} and \ref{prop:stableFEF} (stability under
perturbation), and~\ref{prop:uniqueness} (uniqueness).  These generalize
results that were proved for open books in \cite{Wendl:openbook}
and blown up summed open books in \cite{Wendl:openbook2}, and
will serve as crucial ingredients for the computations of contact
invariants in \S\ref{sec:invariants} and compactness arguments 
in~\S\ref{sec:impressivePart}.

Throughout this section, $(M',\xi)$ is a closed contact $3$-manifold, and
$M = M\paper \cup M\spine \subset M'$ is a compact connected submanifold 
(possibly with boundary) carrying a spinal open book 
$$
\boldsymbol{\pi} = \Big(\pi\spine : M\spine \to \Sigma,
\pi\paper : M\paper \to S^1, \{m_T\}_{T \subset \p M}\Big).
$$
that admits a smooth overlap and supports~$\xi|_M$.

\subsection{A family of stable Hamiltonian structures}
\label{sec:SHSfamily}

This subsection and the next will consist mostly of repackaged notation and 
results from \S\ref{sec:model}.  We shall use the notation from 
Section \ref{sec:collars} for collar neighborhoods, and we also need to recall
from \cite{LisiVanhornWendl1}*{\S 2.2} the open covering
$$
M = \widecheck{M}\spine \cup \widecheck{M}\corner \cup \widecheck{M}\paper \cup 
\widecheck{M}\bndry
$$
defined whenever $M$ carries a spinal open book with smooth overlap.
Here $\widecheck{M}\paper$ denotes the complement of the region
$\{t \ge -1/2\} \subset \nN(\p M\paper)$ in~$M\paper$, $\widecheck{M}\spine$
is the complement of $\{s \ge -1/2\} \subset \nN(\p M\spine)$ in~$M\spine$,
$\widecheck{M}\corner$ is the union of $\nN(\p M\spine)$ with the adjacent
components of $\nN(\p M\paper)$, and $\widecheck{M}\bndry$ is the union of
all components of $\nN(\p M\paper)$ that touch~$\p M$.  The components of
$\widecheck{M}\corner$ carry coordinates
$$
(\rho,\phi,\theta) \in (-1,1) \times S^1 \times S^1 \subset \widecheck{M}\corner
$$
which are related to the collar coordinates from \S\ref{sec:collars}
by $\rho = s$ and $\rho = -t$ on the regions of overlap, and components of
$\widecheck{M}\bndry$ similarly carry coordinates
$$
(\rho,\phi,\theta) \in [0,1) \times S^1 \times S^1 \subset \widecheck{M}\bndry
$$
with $\rho = -t$.  Since $\nN(\p M\paper)$ is contained in $\widecheck{M}\corner
\cup \widecheck{M}\bndry$, we can use the coordinate $\rho = -t$ as an
alternative to $t$ on~$\nN(\p M\paper)$.

Recall that in \S\ref{sec:collars}, the hypersurface
$M^0 \subset E$ was endowed with a similar open covering
$M^0 = \widecheck{M}^0\spine \cup \widecheck{M}^0\corner \cup \widecheck{M}^0\paper$;
in the case $\p M = \emptyset$,
these three regions have obvious canonical identifications with
$\widecheck{M}\spine$, $\widecheck{M}\corner$ and $\widecheck{M}\paper$
respectively, thus defining a diffeomorphism $M \cong M^0$.
This gives rise to a diffeomorphism of $M$ with the hypersurface $M^+ \subset
\widehat{E}$ from \S\ref{sec:nondegPert} after flowing from $M^0$ to $M^+$
along the stabilizing vector field~$Z$.  The idea behind most of the
definitions in this section is to use this identification of $M$ with $M^+$
in order to endow $M$ with the same stable Hamiltonian structure that was defined in
\S\ref{sec:nondegPert}, and its symplectization likewise with the same
almost complex structure as in \S\ref{sec:Jandf}.  Only minor modifications will
be needed for the case $\p M \ne \emptyset$.

The following contact form on $M$ takes on the role that was previously played
by the restriction of $\frac{1}{K}\lambda_K$ to~$M^+$: define
$$
\alpha := \begin{cases}
e^{\nondegParam \widehat{H}} \left( \sigma + \frac{1}{K} \, d\theta\right) & \text{ on $\widecheck{M}\spine$},\\
m e^{F_+(\rho)} \, d\phi + \frac{1}{K} e^{G_+(\rho)} \, d\theta
& \text{ on $\widecheck{M}\corner$},\\
d\pi\paper + \frac{1}{K} \lambda & \text{ on $\widecheck{M}\paper$},\\
f_K(\rho) \, d\theta + m g(\rho) \, d\phi 
& \text{ on $\widecheck{M}\bndry$},
\end{cases}
$$
where the various symbols have the following meanings.
The multiplicity $m \in \NN$ is a number that may vary among different
connected components of $\widecheck{M}\corner \cup \widecheck{M}\bndry$ (see \S\ref{sec:collars}), while $\sigma$ is the
pullback via $\pi\spine : M \to \Sigma$ of the Liouville form on $\Sigma$
defined in \S\ref{sec:Liouville}, $\lambda$ is the fiberwise Liouville
form on $M\paper$ defined in the same subsection, and $K > 0$ is the
(arbitrarily) large constant that was used for building a Liouville form out
of these ingredients via the Thurston trick.
The functions $\widehat{H} : \Sigma \times S^1 \to [0,\infty)$ and
$F_+,G_+ : (-1,1) \to \RR$ were defined for the perturbation of
$M^0$ to $M^+$ in \S\ref{sec:nondegPert}, which also depended on a
constant $\varepsilon > 0$ that may be assumed arbitrarily small.
The only new pieces of data in our definition of $\alpha$ are the functions
$f_K$ and $g$ required for the collar near~$\p M$: to define these, we
first choose smooth functions
$f , g : [0,1) \to [0,1]$ such that:
\begin{itemize}
\item $(f(\rho),g(\rho)) = ( e^{-\rho},1 )$ 
for $\rho \ge 1/4$;
\item $f g' - f' g > 0$;
\item $f'(\rho) < 0$ for $\rho > 0$;
\item $(f(0),g(0)) = (1,0)$, $(f'(0),g'(0)) = (0,1)$, and $f''(0) < 0$.
\end{itemize}
The function $f_K : [0,1) \to [0,1]$ is then defined by
$$
f_K(\rho) := \left[\beta(-\rho) \left(1 - \frac{1}{K}\right) + 
\frac{1}{K}\right] f(\rho),
$$
where $\beta : (-1,0] \to [0,1]$ is the same cutoff function that was used
in \S\ref{sec:SHS} to define a stabilizing vector field.
It follows that
$f_K(\rho) = \frac{1}{K} f(\rho)$ for $\rho \ge 1/2$,
$f_K(\rho) = f(\rho)$ for $\rho \le 1/4$, and
$$
f_K(\rho) = \left[\beta(-\rho) \left(1 - \frac{1}{K}\right) + 
\frac{1}{K}\right] e^{-\rho} \quad
\text{ for $\rho \ge 1/4$},
$$
hence $f_K'(\rho) < 0$ for $\rho \ge 1/4$, so that
$f_K g' - f_K' g > 0$ everywhere.

For $K > 0$ sufficiently large and $\nondegParam > 0$ sufficiently small,
$\alpha$ is a Giroux form for $\boldsymbol{\pi}$, so we can assume after
adjusting $\xi$ by an isotopy that $\alpha$ extends to a contact form on 
$M'$ with
$$
\xi = \ker \alpha.
$$
The Reeb vector field for $\alpha$ will be denoted by $R_\alpha$ and can
be written as
\begin{equation}
\label{eqn:RalphaSpine}
R_\alpha =
e^{-\nondegParam \widehat{H}} \left(
\left[1 + \nondegParam \sigma(X_H)\right] K \p_\theta
- \nondegParam X_H \right) \quad \text{ on $\widecheck{M}\spine$},
\end{equation}
where the Hamiltonian vector field $X_H$ on $\Sigma$ is defined by
$d\sigma(X_H,\cdot) = -dH$.
On the interface, $R_\alpha$ satisfies
\begin{equation}
\label{eqn:RalphaInterface}
R_\alpha = \frac{1}{F_+'(\rho) - G_+'(\rho)} \left(
-\frac{1}{m} e^{-F_+(\rho)} G_+'(\rho) \p_\phi +
K e^{-G_+(\rho)} F_+'(\rho) \p_\theta \right) \quad
\text{ on $\widecheck{M}\corner$},
\end{equation}
while near~$\p M$,
\begin{equation}
\label{eqn:RalphaBndry}
R_\alpha = \frac{1}{f_K(\rho) g'(\rho) - f_K'(\rho) g(\rho)} \left(
g'(\rho) \p_\theta - \frac{f_K'(\rho)}{m} \p_\phi \right) \quad
\text{ on $\widecheck{M}\bndry$}.
\end{equation}
This last formula implies since $f'(0) = 0$ that $\p M$ is foliated by 
closed Reeb orbits in the $\theta$-direction, and these orbits form
Morse-Bott families due to the condition $f''(0) < 0$.  Similarly, the
assumption that $H$ is Morse away from $\p\Sigma$ implies that
Reeb orbits of the form $\{z\} \times S^1$ for $z \in \CritMorse(H)$ are
nondegenerate.  Here we again denote by
$$
\CritMorse(H) \subset \Sigma
$$
the finite set of Morse critical points of $H$, thus excluding the critical 
points in the region near $\p\Sigma$ where $H$ vanishes.

The stable Hamiltonian structure from \S\ref{sec:nondegPert} can be written
in the present context as $\hH = (\Omega,\Lambda)$, where
$$
\Omega := d\alpha + \frac{1}{K C} \eta
$$
for some large constant $C > 0$ and a closed $2$-form $\eta$ that is
assumed to vanish outside of $\widecheck{M}\paper$.  The stabilizing $1$-form is
$$
\Lambda := \begin{cases}
\alpha & \text{ on $\widecheck{M}\spine$},\\
m e^{F_+(\rho)} \, d\phi + \frac{1}{K} e^{G_+(\rho)}
\beta(G_+(\rho)) \, d\theta & \text{ on $\widecheck{M}\corner$},\\
d\pi\paper & \text{ on $\widecheck{M}\paper$},\\
\beta(-\rho) f(\rho) \, d\theta + m g(\rho) \, d\phi 
& \text{ on $\widecheck{M}\bndry$},\\
\alpha & \text{ on $M' \setminus M$}.
\end{cases}
$$
The only feature of this discussion that did not already appear in
\S\ref{sec:model} is the definition of $\hH$ in
$\widecheck{M}\bndry$, since we are now allowing $\p M \ne \emptyset$.
To see that $(\Omega,\Lambda)$ satisfies the conditions of a stable Hamiltonian
structure in this region, we need to check that $\ker \Omega \subset
\ker d\Lambda$: this is obvious whenever either $\Lambda = d\pi\paper$ or
$(\Omega,\Lambda) = (d\alpha,\alpha)$, so we only still need to inspect
the region where $0 < \beta(-\rho) < 1$, which means $\rho \in (1/4,1/2)$.
Here $g(\rho) = 1$, so
\begin{equation}
\label{eqn:interpRegion}
\begin{split}
\Lambda &= \beta(-\rho) f(\rho) \, d\theta + m \, d\phi, \\
\Omega &= d\left( f_K(\rho) \, d\theta + m \, d\phi \right) =
f_K'(\rho) \, d\rho \wedge d\theta
\end{split}
\end{equation}
in $\{ 1/4 \le \rho \le 1/2 \} \subset \widecheck{M}\bndry$, implying 
that $\ker d\Lambda$ and $\ker \Omega$ are both generated by~$\p_\phi$.
This shows that on the region in question, $(\Omega,\Lambda)$ is indeed
a stable Hamiltonian structure and its induced Reeb vector field is simply
$\frac{1}{m} \p_\phi$, which is the same as~$R_\alpha$.
For future reference it will be useful to note that the same result holds
in the region $\{ 1/4 \le \rho \le 1/2 \} \subset \widecheck{M}\corner$, as here
$F_+(\rho) = 0$ and $G_+(\rho) = -\rho$, so that
\eqref{eqn:interpRegion} is valid with $f(\rho)$ replaced by $e^{-\rho}$
and $f_K(\rho)$ replaced by $\frac{1}{K} e^{-\rho}$.
One can also see from this formula and the conditions
imposed on $\beta$ that $\Lambda$ satisfies the contact condition as soon as 
$\beta(-\rho) > 0$, so in particular, $\Lambda \wedge d\Lambda \ge 0$ and
the induced hyperplane distribution
$$
\Xi_0 := \ker \Lambda
$$ 
is therefore a confoliation.

To summarize the discussion so far:

\begin{prop}
The pair $\hH = (\Omega,\Lambda)$ is a confoliation-type stable Hamiltonian 
structure on~$M'$.  Its induced 
Reeb vector field $R_\hH$ matches $R_\alpha$ outside of~$\widecheck{M}\paper$,
and is colinear with $R_\alpha$ if $\eta \equiv 0$. \qed
\end{prop}

The possibly non-exact $2$-form $\eta$ is a harmless but necessary piece of
the setup for applications to weak fillability, though in the present section
we will be interested primarily in the case $\eta \equiv 0$.  Our
stable Hamiltonian structure then has some convenient extra properties
arising from the fact that $R_\hH$ and $R_\alpha$ are in this case
colinear.  This implies
in the first place that in addition to the confoliation condition
$\Lambda \wedge d\Lambda \ge 0$, we have
$$
\Lambda \wedge d\alpha > 0 \quad \text{ and }\quad
\alpha \wedge d\Lambda \ge 0,
$$
and therefore:

\begin{prop}
\label{prop:SHSpert}
For every constant $\param \in [0,1]$, the pair
$$
\hH_\param := (\Omega_\param,\Lambda_\param) :=
(d\alpha , (1 - \param) \Lambda + \param \alpha)
$$
is a stable Hamiltonian structure on~$M'$ whose induced Reeb vector field
is colinear with~$R_\alpha$.  Moreover, $\hH_0 = \hH$ if $\eta \equiv 0$,
while $\Lambda_\param = \alpha$ on $\widecheck{M}\spine$ for all~$\param$, and
$\Lambda_\param$ is a contact form everywhere for $\param > 0$, with an
induced contact structure isotopic to~$\xi$.   \qed
\end{prop}
\begin{remark}
The reader should be cautioned that while the notation $\hH_\param$
makes sense for $\param=0$, $\hH_0$ under this definition is not the same 
SHS that was denoted this way in \S\ref{sec:model}; it corresponds 
rather to what was previously denoted by~$\hH_+$ in the case
$\eta \equiv 0$ and $\p M = \emptyset$.
\end{remark}

For each $\param \ge 0$, we shall denote the induced hyperplane distribution and 
Reeb vector field for $\hH_\param$ by $\Xi_\param$ and $R_\param$ respectively.  
Note that if $\eta \ne 0$, then $R_0 \ne R_\hH$ on $\widecheck{M}\paper$, though 
$\hH$ does induce the same hyperplane distribution $\Xi_0$ as~$\hH_0$.
The scaling of $R_\param$ changes in general with the value
of $\param$, but its direction does not.  Since $\Lambda_\param$
is contact for $\param > 0$, Proposition~\ref{prop:sameJ} implies
$$
\jJ(\hH_\param) = \jJ(\Lambda_\param).
$$

\subsection{The unperturbed finite energy foliation}
\label{sec:FEFunperturbed}

In this subsection we consider the unperturbed stable Hamiltonian structure
$\hH = (\Omega,\Lambda)$ with $\Omega = d\alpha + \frac{1}{K C} \eta$,
where $\eta$ is allowed to be nonzero in~$\widecheck{M}\paper$.
Let us now rewrite the construction of the holomorphic foliation from 
\S\ref{sec:holPages} in the present context.
Choose $J_0 \in \jJ(\hH)$ to satisfy the
same conditions as $J_+ \in \jJ(\hH_+)$ in \S\ref{sec:Jandf} on the region
$\widecheck{M}\spine \cup \widecheck{M}\corner \cup \widecheck{M}\paper$,
while on $\widecheck{M}\bndry$ it is determined by the same condition as
on $\widecheck{M}\corner$, namely
$$
J_0 v_1 = h(\rho) v_2
$$
for some smooth function $h(\rho) > 0$, with $v_1 := \p_\rho$ and
$v_2$ denoting the unique linear combination of $\p_\phi$ and $\p_\theta$
that lies in $\Xi_0$ and satisfies $\Omega(v_1,v_2)=1$.

We can then define a smooth $J_0$-invariant and translation-invariant
distribution $T\fF$ on $\RR \times M$ by 
$$
T\fF_{(r,x)} := \begin{cases}
\Xi_0 & \text{ for $x \in \widecheck{M}\paper$},\\
\Span(\p_\theta, J_0\p_\theta) & \text{ for $x \in \widecheck{M}\spine \cup
\widecheck{M}\corner \cup \widecheck{M}\bndry$}.
\end{cases}
$$
Using the metric $\langle\cdot,\cdot\rangle := d\sigma(\cdot,j\cdot)$ on
$\Sigma$ to define the gradient $\nabla H$ of $H$,
Proposition~\ref{prop:Jfoliation} now adapts to the present setting as
follows:

\begin{prop}
\label{prop:J0foliation}
The distribution $T\fF$ is integrable, and thus defines an $\RR$-invariant
foliation $\fF$ on $\RR \times M$
whose leaves are the images of embedded and asymptotically cylindrical 
$J_0$-holomorphic curves.  
Each leaf of this foliation is one of the following:
\begin{enumerate}
\item A \defin{trivial cylinder} $\RR \times \gamma$, where
$\gamma \subset M$ is either a nondegenerate Reeb orbit of the form
$\gamma = \{z\} \times S^1 \subset \widecheck{M}\spine \subset
\Sigma \times S^1$ for some $z \in \CritMorse(H)$, or part of a Morse-Bott $2$-torus
of Reeb orbits in the $\theta$-direction foliating~$\p M$.
\item A \defin{holomorphic gradient flow cylinder}, admitting a smooth (but not necessarily holomorphic) parametrization 
$u : \RR \times S^1 \hookrightarrow \RR \times M$ of the form
$$
u(s,t) = (a(s),\ell(s),t) \in \RR \times \widecheck{M}\spine \subset
\RR \times \Sigma \times S^1,
$$
where $a : \RR \to \RR$ is a strictly increasing proper function and
$\ell : \RR \to \Sigma$ is a solution of the gradient flow
equation $\dot{\ell} = \nabla H(\ell)$ approaching two distinct critical
points of $H$ as $s \to \pm\infty$.
\item A \defin{holomorphic page}, which is a connected and properly embedded
submanifold formed as a union of subsets of the following type:
\begin{itemize}
\item $\{s\} \times P \subset \RR \times \widecheck{M}\paper$, where
$s \in \RR$ is a constant and $P \subset \widecheck{M}\paper$ is the portion of
a page of $\pi\paper : M\paper \to S^1$ lying in $\widecheck{M}\paper$;
\item Annuli admitting smooth (but not necessarily holomorphic) parametrizations 
$u : (-1,1) \times S^1 \hookrightarrow \RR \times M$ of the form
$$
u(s,t) = (a(s),s,\phi,t) \in \RR \times (-1,1) \times S^1 \times S^1
\subset \RR \times \widecheck{M}\corner
$$
for some bounded function $a : (-1,1) \to \RR$ and constant $\phi \in S^1$;
\item Half-cylinders admitting smooth (but not necessarily holomorphic) parametrizations
$u : [0,\infty) \times S^1 \hookrightarrow \RR \times M$ of the form
$$
u(s,t) = (a(s),\ell(s),t) \in \RR \times \widecheck{M}\spine \subset
\RR \times \Sigma \times S^1,
$$
where $a : [0,\infty) \to \RR$ is a strictly increasing proper function and
$\ell : [0,\infty) \to \Sigma$ is a solution of the gradient flow
equation $\dot{\ell} = \nabla H(\ell)$ that begins at time $s=0$ as a
trajectory in $\nN(\p\Sigma)$ orthogonal to $\p\Sigma$ and
approaches a critical point of $H$ as $s \to \infty$;
\item Half-cylinders admitting smooth (but not necessarily holomorphic) parametrizations
$u : [0,\infty) \times S^1 \hookrightarrow \RR \times M$ of the form
$$
u(s,t) = (a(s),b(s),\phi,t) \in \RR \times [0,1) \times S^1 \times S^1
\subset \RR \times \widecheck{M}\bndry
$$
where $a : [0,\infty) \to \RR$ is a function with
$\lim_{s \to \infty} a(s) = +\infty$ and $b : [0,\infty) \to (0,1)$ is a
strictly decreasing function with $\lim_{s \to \infty} b(s) = 0$.
\end{itemize}
\end{enumerate}
In particular, each of the holomorphic gradient flow cylinders and pages
projects through $\RR \times M \to M$ to an embedded
surface in $M$ whose closure is a compact embedded surface
bounded by Reeb orbits in 
$(\CritMorse(H) \times S^1) \cup \p M$.
\end{prop}
\begin{proof}
The only detail that has not been covered already by
Proposition~\ref{prop:Jfoliation} is the behavior of the holomorphic
pages as they approach~$\p M$.  The relevant calculation here is the
same as for $\widecheck{M}\corner$, but carried out in $\widecheck{M}\bndry$
instead, the key point being that everywhere in the interior of
$\widecheck{M}\bndry$, $\Span(\p_\phi,\p_\theta) = \Span(\p_\theta,R_\hH)$,
hence $J_0 \p_\theta$ is a linear combination of $\p_r$ and~$\p_\rho$.
At $\p M$ this ceases to be true because $R_\hH$ is proportional
to~$\p_\theta$, which implies that the trivial cylinders over orbits
in $\p M$ are tangent to~$T\fF$ and the interior leaves that approach $\p M$
are therefore properly embedded.  Orientation considerations imply that
the ends of those leaves are positive, as their projections to $M$
are embedded surfaces with closures bounded by \emph{positively}
oriented Reeb orbits on~$\p M$.
\end{proof}

The foliation $\fF$ defines a finite energy foliation as in
\cite{HWZ:foliations}, but it is not generally a \emph{stable} finite energy
foliation, because the $J_0$-holomorphic curves forming its leaves may in
general have negative Fredholm index and will thus die under small
perturbations of the data.  We will be examining issues of this type
for the remainder of \S\ref{sec:holOpenBook}.

Each Morse critical point $z \in \CritMorse(H)$ gives rise to an
embedded periodic Reeb orbit parametrized by
$$
\gamma_z : S^1 \to M : t \mapsto (z,t) \subset \widecheck{M}\spine \subset
\Sigma \times S^1.
$$
This is also a periodic orbit of
$R_\param$ for $\param > 0$ since $\hH_\param = \hH$ on $\widecheck{M}\spine$
for all~$\param$.
We will denote the $k$-fold cover of this orbit for any $k \in \NN$
by
$$
\gamma_z^k : S^1 \to M : t \mapsto \gamma_z(kt).
$$
Observe that the natural $S^1$-action on
$\Sigma \times S^1$ induces a preferred trivialization of the contact
bundle along~$\gamma_z$.  We shall denote the Conley-Zehnder index
of $\gamma_z^k$ with respect to this trivialization by
$$
\muCZ(\gamma_z^k) \in \ZZ,
$$
and let $\Morse(z) \in \{0,1,2\}$ denote the Morse index of $z \in \CritMorse(H)$.

\begin{lemma}
\label{lemma:periodsAndIndices}
There exists a number $T_1 > 0$ such that for any $T_0 > 0$,
the above construction can be arranged by choosing $K > 0$ and
$\nondegParam > 0$ sufficiently large and small respectively so that
the dynamics of $R_\hH$ have the following properties:
\begin{enumerate}
\item For each Morse critical point $z \in \CritMorse(H)$,
all orbits in a neighborhood of $\gamma_z : S^1 \to \widecheck{M}\spine$ 
are nondegenerate.
\item Every closed orbit with period less than~$T_1$ is of the form
$\gamma_z^k$ for some Morse critical point $z \in \CritMorse(H)$ and $k \in \NN$, 
and all such orbits satisfy
$$
\muCZ(\gamma_z^k) = \Morse(z) - 1.
$$
\item The orbits $\gamma_z$ for $z \in \CritMorse(H)$ have period less
than~$T_0$.
\item The families of orbits that foliate $\p M$ are Morse-Bott.
\end{enumerate}
Moreover, we can also assume that the dynamics of
$R_\param$ have these same properties for all $\param \ge 0$ sufficiently
small.
\end{lemma}
\begin{proof}
Consider first the dynamics of~$R_\hH$.  On $\widecheck{M}\spine$ we have
$R_\hH = R_\alpha$, so we see from \eqref{eqn:RalphaSpine} that
periodic orbits in $\widecheck{M}\spine \setminus (\CritMorse(H) \times S^1)$
correspond to nonconstant periodic orbits
of the Hamiltonian vector field $X_H$ on $\Sigma$, where $X_H$ is scaled 
by $\nondegParam$ and the
periods of its orbits are therefore scaled by $1/\nondegParam$.  (We can
assume the additional scaling factor of $e^{-\nondegParam \widehat{H}}$
is arbitrarily close to~$1$.)
Since the periods of nonconstant orbits of $X_H$ have positive infimum,
we can make the scaled periods larger than any given $T_1$ by choosing
$\nondegParam$ small.
In $\widecheck{M}\paper$, we have $\Lambda = d\pi\paper$ and thus
all closed orbits of $R_\hH$ have period at least~$1$.
In $\widecheck{M}\corner$, $R_\hH$ matches $R_\alpha$ and is thus given by
\eqref{eqn:RalphaInterface}, so the condition
$G_+' < 0$ from Lemma~\ref{lemma:GnotFlat} implies that 
$R_\hH$ has a nonzero $\p_\phi$ component whose size does not depend on~$K$;
one can therefore find a lower bound independent of $K$ on the periods
in $\widecheck{M}\corner$ and choose $T_1 > 0$ smaller than this bound.
A similar computation works in $\widecheck{M}\bndry$ using \eqref{eqn:RalphaBndry},
since $f_K'(\rho) < 0$ for $\rho > 0$ implies that $R_\alpha$ has a positive
$\p_\phi$-component in the interior, and the dependence of the contact
form on $K$ is limited to the region $\{ \rho \ge 1/4 \}$, where the
Reeb vector field is simply $\frac{1}{m} \p_\phi$.  The orbits that foliate
$\p M$ have period $1$ since $f_K(0) = f(0) = g'(0) = 1$.  With this
understood, let us assume henceforward that $T_1$ is smaller than the
periods of all orbits other than the $\gamma_z$ for $z \in \CritMorse(H)$.
The periods of the latter can however be made arbitrarily small by
increasing~$K$.

The computation of the Conley-Zehnder
index is a standard result from Floer theory, see for example
\cite{SalamonZehnder:Morse} or \cite{Wendl:SFT}*{\S 10.3.2}.
Similarly, the Morse-Bott condition at $\p M$ follows from the
condition $f''(0) < 0$.

All of the above applies immediately to $R_0$, since this is the special
case of $R_\hH$ with $\eta \equiv 0$.  Considering the perturbations
$\hH_\param$ for $\param > 0$, the same conclusions remain valid for
$R_\param$ if $\param$ is sufficiently small:
this is because the perturbation does not change the direction of the
Reeb vector field, so the only change to the dynamics is a very slight
change in the periods of closed orbits.
\end{proof}

From now on we will assume the data to be chosen so that 
Lemma~\ref{lemma:periodsAndIndices} is satisfied for some 
constants $T_0, T_1 > 0$, for which we may assume
$T_1 / T_0$ is as large as needed.

Lemma~\ref{lemma:periodsAndIndices} provides enough information to compute 
the Fredholm indices of the leaves of~$\fF$ by
applying the Riemann-Roch formula to the normal bundles of the curves.
This computation was carried out for the ordinary open book case 
in \cite{Wendl:openbook}, and the result in our setting is:

\begin{prop}
\label{prop:FredholmIndices}
Suppose $u : \dot{S} \to \RR \times M$ represents a holomorphic page in
$\fF$ that has genus $g \ge 0$ and $k + n$ punctures, where $k \ge 0$ of them
are asymptotic to orbits $\gamma_{z_1},\ldots,\gamma_{z_k}$ in 
$\widecheck{M}\spine$ for Morse critical points $z_1,\ldots,z_k \in \CritMorse(H)$,
and the rest are asymptotic to Morse-Bott orbits in~$\p M$.  Then
\begin{equation}
\label{eqn:indexPage}
\ind(u) = 2 - 2g - \sum_{i=1}^k \left[ 2 - \Morse(z_i) \right].
\end{equation}
If $u : \RR \times S^1 \to \RR \times M$ represents a holomorphic gradient
flow cylinder with postive end at $\gamma_{z_+}$ and negative end at
$\gamma_{z_-}$ for $z_\pm \in \CritMorse(H)$, then
$$
\ind(u) = \Morse(z_+) - \Morse(z_-).
$$  \qed
\end{prop}

We shall say that a leaf $u \in \fF$ is an \defin{interior leaf} if all its
asymptotic orbits are in $\widecheck{M}\spine$, none in~$\p M$.  This
includes all the trivial cylinders over orbits in $\CritMorse(H) \times S^1$,
all gradient flow cylinders and all holomorphic pages that are modelled on
pages with boundary contained in~$M\spine$.

Observe that if one starts from any point $z \in \p\Sigma$ and traces an
inward curve in $\nN(\p\Sigma)$, orthogonal to the boundary, it soon becomes
a gradient flow line that ends at a critical point of index either~$1$ or~$2$, 
and the former is the case only for finitely many starting points in~$\p\Sigma$.  
This implies that $\fF$ contains at most finitely many leaves (up to
$\RR$-translation) whose asymptotic orbits include index~$1$ critical points.
The \emph{generic} leaf is therefore either a
holomorphic page with positive ends approaching orbits in~$\p M$ and
index~$2$ critical points, or a gradient flow cylinder connecting a
critical point of index~$0$ to one of index~$2$.  Since every component
of $\Sigma$ has nonempty boundary, $H$ can always be
chosen such that
$$
\Morse(z) \in \{1,2\} \quad \text{ for all $z \in \CritMorse(H)$},
$$
and we shall assume this from now on so that \emph{all} generic leaves are
holomorphic pages rather than gradient flow cylinders.
The index formula \eqref{eqn:indexPage}
shows that only the \emph{planar} holomorphic pages in $\fF$ can have 
positive index, and in general some of these may even have
$\ind(u) \le 0$ if they have multiple ends approaching
index~$1$ critical points.  This can be avoided by a generic
perturbation of~$H$ and~$j$ away from the boundary to arrange the
following conditions:
\begin{enumerate}[label=(\roman{enumi})]
\item $\nabla H$ is Morse-Smale,
\item No two gradient flow lines approaching index~$1$ critical points 
enter $\p\Sigma$ at points with the same value of $m \phi$, where
$\phi$ denotes the usual collar coordinate at $\p\Sigma$ and $m$ is the
relevant multiplicity determined by the adjacent component of
$\pi\paper : M\paper \to S^1$ as in \eqref{eqn:localMult}.
\end{enumerate}

\begin{defn}
\label{defn:GP}
We will say that the pair $(H,j)$ are in \defin{general position} whenever
$H$ has no index~$0$ Morse critical points and
both of the above conditions are satisfied.  
\end{defn}

Plugging in the index formulas
above and applying the automatic transversality criterion from
\cite{Wendl:automatic}, we find:

\begin{lemma}
\label{lemma:stable}
If $(H,j)$ are in general position, then every gradient flow cylinder
and every holomorphic page with genus zero has
index~$1$ or~$2$ and is Fredholm regular.  Moreover, these are the only
leaves of~$\fF$ with positive index.  \qed
\end{lemma}

For later arguments in \S\ref{sec:invariants} and \S\ref{sec:impressivePart},
we will also need some control over the indices of multiple covers of
curves in $\fF$, especially the trivial cylinders and gradient flow cylinders.

\begin{lemma}
\label{lemma:wS}
Suppose $(H,j)$ are in general position, and $u$ is a connected stable 
holomorphic building in $\RR \times M'$ with no nodes, with arithmetic genus 
$g \ge 0$, and whose connected 
components are all covers of leaves of $\fF$ contained in $\RR \times
\widecheck{M}\spine$.  Assume moreover that
the sum of the periods of the positive asymptotic orbits of $u$ is less 
than~$T_1$.  Then
$$
\ind(u) = 2g - 2 + \#\Gamma_0^+ + 2 \#\Gamma_1^+ + \#\Gamma_0^-,
$$
where $\Gamma^\pm_0$ and $\Gamma^\pm_1$ denote the sets of positive/negative
punctures of $u$ at which the asymptotic orbit has even or odd Conley-Zehnder
index respectively.  In particular:
\begin{enumerate} 
\item
The index is nonnegative, 
with equality if and only if $u$ is either
a trivial cylinder or a building composed entirely of branched covers
of trivial cylinders over an orbit $\{z\} \times S^1$ with $\Morse(z)=2$,
each connected component having exactly one positive puncture.
\item If $u$ is a cover of a gradient flow cylinder, then
$\ind(u) \ge 1$, with equality if and only if $u$ itself is a cylinder and 
the cover is unbranched, in which case $u$ is also Fredholm regular.
\end{enumerate}
\end{lemma}
\begin{proof}
By Proposition~\ref{prop:posPuncture}, all the negative asymptotic orbits
of $u$ also have periods less than~$T_1$, so fixing the $S^1$-invariant
trivialization, the Conley-Zehnder indices of all asymptotic orbits are
given by Lemma~\ref{lemma:periodsAndIndices}, i.e.~they are $1$ for 
punctures in $\Gamma_1$ and $0$ for the others.
The index computation then follows easily from the standard formula
\eqref{eqn:index} after observing that the relative first Chern number
vanishes, as the $S^1$-invariant trivialization extends globally 
over~$\widecheck{M}\spine$: writing the Euler characterstic as
$2 - 2g - \#\Gamma$, we thus obtain
$$
\ind(u) = - \left( 2 - 2g - \#\Gamma \right) + \#\Gamma_1^+ - \#\Gamma_1^-,
$$
which reduces to the stated formula.  

To understand the consequences of this formula,
observe first that we always have
\begin{equation}
\label{eqn:impossible}
\begin{split}
&\Gamma^+ \ne \emptyset \text{ and } \Gamma^- \ne \emptyset,\\
&\#\Gamma_1^+ = 0 \quad \text{ implies } \quad \#\Gamma_1^- = 0.
\end{split}
\end{equation}
Indeed, these statements follow from the fact that they manifestly hold for 
all of the trivial cylinders and gradient flow
cylinders that constitute the somewhere injective curves covered by
components of~$u$ (the second statement depends on the fact that 
$\nabla H$ is Morse-Smale).
Thus if $\ind(u) < 0$, we necessarily have $g=0$
and $\#\Gamma_1^+ = 0$, implying $\#\Gamma_1^- = 0$, but then
$\#\Gamma_0^+$ and $\#\Gamma_0^-$ must both be positive and we have a
contradiction.  The index must also be strictly positive if $g \ge 1$ since
$\#\Gamma_0^+$ and $\#\Gamma_1^+$ will never both be zero.
Now suppose $g=0$ and $\ind(u) = 0$.  We have the following possibilities:
\begin{enumerate}
\item If $\#\Gamma_0^+ = 2$, then $\#\Gamma_1^+ = \#\Gamma_0^- = 0$,
but the implication in \eqref{eqn:impossible} then gives $\#\Gamma_1^- = 0$
and thus contradicts the fact that $\Gamma^- \ne \emptyset$.
\item If $\#\Gamma_0^- = 2$, then $\#\Gamma_0^+ = \#\Gamma_1^+ = 0$,
which is impossible since $\Gamma^+ \ne \emptyset$.
\item If $\#\Gamma_0^+ = \#\Gamma_0^- = 1$ and $\#\Gamma_1^+ = 0$, then
\eqref{eqn:impossible} implies $\#\Gamma_1^- = 0$, so $u$ is a trivial
cylinder over an even orbit, meaning a cover of some
$\gamma_z$ with $\Morse(z)=1$.
\item If $\#\Gamma_0^+ = \#\Gamma_0^- = 0$ and $\#\Gamma_1^+ = 1$, then
no components of $u$ can be covers of gradient flow cylinders since these
always have even negative punctures, thus we have a building whose
components are all covers of trivial cylinders over an odd orbit
(hence $\gamma_z$ with $\Morse(z)=2$), and the building has exactly
one positive puncture.  Since $g=0$, the latter implies that each of its 
components also has exactly one positive puncture.
\end{enumerate}
Finally, applying the index formula to the case where 
$u : \dot{S} \to \RR \times \widecheck{M}\spine$ is a cover of a
gradient flow cylinder, we have $\#\Gamma_0^+ = \#\Gamma_1^- = 0$ and thus
\begin{equation*}
\begin{split}
\ind(u) &= 2g - 2 + 2 \#\Gamma^+ + \#\Gamma^- =
2g - 2 + \#\Gamma + \#\Gamma^+ \\
&= -\chi(\dot{S}) + \#\Gamma^+,
\end{split}
\end{equation*}
which equals at least~$1$ since $\chi(\dot{S}) \le 0$ and
$\#\Gamma^+ \ge 1$.  Equality then holds if and only if $\dot{S}$ is
a cylinder, and the Riemann-Hurwitz formula then implies that the cover is 
unbranched.  Since $u$ is immersed in this case, Fredholm
regularity follows from the main result of \cite{Wendl:automatic}.
\end{proof}

\subsection{Perturbation of stable leaves}
\label{sec:perturbation}

For the SFT and ECH computations in \S\ref{sec:invariants}, we will need
to perturb the planar holomorphic curves in~$\fF$
as the confoliation $\Xi_0$ changes to the contact 
structure~$\Xi_\param$ for $\param > 0$.  To this end,
assume $(H,j)$ are in general position, and denote 
$$
\jJ(M' ; H) := \bigcup_{\hH'} \jJ(\hH'),
$$
where the union is over all confoliation-type stable Hamiltonian structures
$\hH'$ on $M'$ that match $\hH$ on a neighborhood
of $\CritMorse(H) \times S^1 \subset \widecheck{M}\spine$; 
note that this is true in particular for all $\hH_\param$ with $\param \ge 0$.
We assign to $\jJ(M' ; H)$ the
natural $C^\infty$-topology as a subset of the space of all smooth
translation-invariant almost complex structures on $\RR \times M'$.
Then for $J \in \jJ(M' ; H)$, denote by
$$
\mM(J ; H)
$$
the moduli space of $\RR$-equivalence classes of nonconstant, connected, 
unparametrized finite-energy $J$-holomorphic curves whose asymptotic orbits
are all contained in $\CritMorse(H) \times S^1$.\footnote{See \S\ref{sec:energy}
for some clarifying remarks on moduli spaces of unparametrized finite-energy
$J$-holomorphic curves and their topologies.}
This moduli space is naturally contained in the corresponding space of 
stable $J$-holomorphic buildings as defined in \cite{SFTcompactness},
$$
\mM(J ; H) \subset \overline{\mM}(J ; H).
$$
We shall consider the
resulting \emph{universal} moduli spaces
\begin{equation*}
\begin{split}
\mM(\jJ(M' ; H)) &:= \left\{ (J,u) \ \big|\ 
J \in \jJ(M' ; H), \ u \in \mM(J ; H) \right\}, \\
\overline{\mM}(\jJ(M' ; H)) &:= \left\{ (J,u) \ \big|\ J \in \jJ(M' ; H),
\ u \in \overline{\mM}(J ; H) \right\},
\end{split}
\end{equation*}
which inherit natural topologies.  Let
$$
\mM^\fF(J_0) \subset \mM(J_0 ; H)
$$
denote the subset consisting of $\RR$-equivalence classes of
embedded curves whose images are leaves of the foliation~$\fF$.
Its closure $\overline{\mM}^\fF(J_0) \subset \overline{\mM}(J_0 ; H)$
in the compactified moduli space consists of stable holomorphic
buildings whose levels are likewise disjoint unions of leaves of~$\fF$.
We will see in Lemma~\ref{lemma:uniqueness} that
$\overline{\mM}^\fF(J_0)$ is an open and closed subset of $\overline{\mM}(J_0 ; H)$;
note that this is clear
already for the components with arithmetic genus zero, due to
Fredholm regularity (Lemma~\ref{lemma:stable}).
Choose an open neighborhood
$$
\overline{\mM}^\fF(\jJ(M';H)) \subset \overline{\mM}(\jJ(M';H))
$$
of $\{J_0\} \times \overline{\mM}^\fF(J_0)$, let
$\mM^\fF(\jJ(M';H)) \subset \overline{\mM}^\fF(\jJ(M';H))$ denote the
open subset consisting of smooth $1$-level curves with no nodes,
and for each $J \in \jJ(M' ; H)$, denote\footnote{To clarify:
depending on the size of the neighborhood 
$\overline{\mM}^\fF(\jJ(M';H))$, one should expect $\mM^\fF(J)$
and $\overline{\mM}^\fF(J)$ to be empty unless $J$ is close to~$J_0$.
(The latter will of course be the main case of interest.)}
\begin{equation*}
\begin{split}
\mM^\fF(J) &:= \left\{ u \in \mM(J ; H)\ \big|\ (J,u) \in
\mM^\fF(\jJ(M';H)) \right\}, \\
\overline{\mM}^\fF(J) &:= \left\{ u \in \overline{\mM}(J ; H)\ \big|\ 
(J,u) \in \overline{\mM}^\fF(\jJ(M';H)) \right\}.
\end{split}
\end{equation*}
The components of each of these spaces consisting of curves or buildings with 
a prescribed index $i \in \ZZ$ will be written as
$$
\mM_i^\fF(J),\quad 
\overline{\mM}_i^\fF(J),\quad 
\mM_i^\fF(\jJ(M';H)),\quad
\overline{\mM}_i^\fF(\jJ(M';H)).
$$
After shrinking the neighborhood $\overline{\mM}^\fF(\jJ(M';H))$ if necessary,
we shall assume without loss of generality that the following conditions hold
for all $u \in\overline{\mM}^\fF(J)$:
\begin{itemize}
\item
All components of levels in $u$ are somewhere injective;
\item
If $u$ has arithmetic genus zero, then all components of levels in $u$ are
Fredholm regular.
\end{itemize}
Both conditions follow from the fact that they are already known to hold
for $\overline{\mM}^\fF(J_0)$.  If additionally $J \in \jJ(M';H)$ is generic,
then it now follows that $\mM^\fF_i(J)$ is empty for all $i \le 0$.
With or without genericity, we can also conclude that
$\mM^\fF_1(J)$ and $\mM^\fF_2(J)$ are
smooth manifolds of dimensions~$0$ and~$1$
respectively (recall that we divided out the $\RR$-action in defining these spaces).
By inspection of the foliation~$\fF$, we see that 
$\mM^\fF_1(J_0) = \overline{\mM}^\fF_1(J_0)$ has 
finitely many elements, and $\overline{\mM}^\fF_2(J_0)$ is homeomorphic to a
disjoint union of finitely many circles and compact intervals whose endpoints
are $2$-level buildings, 
both levels being unions of trivial cylinders with curves in 
$\mM^\fF_1(J_0)$.  The implicit function theorem now implies that
this description also holds for $\overline{\mM}^\fF_i(J)$ 
whenever $J$ is sufficiently close to~$J_0$:

\begin{prop}
\label{prop:IFT}
For all $J \in \jJ(M';H)$ sufficiently close to~$J_0$,
there exist families of homeomorphisms
$$
\Psi_J : \mM^\fF_1(J_0) \to \mM^\fF_1(J),\qquad
\Psi_J : \overline{\mM}^\fF_2(J_0) \to \overline{\mM}^\fF_2(J)
$$
that depend continuously on $J \in \jJ(M';H)$ and satisfy $\Psi_{J_0} = \Id$.
In particular, $\mM^\fF_1(J)$ contains
finitely many elements, and $\overline{\mM}^\fF_2(J)$ is homeomorphic
to a disjoint union of finitely many circles and compact intervals whose 
endpoints consist of
$2$-level holomorphic buildings in which each level is a union of
trivial cylinders with a curve in~$\mM^\fF_1(J)$.  

Moreover, if $J$ is also generic, then $\mM^\fF_i(J) = \emptyset$ 
for all~$i \le 0$.
\qed
\end{prop}

We obtain a stronger result in the special case where all pages are
both interior and planar.  The proof of the following will be postponed until
the end of~\S\ref{sec:intersections}, since it 
requires a bit of intersection theory.

\begin{prop}
\label{prop:stableFEF}
Assume $\p M = \emptyset$ and every page in $M\paper$ has genus zero.
Then for $J \in \jJ(M';H)$ sufficiently close to~$J_0$,
the trivial cylinders over
orbits in $\CritMorse(H) \times S^1$, together with the curves
in $\mM_1^\fF(J) \cup \mM_2^\fF(J)$ foliate $\RR\times M$, and thus
form a stable finite energy foliation of $(\RR\times M,J)$.
\end{prop}

\subsection{Intersection-theoretic properties}
\label{sec:intersections}

We now examine the properties of the foliation $\fF$ and perturbed moduli spaces
$\mM^\fF(J)$ in terms of Siefring's intersection theory
of punctured holomorphic curves (see \S\ref{sec:Siefring}).
We shall assume throughout the following that the period
bounds and index formula of Lemma~\ref{lemma:periodsAndIndices} 
hold for some constants $T_1 > T_0 > 0$.
We first observe the following
immediate consequence of Lemma~\ref{lemma:periodsAndIndices} and
\eqref{eqn:CZwinding}.  Note that whenever $\gamma$ is a nondegenerate
Reeb orbit with $\muCZ(\gamma)=0$, the same holds for all covers
of~$\gamma$.

\begin{lemma}
\label{lemma:winding}
For any $z \in \CritMorse(H)$ with Morse index~$0$ or~$2$ and $k \in \NN$ such 
that $\gamma_z^k$ has period
less than~$T_1$, the extremal 
winding numbers $\alpha_\pm(\gamma_z^k)$ behave as follows:
\begin{itemize}
\item If $\Morse(z) = 2$, then $\alpha_-(\gamma_z^k) = 0$ and 
$\alpha_+(\gamma_z^k) = 1$.
\item If $\Morse(z) = 0$, then $\alpha_-(\gamma_z^k) = -1$ and 
$\alpha_+(\gamma_z^k) = 0$.
\end{itemize}
Moreover if $\Morse(z)=1$, then 
$\alpha_-(\gamma_z^k) = \alpha_+(\gamma_z^k) = 0$ for all $k \in\NN$.
\qed
\end{lemma}

\begin{lemma}
\label{lemma:extremalWinding}
For every $J \in \jJ(M';H)$ and $u \in \mM^\fF(J)$, if~$u$ is not a
trivial cylinder, then $c_N(u)=0$ and $u$ has zero asymptotic winding 
(in the $S^1$-invariant trivialization) at each puncture.
\end{lemma}
\begin{proof}
For any leaf $u \in \fF$ which is not a trivial cylinder, the projection
of~$u$ to~$M'$ is embedded and thus $\windpi(u)=0$.  Since each end
of~$u$ is an
$S^1$-invariant cover of a gradient flow line and thus has asymptotic
winding zero, this winding is extremal by Lemma~\ref{lemma:winding} and
we have $\asympdef(u)=0$.  Then \eqref{eqn:windpi} implies $c_N(u)=0$.
Since $c_N(u)$ depends only on the asymptotic orbits and relative
homology class, it remains zero for any $u \in \mM^\fF(J)$ which is
not a trivial cylinder, so \eqref{eqn:windpi} then implies 
$\asympdef(u)=0$ and the result on asymptotic winding follows.
\end{proof}

\begin{lemma}
\label{lemma:leavesOrbits}
Suppose $k,m \in \NN$, $J \in \jJ(M';H)$, $u \in \mM^\fF(J)$ is not a 
trivial cylinder and $u^k$ denotes any $J$-holomorphic 
$k$-fold cover of~$u$.  Then for any
$z \in \CritMorse(H)$ and $m \in \NN$ such that $\gamma_z^{km}$ has period
less than~$T_1$,  
$$
u^k * (\RR\times \gamma_z^m) = 0.
$$
If $\Morse(z)=1$, then this also holds without any restriction on
$k , m \in \NN$.
\end{lemma}
\begin{proof}
By homotopy invariance, it suffices to show that this holds for
$J=J_0$ and any leaf $u \in \fF$ which is not a trivial cylinder.
The image of~$u^k$ then covers a gradient flow line wherever it intersects
$M\spine$, so $u^k$ has no actual intersections with
$\RR\times \gamma_z^m$, and it remains to rule out asymptotic contributions.
By the definition in \cite{Siefring:intersection}, these can exist only
if~$u$ has an end approaching~$\gamma_z$, in which case
$u^k$ has an end approaching $\gamma_z^n$ for some $n \le k$.  Moreover,
the asymptotic contribution is then zero if and only if the asymptotic
winding of the $m$-fold cover of this end differs from the a priori
bound set by $\alpha_\pm(\gamma_z^{mn})$.  In the $S^1$-invariant trivialization,
this asymptotic winding is manifestly zero, so in the case
$\Morse(z) \in \{0,2\}$, $nm \le km$ implies that
$\gamma_z^{mn}$ has period less than~$T_1$, and Lemma~\ref{lemma:winding}
then gives $\alpha_\pm(\gamma^{mn}) = 0$ for the appropriate choice of sign.
For $\Morse(z) = 1$, the same holds with no restrictions on multiplicities
since $\alpha_\pm(\gamma^k)=0$ for all $k \in \NN$.
\end{proof}

\begin{prop}
\label{prop:int0}
For any $J \in \jJ(M';H)$, every curve $u \in \mM^\fF(J)$ is embedded,
and any two such curves $u , v$ satisfy $u*v = 0$ unless 
both are trivial cylinders.
\end{prop}
\begin{proof}
The proof of $u*v=0$ is immediate from Lemmas~\ref{lemma:nicelyEmbedded}
and~\ref{lemma:leavesOrbits}.  If~$u$ is not a trivial cylinder, then
this also implies $u*u=0$, and since $u$ is somewhere injective by the
definition of $\overline{\mM}^\fF(\jJ(M';H))$,
it now follows from the adjunction formula
\eqref{eqn:adjunction} that $u$ is nicely embedded, hence also embedded.
\end{proof}

\begin{lemma}
\label{lemma:intersectingPages}
Assume $J \in \jJ(M';H)$, and
$v \in \overline{\mM}(J;H)$ is a $J$-holomorphic curve whose 
positive ends are all asymptotic to orbits of the
form $\gamma_z^k$ with $z \in \CritMorse(H)$ and $k \in \NN$ having period less
than~$T_1$.
Then for any curve $u \in \mM^\fF(J)$ with no negative ends, $v * u = 0$.
\end{lemma}
\begin{proof}
After an $\RR$-translation we can assume the image of $u$ is contained in
$[0,\infty) \times M$.  Likewise, we can homotop~$v$ through asymptotically
cylindrical maps to a (non-holomorphic) map $v'$ which looks the same
near its negative ends but whose intersection with
$[0,\infty) \times M$ consists only of the trivial cylinders over its
positive asymptotic orbits.  Thus by the homotopy invariance of the
intersection number,
$$
v * u = v' * u = 
\sum_i (\RR\times \gamma_{z_i}^{k_i}) * u,
$$
for some finite set of critical points $z_i \in \CritMorse(H)$ and 
natural numbers 
$k_i$ such that $\gamma_{z_i}^{k_i}$ has period less than~$T_1$.
This is zero by Lemma~\ref{lemma:leavesOrbits}.
\end{proof}

We are now in a position to prove Proposition~\ref{prop:stableFEF}.

\begin{proof}[Proof of Proposition~\ref{prop:stableFEF}]
By Prop.~\ref{prop:int0}, every curve in $\mM^\fF_1(J) \cup
\mM^\fF_2(J)$ is nicely embedded and disjoint from its own asymptotic
orbits, and any two such curves are either identical 
(up to $\RR$-translation) or disjoint; in fact the latter is also true
for arbitrary $\RR$-translations, implying that their projections to~$M$
are either identical or disjoint.  It remains to show that these curves
fill the entirety of~$M$.  Let $\Delta \subset M$ denote the compact
set consisting of all points that lie either in 
$\CritMorse(H) \times S^1$ or in the projection of any curve
in $\mM^\fF_1(J)$.  Also, let $\oO \subset M \setminus\Delta$ denote
the set of points that lie in the projection of any curve in
$\mM^\fF_2(J)$.  Applying the implicit function theorem to finitely
many index~$2$ leaves of~$\fF$, we find that for $J$ sufficiently close to~$J_0$,
every connected component of $M \setminus \Delta$ contains points in~$\oO$.

To finish, we claim that~$\oO$ is an open and closed subset of
$M\setminus \Delta$.  To see that it is closed, suppose $x_k \in \oO$ 
converges to $x \in M\setminus \Delta$.  Then the points $x_k$ are contained 
in the projections of curves $u_k \in \mM^\fF_2(J)$, and these have
a subsequence converging to either another curve $u \in \mM^\fF_2(J)$
or a building $u \in \overline{\mM}^\fF_2(J)$.  In the former case we
conclude $x \in \oO$, and in the latter case, $u$ has two levels which
are each unions of trivial cylinders with curves in $\mM^\fF_1(J)$,
so $x \in \Delta$ and we have a contradiction.  That~$\oO$ is open follows
from an implicit function theorem for nicely embedded index~$2$ curves,
cf.~\cite{Wendl:thesis}*{Theorem~4.5.44}.
The main point is that since any $u \in \mM^\fF_2(J)$ is embedded, the nearby curves 
in $\mM^\fF_2(J)$ can all be identified with sections of the normal
bundle of~$u$, and the condition $c_N(u)=0$ from
Lemma~\ref{lemma:extremalWinding} implies that these sections must be
nowhere zero, cf.~\cite{Wendl:automatic}*{Equation~(2.7)}.
\end{proof}

\subsection{Uniqueness}
\label{sec:uniqueness}

For any $J \in \jJ(M';H)$, define the spaces
$$
\mM(J ; H,T_1), \qquad \overline{\mM}(J ; H,T_1)
$$
to consist of all connected $\RR$-equivalence classes of unparametrized
finite energy $J$-holomorphic curves or buildings respectively
in $\RR\times M'$
whose positive asymptotic orbits are in $\CritMorse(H) \times S^1$ and have
periods adding up to less than~$T_1$.  Proposition~\ref{prop:posPuncture}
then implies that the same condition holds at the negative ends, thus
all negative asymptotic orbits of
curves in $\mM(J ; H,T_1)$ or $\overline{\mM}(J ; H,T_1)$
are also in $\CritMorse(H) \times S^1$ due to 
Lemma~\ref{lemma:periodsAndIndices}.
We denote by
$$
\mM^*(J ; H,T_1) \subset \mM(J ; H,T_1)
$$
the set of somewhere injective curves in $\mM(J ; H,T_1)$.
The following is now an easy consequence of the intersection theory
from~\S\ref{sec:intersections}.

\begin{lemma}
\label{lemma:uniqueness}
Every curve in $\mM^*(J_0 ; H,T_1)$ is an interior leaf of~$\fF$.
\end{lemma}
\begin{proof}
Suppose $u \in \mM^*(J_0 ; H,T_1)$ is not a leaf
of~$\fF$.  By Prop.~\ref{prop:posPuncture} it must have at least one 
positive puncture, which by assumption is asymptotic to an orbit of the form 
$\gamma_{z}^k$ for $z \in \CritMorse(H)$ and $k \in \NN$, with period less
than~$T_1$.  All trivial cylinders over orbits in $\CritMorse(H) \times S^1$
are leaves of~$\fF$, so we may assume~$u$ is not a trivial cylinder.
Then as~$u$ approaches $\gamma_z^k$, it has isolated intersections with 
infinitely many leaves of~$\fF$.  In particular, we can find a \emph{generic}
leaf $v \in \fF$ with $u * v > 0$, i.e.~$v$ has no ends asymptotic to
orbits~$\gamma_\zeta$ with $\Morse(\zeta)=1$.  Since there are no index~$0$
critical points, this implies every end of~$v$ is positive.  Then
Lemma~\ref{lemma:intersectingPages} implies $u * v = 0$, so we have a
contradiction.
\end{proof}

To generalize the above lemma to the perturbation of~$J_0$, we will need
to specialize to the case where $(M,\xi)$ is a partially planar domain,
i.e.~we assume there exists a connected component $M\planar\paper \subset M\paper$
which has genus zero pages and does not touch~$\p M$.  

\begin{remark}
\label{remark:shrinkIt}
If $M\planar\paper \subset M\setminus \p M$ is a connected component of
$M\paper$, then we can always shrink~$M$ to a smaller subdomain 
on which~$\xi$ is still supported by a spinal open book containing the
pages of $M\planar\paper$ in the interior, but with the additional property that
every spinal component intersects $M\planar\paper$.  Indeed, if
$\Sigma_1 \times S^1 \subset M\spine$ is any spinal component disjoint from
$M\planar\paper$, then since~$\xi$ is transverse to the $S^1$-direction at
$\p(\Sigma_1 \times S^1)$, we can replace $M$ with a smaller domain whose
boundary includes all components of $\p M\paper$ that touch 
$\Sigma_1 \times S^1$, and assume after an isotopy of $\xi$ that it is
supported by a spinal open book on the shrunken domain
(cf.~\cite{LisiVanhornWendl1}*{Example~1.11}).
\end{remark}

By the above remark, we lose no generality by imposing
the following conditions on our partially planar domain and the
chosen geometric data:

\begin{assumptions}
\label{ass:partiallyPlanar}
Suppose the spinal open book on $(M,\xi)$, the data $H$ and~$j$ and
constants $K,\varepsilon > 0$ (cf.~Lemma~\ref{lemma:periodsAndIndices})
satisfy the following conditions:
\begin{enumerate}
\item $M\paper \subset M$ contains a connected component
$M\planar\paper$ which has genus zero pages and 
$\p M\planar\paper \subset M\spine$;
\item Every page in the interior of~$M$ has
fewer than $T_1 / T_0$ boundary components;
\item Every component of $M\spine$ intersects~$M\planar\paper$;
\item $H : \Sigma \to [0,\infty)$ has exactly one index~$2$ critical point 
on every connected component of~$\Sigma$.
\end{enumerate}
\end{assumptions}

For the main result of this section, we consider sequences
$\param_\nu > 0$ and $J_\nu \in \jJ(\hH_{\param_\nu}) \subset \jJ(M';H)$ such that
$$
\param_\nu \to 0 \qquad\text{ and }\qquad J_\nu \to J_0.
$$
\begin{prop}
\label{prop:uniqueness}
If Assumptions~\ref{ass:partiallyPlanar} hold, then
for~$\nu$ sufficiently large, $\mM^*(J_\nu ; H,T_1) = \mM^\fF(J_\nu)$.
\end{prop}
\begin{proof}
The claim that $\mM^\fF(J_\nu) \subset \mM^*(J_\nu ; H,T_1)$
follows immediately from Assumptions~\ref{ass:partiallyPlanar} since
every curve in $\mM^\fF(J_\nu)$ has fewer than $T_1 / T_0$ ends,
all approaching simply covered orbits in $\CritMorse(H) \times S^1$ with
period less than~$T_0$.  We will prove
the converse by showing that for
$\nu$ sufficiently large, the presence of the embedded planar 
curves in $\mM^\fF_2(J_\nu)$ forces all other curves in
$\mM^*(J_\nu ; H,T_1)$ to be nicely embedded.  Then
the compactness theorem in \cite{Wendl:compactnessRinvt} implies essentially
that if $\nu$ is large enough, then every such curve in 
$\mM^*(J_\nu ; H,T_1)$ is a perturbation of a nicely embedded 
$J_0$-holomorphic building, whose components must be leaves 
of~$\fF$ due to Lemma~\ref{lemma:uniqueness}.  Here are the details.

Arguing by contradiction, assume there is a sequence of $J_\nu$-holomorphic
curves
$$
u_\nu : \dot{S}_\nu \to \RR\times M'
$$
which define elements in
$\mM^*(J_\nu ; H,T_1) \setminus \mM^\fF(J_\nu)$ as $\nu \to\infty$.
Since the trivial cylinders over orbits in $\CritMorse(H) \times S^1$ are all
leaves of~$\fF$, we may assume $u_\nu$ is never a trivial cylinder.
After taking a subsequence, we may also assume that $u_\nu$ has fixed
numbers of positive and negative ends, always approaching the same collection
of orbits; this follows from the period bound at the positive ends since
there are finitely many combinations of orbits in $\CritMorse(H) \times S^1$
for which the required bound is satisfied.

\textsl{Step~1: Compactness.}  We claim that a subsequence of~$u_\nu$
converges to a $J_0$-holomorphic building $u_\infty$ whose connected
components are all covers of interior leaves in~$\fF$.
The convergence does not immediately follow from \cite{SFTcompactness}, 
for three reasons:
\begin{enumerate}
\item We must check that $u_\nu$ satisfy a suitable energy bound as the
contact structures $\Xi_{\param_\nu}$ degenerate to the confoliation~$\Xi_0$.
\item We have not assumed any bound on the genus of~$\dot{S}_\nu$.
\item The dynamics of~$R_0$ are degenerate.
\end{enumerate}
The first issue is the main reason we have introduced the stable Hamiltonian
structures $\hH_\param = (\Omega_\param,\Lambda_\param)$.
A reasonable notion of energy can be defined by
$$
E_\nu(u_\nu) := \sup_{\varphi \in \tT} \int_{\dot{S}_\nu} u_\nu^* 
\left( d(\varphi(r) \Lambda_{\param_\nu}) + \Omega_{\param_\nu} \right),
$$
where $\tT$ denotes the space of smooth strictly increasing functions
$\varphi : \RR \to (-\delta,\delta)$ for some constant $\delta > 0$ chosen
sufficiently small to make sure that the integrand is nonnegative.
This notion is equivalent to the energy defined in
\cite{SFTcompactness}, in the sense that either satisfies uniform bounds
if and only if the other does.  Since $\Omega_{\param_\nu} = d\alpha$, we
can write
$$
E_\nu(u_\nu) = \sup_{\varphi \in \tT} \int_{\dot{S}_\nu}
u_\nu^*d\left( \varphi(r) \Lambda_{\param_\nu} + \alpha \right)
$$
and thus conclude from Stokes' theorem and
the bound on the periods at the positive ends
(cf.~Prop.~\ref{prop:posPuncture}) that $E_\nu(u_\nu)$ is uniformly bounded.

The second issue is a larger danger.  In order to bound the genus 
of~$\dot{S}_\nu$, we use the following
argument, originally suggested by Michael Hutchings and used already
in \cite{Wendl:openbook2}. 
Since $E_\nu(u_\nu)$ is bounded and we have
convergence of the data $\hH_{\param_\nu} \to \hH_0$ and $J_\nu \to
J_0$, a compactness theorem of Taubes \cite{Taubes:currents}*{Prop.~3.3}
(see also \cite{Hutchings:index}*{Lemma~9.9}) implies that the sequence of
currents represented by~$u_\nu$ has a convergent subsequence.  This implies
in particular that the relative homology classes of~$u_\nu$ have a convergent
subsequence, so taking advantage of the assumption that $u_\nu$ is
somewhere injective, we can write down the adjunction inequality
\eqref{eqn:adjunction},
$$
u_\nu * u_\nu \ge 2\left[\delta(u_\nu) + \delta_\infty(u_\nu)\right] + c_N(u_\nu)
\ge c_N(u_\nu)
$$
and observe that the left hand side is bounded.  Now plugging in the
definition of the normal Chern number from \cite{Wendl:compactnessRinvt},
we have
$$
c_N(u_\nu) = c_1(u^*\Xi_{\param_\nu}) - \chi(\dot{S}_\nu) + C
$$
where the constant $C \in \ZZ$ depends only on the extremal winding
numbers at the asymptotic orbits and is thus fixed, and
$c_1(u^*\Xi_{\param_\nu})$ is the relative first Chern number of
the bundle $u^*\Xi_{\param_\nu} \to \dot{S}_\nu$ with respect to the $S^1$-invariant
trivializations at the asymptotic orbits.  The latter also depends
only on the relative homology class of~$u_\nu$, so we conclude that
$\chi(\dot{S}_\nu)$ is bounded from below, giving a bound on the
genus of~$\dot{S}_\nu$ from above.  Passing again to a subsequence, we may
now assume all the surfaces $\dot{S}_\nu$ are diffeomorphic.

To conclude, we observe that the dynamics of $R_0$ are indeed nondegenerate
up to period~$T_1$.  By Prop.~\ref{prop:posPuncture} and the period 
bound imposed on the positive asymptotic orbits of~$u_\nu$, every orbit 
that can appear in bubbling or breaking is therefore nondegenerate, in which 
case the proof of the main compactness theorem in \cite{SFTcompactness} 
goes through and gives a subsequence 
convergent in the usual sense to a $J_0$-holomorphic building~$u_\infty$.

The fact that all components of $u_\infty$ are covers of interior
leaves in~$\fF$ follows now from Lemma~\ref{lemma:uniqueness}.

\textsl{Step~2: Intersection theory.}
The goal of this step is to show that $u_\nu * u_\nu = 0$ for all
$\nu$ sufficiently large. 
Since $H$ has no Morse critical points of Morse index~$0$, every curve in
$\mM^\fF_2(J_\nu)$ has all its ends at critical points of Morse index~$2$,
thus they are all positive.  Then Lemma~\ref{lemma:intersectingPages}
implies that for any $v \in \mM^\fF_2(J_\nu)$,
\begin{equation}
\label{eqn:zero}
u_\nu * v = 0.
\end{equation}
We claim next that no negative end of $u_\nu$ approaches any orbit of
the form $\gamma_z^k$ with $\Morse(z)=2$, and any positive end that
approaches such an orbit has asymptotic winding zero.  Here we
use the assumption that every connected component of $M\spine$
intersects the planar piece
$M\planar\paper$ and has a unique index~$2$ critical point:
it follows via Proposition~\ref{prop:IFT}
that whenever $\Morse(z)=2$, for sufficiently large~$\nu$
there exists a curve $v \in \mM^\fF_2(J_\nu)$ whose asymptotic orbits
include~$\gamma_z$.  By Lemma~\ref{lemma:extremalWinding}, $v$
approaches~$\gamma_z$ with zero asymptotic winding, so if the asymptotic
winding of $u_\nu$ approaching $\gamma_z^k$ is nonzero, then the
projections of $u_\nu$ and~$v$ to~$M'$ intersect, implying
$u_\nu * v > 0$ and thus contradicting \eqref{eqn:zero}.
Moreover, the end approaching $\gamma_z^k$ cannot be negative, as its
asymptotic winding would then be bounded from below by
$\alpha_+(\gamma^k)$, which is~$1$ by Lemma~\ref{lemma:winding}.

Finally, we claim that for every orbit~$\gamma$ which occurs as an
asymptotic orbit of~$u_\nu$,
\begin{equation}
\label{eqn:zero1}
u_\nu * (\RR\times \gamma) = 0.
\end{equation}
The orbit~$\gamma$ is necessarily of the form $\gamma_z^k$ for
$z \in \CritMorse(H)$ and $k\in \NN$ and has period less than~$T_1$.
If $\Morse(z)=2$ then we may again assume due to Proposition~\ref{prop:IFT}
that $\gamma_z$ is an
asymptotic orbit for some $v \in \mM_2^\fF(J_\nu)$.  Then any
intersection of~$u_\nu$ with $\RR\times \gamma_z$ is necessarily
positive and thus causes an intersection of~$u_\nu$ with an
end of~$v$ approaching~$\gamma_z$, again contradicting
\eqref{eqn:zero}.  Asymptotic contributions to
$u_\nu * (\RR\times \gamma_z^k)$ are also ruled out since, as was
just shown, any end of $u_\nu$ approaching a cover 
of $\gamma_z$ has asymptotic winding zero, and this matches
$\alpha_-(\gamma_z^k)$ by Lemma~\ref{lemma:winding}.

For the case $\Morse(z)=1$ we argue slightly differently: we pass
to the limit and show that $u_\infty * (\RR\times \gamma_z^k) = 0$,
which implies \eqref{eqn:zero1} for sufficiently large~$\nu$.  Recall that
every connected component~$w$ of the building $u_\infty$ is a cover of
some leaf of~$\fF$.  If this leaf is a trivial cylinder, then
$w * (\RR\times \gamma_z^k) = 0$ by Lemma~\ref{lemma:evenOrbits} since
$\muCZ(\gamma_z)$ is even.  If it is not a trivial cylinder, then we instead
obtain the same result from Lemma~\ref{lemma:leavesOrbits}.
There are also no breaking contributions to
$u_\infty * (\RR\times \gamma_z^k)$ since there are no common breaking
orbits with odd Conley-Zehnder index.

We have now established all the conditions to apply
Lemma~\ref{lemma:nicelyEmbedded} and conclude
$$
u_\nu * u_\nu = 0.
$$

\textsl{Step~3: Nicely embedded curves degenerate nicely.}
By the main result of \cite{Wendl:compactnessRinvt}, the limit
building $u_\infty$ must also be nicely embedded, in the sense that
all of its levels are unions of trivial cylinders with nicely embedded
curves: in particular, this means every component of $u_\infty$ is
an interior leaf of~$\fF$.  By inspection of~$\fF$,
the only connected multi-level buildings
one can construct out of leaves have exactly two levels:
the top consists of a disjoint union of trivial cylinders with
gradient flow cylinders, and the bottom is a single holomorphic page.
Any such building is a limit of a sequence of holomorphic pages
in~$\fF$, and thus belongs to $\overline{\mM}^\fF(J_0)$, so we conclude that
$u_\nu \in \mM^\fF(J_\nu)$ for sufficiently large~$\nu$.
\end{proof}

\begin{remark}
\label{remark:nondegenerate}
Proposition~\ref{prop:uniqueness} also holds for any
sufficiently small perturbations $\Lambda_\nu'$ of $\Lambda_{\param_\nu}$ 
(fixed in a neighborhood of $\CritMorse(H) \times S^1$) and
$J_\nu' \in \jJ(d\Lambda_\nu',\Lambda_\nu')$ of $J_\nu$.  In particular we can arrange
in this way for $\Lambda_\nu'$ to be a sequence of nondegenerate contact forms.  
The uniqueness result is proved by repeating the above argument for sequences 
$\Lambda_{\param_\nu}^\mu \to \Lambda_{\param_\nu}$ and $J_\nu^\mu \in 
\jJ(d\Lambda_{\param_\nu}^\mu,\Lambda_{\param_\nu}^\mu)$,
$J_\nu^\mu \to J_\nu$ as $\mu \to \infty$.  The only reason we did not
state Prop.~\ref{prop:uniqueness} to allow this perturbation in the first
place is that there is no obvious way to perturb the stable Hamiltonian
structure $\hH_{\param_\nu}$ together with $\Lambda_{\param_\nu}$---instead,
the compactness argument in the proof as 
$\Lambda_{\param_\nu}^\mu \to \Lambda_{\param_\nu}$ requires the usual
notion of energy for almost complex structures compatible with contact
forms as in \cite{Hofer:weinstein}, taking advantage of
Proposition~\ref{prop:sameJ}.
\end{remark}

\section{Computations in ECH and SFT}
\label{sec:invariants}

We now apply the holomorphic curve construction of the previous section
to prove Theorems~\ref{thm:algTorsion}, \ref{thm:ECHvanish} and~\ref{thm:Umap}.  

Adopting the notation of \S\ref{sec:holOpenBook}, assume $M$ is a
compact $3$-manifold contained in a closed and connected contact 
$3$-manifold $(M',\xi)$,
$\Omega$ is a closed $2$-form on $M'$
and $\xi|_M$ is supported by a an $\Omega$-separating partially planar 
spinal open book~$\boldsymbol{\pi}$.
Fix all data necessary for defining the exact stable Hamiltonian 
structures $\hH_\param = (d\alpha,\Lambda_\param)$, along with 
$J_0 \in \jJ(\hH_0)$ admitting the $J_0$-holomorphic finite energy 
foliation $\fF$ on $\RR \times M$ and the perturbed moduli spaces
defined in \S\ref{sec:perturbation} and~\S\ref{sec:uniqueness}.
After possibly shrinking $M$ to a
smaller domain, we can and shall take
Assumptions~\ref{ass:partiallyPlanar} as given.
The assumptions also imply that $\Omega$ is exact on $M\spine$.

Denote the connected components of~$\Sigma$ by
$$
\Sigma = \Sigma_1 \cup \ldots \cup \Sigma_r
$$
and for each $i=1,\ldots,r$, let $z_i \in \CritMorse(H)$ denote the unique
index~$2$ critical point in~$\Sigma_i$.  Suppose the pages in
$M\planar\paper \subset M\paper$ have $k + 1 \ge 1$ boundary components.
Without loss of generality, we may assume that~$k$ is \emph{minimal}
in the sense that for any other connected component $M\other\paper \subset M\paper$
with planar pages and $\p M\other\paper \subset M\spine$, the pages have
\emph{at least} $k+1$ boundary components.  For $j=1,\ldots,r$, let
$$
m_j \in \NN
$$
denote the number of boundary components of each page in $M\planar\paper$
that lie in $\Sigma_j \times S^1 \subset M\spine$.

For Theorems~\ref{thm:ECHvanish} and~\ref{thm:algTorsion}, we add the
assumption that $M$ is a planar $k$-torsion domain.  In this case there
is at least one other connected component $M\other\paper$ which is
``different'' from $M\planar\paper$ in the sense that at least one of the
following conditions holds:
\begin{enumerate}
\item The pages in $M\other\paper$ are not diffeomorphic to those in~$M\planar\paper$,
\item For some $j \in \{1,\ldots,r\}$, the pages of $M\other\paper$ do not 
have exactly $m_j$ boundary components contained in $\Sigma_j \times S^1$.
\end{enumerate}
These assumptions imply that at least one connected component of $\Sigma$
has disconnected boundary, so after reordering the labels, assume this is
true of~$\Sigma_r$.  We may then assume $\Sigma_r$ contains at least one
extra critical point
$$
\zeta \in \Sigma_r, \qquad
\Morse(\zeta) = 1,
$$
such that the two gradient flow lines ending at $\zeta$ enter
through different components of~$\p\Sigma_r$, one from $M\planar\paper$ and the
other from $M\other\paper$.

Choose a nondegenerate contact form $\Lambda_\nu'$ and almost complex structure 
$J_\nu' \in \jJ(\Lambda_\nu')$ for which Propositions~\ref{prop:IFT}
and~\ref{prop:uniqueness} both hold (see also Remark~\ref{remark:nondegenerate}), 
and assume additionally that $J_\nu'$ is generic, so $\mM^\fF_i(J_\nu')$ is empty for all
$i \le 0$.  For any set of integers $n_1,\ldots,n_r \ge 0$,
define
\begin{equation}
\begin{split}
\mM(J_\nu' ; H ; n_1,\ldots,n_r) &\subset \mM(J_\nu' ; H,T_1),\\
\overline{\mM}(J_\nu' ; H ; n_1,\ldots,n_r) &\subset 
\overline{\mM}(J_\nu' ; H,T_1),
\end{split}
\end{equation}
to consist of all connected curves or buildings respectively such that for 
each $i \in \{1,\ldots,r\}$, the sum of the covering multiplicities of all
positive asymptotic orbits in $\Sigma_i \times S^1$ is less than or
equal to~$n_i$.  Applying Propositions~\ref{prop:IFT}
and~\ref{prop:uniqueness} under the above assumptions, we can now
completely classify the somewhere injective curves in
$\mM(J_\nu' ; H ; m_1,\ldots,m_r)$ as follows.  
The generic curve in this space is an embedded index~$2$ punctured sphere
with no negative ends and $k+1$ positive ends, of which $m_j$ ends are 
asymptotic to $\gamma_{z_j}$ for $j=1,\ldots,r$.
Aside from trivial cylinders, the only other somewhere injective curves
in $\mM(J_\nu' ; H ; m_1,\ldots,m_r)$ are the following: for every
$i=1,\ldots,r$ and every index~$1$ critical point $y \in \Sigma_i$,
\begin{itemize}
\item 
Each gradient flow line~$\ell$ entering through $\p \Sigma_i$ from
$M\planar\paper$ and ending at $y$ corresponds to a unique
index~$1$ punctured sphere $u_\ell$
with no negative ends, and $k+1$ positive ends asymptotic to the
same collection of simply covered orbits as the generic curves,
except with one copy of $\gamma_{z_i}$ replaced by $\gamma_y$;
\item There are exactly two embedded index~$1$ cylinders $v_y^+,v_y^-$, 
each with a positive end 
at $\gamma_{z_i}$ and negative end at $\gamma_{y}$, such that the closed
cycle $[v_y^+] - [v_y^-] \in H_2(M')$ defined by the two relative 
homology classes satisfies
$$
\int_{[v_y^+] - [v_y^-]} \Omega = 0.
$$
\end{itemize}
All of these curves are ECH-admissible, i.e.~they satisfy $\ind(u)=I(u)$.

\begin{proof}[Proof of Theorem~\ref{thm:ECHvanish}]
Consider the orbit set
$$
\boldsymbol{\gamma} = \{(\gamma_{z_1},m_1),\ldots,(\gamma_{z_{r-1}},m_{r-1}),
(\gamma_{z_r},m_r - 1),(\gamma_\zeta,1) \}
$$
as a generator of the ECH chain complex for 
$(M',\Lambda_\nu',J_\nu')$ with coefficients in
$\ZZ[H_2(M') / \ker\Omega]$.  Here we are abusing notation slightly by
allowing the possibility $m_r - 1 = 0$; if this is the case then
$\gamma_{z_r}$ should be removed from the orbit set altogether.
By the above classification, $\p_\ECH\boldsymbol{\gamma}$ counts two index~$1$
cylinders $v_y^+$ and $v_y^-$ for every $y \in \CritMorse(H)$ with 
Morse index~$1$, but these are
homologous in $H_2(M') / \ker\Omega$, and Proposition~\ref{prop:orientIndex1}
implies that for any choice of coherent orientations provided by
\cite{BourgeoisMohnke}, they cancel each other out.  
Thus the only index~$1$ curve remaining to count is the
punctured sphere $u_\ell$ corresponding to the unique gradient flow line~$\ell$
that enters $\p\Sigma_r$ from $M\planar\paper$ and ends at~$\zeta$.
Since~$u_\ell$ has no positive ends,
we find $\p_{\ECH}\boldsymbol{\gamma} = \boldsymbol{\emptyset}$.
\end{proof}
\begin{remark}
\label{remark:hintforAgustin}
In the above proof, we achieved cancelation for the cylinders $v_y^+$ and
$v_y^-$ by appealing to Proposition~\ref{prop:orientIndex1}, which is a
distinctly low-dimensional result, but there are also other ways to
see that the paired cylinders in this particular setting must be
oppositely oriented.  One such approach is to cap off $\Sigma$ by disks
and extend the function $H$ with a single index~$0$ critical point on
each cap, and then identify the normal
Cauchy-Riemann operators for the gradient flow cylinders with linearizations
of the Floer equation with respect to a $C^2$-small time-independent Hamiltonian
on the resulting closed surface.  The computation of
Hamiltonian Floer homology on this surface then implies that paired
cylinders must cancel because the index~$2$ critical point (viewed as a
constant Hamiltonian orbit) is a closed generator of the Floer chain complex.
This approach can also work in higher-dimensional settings, 
cf.~\cite{Moreno:algebraicGiroux}.
\end{remark}

To complete the analogous computation in SFT, we must be a bit more careful
since SFT in principle counts \emph{all} holomorphic curves, not only those
which are somewhere injective.  To be fully correct, the computation of
SFT requires an abstract perturbation of the Cauchy-Riemann equation
to achieve transversality for all solutions,
e.g.~this can be done following the polyfold scheme under development by
Hofer-Wysocki-Zehnder, cf.~\cite{HWZ:polyfoldBook}.  We will not need to
know any details about this perturbation, but only the following general
principles:
\begin{itemize}
\item Any Fredholm regular holomorphic curve with index~$1$ gives rise
uniquely to a solution of the perturbed problem for sufficiently small
perturbations.
\item If solutions of the perturbed problem with given asymptotic
behavior exist for arbitrarily small perturbations, then as the perturbation
is switched off we find a subsequence convergent to a holomorphic building
with the same asymptotic behavior.
\end{itemize}
This understood, counting the solutions of the perturbed problem
requires a precise description of the corresponding
space of index~$1$ $J_\nu'$-holomorphic buildings.

\begin{prop}
\label{prop:buildings}
Suppose $u \in \overline{\mM}(J_\nu'; H; m_1,\ldots,m_r)$ has index~$1$
and only simply covered orbits at its positive ends, including at most
one such end asymptotic to $\gamma_\zeta$ and the others all asymptotic
to $\gamma_{z_i}$ for $i \in \{1,\ldots,r\}$.  Then $u$ has only one level
and no nodes, and is somewhere injective: in particular, it is one of
the curves $u_\ell$ or $v_y^\pm$ that were counted in the proof of
Theorem~\ref{thm:ECHvanish}.
\end{prop}
\begin{proof}
We observe first that~$u$ must have at least one connected component
that is not a cover of a trivial cylinder: were it otherwise,
then since every positive asymptotic orbit is simply covered and at most
one of these is at an index~$1$ critical point, every component would
be either a trivial cylinder or a branched cover of $\RR\times \gamma_z$
for $\Morse(z)=2$.  Since the relevant covers of $\gamma_z$ all have
odd Conley-Zehnder index by Lemma~\ref{lemma:periodsAndIndices}, this
would imply that $\ind(u)$ is even and thus gives a contradiction.

Next, observe that every nonconstant component of the top level belongs to
the moduli space $\mM(J_\nu' ; H ; m_1,\ldots,m_r)$ and is thus a curve in one of the
perturbed moduli spaces arising from the foliation via
Proposition~\ref{prop:IFT}.  By induction, it follows that the nonconstant
components of  all other levels are also covers of such curves, so by
Lemma~\ref{lemma:wS}, they all have nonnegative index.  Since
$\ind(u) = 1$, Proposition~\ref{prop:indexNodes4} now implies that
$u$ cannot have any nodes and therefore (by stability) also has no
constant components.

By assumption, the total multiplicities of the positive ends of $u$ in each 
spinal component are bounded above by those of the holomorphic pages, thus
at most one component of $u$
can be a (perturbed) holomorphic page, and multiple covers of such curves
cannot appear.  If $u$ does have a component that is a page, that component
must be $u_\ell$, it must occupy the bottommost level,
and all other components then must have index~$0$, implying
via Lemma~\ref{lemma:wS} that all other nontrivial components are branched
covers of trivial cylinders with one positive end.  A nontrivial 
cover of this type
cannot appear in the top level since the positive asymptotic orbits are simply
covered; by induction, it follows that such covers cannot appear anywhere,
and we are left with $u = u_\ell$.

If no component of $u$ is a perturbed page, then exactly one component is
a cover of a (perturbed) gradient flow cylinder, which by
Lemma~\ref{lemma:wS} is then the unique component with index~$1$, while all
others have index~$0$.  Now the same argument again rules out any
nontrivial index~$0$ components since the positive asymptotic orbits are simply
covered, and implies at the same time that the index~$1$ component is
somewhere injective, hence $u = v_y^\pm$.
\end{proof}

We briefly recall from \cite{LatschevWendl} the necessary notation for the
version of the SFT 
chain complex that is involved in the definition of algebraic torsion.
A closed Reeb orbit for $\Lambda'_\nu$ is called \defin{good} if it is not
a double cover of an orbit whose odd/even parity (defined in terms of the
Conley-Zehnder index) is different from its own.
Let $\aA$ denote the $\ZZ_2$-graded supercommutative algebra with unit
over the group ring $\RR[H_2(M') / \ker\Omega]$, generated by all formal
variables of the form $q_\gamma$ where $\gamma$ is a good orbit.  
Writing $\dim M' = 2n-1$, the
degree $|q_\gamma| \in \ZZ_2$ of a generator is defined in general to be 
$n - 3 + \muCZ(\gamma)$ (mod~$2$), thus in the present case, $n=2$ and
the odd/even degree of $q_\gamma$ is opposite the parity of $\gamma$ as defined by
the Conley-Zehnder index.  The actual chain complex is the algebra of
formal power series $\aA[[\hbar]]$, where $\hbar$ is a formal variable defined
to have even degree.  Counting index~$1$ solutions to a
small abstract perturbation of the $J_\nu'$-holomorphic curve equation
in the symplectization then gives
rise to a differential operator $D_{\SFT} : \aA[[\hbar]] \to \aA[[\hbar]]$
satisfying $(D_{\SFT})^2 = 0$,
and the contact manifold $(M,\ker \Lambda'_\nu)$ is said to have
$\Omega$-twisted algebraic $k$-torsion for an integer $k \ge 0$ if and only if 
$\hbar^k$ is an exact element in the chain complex
$(\aA[[\hbar]],D_{\SFT})$.

\begin{proof}[Proof of Theorem~\ref{thm:algTorsion}]
We compute the operation of $D_{\SFT}$ on the element
$$
Q := q_{\gamma_{z_1}}^{m_1}\ldots q_{\gamma_{z_{r-1}}}^{m_{r-1}}
q_{\gamma_{z_r}}^{m_r - 1} q_\zeta \in \aA
$$
in the chain complex $(\aA[[\hbar]],D_{\SFT})$ outlined above.
By Proposition~\ref{prop:buildings},
all relevant index~$1$ solutions of the perturbed equation
can be identified with the Fredholm regular $J_\nu'$-holomorphic
curves $u_\ell$, $v_y^\pm$ that were counted in the proof of
Theorem~\ref{thm:ECHvanish}.  Once again the coherent orientations
give $v_y^+$ and $v_y^-$ opposite signs due to Prop.~\ref{prop:orientIndex1}
or Remark~\ref{remark:hintforAgustin},
so these cancel, and what's left is a single curve~$u_\ell$ with genus zero,
$k+1$ positive punctures (one for each generator in~$Q$), and no negative
punctures.   This is exactly the same situation that arose in the more
specialized computations of \cites{LatschevWendl,Wendl:openbook2}, and for the
same reasons, it gives $D_\SFT Q = \hbar^k$.
\end{proof}

\begin{proof}[Proof of Theorem~\ref{thm:Umap}]
Since the theorem is trivial whenever the ECH contact invariant
vanishes, the case with planar torsion is implied by
Theorem~\ref{thm:ECHvanish}.  Assume therefore that $M \subset M'$
is not a planar torsion domain: in this case $M = M'$, there is no
boundary, and all pages are planar and diffeomorphic to each other.  
Given $d \in \NN$, we can without loss of generality arrange $T_0 > 0$
in Lemma~\ref{lemma:periodsAndIndices} sufficiently small so that
$$
(k + 1) d < T_1 / T_0.
$$
Now by Proposition~\ref{prop:stableFEF}
(see also Remark~\ref{remark:nondegenerate}),
we can pick a nondegenerate contact form $\Lambda_\nu'$ and generic
$J_\nu' \in \jJ(\Lambda_\nu')$ so that for a generic point
$x \in M \setminus (\CritMorse(H) \times S^1)$, $(0,x) \in \RR\times M$ is
in the image of a unique index~$2$ $J_\nu'$-holomorphic curve
$$
u_x \in \mM_2^\fF(J_\nu'),
$$
and by Prop.~\ref{prop:uniqueness} $u_x$ is the only such curve in
$\mM^*(J_\nu' ; H,T_1)$.  We consider for $n=1,\ldots,d$ the generator
$$
\boldsymbol{\gamma}_n = \{(\gamma_{z_1},n m_1),\ldots,(\gamma_{z_r},n m_r)\}
$$
in the ECH chain complex for $\Lambda_\nu'$, $J_\nu'$ with coefficients
in $\ZZ[H_2(M) / \ker\Omega]$.  Then $\p_\ECH\boldsymbol{\gamma}_n$
counts only the pairs of cylinders $v_y^+$ and $v_y^-$ (combined with
trivial cylinders) which cancel each other out due to Prop.~\ref{prop:orientIndex1}
or Remark~\ref{remark:hintforAgustin}, thus
$$
\p_\ECH\boldsymbol{\gamma}_n = 0,
$$
so $\boldsymbol{\gamma}_n$ represents a homology class in ECH.  Defining
the $U$-map by counting admissible index~$2$ curves through $(0,x)$, the
action of~$U$ on $\boldsymbol{\gamma}_n$ then counts unions of 
trivial cylinders with the curve~$u_x$ and nothing else, hence
$$
U\boldsymbol{\gamma}_n = \boldsymbol{\gamma}_{n-1},
$$
implying $U^d \boldsymbol{\gamma}_d = \boldsymbol{\emptyset}$.
Since one can choose the data to make this true for arbitrarily large~$d$,
the result follows.
\end{proof}

\section{Spinal open books $\Rightarrow$ Lefschetz fibrations}
\label{sec:impressivePart}

In this section we complete the proofs of
Theorems~\ref{thm:classification}, \ref{thm:weak} and~\ref{thm:SteinDeformation}.  
By the
non-fillability results proved in \cite{LisiVanhornWendl1} via spine removal,
we can restrict
our attention to partially planar spinal open books that do not have planar 
torsion, i.e.~from now on, $M = M'$ has no boundary and $\boldsymbol{\pi}$
is symmetric.  The main idea in the proofs will be to attach to a given
filling $(W,\omega)$ the special
cylindrical end constructed in \S\ref{sec:model}, which contains a
pseudoholomorphic foliation, and then push this foliation into the filling~$W$.
The goal will be to obtain a Lefschetz fibration whose fibres are
the leaves of the foliation and whose base is the moduli space itself.
By looking at intersections of the leaves with each holomorphic vertebra,
we will then show that the moduli space defines a branched cover of each 
vertebra, which will necessarily be unbranched if the spinal open book
is Lefschetz-amenable (see Definition~\ref{defn:simple}).
Then, in order to complete the proof of Theorem~\ref{thm:classification}, it 
will be necessary to understand how the moduli space deforms under a 
generic homotopy of almost complex structures associated to a homotopy
of the symplectic data on~$W$.

The argument is similar to the one in \cite{Wendl:fillable}, and should be
thought of as a punctured version of McDuff's classification of ruled
symplectic manifolds \cite{McDuff:rationalRuled}.  There are two new ingredients in
the spinal setting, however. The first and main new ingredient is that
the moduli space of index~$2$ curves coming from the planar pages of the
open book has codimension~$1$ boundary in addition to codimension~$2$
nodal curves.  As in \cite{Wendl:fillable}, the codimension~$2$ nodal curves
correspond to Lefschetz critical fibers (a proof of this fact is sketched
in the appendix of \cite{Wendl:rationalRuled}).
The codimension~$1$ boundary consists of index~$1$ 
buildings in the filling attached to holomorphic gradient flow
cylinders in $\R \times M$; this phenomenon arises due to the presence
of index~$1$ critical points on vertebrae, thus it can be avoided in the
setting of ordinary open books (where all vertebrae are disks) but not in
the general case.  The key observation however is that
these buildings come in canceling pairs,
since the same can be assumed to be true for the gradient flow cylinders.
The base of the Lefschetz fibration will thus be a quotient moduli space, obtained
by ``sewing together'' the moduli space of index~$2$ curves along canceling
boundary components.  (We note that the actual situation is slightly 
more delicate since there may also be corner points to the moduli space,
at which two boundary strata intersect.)

The second new ingredient compared with \cite{Wendl:fillable} is that in 
the spinal open book setting, it
makes sense to consider \emph{weak} fillings that are exact only on the 
spine.  For a general weak filling, it is not
possible to attach a symplectization end with a holomorphic foliation,
but non-exactness away from the spine was already incorporated into the
stable Hamiltonian model constructed in \S\ref{sec:model}; we will take 
advantage of this by working directly with stable Hamiltonian data 
instead of contact data at infinity.  In the more specialized setting of
blown up summed open books, weak fillings were handled via a different
and less powerful approach in \cite{NiederkruegerWendl}.

\subsection{The completed filling and the moduli space}
\label{sec:moduliSpace}

Our standing assumptions will be as follows.
Assume $(M,\xi)$ is a \emph{closed} contact $3$-manifold with a supporting 
symmetric spinal open book
$$
\boldsymbol{\pi} := \Big(\pi\spine : M\spine \to \Sigma,
\pi\paper : M\paper \to S^1 \Big)
$$
whose pages have genus zero.  As in \S\ref{sec:invariants}, denote the
connected components of $\Sigma$ by 
$$
\Sigma_1,\ldots,\Sigma_r \subset \Sigma,
$$
and let
$$
m_i \in \NN, \qquad i=1,\ldots,r
$$
denote the number of boundary components
that pages have in the component $\Sigma_i \times S^1 \subset M\spine$;
note that this definition does not depend on the choice of a page since
$\boldsymbol{\pi}$ is symmetric.  Assume $\Omega$ is a closed $2$-form
on $M$ such that $\Omega|_{\xi} > 0$ and $\Omega|_{M\spine}$ is exact, and
$(W,\omega)$ is a compact symplectic manifold with boundary $\p W = M$
such that $\omega|_{TM} = \Omega$.  For the strong filling case
of Theorem~\ref{thm:classification}, we will sometimes also
require $\omega = d\lambda$ near $\p W$ for some $1$-form $\lambda$
such that 
$$
\alpha := \lambda|_{TM}
$$
is a contact form for~$\xi$.  The dual Liouville vector field in this case
will be denoted by $V_\lambda$, where by definition
$$
\omega(V_\lambda,\cdot) = \lambda.
$$
For the Liouville case, $\lambda$ will be
assumed to extend to a global primitive of $\omega$ on~$W$, and for the
almost Stein case, $\lambda$ will also have the form $-df \circ J$ for some
smooth function $f : W \to \RR$ and $\omega$-tame almost complex structure~$J$.

In \S\ref{sec:model}, we constructed a noncompact symplectic
model $(\widehat{E},\omega_E)$ containing a weakly contact hypersurface
$$
(M^-,\xi_-) \subset (\widehat{E},\omega_E)
$$
that is contactomorphic to $(M,\xi)$; let us fix such a contactomorphism
and identify $M = M^-$ henceforward.  The symplectic structure takes the form
$$
\omega_E = \frac{1}{KC} \left( C\, d\lambda_K + \eta \right),
$$
where $\eta$ is a closed $2$-form on $M\paper \subset M$ 
with $[\eta] = [\Omega] \in H^2_\dR(M)$,
$\lambda_K$ is a Liouville form whose restriction to $M^-$ is a
contact form for $\xi_-$, and $C > 0$ and $K > 0$ are large
constants.  Let
$$
\widehat{\nN}_-(\p E) \subset \widehat{E}
$$
denote the unbounded region in $\widehat{E}$ with 
$\p \widehat{\nN}_-(\p E) = - M^-$.

In general we only care about the deformation class of the symplectic 
data on~$W$, thus we are free to make modifications in a collar neighborhood
of~$\p W$ and then rescale globally so as to produce any desired
contact form~$\alpha$ at the boundary.  
By \cite{MassotNiederkruegerWendl}*{Lemma~2.10}, we can deform
$\omega$ near $\p W$ and subsequently rescale so
that without loss of generality,
$$
\Omega = \omega_E|_{TM^-}
$$
under the chosen contactomorphism identifying $M$ with $M^-$.  We then
define a completion of $(W,\omega)$ by
$$
(\W,\widehat{\omega}) := (W,\omega) \cup_{M = M^-} 
\left( \widehat{\nN}_-(\p E) , \omega_E \right),
$$
where a standard application of the Moser deformation trick 
(see for example \cite{NiederkruegerWendl}*{Lemma~2.3}) produces collars
near $\p W$ and $\p \widehat{\nN}_-(\p E)$ that permit a smooth symplectic
gluing of the two pieces.    The gluing is simpler to describe 
if $(W,\omega)$ is a strong or exact filling, as we can then use collar
neighborhoods constructed by flowing along Liouville vector fields.
In these cases we can assume $\eta \equiv 0$ so that $\omega_E$ is the
exterior derivative of the Liouville form $\frac{1}{K} \lambda_K$,
and $\lambda$ can then (after a global rescaling) 
be deformed near $\p W$ so that it glues together
smoothly with $\frac{1}{K} \lambda_K$.  Denote the resulting Liouville
form on a neighborhood of $\widehat{\nN}_-(\p E) \subset \W$ by
$\widehat{\lambda}$, so we have
$$
\widehat{\omega} = d\widehat{\lambda}
$$
on this neighborhood if $(W,\omega)$ is a strong filling, and the same
holds globally on $\W$ if the filling is exact.

We must do something slightly different in the almost Stein case:
recall from \S\ref{sec:Jandf} that $(\widehat{E},\omega_E)$ comes equipped 
with a compatible almost complex structure $J_+$ and a $J_+$-convex function
$f_+ : \widehat{E} \to \RR$ such that the induced Liouville form
$\lambda_+ = - df_+ \circ J_+$ matches $\frac{1}{K} \lambda_K$ on the
region $\widehat{\nN}(\p_h E) \subset \widehat{E}$ but not everywhere else.
We shall therefore forget temporarily about $\omega_E$ and glue the almost
Stein manifolds $(W,J,f)$ and $(\widehat{\nN}_-(\p E),J_+,f_+)$ together
along $M = M^-$.  To enable this, one can first deform the Weinstein structures
$(\omega,V_\lambda,f)$ near $\p W$ and rescale $\lambda$ and $f$ globally so 
that these data glue together smoothly with $\lambda_+$ and $f_+$;
this can be done without introducing any critical points of $f$ in the
collar, thus one can then apply Lemma~\ref{lemma:almostStein2} to produce a
deformed $d\lambda$-tame almost complex structure $J$ with
$\lambda = -df \circ J$ such that $J$ glues together smoothly with~$J_+$.
The result is an almost Stein completion
$$
(\W,\widehat{J},\widehat{f}) = (W,J,f) \cup_{M = M^-} \left( \widehat{\nN}_-(\p E),
J_+,f_+ \right)
$$
such that $\widehat{\lambda} := - d\widehat{f} \circ \widehat{J}$ matches the modified
Liouville form $\lambda_+$ in the cylindrical end.  
By gluing $\lambda$ together with the interpolated Liouville form
$\Theta$ provided by Lemma~\ref{lemma:LiouvilleInterp}, we also obtain a
Liouville form $\widehat{\Theta}$ on $\W$ with
\begin{itemize}
\item $\widehat{\Theta} = \widehat{\lambda}$ on~$W$,
\item $\widehat{\Theta} = \frac{1}{K} \lambda_K$ near infinity, and
\item $d\widehat{\Theta}$ tames~$\widehat{J}$ everywhere.
\end{itemize}
We will use $\widehat{\Theta}$ below to define energy for
$\widehat{J}$-holomorphic curves in~$\W$.

Recall now that the end we just attached to form the completion contains a 
region
$$
\widehat{\nN}_+(\p E) \subset \widehat{\nN}_-(\p E)
$$
that is identified with the half-symplectization $[0,\infty) \times M^+$
of a certain stable hypersurface $M^+ = - \p \widehat{\nN}_+(\p E)
\subset \widehat{E}$, carrying a stable Hamiltonian structure
$\hH_+ = (\Omega_+,\Lambda_+)$.  In \S\ref{sec:Jandf} and \S\ref{sec:holPages},
we constructed the compatible almost complex structure $J_+$ such that its
restriction to $\widehat{\nN}_+(\p E)$ is in $\jJ(\hH_+)$, and
$(\widehat{\nN}_-(\p E),J_+)$ contains \emph{holomorphic vertebrae} and
\emph{holomorphic pages}.  Indeed, let us select a holomorphic vertebra
from Proposition~\ref{prop:vertebrae} corresponding to each component
$\Sigma_i \subset \Sigma$, and denote it by
$$
\dot{\Sigma}_i \subset \widehat{\nN}_-(\p E) \subset \W, \qquad i=1,\ldots,r.
$$
Meanwhile, the holomorphic pages form a foliation~$\fF_+$ on $\widehat{E}$ 
whose restriction to $\widehat{\nN}_+(\p E) = [0,\infty) \times M^+$
has the same form as the foliation $\fF$ that we considered 
in \S\ref{sec:holOpenBook}, thus we are free to use the analytical results
of that section, including the index and intersection-theoretic computations.
We are also free to impose 
Assumptions~\ref{ass:partiallyPlanar}---this mostly follows already from 
the premise that $\boldsymbol{\pi}$ is a partially planar domain without 
planar torsion, but it includes also the following conditions on the
Hamiltonian function $H : \Sigma \to [0,\infty)$ and complex structure $j$
on $\Sigma$ that play key roles in the construction of~$\fF_+$:
\begin{itemize}
\item $H : \Sigma \to [0,\infty)$ has no index~$0$ Morse critical points,
and it has exactly one index~$2$ critical point on every connected component 
of~$\Sigma$;
\item $(H,j)$ are in general position (see Definition~\ref{defn:GP}).
\end{itemize}
For technical reasons, it will be convenient (though not essential) to add
one more assumption:\footnote{The purpose of this extra assumption is to
simplify Lemma~\ref{lemma:genAsymp}, which implies a special case of the
unpublished folk theorem that asymptotic contributions to Siefring's
intersection numbers (see \S\ref{sec:Siefring}) are a non-generic phenomenon
for somewhere injective curves.}
\begin{itemize}
\item The Hessian $\nabla^2 H(z) : T_z\Sigma \to T_z\Sigma$ commutes with
$j$ at every $z \in \CritMorse(H)$ with $\Morse(z)=2$.
\end{itemize}
The assumptions on $H$ imply that each
index~$1$ critical point of $H$ is connected to the unique index~$2$
critical point in the same connected component of $\Sigma$ by exactly
two gradient flow lines.  The resulting gradient flow cylinders receive
opposite orientations by Prop.~\ref{prop:orientIndex1}
or Remark~\ref{remark:hintforAgustin}.  It will be useful to note that
the unbranched multiple covers of these cylinders also satisfy the
automatic transversality criterion of Prop.~\ref{prop:automatic} and
have the same properties with regard to orientations, hence:
\begin{lemma}
\label{lemma:cancel}
Every holomorphic gradient flow cylinder $u$ in $\fF_+$ is a Fredholm
regular index~$1$ curve, and so are its unbranched $k$-fold covers for
every $k \in \NN$.  Moreover, $\fF_+$ contains exactly two holomorphic
gradient flow cylinders asymptotic to the same pair of orbits, and for
any choice of coherent orientations from \cite{BourgeoisMohnke} and
for each $k \in \NN$, the unbranched $k$-fold covers of these two
gradient flow cylinders are oppositely oriented.  \qed
\end{lemma}
We will refer to the pairs of gradient flow cylinders described
in this lemma as \defin{canceling pairs}.

Assumptions~\ref{ass:partiallyPlanar} also presume that
Lemma~\ref{lemma:periodsAndIndices} is applicable, imposing dynamical
conditions on the stable Hamiltonian structure~$\hH_+$: in particular, this
provides constants
$T_0, T_1 > 0$ such that all Reeb orbits of period less than $T_1$ are
covers of
$\{z\} \times S^1 \subset \widecheck{M}^+\spine$,
and each simply covered orbit of this form has period less than~$T_0$,
where $T_1 / T_0$ may be assumed arbitrarily large.  More specifically,
$T_1 / T_0$ is assumed to be larger than the number of boundary
components of any page.  Since $\hH_+$ is of confoliation type,
Proposition~\ref{prop:posPuncture} then implies that all breaking orbits
appearing in the holomorphic buildings discussed below will be covers
of $\{z\} \times S^1$ for various $z \in \CritMorse(H)$.  These orbits are
\defin{elliptic} if $\Morse(z)=2$ and \defin{hyperbolic} if $\Morse(z)=1$,
so we will refer to them as such.

In the almost Stein case, we have already extended $J_+$ to a
$d\widehat{\Theta}$-tame almost complex structure $\widehat{J}$ on $\W$,
and we shall allow a generic $d\widehat{\Theta}$-tame perturbation of 
$\widehat{J}$ in the interior of~$W$; note that this perturbs the Liouville form
$\widehat{\lambda} = - d\widehat{f} \circ \widehat{J}$, but such a change is harmless
since the Liouville condition is open.
In the weak, strong and exact cases, we simply extend $J_+$ arbitrarily to an
$\widehat{\omega}$-tame almost complex structure $\widehat{J}$ on $\W$
which is generic in the interior of~$W$.  In particular, $\widehat{J}$ satisfies
$$
\widehat{J} \equiv J_+ \quad \text{ in $\widehat{\nN}_-(\p E)$},
$$
and this gives $(\W,\widehat{J})$ the structure of an almost complex manifold with
a cylindrical end $([0,\infty) \times M^+,J_+)$ compatible with the stable
Hamiltonian structure~$\hH_+ = (\Omega_+,\Lambda_+)$.  We 
define the \defin{energy} of a
punctured $\widehat{J}$-holomorphic curve $u : \dot{S} \to \W$ as
$$
E(u) := \sup_{\varphi \in \tT} \int_{\dot{S}} u^*\widehat{\omega}_\varphi,
$$
where for some $T > 0$ chosen large enough such that
$\widehat{\Theta} = \frac{1}{K}\lambda_K$ in $[T,\infty) \times M^+ \subset
\widehat{\nN}_+(\p E)$,
$$
\tT := \left\{ \varphi \in C^\infty([T,\infty) \to [T,T+1)) \ \Big|\ 
\text{$\varphi' > 0$ and $\varphi(r) = r$ for $r$ near~$T$} \right\},
$$
and the modified symplectic form $\widehat{\omega}_\varphi$ is defined such that
$$
\widehat{\omega}_\varphi = d \left( ( e^{\varphi(r)} - 1) \Lambda_+ \right) +
\Omega_+ \quad \text{ on $[T,\infty) \times M^+$},
$$
while on the rest of $\W$, $\widehat{\omega}_\varphi$ is defined to match
$d\widehat{\Theta}$ in the almost Stein case or $\widehat{\omega}$ in weak,
strong and Liouville cases.  The point of this definition is that
curves with $E(u) < \infty$ will now have asymptotically
cylindrical behavior and obey the compactness theory in \cite{SFTcompactness}.

We will also need to consider smooth deformations 
$$
\omega_\tau, \quad \lambda_\tau, \quad J_\tau, \quad f_\tau \qquad
0 \le \tau \le 1
$$
of the symplectic data on~$W$.  After suitable modifications near $\p W$
and global rescaling, we can fit this into the above picture by 
considering a smooth $1$-parameter family of completed fillings
$$
(\W,\widehat{\omega}_\tau) \text{ or } 
(\W,\widehat{J}_\tau,\widehat{f}_\tau), \qquad 0 \le \tau \le 1
$$
with a generic $1$-parameter family of tame almost 
complex structures $\widehat{J}_\tau$ such that all the data on 
$\widehat{\nN}_-(\p E) \subset \W$ is $\tau$-independent and matches
the construction above.  The definition of energy is then also
$\tau$-dependent, but this does not affect the existence of uniform
energy bounds since the data is $\tau$-independent in the cylindrical end.

With this setting in place, our moduli spaces of curves 
in $\W$ will now be defined in terms of the $J_+$-holomorphic 
foliation $\fF_+$ of $\RR \times M^+$.
Let $\ModPlus{}$ denote the moduli space of unparametrized
$J_+$-holomorphic curves in $\RR \times M^+$ modulo $\RR$-translation
that belong to the foliation~$\fF_+$, and let
$$
\ModPlus{i} \subset \ModPlus{} \quad\text{ for $i=1,2$}
$$
denote the components with virtual dimension~$i$.
We denote by $\cpctModPlus{}$ and $\cpctModPlus{i} \subset 
\cpctModPlus{}$ respectively the closures of these in the space of stable 
holomorphic buildings in $(\RR \times M^+,J_+)$, see~\S\ref{sec:energy}.  
The structure of the foliation $\fF_+$ implies that
$\ModPlus{1} = \cpctModPlus{1}$ is a finite set, consisting of all holomorphic
gradient flow cylinders and the exceptional holomorphic pages that have
one end asymptotic to a hyperbolic orbit and the rest asymptotic to
elliptic orbits.  Each connected 
component of $\cpctModPlus{2}$ is either a circle or a compact interval 
bounded by buildings with two levels whose unique nontrivial components are
each curves in~$\ModPlus{1}$.  Define an equivalence relation
by saying that for two buildings $u$ and $u'$, $u \sim u'$ if and only if
their positive asymptotic orbits coincide up to a permutation of the punctures
and their bottommost levels are identical; we shall write
$$ 
\quotModPlus := \cpctModPlus{2} / \sim.
$$
This quotient moduli space has the topology of a disjoint union of circles:
indeed, Lemma~\ref{lemma:cancel} implies that the equivalence relation 
identifies pairs of buildings in $\p\cpctModPlus{2}$
having the same index~$1$ holomorphic page in their lower levels and
a canceling pair of gradient flow cylinders in their upper levels.
Note that gradient flow cylinders never appear as bottom levels
of these buildings, thus the elements of $\quotModPlus$ have no
negative ends.

Let $\mM(\widehat{J})$ and $\overline{\mM}(\widehat{J})$ denote the spaces of unparametrized
finite-energy holomorphic curves or stable buildings respectively 
with arithmetic genus zero in $(\W,\widehat{J})$, 
and define a similar equivalence relation by
$$
\widehat{\mM}(\widehat{J}) := \overline{\mM}(\widehat{J}) / \sim,
$$
where $u$ and $u'$ are considered equivalent if and only if their 
asymptotic orbits coincide up to permutation and their bottommost
nonempty levels are identical.  Since the main level is allowed to be empty
in general, it may happen that the bottommost nonempty level of
$u \in \overline{\mM}(\widehat{J})$ is an upper level, and viewing buildings in
$(\RR \times M^+,J_+)$ without negative ends as buildings in $(\W,\widehat{J})$ with empty main levels gives
rise to a natural inclusion
\begin{equation}
\label{eqn:inclusion}
\quotModPlus \subset \widehat{\mM}(\widehat{J}).
\end{equation}
With this inclusion in mind, we define
$$
\quotMod{}{} \subset \widehat{\mM}(\widehat{J})
$$
to be the smallest open and closed subset that contains~$\quotModPlus$.
Observe that since the holomorphic pages in $\ModPlus{}$ have only
positive ends, they can also be regarded as $\widehat{J}$-holomorphic
curves in $[0,\infty) \times M^+ = \widehat{\nN}_+(\p E) \subset \W$, 
so that each $u \in \quotModPlus$ gives
rise to a $1$-parameter family of elements in $\quotMod{}{}$ that converge
in the SFT-topology to~$u$ as their main levels
are pushed to infinity.  All these families of holomorphic pages therefore
belong to~$\quotMod{}{}$, and we will see in Proposition~\ref{T:moduliSurface} 
below that they form collar neighborhoods of the boundary of~$\quotMod{}{}$.

Define subsets
$$
\quotMod{\reg}{},\ \quotMod{\sing}{},\ \quotMod{\exot}{} \subset \quotMod{}{},
$$
where:
\begin{itemize}
\item $u \in \quotMod{\reg}{}$ if its main level is a smooth embedded 
$\widehat{J}$-holomorphic curve with one connected component and
only simply covered asymptotic orbits;
\item $u \in \quotMod{\sing}{}$ if its main level is a nodal $\widehat{J}$-holomorphic
curve with two embedded connected components that intersect each other 
transversely at a single node and nowhere else, and both have only simply 
covered asymptotic orbits;
\item $u \in \quotMod{\exot}{}$ if its main level is a smooth embedded
$\widehat{J}$-holomorphic curve with one connected component such that one of its
asymptotic orbits is doubly covered, and the rest are simply covered.
\end{itemize}
As the notation should suggest, elements of $\quotMod{\reg}{}$ and $\quotMod{\sing}{}$
will give rise to the regular and singular fibers respectively of a
Lefschetz fibration on~$W$ when the spinal open book is {Lefschetz-amenable}.  Elements of 
$\quotMod{\exot}{}$ are a slightly different kind of object that we will refer
to as \defin{exotic fibers}: we will see that they can occur only in the 
{non-amenable} case, thus producing a topological decomposition of $W$ that
is more general than a Lefschetz fibration.

\begin{defn}
\label{defn:genericBC}
For $i=1,2$, assume $\Sigma_i$ are closed oriented surfaces and $\dot{\Sigma}_i \subset \Sigma_i$
are obtained by deleting finitely many points.  A continuous map
$\pi : \dot{\Sigma}_1 \to \dot{\Sigma}_2$ will be called a \defin{branched cover (with degree~$d \in \NN$) of surfaces
with cylindrical ends} if it is proper and its unique extension to a map
$\Sigma_1 \to \Sigma_2$ is a branched cover (with degree~$d$).  
We will say that a branched cover $\pi : \dot{\Sigma}_1 \to \dot{\Sigma}_2$ is
\defin{generic} if its branch points are all simple (i.e.~they have branching order~$2$)
and all have distinct images.\footnote{Note that the generic conditions imposed on
branch points do not apply to all branch points of the \emph{extended} map
$\Sigma_1 \to \Sigma_2$, which can include puntures.}
\end{defn}

\begin{prop} \label{T:moduliSurface}
For generic choices of $\widehat{J}$ on $\W$ satisfying the conditions specified above,
$\quotMod{}{}$ decomposes into disjoint subsets
$$
\quotMod{}{} = \quotMod{\reg}{} \cup \quotMod{\sing}{} \cup \quotMod{\exot}{}
\cup \quotModPlus,
$$
where $\quotMod{\reg}{}$ is an open subset, $\quotMod{\sing}{}$ and 
$\quotMod{\exot}{}$ are each finite, and $\quotMod{}{}$ has the
topology of a compact, connected and oriented surface with boundary
$$
\p\quotMod{}{} = \quotModPlus.
$$
Moreover, every point in $\W$ is in the image of the main level for a unique
curve in the interior of $\quotMod{}{}$, and this interior admits a 
smooth structure such that the resulting continuous surjection
$$
\Pi : \W \to \quotMod{}{} \setminus \quotModPlus : x \mapsto 
\text{the curve through~$x$}
$$
is smooth outside the finitely many nodes of the curves
in~$\quotMod{\sing}{}$.  For each $i=1,\ldots,r$, the holomorphic vertebra
$\dot{\Sigma}_i \subset \W$ is disjoint from the nodes of curves
in $\quotMod{\sing}{}$, and the map $\Pi$ restricts to $\dot{\Sigma}_i$
as a generic branched cover with degree~$m_i$ of surfaces with cylindrical ends,
whose branch points all have image in~$\quotMod{\reg}{}$. 
Finally, $\quotMod{\exot}{}$ is empty if and only if the branched covers
$\Pi|_{\dot{\Sigma}_i}$ have no branch points for every $i=1,\ldots,r$.
\end{prop}

This will be enough to conclude the first part of 
Theorem~\ref{thm:classification}, that a planar spinal open book must be 
uniform if $M$ is fillable, and we will explain in \S\ref{sec:Lefschetz}
how to turn $\Pi : \W \to \quotMod{}{} \setminus \p\quotMod{}{}$ into a
bordered Lefschetz fibration on $W$ whenever there are no exotic fibers.

We now state a corresponding result for $1$-parameter deformations of the
data.  Consider a $1$-parameter family of almost complex structures
$\{ \widehat{J}_\tau \}_{\tau \in [0,1]}$ on $\W$ such that
\begin{itemize}
\item $\widehat{J}_\tau|_{\widehat{\nN}_-(\p E)} = J_+$ for all~$\tau$;
\item $\widehat{J}_\tau$ is $\widehat{\omega}_\tau$-tame 
(or in the almost Stein case $d\widehat{\Theta}_\tau$-tame) for all~$\tau$;
\item $\widehat{J}_0 = \widehat{J}$; 
\item $\widehat{J}_1$ and the homotopy $\{\widehat{J}_\tau\}$ are both generic on the
interior of~$W$.
\end{itemize}
Let $\mM(\{\widehat{J}_\tau\})$, $\overline{\mM}(\{\widehat{J}_\tau\})$ and
$\widehat{\mM}(\{\widehat{J}_\tau\})$ denote the spaces of pairs 
$(u,\tau)$ with $\tau \in [0,1]$ and $u \in \mM(\widehat{J}_\tau)$,
$u \in \overline{\mM}(\widehat{J}_\tau)$ or $u \in \widehat{\mM}(\widehat{J}_\tau)$
respectively.  Since $\widehat{J}_\tau$ is independent of $\tau$ near infinity, the 
inclusion \eqref{eqn:inclusion} generalizes to this parametrized setting as
$$
\quotModPlus \times [0,1] \subset \widehat{\mM}(\{\widehat{J}_\tau\}),
$$
and we define
$$
\quotModHomotop{} \subset \widehat{\mM}(\{\widehat{J}_\tau\})
$$
as the smallest open and closed subset containing $\quotModPlus \times [0,1]$,
along with subsets
$$
\quotModHomotop{\reg}, \ \quotModHomotop{\sing}, \ 
\quotModHomotop{\exot} \subset \quotModHomotop{}
$$
defined by the same criteria as before.  For $\tau \in [0,1]$, let
$$
\quotMod{}{\tau} := \left\{ u \in \widehat{\mM}(\widehat{J}_\tau)\ \Big|\ 
(u,\tau) \in \quotModHomotop{} \right\},
$$
with corresponding subsets $\quotMod{\reg}{\tau}$, $\quotMod{\sing}{\tau}$
and $\quotMod{\exot}{\tau}$.  We will sometimes identify $\quotMod{}{}$
and $\quotMod{}{\tau}$ with the corresponding subsets of~$\quotModHomotop{}$.

\begin{prop} 
\label{prop:deform}
For generic families $\{\widehat{J}_\tau\}_{\tau \in [0,1]}$ satisfying the conditions
specified above, there exists a homeomorphism
$$
\Psi : \quotMod{}{} \times [0,1] \to \quotModHomotop{}
$$
satisfying
\begin{equation*}
\begin{split}
\Psi\big(\quotMod{\reg}{} \times [0,1]\big) &= \quotModHomotop{\reg},\\
\Psi\big(\quotMod{\sing}{} \times [0,1]\big) &= \quotModHomotop{\sing},\\
\Psi\big(\quotMod{\exot}{} \times [0,1]\big) &= \quotModHomotop{\exot}, \\
\Psi\big(\quotMod{}{} \times \{\tau\}\big) &= \quotMod{}{\tau} \text{ for every
$\tau \in [0,1]$},
\end{split}
\end{equation*}
and $\Psi$ restricts to the identity map on 
the subsets $\quotModPlus \times [0,1]$ and $\quotMod{}{} \times \{0\}$ 
in $\quotModHomotop{}$.  Moreover, outside of $\quotModPlus \times [0,1]$,
one can define a natural smooth structure on $\quotModHomotop{}$ 
for which $\Psi$ is smooth outside of possibly finitely many
points in $\quotMod{\sing}{} \times (0,1)$ and $\quotMod{\exot}{} \times (0,1)$,
and the conclusions of Proposition~\ref{T:moduliSurface} hold for
$\quotMod{}{\tau}$ for every $\tau \in [0,1]$, giving a continuous surjection
\begin{equation*}
\begin{split}
\Pi : \W \times [0,1] &\to \quotModHomotop{} \setminus
\left( \quotModPlus \times [0,1] \right) \\
(x,\tau) &\mapsto (u,\tau) \text{ where $x \in \im(u)$}
\end{split}
\end{equation*}
which is smooth outside of the finite collection of continuous and
piecewise smooth paths in $\W \times [0,1]$ traced out
by the nodes of curves in $\quotModHomotop{\sing}$.  The restriction
of $\Pi$ to $\dot{\Sigma}_i \times [0,1]$ for $i=1,\ldots,r$ defines
a smooth deformation of generic branched covers of surfaces with cylindrical ends.
\end{prop}

\begin{remark}
The caveat about the smoothness of $\Psi$ in Proposition~\ref{prop:deform} has 
to do with isolated breaking configurations in $\quotModHomotop{\sing}$ and
$\quotModHomotop{\exot}$ that will be dealt with in Lemma~\ref{lemma:singLeaves}.
Here there are always two ways to glue such configurations and thus move the
parameter $\tau$ forward or backward, producing a $1$-parameter family that is
manifestly continuous, but we have chosen not to worry about whether it is
smooth.  This is in any case immaterial in our main applications, for
which the moduli space does not need to have a canonical smooth structure
as long as the foliation it produces on $\W$ is smooth.
(In places where the latter is in doubt, i.e.~at nodal points, the foliation
can always be smoothed by hand with a small perturbation.)
\end{remark}

\subsection{Generic conditions}
\label{sec:genericity}

Before stating the main compactness results, let us clarify
the role that our genericity conditions on $\widehat{J}$ and $\{\widehat{J}_\tau\}$ are
going to play.  As usual such assumptions guarantee that 
moduli spaces of somewhere injective curves are smooth and have
dimension equal to the index, with the consequence that this index is
bounded from below.  In addition to this, we will need to use genericity on 
occasion to limit non-transverse intersections and asymptotic intersection 
contributions in the sense of 
Siefring.  The results of this subsection should be understood to be true after
choosing $\widehat{J}$ and $\{\widehat{J}_\tau\}$ from comeager subsets of the sets of all 
almost complex structures or smooth homotopies thereof
with the properties specified in \S\ref{sec:moduliSpace}.

\begin{lemma}
\label{lemma:genericity}
For every $\tau \in [0,1]$, all somewhere injective $\widehat{J}_\tau$-holomorphic
curves $v$ in $\W$ that intersect the interior of~$W$ satisfy 
$\ind(v) \ge -1$, and there exists at most one such curve (up to
parametrization) with
$\ind(v) = -1$.  Moreover, for almost every $\tau \in [0,1]$, and in
particular for $\tau=0$ and $\tau=1$, all such curves satisfy $\ind(v) \ge 0$.
\end{lemma}
\begin{proof}
The almost complex stuctures $\widehat{J}_\tau$ are fixed on 
$\widehat{\nN}_-(\p E) \subset \W$
but generic perturbations are allowed in the interior of~$W \subset \W$,
thus the inequalities $\ind(v) \ge -1$ and $\ind(v) \ge 0$ follow from
standard transversality arguments as in \cite{McDuffSalamon:Jhol}.
The fact that no individual $\widehat{J}_\tau$ admits more than one simple
curve of index~$-1$ follows by showing that for generic families $\{\widehat{J}_\tau\}$,
the map
$$
\mM^*(\{\widehat{J}_\tau\}) \times \mM^*(\{\widehat{J}_\tau\}) \to [0,1] \times [0,1] :
\left( (u,\tau) , (u',\tau') \right) \mapsto (\tau,\tau')
$$
is transverse to the diagonal in $[0,1] \times [0,1]$ outside of the
diagonal in its domain.  Here $\mM^*(\{\widehat{J}_\tau\})$ denotes the space of all
pairs $(u,\tau)$ such that $\tau \in [0,1]$ and $u$ is an unparametrized
somewhere injective finite-energy $\widehat{J}_\tau$-holomorphic curve that intersects
the interior of~$W$.  This transversality result is probably also standard,
but since we do not know a good reference for the proof, here is a sketch.
One starts by defining a universal moduli space $\univ^*$ consisting of
tuples $(u,\tau,u',\tau',\{\widehat{J}_\tau\})$, where $\{\widehat{J}_\tau\}$ belongs to a suitable
Banach manifold $\jJ_{[0,1]}$ of homotopies of almost complex structures
(e.g.~of Floer $C_\epsilon$ class or in $C^k$ for some large 
$k \in \NN$), and $(u,\tau)$ and $(u',\tau')$ are two distinct elements
of $\mM^*(\{\widehat{J}_\tau\})$.
The fact that they are distinct implies in particular that whenever $\tau=\tau'$,
each of $u$ and $u'$ has an injective point where it does not intersect the 
other curve.  Standard arguments via elliptic regularity and
the implicit function theorem then show that $\univ^*$ is a 
differentiable Banach manifold and, moreover, that the map
$$
\univ^* \to [0,1] \times [0,1] : \left( u,\tau,u',\tau',\{\widehat{J}_\tau\}\right)
\mapsto (\tau,\tau')
$$
is a submersion.  It follows that the preimage of the diagonal under this
map is a submanifold $\univ^*_\Delta \subset \univ^*$, so applying the
Sard-Smale theorem to the natural projection $\univ^*_\Delta \to \jJ_{[0,1]}$
provides a comeager subset of $\jJ_{[0,1]}$ for which the desired
transversality result is satisfied.  In the final step, one can use the
``Taubes trick'' (cf.~\cite{McDuffSalamon:Jhol}*{\S 3.2}
or \cite{Wendl:lecturesV2}*{\S 4.4.2}) to replace $\jJ_{[0,1]}$ with a
suitable Fr\'echet manifold of \emph{smooth} homotopies $\{\widehat{J}_\tau\}$.
\end{proof}

Genericity also implies that the existence of non-transverse intersections of
somewhere injective $\widehat{J}$-holomorphic curves with fixed holomorphic 
hypersurfaces is a ``codimension two phenomenon''.  We will be interested
especially in controlling intersections with the union of the holomorphic
vertebrae
$$
\dot{\Sigma} := \dot{\Sigma}_1 \cup \ldots \cup \dot{\Sigma}_r.
$$
The next statement follows directly from the results of
\cite{CieliebakMohnke:transversality}*{\S 6}.

\begin{lemma}
\label{lemma:genVertebrae}
For $i \in \ZZ$, $\ell \in \NN$, $\mathbf{k} := (k_1,\ldots,k_\ell) \in \NN^\ell$ 
and $\tau \in [0,1]$, let 
$\mM^*_i(\widehat{J}_\tau ; \dot{\Sigma},\mathbf{k})$ denote the following moduli space of
constrained $\widehat{J}_\tau$-holomorphic curves with $\ell$ marked points: elements of
$\mM^*_i(\widehat{J}_\tau ; \dot{\Sigma},\mathbf{k})$ are represented by tuples
$$
(S,j,\Gamma,(\zeta_1,\ldots,\zeta_\ell),u)
$$
such that $u : (\dot{S} := S \setminus \Gamma,j) \to (\W,\widehat{J}_\tau)$ is a 
somewhere injective finite-energy $\widehat{J}_\tau$-holomorphic curve of index~$i$ 
intersecting the interior of~$W$, $\zeta_1,\ldots,\zeta_\ell \in \dot{S}$ 
are distinct points,
two tuples are equivalent if they are related by a biholomorphic map of their
domains preserving
the ordered sets of punctures $\Gamma$ and marked points $(\zeta_1,\ldots,\zeta_\ell)$,
and $u$ also satisfies the constraints
$$
u(\zeta_j) \in \dot{\Sigma}
$$
such that for each $j=1,\ldots,\ell$,
the local intersection index of $u$ with $\dot{\Sigma}$ at $\zeta_j$
is at least~$k_j$.  Then for almost every $\tau$, and in particular for
$\tau \in \{0,1\}$, $\mM^*_i(\widehat{J}_\tau;\dot{\Sigma},\mathbf{k})$ is a smooth 
manifold with
$$
\dim \mM^*_i(\widehat{J}_\tau;\dot{\Sigma},\mathbf{k}) = i - 2 \sum_{j=1}^\ell (k_j-1).
$$ 
Moreover, the space $\mM^*_i(\{\widehat{J}_\tau\};\dot{\Sigma},\mathbf{k})$ of pairs
$(u,\tau)$ such that $\tau \in [0,1]$ and
$u \in \mM^*_i(\widehat{J}_\tau;\dot{\Sigma},\mathbf{k})$ is a smooth manifold
of dimension $i + 1 - 2 \sum_j (k_j-1)$, so in particular, this
space is empty whenever $i + 1 - 2 \sum_j (k_j-1) < 0$.
\qed
\end{lemma}

Note that any curve in the ordinary moduli space without marked points
gives rise to an element of the space in the above lemma whenever it
intersects~$\dot{\Sigma}$: one can simply add marked points wherever these
intersections occur.  Adding a marked point $\zeta$ with the constraint
$u(\zeta) \in \dot{\Sigma}$ but without any constraint on the local intersection
index does not change the dimension of the moduli space.
Combining this observation with the usual results about 
generic transversality of
the evaluation map from \cite{McDuffSalamon:Jhol}, we obtain:

\begin{lemma}
\label{lemma:nodalBranching}
Suppose $\tau \in [0,1]$ and $u_0$ and $u_1$ are somewhere injective
$\widehat{J}_\tau$-holomorphic curves that both intersect the interior of $W$
such that for each $j=0,1$, we have $\ind(u_j) \in \{-1,0\}$, 
$u_j$ intersects $\dot{\Sigma}$ transversely, and all its asymptotic
Reeb orbits are disjoint from those of~$\dot{\Sigma}$.  Then the sets
$\im u_0 \cap \dot{\Sigma}$ and $\im u_1 \cap \dot{\Sigma}$ are
disjoint.   \qed
\end{lemma}

A similar phenomenon in Siefring's intersection theory guarantees that
generically, asymptotic contributions to the intersection counts
$u*v$ and $\delta(u) + \delta_\infty(u)$ are zero whenever $u$ and $v$ are 
somewhere injective curves of sufficiently low index.  Since no proof of this 
fact is available in the current literature, we shall only address the following
simpler special case which suffices for our purposes.  Fix a simply covered
elliptic orbit 
$\gamma : S^1 \to \Sigma \times S^1 \subset \widecheck{M}^+\spine : t \mapsto (z,t)$,
where $\Morse(z)=2$.  Using the formula \eqref{eqn:ReebSigma} for the Reeb
vector field and the natural trivialization $\gamma^*\Xi_+ = S^1 \times T_z\Sigma$,
the associated \defin{asymptotic operator} 
$\mathbf{A}_\gamma : \Gamma(\gamma^*\Xi_+) \to \Gamma(\gamma^*\Xi_+)$
(see e.g.~\cite{Wendl:automatic}*{\S 3.2}) is identified with
$$
C^\infty(S^1,T_z\Sigma) \to C^\infty(S^1,T_z\Sigma) : 
v \mapsto - j \dot{v} + \frac{\nondegParam}{K} \nabla_v \nabla H.
$$
In light of the added assumption in \S\ref{sec:moduliSpace} 
that $\nabla^2 H : T_z\Sigma \to T_z\Sigma$ is $j$-linear when
$\Morse(z)=2$, $\mathbf{A}_\gamma$ is therefore complex linear and thus has
real $2$-dimensional eigenspaces.  Let $V^-_\gamma \subset \Gamma(\gamma^*\Xi_+)$
denote the eigenspace with the largest negative eigenvalue, and suppose
$\mM^*(\{\widehat{J}_\tau\})$ denotes any moduli space consisting of pairs
$(u,\tau)$ such that $\tau \in [0,1]$ and $u$ is an unparametrized
and possibly disconnected somewhere injective finite-energy 
$\widehat{J}_\tau$-holomorphic curve that intersects
the interior of~$W$ with each of its connected components
and has at least two punctures $z_1,z_2$ asymptotic
to~$\gamma$.  Then, as was discussed in \S\ref{sec:automatic},
the asymptotic formulas of \cites{HWZ:props1,Mora,Siefring:asymptotics} 
give rise to an \defin{asymptotic evaluation map}
$$
\ev^\infty = (\ev^\infty_1,\ev^\infty_2) : \mM^*(\{\widehat{J}_\tau\}) \to V^-_\gamma \times V^-_\gamma,
$$
where for $i=1,2$, $\ev^\infty_i(u,\tau)$ associates to $u$ the leading
asymptotic eigenfunction of $u$ at~$z_i$.  As with ordinary evaluation maps
as in \cite{McDuffSalamon:Jhol}, one can show that this asymptotic evaluation 
map is a submersion when extended to the universal moduli space, hence 
generic choices of $\{\widehat{J}_\tau\}$ can make it transverse to any given submanifold
of $V^-_\gamma \times V^-_\gamma$, in particular the diagonal.  This leads to the following
result, which is essentially Proposition~3.9 in \cite{HutchingsTaubes:gluing2}.

\begin{lemma}
\label{lemma:genAsymp}
For every $\tau \in [0,1]$ and every pair of somewhere injective 
$\widehat{J}_\tau$-holomorphic
curves $u$ and $v$ in $\W$ of index~$-1$ or~$0$ that intersect the interior 
of~$W$ and each have a puncture asymptotic to the same simply covered elliptic 
orbit in $\CritMorse(H) \times S^1 \subset \widecheck{M}^+\spine$, the values
of the asymptotic evaluation maps at these two punctures are distinct.  
Moreover, the same holds for two punctures of a single curve with these
same properties.  \qed
\end{lemma}

The main results of \cite{Siefring:asymptotics} imply that whenever two
punctures of the same sign are asymptotic to the same orbit, the asymptotic eigenvalue 
controlling the relative exponential decay rate of the ends to each other is 
extremal if and only if the asymptotic evaluation map for both punctures
has distinct values.  Since the relevant eigenspace in the case at hand is
$2$-dimensional, a non-extremal decay rate is equivalent to non-extremal
asymptotic winding, so by the definitions of $\delta_\infty(u)$ and
$u*v$ in \cite{Siefring:intersection}, Lemma~\ref{lemma:genAsymp} implies:

\begin{lemma}
\label{lemma:noneHidden}
Suppose $\tau \in [0,1]$, $u$ and $v$ are somewhere injective
finite-energy $\widehat{J}_\tau$-holomorphic curves of index~$-1$ or~$0$ that intersect 
the interior of~$W$ and have non-identical images, and every Reeb orbit that 
occurs as an asymptotic orbit
for both $u$ and $v$ is a simply covered elliptic orbit in
$\CritMorse(H) \times S^1 \subset \widecheck{M}^+\spine$.  Then $u * v$ is the
algebraic count of actual intersections of $u$ and $v$, i.e.~it includes
no asymptotic contributions.  Moreover, if every orbit occurring as an
asymptotic orbit for two distinct ends of $u$ is also a simple elliptic
orbit in $\CritMorse(H) \times S^1$, then $\delta_\infty(u) = 0$.
\qed
\end{lemma}

\subsection{Compactness for nicely embedded curves}
\label{sec:compactnessProofs}

In this section we will state and prove two compactness results for
certain classes of nicely embedded holomorphic curves in~$\W$, which
will be used in \S\ref{sec:moduliSurface} to describe the global
structure of the quotient moduli spaces $\quotMod{}{}$ 
and~$\quotModHomotop{}$.  Recall from \S\ref{sec:Siefring} that a
somewhere injective finite-energy $\widehat{J}$-holomorphic curve 
$u$ in $\W$ is \emph{nicely embedded} if 
its intersection numbers as defined by Siefring \cite{Siefring:intersection}
satisfy
$$
\delta(u) = \delta_\infty(u) = 0 \quad\text{ and }\quad u * u \le 0;
$$
moreover, if $u$ is in the $\RR$-invariant setting $(\RR \times M^+,J_+)$
and is not a trivial cylinder, then the condition reduces to $u * u = 0$.
We saw in Proposition~\ref{prop:int0} that the latter is satisfied by
every holomorphic page in $\ModPlus{}$, so they are nicely embedded,
and we will see that the same is therefore true for all the smooth somewhere 
injective curves in~$\quotMod{}{}$.  Thus in order to understand the
strata of $\quotMod{}{}$ that arise from nontrivial holomorphic buildings,
it suffices to understand the closure of the space of nicely embedded
curves.

\subsubsection{Moduli spaces of nicely embedded curves}

Adapting some notation from \S\ref{sec:invariants}, 
we shall abbreviate $\mathbf{m} := (m_1,\ldots,m_r)$ and define
$$
\mM(\widehat{J} ; H ; \mathbf{m}) \subset \mM(\widehat{J}), \qquad
\overline{\mM}(\widehat{J} ; H ; \mathbf{m}) \subset \overline{\mM}(\widehat{J})
$$
as the spaces of curves/buildings $u$ in $(\W,\widehat{J})$ that satisfy the following 
conditions:
\begin{enumerate}
\item All asymptotic orbits of $u$ are in $\Crit(H) \times S^1 \subset
\widecheck{M}^+\spine$ and the sum of all their periods is less than the
bound $T_1$ from Lemma~\ref{lemma:periodsAndIndices};
\item For each $i \in \{1,\ldots,r\}$, the sum of the covering multiplicities
of all asymptotic orbits of $u$ in the component $\Sigma_i \times S^1
\subset \widecheck{M}^+\spine$ is at most~$m_i$;
\item $u$ has (arithmetic) genus~$0$.
\end{enumerate}
Define subsets
$$
\mM\nice(\widehat{J} ; H ; \mathbf{m}) \subset \mM(\widehat{J} ; H ; \mathbf{m}), \qquad
\overline{\mM}\nice(\widehat{J} ; H ; \mathbf{m}) \subset 
\overline{\mM}(\widehat{J} ; H ; \mathbf{m}),
$$
where the first consists of all $u \in \mM(\widehat{J};H;\mathbf{m})$ that are
nicely embedded, and the second is the closure of the first with respect
to the SFT-topology.
Given the $1$-parameter family $\{\widehat{J}_\tau\}$, we analogously 
define spaces of pairs $(u,\tau)$ for $\tau \in [0,1]$, which form subsets
\begin{equation*}
\begin{split}
\mM\nice(\{\widehat{J}_\tau\} ; H ; \mathbf{m}) &\subset \mM(\{\widehat{J}_\tau\} ; H ; \mathbf{m})
\subset \mM(\{\widehat{J}_\tau\}), \\
\overline{\mM}\nice(\{\widehat{J}_\tau\};H;\mathbf{m}) &\subset
\overline{\mM}(\{\widehat{J}_\tau\};H;\mathbf{m}) \subset \overline{\mM}(\{\widehat{J}_\tau\}),
\end{split}
\end{equation*}
where we should clarify that $\overline{\mM}\nice(\{\widehat{J}_\tau\};H;\mathbf{m})$ 
is defined as the closure of $\mM\nice(\{\widehat{J}_\tau\};H;\mathbf{m})$ in
$\overline{\mM}(\{\widehat{J}_\tau\})$; note that this may in general be larger than the 
space of pairs $(u,\tau)$ with $u \in \overline{\mM}\nice(\widehat{J}_\tau;H;\mathbf{m})$,
since it includes all limits of SFT-convergent sequences $(u_\nu,\tau_\nu) \in
\mM\nice(\{\widehat{J}_\tau\};H;\mathbf{m})$, where the $\tau_\nu \in [0,1]$ can vary.
For each $i \in \ZZ$, we denote by
\begin{equation*}
\begin{split}
\mM\nice_i(\widehat{J} ; H ; \mathbf{m}) &\subset \mM\nice(\widehat{J} ; H ; \mathbf{m}),\\
\overline{\mM}_i(\{\widehat{J}_\tau\} ; H ; \mathbf{m}) &\subset
\overline{\mM}(\{\widehat{J}_\tau\} ; H ; \mathbf{m})
\end{split}
\end{equation*}
and so forth the subsets defined by the condition $\ind(u)=i$.  
Our genericity assumptions, in particular Lemma~\ref{lemma:genericity},
imply that a somewhere injective curve in $\mM(\widehat{J}_\tau;H;\mathbf{m})$ for
$\tau \in [0,1]$ will never have index less than~$-1$ if it intersects
the interior of~$W$.  Outside of this region, 
i.e.~in $\widehat{\nN}_-(\p E) \subset \W$, $\widehat{J}_\tau$ was defined 
to match the specially constructed model $J_+$ from \S\ref{sec:Jandf} 
and is thus neither generic nor
$\tau$-dependent, so we must still say something about indices of curves 
with images contained entirely in~$\widehat{\nN}_-(\p E)$.

\begin{lemma}
\label{lemma:index-1}
For every $\tau \in [0,1]$, every curve
$u \in \mM(\widehat{J}_\tau;H;\mathbf{m})$ with image contained in
$\widehat{\nN}_-(\p E)$ is an embedded leaf of $\fF_+$ and is isotopic
to one of the holomorphic pages in $\ModPlus{}$.  In particular,
it has index~$1$ or~$2$.  
\end{lemma}
\begin{proof}
Recall that the subset
$\widehat{\nN}_+(\p E) \subset \widehat{\nN}_-(\p E)$ is identified canonically
with the half-symplectization $[0,\infty) \times M^+$, and it is also a 
retraction of $\widehat{\nN}_-(\p E)$, thus we can choose a diffeomorphism 
$\widehat{\nN}_-(\p E) \cong [0,\infty) \times M^+$ that matches the 
canonical one near infinity.  Under this identification, we observe that
a curve $u \in \mM(\widehat{J}_\tau;H;\mathbf{m})$ contained in $\widehat{\nN}_-(\p E)$
must intersect $[0,\infty) \times \widecheck{M}^+\paper$,
as otherwise it would be confined to the neighborhood of a spinal component
in which the homological sum of all its asymptotic orbits is nonzero,
producing a contradiction.  Now if $u$ is not a leaf of $\fF_+$, observe 
that it also cannot be a multiple cover of any leaf since the total
multiplicities of the orbits in each component of the spine would then
be greater than what is allowed for curves in $\mM(\widehat{J}_\tau;H;\mathbf{m})$.
It follows that 
$u$ has at least one isolated intersection with some leaf $v \in \fF_+$
that stays away from $M^-$ and is thus an asymptotically cylindrical
$\widehat{J}$-holomorphic curve; indeed, the entirety of the region
$(-1,\infty) \times \widecheck{M}\paper \subset \widehat{E}$ is foliated by
leaves of this type, which are tangent to~$\Xi_+$.  Positivity of
intersections therefore implies $u * v > 0$.  However, 
a small alteration to the proof of Lemma~\ref{lemma:intersectingPages}
shows that $u * v = 0$ for every holomorphic page $v \in \fF_+$.
Indeed, this is immediate if $u$ has no punctures, as one can then translate
$v$ upward in $[0,\infty) \times M^+$ to make it disjoint from~$u$.
If $u$ does have punctures, then
one can modify it as in the proof of Lemma~\ref{lemma:intersectingPages}
by a homotopy through asymptotically cylindrical
maps so that its intersection with $\widehat{\nN}_+(\p E)$ is a union of
trivial cylinders, and then compute the intersection number again via
Lemma~\ref{lemma:leavesOrbits}.  This contradiction proves the lemma.
\end{proof}

Further constraints on indices hold for nicely embedded curves.
The following lemma implies that $\overline{\mM}\nice_i(\widehat{J}_\tau;H;\mathbf{m})$
is always empty for $i > 2$, and moreover, all buildings $u$ in 
$\overline{\mM}\nice_1(\widehat{J}_\tau;H;\mathbf{m})$
or $\overline{\mM}\nice_2(\widehat{J}_\tau;H;\mathbf{m})$ have only simply covered
asymptotic orbits, all elliptic in the latter case, with exactly one
hyperbolic orbit in the former case, and $u*u = 0$.  For $u$ in
$\overline{\mM}\nice_{-1}(\widehat{J}_\tau;H;\mathbf{m})$ or
$\overline{\mM}\nice_{0}(\widehat{J}_\tau;H;\mathbf{m})$, $u * u$ can be either $-1$
or $0$ depending on an easily denumerable list of combinations of
elliptic/hyperbolic orbits with at most one doubly covered orbit.

\begin{lemma}
\label{lemma:upperBound}
For any $(u,\tau) \in \overline{\mM}\nice(\{\widehat{J}_\tau\} ; H ; \mathbf{m})$,
all asymptotic orbits of $u$ are at most doubly covered,
$-1 \le \ind(u) \le 2$, $u * u \in \{-1,0\}$ and
$$
2 \left( u * u + 1 \right) =
\ind(u) + \#\Gamma_0 + 2 \#\Gamma^2 \in \{0,2\},
$$
where $\Gamma_0$ denotes the set of punctures of $u$ at which the
asymptotic orbit has even Conley-Zehnder index, and $\Gamma^2$ is the
set of punctures at which the orbit is doubly covered.
\end{lemma}
\begin{proof}
Denote the set of punctures of $u$ by $\Gamma$ and for each $m \in \NN$,
let $\Gamma^m \subset \Gamma$ denote the subset at which the orbit has
covering multiplicity~$m$.
By assumption, $u$ is either nicely embedded or is the limit in the
SFT-topology of a sequence $u^\nu$ of nicely embedded curves as 
$\nu \to \infty$, thus it suffices to prove the lemma under the assumption
that $u$ itself is nicely embedded.  Given this, we have
$\delta(u) = \delta_\infty(u) = 0$ and $u * u \le 0$, and we already know
$\ind(u) \ge -1$ due to Lemmas~\ref{lemma:genericity} and~\ref{lemma:index-1}.  
To compute 
$u * u$, we first plug the stated conditions
into the adjunction formula \eqref{eqn:adjunctionEq}, obtaining
$c_N(u) = u * u - \left[ \bar{\sigma}(u) - \#\Gamma \right]$,
thus by \eqref{eqn:2cN},
\begin{equation}
\label{eqn:unu}
\begin{split}
\ind(u) -2 + \#\Gamma_0 &= 2 \left( u * u - 
\left[ \bar{\sigma}(u) - \#\Gamma \right] \right) \\
&\le - 2 \left[ \bar{\sigma}(u) - \#\Gamma \right] .
\end{split}
\end{equation}
Since the index formula implies that $\ind(u)$ and $\#\Gamma_0$ always have 
the same parity, and
$\ind(u) \ge -1$, the left hand side of this inequality is at least~$-2$,
implying $\bar{\sigma}(u) - \#\Gamma \le 1$.
But Lemma~\ref{lemma:winding} implies that all the asymptotic orbits 
$\gamma$ of $u$ satisfy $\alpha_-(\gamma) = 0$ in the $S^1$-invariant
trivialization, so Lemma~\ref{lemma:spectralCovering} then implies
$$
\bar{\sigma}(u) - \#\Gamma = \sum_{m \in \NN} m \#\Gamma^m - \#\Gamma
= \sum_{m \ge 2} (m-1) \#\Gamma^m \le 1,
$$
thus $\#\Gamma^m = 0$ for all $m \ge 3$ and
$\bar{\sigma}(u) - \#\Gamma = \#\Gamma^2 \in \{0,1\}$.
The stated identity now follows from \eqref{eqn:unu}.
\end{proof}

\subsubsection{Statements of the main compactness results}

To simplify the wording in the following statements, 
we will describe only the \emph{nontrivial}
components in each level of a holomorphic building, so that each level should 
be understood to
consist of the disjoint union of the specified curves with some trivial
cylinders.  In cases where multiple nontrivial curves appear in 
upper levels (e.g.~case~\ref{item:compactness2doublegradflow} in
Prop.~\ref{T:compactness} below), the actual number of upper levels may vary
depending on whether these curves occupy the same level or not.
Schematic representations of the index~$2$ buildings described in the following two
results are shown in Figures~\ref{fig:compactnessGeneric} and~\ref{fig:compactnessExceptional},
where the correct labelling of the elliptic and hyperbolic orbits in these pictures
can be deduced from Lemma~\ref{lemma:upperBound} above.
In each case, minor simplifications of the same pictures produce representations
of the relevant buildings with lower index as well,
e.g.~the index~$1$ buildings in
Proposition~\ref{T:compactness} look like Figures~\ref{fig:compactness2gradflow}
or~\ref{fig:compactness2doublegradflow} with one gradient flow cylinder removed
and possibly everything shifted up one level.

\begin{prop} \label{T:compactness}
Assume $\widehat{J}$ is generic so that the results of \S\ref{sec:genericity} hold.
Then $\mM\nice_0(\widehat{J};H;\mathbf{m})$ is a finite set and thus matches
$\overline{\mM}\nice_0(\widehat{J};H;\mathbf{m})$.  Buildings in
$\overline{\mM}\nice_1(\widehat{J};H;\mathbf{m}) \setminus 
\mM\nice_1(\widehat{J};H;\mathbf{m})$ all fit either of the following descriptions:
\begin{enumerate}[label=(1\alph{enumi})]
\item The main level is empty and the upper level is a
holomorphic page in~$\ModPlus{1}$;
\label{item:compactness1escape}
\item The main level is a curve $u_0 \in \mM\nice_0(\widehat{J};H;\mathbf{m})$
with $u_0 * u_0 = 0$, and the upper level contains a single
gradient flow cylinder in~$\ModPlus{1}$.
\label{item:compactness1gradflow}
\setcounter{building1}{\value{enumi}}
\end{enumerate}
Finally, every building in
$\overline{\mM}\nice_2(\widehat{J};H;\mathbf{m}) \setminus
\mM\nice_2(\widehat{J};H;\mathbf{m})$ fits one of the following descriptions:
\begin{enumerate}[label=(2\alph{enumi})]
\item The main level is empty and there are either one or two upper levels 
representing an element of $\cpctModPlus{2}$ (the variant with only one upper level is
shown in Figure~\ref{fig:compactness2escape});\footnote{Note that the same diagram, but with a
nonempty curve in $W$ instead of $\RR \times M$, would represent an element of $\mM\nice_2(\widehat{J};H;\mathbf{m})$.}
\label{item:compactness2escape}
\item There are no upper levels, and the main level is a nodal curve in 
$\W$, having
two connected components $u_\pm \in \mM\nice_0(\widehat{J};H;\mathbf{m})$ with
$u_+ * u_+ = u_- * u_- = -1$ and $u_+ * u_- = 1$, and intersecting transversely
at a single node (Figure~\ref{fig:compactness2nodal});
\label{item:compactness2nodal}
\item 
The main level consists of a curve
$u_0 \in \mM\nice_1(\widehat{J};H;\mathbf{m})$, and there is one upper level containing
a single gradient flow cylinder in~$\ModPlus{1}$ (Figure~\ref{fig:compactness2gradflow}); 
\label{item:compactness2gradflow}
\item Case~\ref{item:compactness1gradflow} with a second gradient flow
cylinder in~$\ModPlus{1}$ added in an upper level (Figure~\ref{fig:compactness2doublegradflow});
\label{item:compactness2doublegradflow}
\item 
The main level consists of a curve
$u_0 \in \mM\nice_0(\widehat{J};H:\mathbf{m})$ which has $u_0 * u_0 = 0$ and
elliptic asymptotic orbits including one that is doubly covered,
and there is one upper level containing an index~$2$ branched double cover of 
the trivial cylinder over this orbit, with two positive punctures and one 
negative (Figure~\ref{fig:compactness2branched}).
\label{item:compactness2branched}
\setcounter{building2}{\value{enumi}}
\end{enumerate}
\end{prop}

\begin{figure}
    \begin{subfigure}[t]{0.45\linewidth}
        \includegraphics{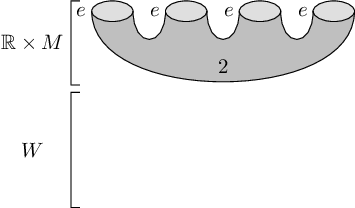}
    \caption{Case~\ref{item:compactness2escape}. 
	\label{fig:compactness2escape}
    }
\end{subfigure}
\hfill
\begin{subfigure}[t]{0.45\linewidth}
    \includegraphics{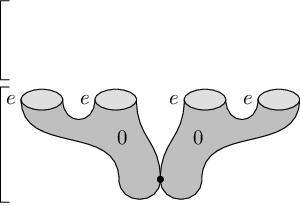}
    \caption{
        Case~\ref{item:compactness2nodal}.\label{fig:compactness2nodal}
    }
\end{subfigure}
\hfill\phantom{ }

\begin{subfigure}[t]{0.45\linewidth}
    \includegraphics{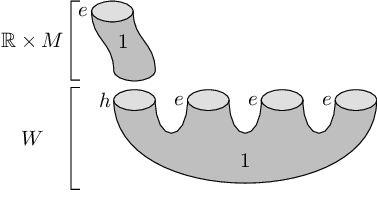}
    \caption{
        Case~\ref{item:compactness2gradflow}.\label{fig:compactness2gradflow}
    }
\end{subfigure}
\hfill
\begin{subfigure}[t]{0.45\linewidth}
    \includegraphics{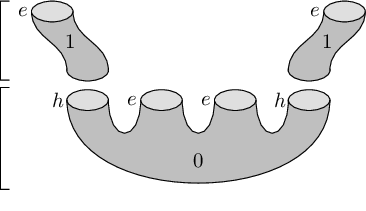}
    \caption{Case~\ref{item:compactness2doublegradflow}.\label{fig:compactness2doublegradflow}
    }
\end{subfigure}

\hfill
\begin{subfigure}[t]{0.45\linewidth}
    \includegraphics{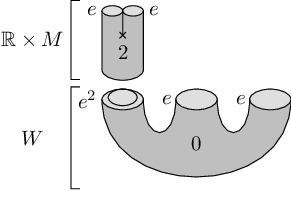}
    \caption{Case~\ref{item:compactness2branched}.\label{fig:compactness2branched}
    }
\end{subfigure} \hfill \null

    \caption{The possible index 2 buildings that can arise in the compactification of the moduli space for generic $J$. 
        The asymptotic orbits are labelled with $e$ or $h$ to indicate elliptic/hyperbolic simple
orbits in the spine. The notation of $e^2$ denotes a double cover of an elliptic orbit. The components of each 
building are labelled with their Fredholm index.}
    \label{fig:compactnessGeneric}
\end{figure}

\begin{prop} \label{T:deformationJ}
Assume the homotopy $\{\widehat{J}_\tau\}$ is generic so that the results of
\S\ref{sec:genericity} hold.  Then there exists a finite set of parameter 
values
$$
I^\sing \subset (0,1)
$$
such that $\mM\nice_{-1}(\{\widehat{J}_\tau\};H;\mathbf{m})$ contains exactly one
pair $(u,\tau)$ for each $\tau \in I^\sing$ and no other elements, hence
$\overline{\mM}\nice_{-1}(\{\widehat{J}_\tau\};H;\mathbf{m}) = 
\mM\nice_{-1}(\{\widehat{J}_\tau\};H;\mathbf{m})$.  Pairs 
$(u,\tau) \in \overline{\mM}\nice_i(\{\widehat{J}_\tau\};H;\mathbf{m})$ for $i=0,1,2$
and $\tau \in [0,1] \setminus I^\sing$ are described by the list in
Proposition~\ref{T:compactness}, while if $\tau \in I^\sing$, the list must
be supplemented as follows.  For $i=0$, $u$ can either be a smooth curve or
one of the following:
\begin{enumerate}[label=(0\alph{enumi})]
\item
The main level is a curve $u_0 \in \mM\nice_{-1}(\widehat{J}_\tau;H;\mathbf{m})$
with $u_0 * u_0 = -1$ or~$0$, and there is one
upper level containing a gradient flow cylinder in~$\ModPlus{1}$;
\label{item:deformation0gradflow}
\item
The main level is a curve $u_0 \in \mM\nice_{-1}(\widehat{J}_\tau;H;\mathbf{m})$
with $u_0 * u_0 = 0$ whose asymptotic orbits include one that is hyperbolic
and doubly covered, while the upper level consists
of an index~$1$ unbranched double cover of a gradient flow cylinder
in~$\ModPlus{1}$.
\label{item:deformation0unbranched}
\end{enumerate}
The cases with index~$1$ not described
in Proposition~\ref{T:compactness} can include the following:
\begin{enumerate}[label=(1\alph{enumi})]
\setcounter{enumi}{\value{building1}}
\item There are no upper levels, and the main level is a nodal curve
in $\W$, having two connected components 
$u_0 \in \mM\nice_0(\widehat{J}_\tau;H;\mathbf{m})$ and
$u_{-1} \in \mM\nice_{-1}(\widehat{J}_\tau;H;\mathbf{m})$ with
$u_0 * u_0 = u_{-1} * u_{-1} = -1$ and $u_0 * u_{-1} = 1$, and intersecting
transversely at a single node;
\label{item:deformation1nodal}
\item Case~\ref{item:deformation0gradflow} with $u_0 * u_0 = 0$ and
a second gradient flow cylinder in $\ModPlus{1}$ added in an upper level;
\label{item:deformation1doublegradflow}
\item The main level is a curve $u_0 \in \mM\nice_{-1}(\widehat{J}_\tau;H;\mathbf{m})$
with $u_0 * u_0 = 0$ whose asymptotic orbits include one that is elliptic
and doubly covered, and there is one
upper level containing an index~$2$ branched double cover of the
trivial cylinder over this elliptic orbit;
\label{item:deformation1branched}
\item The main level is a curve $u_0 \in \mM\nice_{-1}(\widehat{J}_\tau;H;\mathbf{m})$
with $u_0 * u_0 = 0$ whose asymptotic orbits include one that is hyperbolic
and doubly covered, and there are two upper levels:
the first contains an index~$1$ 
branched double cover of the trivial cylinder over the
hyperbolic orbit with two positive punctures and one negative, and 
a gradient flow cylinder in~$\ModPlus{1}$ is stacked on top of this in a
second upper level;
\label{item:deformation1vertigo}
\end{enumerate}
Finally, the following additional possibilities for buildings of index~$2$
can occur:
\begin{enumerate}[label=(2\alph{enumi})]
\setcounter{enumi}{\value{building2}}
\item Case~\ref{item:deformation1doublegradflow} with a third gradient
flow cylinder in $\ModPlus{1}$ added in an upper level (Figure~\ref{fig:deformation2triplegradflow});
\label{item:deformation2triplegradflow}
\item 
Case~\ref{item:deformation1branched} with a gradient flow cylinder
in $\ModPlus{1}$ added in an upper level (Figure~\ref{fig:deformation2thebirds});
\label{item:deformation2thebirds}
\item 
Case~\ref{item:deformation1nodal} with a gradient flow cylinder
in $\ModPlus{1}$ added in an upper level, connected to the index~$-1$
curve along its unique hyperbolic orbit (Figure~\ref{fig:deformation2nodalgradflow});
\label{item:deformation2nodalgradflow}
\item
Case~\ref{item:deformation1vertigo} with an
extra gradient flow cylinder in $\ModPlus{1}$ added on top (Figure~\ref{fig:deformation2vertigo});
\label{item:deformation2vertigo}
\item
Case~\ref{item:deformation0unbranched} with an index~$2$ 
branched double cover of
the trivial cylinder over an elliptic orbit stacked on top of the
unbranched cover (Figure~\ref{fig:deformation2unbranched});
\label{item:deformation2unbranched}
\item 
The main level is a curve $u_0 \in \mM\nice_{-1}(\widehat{J}_\tau;H;\mathbf{m})$
with $u_0 * u_0 = 0$ whose asymptotic orbits include one that is hyperbolic
and doubly covered, and there is one upper level containing an index~$3$
branched double cover of a gradient flow cylinder
in~$\ModPlus{1}$ (Figure~\ref{fig:deformation2ihatethisone}).\footnote{Notice 
that cases~\ref{item:deformation2vertigo} and~\ref{item:deformation2unbranched} 
can arise as the boundary of the space of configurations from case~\ref{item:deformation2ihatethisone}.}
\label{item:deformation2ihatethisone}
\end{enumerate}
\end{prop}

\begin{remark}
In the scenarios in Propositions~\ref{T:compactness} and~\ref{T:deformationJ}
involving multiple components in one level without nodes, it may happen that
some of these components are identical, but this is only possible if
at least one of the multiplicities $m_1,\ldots,m_r$ is greater than~$1$.
The latter is also a necessary condition for any of the scenarios that involve
doubly covered curves.
\end{remark}

\begin{figure} 
    \begin{subfigure}[t]{0.475\linewidth}
\setcounter{subfigure}{5}
        \includegraphics{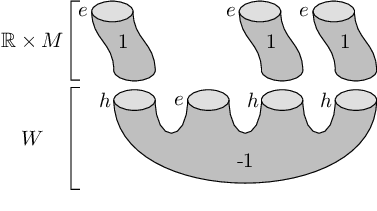}
    \caption{
        Case~\ref{item:deformation2triplegradflow}.\label{fig:deformation2triplegradflow}
    }
\end{subfigure}
\hfill
\begin{subfigure}[t]{0.475\linewidth}
    \includegraphics{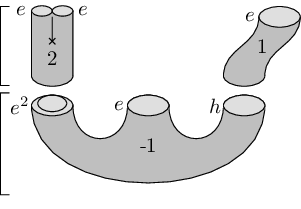}
    \caption{
    Case~\ref{item:deformation2thebirds}.\label{fig:deformation2thebirds}
    }
\end{subfigure}

\begin{subfigure}[t]{0.475\linewidth}
    \includegraphics{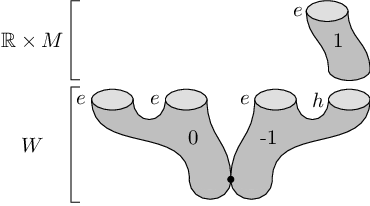}
    \caption{
    Case~\ref{item:deformation2nodalgradflow}.\label{fig:deformation2nodalgradflow}
    }
\end{subfigure}
\hfill
\begin{subfigure}[t]{0.475\linewidth}
    \includegraphics{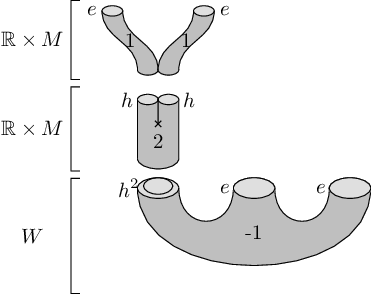}
    \caption{ 
    Case~\ref{item:deformation2vertigo}.\label{fig:deformation2vertigo}
}
\end{subfigure}
\begin{subfigure}[t]{0.475\linewidth}
    \includegraphics{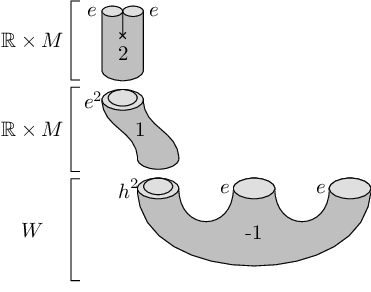}
    \caption{Case~\ref{item:deformation2unbranched}.\label{fig:deformation2unbranched}
    }
\end{subfigure} 
\hfill
\begin{subfigure}[t]{0.475\linewidth}
    \includegraphics{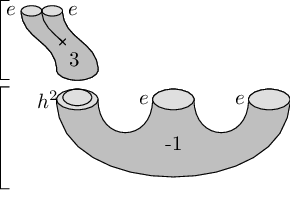}
    \caption{Case~\ref{item:deformation2ihatethisone}.\label{fig:deformation2ihatethisone}
    }
\end{subfigure}

    \caption{The buildings that can arise for exceptional homotopy values, but
    that are not in Figure \ref{fig:compactnessGeneric}.
    Notice that translation invariance of $J$ in $\R \times M$ precludes 
    $-1$ curves in upper levels, so the exceptional buildings must feature a $-1$ curve in $W$. 
}
    \label{fig:compactnessExceptional}
\end{figure}

\subsubsection{The upper levels}

The proofs of Propositions~\ref{T:compactness} and~\ref{T:deformationJ} 
will follow by considering 
inequalities that relate the Fredholm index and self-intersection numbers.
Notice that the statement of Proposition~\ref{T:compactness} is
precisely what remains of Proposition~\ref{T:deformationJ} if one adds the
assumption that somewhere injective curves of index~$-1$ in $\W$ do not exist.
With this understood and Lemmas~\ref{lemma:genericity} 
and~\ref{lemma:index-1} in hand, we will
use the same argument to prove both compactness results.

For the rest of \S\ref{sec:compactnessProofs}, we shall
fix $\tau_\infty \in [0,1]$ and a holomorphic building
$$
u^\infty \in \overline{\mM}(\widehat{J}_{\tau_\infty};H;\mathbf{m})
$$
and work on deriving constraints on the structure of~$u^\infty$.  We will later
add the assumption that $(u^\infty,\tau_\infty) \in 
\overline{\mM}\nice(\{\widehat{J}_\tau\};H;\mathbf{m})$,
but it will be useful to wait a bit before imposing such a restriction.

\begin{lemma}
\label{lemma:up}
All components in upper levels of 
$u^\infty \in \overline{\mM}(\widehat{J}_{\tau_\infty};H;\mathbf{m})$ are covers of leaves
of~$\fF_+$.  Moreover, the main level of $u^\infty$ is empty if and only
if $u^\infty$ contains a component whose image is a holomorphic page; 
in that case, the bottommost nonempty level of $u^\infty$ consists of a
single embedded holomorphic page, all other components are either trivial
cylinders or embedded gradient flow cylinders, and there are no nodes.
\end{lemma}
\begin{proof}
The period constraint on the asymptotic orbits implies the same constraint
on all breaking orbits via Proposition~\ref{prop:posPuncture}, so
the first claim then follows from Lemma~\ref{lemma:uniqueness}.  Given this,
the rest
follows from the assumptions about total multiplicities of orbits in the 
components $\Sigma_i \times S^1$: in particular, these constraints
imply that $u^\infty$ can contain no more than one holomorphic page, which
cannot be multiply covered, and if it is present then there can be nothing
in any level below it.  Moreover, the fact that these pages have simply
covered asymptotic orbits implies since the arithmetic genus is zero that
all curves in levels above the page are simply covered cylinders, and
nodes cannot appear.  Conversely, if the main level is empty, then some
component must be a holomorphic page since all other kinds of leaves have
negative punctures.
\end{proof}

\subsubsection{Index relations}

Let us now fix some notation.  We shall write
the set of (necessarily positive) punctures of 
$u^\infty$ as $\Gamma(u^\infty)$, and for
each $z \in \Gamma(u^\infty)$, let $\gamma_z$ denote the corresponding
asymptotic orbit, and $m_z \in \NN$ its covering multiplicity.\footnote{The
reader should beware that the notation $\gamma_z$ was used with a slightly
different meaning in earlier sections.}  For each $m \in \NN$, define the
subset
$$
\Gamma^m(u^\infty) := \left\{ z \in \Gamma(u^\infty)\ \big|\ m_z = m \right\},
$$
hence $\Gamma(u^\infty) = \Gamma^1(u^\infty) \coprod \Gamma^2(u^\infty)
\coprod \ldots$.  We shall also write 
$$
\Gamma(u^\infty) = \Gamma_0(u^\infty) \coprod \Gamma_1(u^\infty),
$$
where $\Gamma_0(u^\infty)$ and $\Gamma_1(u^\infty)$ denote the sets of
punctures $z$ at which $\muCZ(\gamma_z)$ is even or odd respectively.
By Lemma~\ref{lemma:periodsAndIndices}, this means
$\muCZ(\gamma_z) = \ell$ with respect to the $S^1$-invariant trivialization
for $z \in \Gamma_\ell(u^\infty)$, $\ell=0,1$.

In light of Lemma~\ref{lemma:up}, we shall assume from now on that the main
level of $u^\infty$ is nonempty, all upper levels consist of covers of
trivial cylinders or gradient flow cylinders, and all breaking orbits are
covers of orbits in $\CritMorse(H) \times S^1 \subset \widecheck{M}^+\spine$
with period less than~$T_1$.
Denote the nonconstant connected components of the main level
by $u_1,\ldots,u_L$, and write each of these as
$$
u_i = v_i \circ \varphi_i,
$$
where $v_i$ is a somewhere injective curve and $\varphi_i$ is a holomorphic
branched cover of punctured Riemann surfaces whose unique extension over the
punctures gives a map of closed Riemann surfaces with degree
$$
k_i := \deg(\varphi_i) \in \NN.
$$
We shall use the same notation $\Gamma(u_i)$, $\Gamma(v_i)$ with subsets
$\Gamma^m(u_i)$, $\Gamma_0(u_i)$ etc.~for the (again positive)
punctures of $u_i$
and~$v_i$ and their asymptotic orbits and covering multiplicities.

Lemmas~\ref{lemma:genericity} and~\ref{lemma:index-1} imply $\ind(v_i) \ge -1$, and
$\ind(v_i) \ge 0$ if $\widehat{J}_{\tau_\infty}$ is generic.  If $\ind(v_i) = -1$,
then we notice from the index formula \eqref{eqn:index} that
$\Gamma_0(v_i)$ cannot be empty, thus in general, our genericity assumptions
always imply
\begin{equation}
\label{eqn:lowerBound}
\ind(v_i) + \#\Gamma_0(v_i) \ge 0.
\end{equation}
We will later make use of the fact that if $(u^\infty,\tau_\infty) \in 
\overline{\mM}\nice(\{\widehat{J}_\tau\};H;\mathbf{m})$, then Lemma~\ref{lemma:upperBound}
implies a
corresponding upper bound for $\ind(u^\infty) + \#\Gamma_0(u^\infty)$
plus associated counts of multiply covered punctures,
and our strategy will be to combine these relations with \eqref{eqn:lowerBound}
for deriving constraints on~$u^\infty$.
The workhorse result for this purpose will be Lemma~\ref{L:index} below.  

As preparation, we must first relate the indices of the curves
$v_i$ and their multiple covers~$u_i$.
Given $z \in \Gamma(u_i)$, let $k_z \in \NN$ denote the \emph{branching
order} of $\varphi_i$ at~$z$, meaning that $\varphi_i$ is a $k_i$-to-$1$
covering map on the cylindrical end near this puncture.  These numbers are
related to the total degree of $\varphi_i$ by
\begin{equation}
\label{eqn:ks}
k_i = \sum_{z \in \varphi_i^{-1}(\zeta)} k_z \quad
\text{ for any $\zeta \in \Gamma(v_i)$}.
\end{equation}
We shall use the same notation for branching orders
at critical points $z \in \Crit(\varphi_i)$, so that the algebraic
count of critical points of $\varphi_i$ is
$$
Z(d\varphi_i) := \sum_{z \in \Crit(\varphi_i)} (k_z - 1) \ge 0;
$$
we emphasize that $\varphi_i$ is being viewed here as a branched cover
between \emph{punctured} Riemann surfaces, so the sum over $z \in \Crit(\varphi_i)$
does not include points in~$\Gamma(u_i)$.  The asymptotic orbits of $u_i$
and $v_i$ are now related by
$$
\gamma_z = \gamma_{\zeta}^{k_z} \quad \text{ for $z \in \Gamma(u_i)$ and
$\zeta = \varphi_i(z) \in \Gamma(v_i)$},
$$
where we are abusing notation by identifying $\varphi_i$ with its
holomorphic extension over the punctures.  Since $u^\infty$ has arithmetic
genus~$0$, all the components $u_i$ and $v_i$ also have genus zero and
the Riemann-Hurwitz formula therefore implies
\begin{equation}
\label{eqn:RiemannHurwitz}
Z(d\varphi_i) + \sum_{z \in \Gamma(u_i)} (k_z - 1) = 2 (k_i - 1).
\end{equation}

Since the breaking orbits are all in $\CritMorse(H) \times S^1$ with
periods less than~$T_1$, the orbits $\gamma_z$ for $z \in
\Gamma_\ell(u_i)$ or $z \in \Gamma_\ell(v_i)$ with $\ell=0,1$ also satisfy
$\muCZ(\gamma_z) = \ell$ with respect to the $S^1$-invariant
trivialization.  Moreover, the extension of $\varphi_i$
over the punctures maps $\Gamma_\ell(u_i)$ to $\Gamma_\ell(v_i)$ for
each $\ell=0,1$.
Plugging this information into the index formula \eqref{eqn:index} gives
\begin{equation*}
\begin{split}
\ind(u_i) &= -2 + \#\Gamma(u_i) + \#\Gamma_1(u_i) + 2 c_1(u_i), \\
\ind(v_i) &= -2 + \#\Gamma(v_i) + \#\Gamma_1(v_i) + 2 c_1(v_i),
\end{split}
\end{equation*}
where $c_1(v_i)$ and $c_1(u_i) = k_i c_1(v_i) \in \ZZ$ are abbreviations for
the relative first Chern numbers of the pulled back tangent bundles with
respect to the $S^1$-invariant trivialization at the ends.  A quick computation
combining this with \eqref{eqn:ks} and \eqref{eqn:RiemannHurwitz} yields:

\begin{lemma} \label{L:IndexCover}
For each $i=1,\ldots,L$,
\[
\begin{split}
\ind(u_i) &=
k_i \ind(v_i) + Z(d\varphi_i) 
     - \sum_{z \in \Gamma_1(u_i)} \left(k_z - 1\right)\\
&= 
k_i \left[\ind(v_i) + \# \Gamma_0(v_i) \right] + Z(d\varphi_i) 
     - \sum_{z \in \Gamma_1(u_i)} \left(k_z - 1 \right) 
     - \sum_{z \in \Gamma_0(u_i)} k_z .
\end{split}
\]
\qed
\end{lemma}

\begin{lemma} \label{L:index}
If $u^\infty \in \overline{\mM}(\widehat{J}_{\tau_\infty};H;\mathbf{m})$ 
has no nodes, then it satisfies
\[
\begin{split}
\ind(u^\infty) + \# \Gamma_0(u^\infty) + 2 \sum_{m \ge 2} & 
(m-1)\# \Gamma^m(u^\infty) \\
= 2 (L - 1) & + \sum_{i=1}^L \Bigg(
k_i \left[ \ind(v_i) + \# \Gamma_0(v_i) \right] + Z(d\varphi_i) \\
& + \sum_{\zeta \in \Gamma(v_i)} \sum_{z \in \varphi_i^{-1}(\zeta)} \left[ k_{z}(2m_\zeta -1) - 1 \right]
\Bigg).
\end{split}
\]
\end{lemma}
\begin{proof}
If $u^\infty$ has no upper levels, then we can replace it with an unstable
building having one upper level that consists only of trivial cylinders.
Let us therefore denote by $u^+$ the (possibly disconnected)
holomorphic building in
$(\RR \times M^+,J_+)$ consisting of all upper levels of $u^\infty$, and
assume without loss of generality that $u^+$ is a disjoint union
of $L_+ \ge 1$ connected buildings.  
We can compute $L_+$ from the fact that $u^\infty$ has arithmetic genus
zero: indeed, $u^\infty$ gives rise to a contractible graph whose vertices
correspond to the connected components $u_1,\ldots,u_L$ and the $L_+$
connected buildings forming $u^+$, and with edges corresponding to the
punctures of $u_1,\ldots,u_L$, so its Euler characteristic is
$1 = L + L_+ - \sum_{i=1}^L \#\Gamma(u_i)$, implying
$$
L_+ = -(L-1) + \sum_{i=1}^L \#\Gamma(u_i).
$$
Meanwhile, the positive punctures of $u^+$ are in one-to-one correspondence 
with those of $u^\infty$, and its negative punctures correspond to the 
punctures of $u_1,\ldots,u_L$, thus Lemma~\ref{lemma:wS} combines with
the above relation and gives
\begin{equation}
\label{eqn:indexu+}
\begin{split}
\ind(u^+) &= - 2 L_+ + \#\Gamma_0(u^\infty) + 
2 \#\Gamma_1(u^\infty) + \sum_{i=1}^L \#\Gamma_0(u_i) \\
&= 2(L-1) - \sum_{i=1}^L \left[ \#\Gamma_0(u_i) + 2\#\Gamma_1(u_i)\right] + 
\#\Gamma_0(u^\infty) + 2\#\Gamma_1(u^\infty) . 
\end{split}
\end{equation}
Note that since each component of $u^+$ has image in $\RR \times 
\widecheck{M}^+\spine$ and the latter fibers over $S^1$,
there is a well-defined degree of the projection to~$S^1$.
In particular, the total degree of the positive ends agrees with the
total degree of the negative ends, implying
$$
\sum_{z \in \Gamma(u^\infty)} m_z = \sum_{i=1}^L
\sum_{z \in \Gamma(u_i)} m_z = \sum_{i=1}^L \sum_{\zeta \in \Gamma(v_i)}
\sum_{z \in \varphi_i^{-1}(\zeta)} k_z m_\zeta.
$$
The expression on the left hand side of this equation can also be rewritten as
$\sum_{m \in \NN} m \#\Gamma^m(u^\infty)$, thus
\begin{equation}
\label{eqn:multiplicityRelation}
\begin{split}
\#\Gamma_0(u^\infty) &+ 2 \sum_{m \ge 2} (m-1) \#\Gamma^m(u^\infty) =
\#\Gamma_0(u^\infty) + 2 \sum_{z \in \Gamma(u^\infty)} m_z - 2 \#\Gamma(u^\infty) \\
&= - \#\Gamma_0(u^\infty) - 2\#\Gamma_1(u^\infty) +
2 \sum_{i=1}^L \sum_{\zeta \in \Gamma(v_i)}
\sum_{z \in \varphi_i^{-1}(\zeta)} k_z m_\zeta.
\end{split}
\end{equation}
Now writing $\ind(u^\infty) = \ind(u^+) + \sum_{i=1}^L \ind(u_i)$ and
applying Lemma~\ref{L:IndexCover} along with \eqref{eqn:indexu+} and
\eqref{eqn:multiplicityRelation}, we find that
$\ind(u^\infty) + \#\Gamma_0(u^\infty) + 2 \sum_{m\ge 2} (m-1) \#\Gamma^m(u^\infty)$
equals
\begin{equation*}
\begin{split}
2 & (L-1) - \sum_{i=1}^L \left[ \#\Gamma_0(u_i) + 2\#\Gamma_1(u_i)\right] + 
\#\Gamma_0(u^\infty) + 2\#\Gamma_1(u^\infty) \\
& \quad + \sum_{i=1}^L \Bigg(
k_i \left[\ind(v_i) + \# \Gamma_0(v_i) \right] + Z(d\varphi_i) 
     - \sum_{z \in \Gamma_1(u_i)} \left(k_z - 1 \right) 
     - \sum_{z \in \Gamma_0(u_i)} k_z
\Bigg) \\
& \quad - \#\Gamma_0(u^\infty) - 2\#\Gamma_1(u^\infty) +
2 \sum_{i=1}^L \sum_{\zeta \in \Gamma(v_i)}
\sum_{z \in \varphi_i^{-1}(\zeta)} k_z m_\zeta \\
&= 2(L-1) + \sum_{i=1}^L \Bigg( k_i \left[\ind(v_i) + 
\# \Gamma_0(v_i) \right] + Z(d\varphi_i) \\
& \qquad + \sum_{\zeta \in \Gamma_0(v_i)} \sum_{z \in \varphi_i^{-1}(\zeta)}
\left( 2 k_z m_\zeta - k_z - 1 \right) \\
& \qquad + \sum_{\zeta \in \Gamma_1(v_i)} \sum_{z \in \varphi_i^{-1}(\zeta)}
\left( 2 k_z m_\zeta - (k_z - 1) - 2 \right) \Bigg) .
\end{split}
\end{equation*}
\end{proof}

The next step is to combine the above lemma with the lower bounds on indices
arising from Lemmas~\ref{lemma:genericity} and~\ref{lemma:index-1} and some 
information from intersection theory.

\begin{lemma}
\label{lemma:noNodes}
Assume $(u^\infty,\tau_\infty) \in 
\overline{\mM}\nice(\{\widehat{J}_\tau\};H;\mathbf{m})$ and that
$u^\infty$ has nodes.  
Then there is exactly one node, there are no ghost 
bubbles, and the main level has exactly two connected
components $u_1$ and $u_2$, each somewhere injective with all
asymptotic orbits simply covered and satisfying 
$\ind(u_i) + \#\Gamma_0(u_i) = 0$ for $i=1,2$.
\end{lemma}
\begin{proof}
If $u^\infty$ has nodes, then Lemma~\ref{L:index} applies to each of its
nonconstant maximal non-nodal subbuildings (see \S\ref{sec:energy}),
and the right hand side of this relation is nonnegative due to
\eqref{eqn:lowerBound}.  The sum of the left hand sides over all these
subbuildings is meanwhile at most $2$ due to Lemma~\ref{lemma:upperBound}.
Proposition~\ref{prop:indexNodes4} thus implies that there
are no ghost bubbles, there is exactly one node
connecting two maximal non-nodal subbuildings, and for both of these
the expression on the right hand side in Lemma~\ref{L:index} vanishes.
This implies that each has exactly one component $u_i = v_i \circ \varphi_i$ 
in the main level, where the underlying simple curve $v_i$
satisfies $\ind(v_i) + \#\Gamma_0(v_i) = 0$ and has $m_\zeta = 1$ 
for all $\zeta \in \Gamma(v_i)$.  Moreover, $Z(d\varphi_i) = 0$ and
$k_z = 1$ for all $z \in \Gamma(u_i)$, hence the Riemann-Hurwitz formula
\eqref{eqn:RiemannHurwitz} implies $k_i = 1$, i.e.~$u_i$ 
is somewhere injective and $u_i = v_i$.
\end{proof}

Here is a more complete inventory of the consequences of
Lemma~\ref{L:index}.

\begin{lemma} \label{L:indexGames}
If $(u^\infty,\tau_\infty) \in \overline{\mM}\nice(\{\widehat{J}_\tau\};H;\mathbf{m})$, 
then the following constraints hold:
\begin{itemize}
\item The number of connected components in the main level is at most~$2$;
\item The asymptotic orbits of all the curves $v_i$ are either simply or 
doubly covered;
\item $\ind(v_i) + \# \Gamma_0(v_i) = 0$ or $2$ for each~$i$;
\item All punctures $z \in \Gamma(u_i)$ have $k_z \le 2$, and furthermore
  $k_z m_{\varphi_i(z)} \le 2$.
\end{itemize}
Moreover:
\begin{enumerate}
    \item \label{i}
    If any $v_i$ has $\ind(v_i) + \# \Gamma_0(v_i) =2$, 
      then $L=1$ and $u_1$ is somewhere injective (i.e.~$u_1 = v_1$).
    \item If all the $v_i$ have $\ind(v_i) + \# \Gamma_0(v_i) = 0$, then:
    	\begin{enumerate}
	\item  \label{i:i1}
	       If any of the $v_i$ has a doubly covered orbit,
	      then it is the only doubly covered orbit, $L=1$,
	      and $u_1$ is somewhere injective (i.e.~$u_1=v_1$).
	    \item If on the other hand $m_\zeta = 1$ for all $\zeta \in \Gamma(v_i)$, 
	      we consider two cases:
	      \begin{enumerate}
		\item \label{i:coverCase}
		  If $L=1$, let $\ell \ge 0$ denote the number of punctures
		  $z \in \Gamma(u_1)$ at which $k_z=2$: then $\ell \le 2$
		  and all other punctures have $k_z = 1$.  Moreover,
		  $u_1$ is somewhere injective if $\ell=0$, and otherwise
		  $\ell + Z(d\varphi_1) = 2$ and $k_1 = 2$.
		\item \label{i:node}
		  If $L=2$, then both components in the main level 
		  are somewhere injective, all their asymptotic orbits are
		  simply covered and $\ind(u_i) + \#\Gamma_0(u_i) = 0$
		  for $i=1,2$.
	       \end{enumerate}
	  \end{enumerate}
      \end{enumerate}
      In particular, the components $u_i$ in the main level are somewhere
      injective except possibly in case \ref{i:coverCase}.
\end{lemma}
\begin{proof}
If $u^\infty$ has any nodes then Lemma~\ref{lemma:noNodes} applies and
produces a result consistent with case~\ref{i:node}.  Let us therefore
assume there are no nodes, so that the relation in Lemma~\ref{L:index} applies
to $u^\infty$ directly, and its left hand side is at most $2$ by
Lemma~\ref{lemma:upperBound}.  The first three bullet points are
then immediate from the lemma since $\ind(v_i) + \#\Gamma_0(v_i) \ge 0$
for all~$i$.  For the fourth bullet point, we also see
an immediate contradiction if $k_z \ge 4$ for some $z \in \Gamma(u_i)$, 
so suppose $k_z = 3$, with $\zeta = \varphi_i(z) \in \Gamma(v_i)$.
Then $m_\zeta$ must be~$1$ and $k_z$ must also be~$1$ for all other
$z \in \Gamma(u_i)$, and $Z(d\varphi_i) = 0$, but then 
\eqref{eqn:RiemannHurwitz} gives
$$
Z(d\varphi_i) + \sum_{z \in \Gamma(u_i)} (k_z - 1) = 2 = 2 (k_i - 1),
$$
hence $k_i=2$ and there cannot be a branch point of order~$3$.  We therefore
have both $k_z \le 2$ and $m_\zeta \le 2$ for all punctures; if ever
$k_z = m_\zeta = 2$, then Lemma~\ref{L:index} provides another immediate
contradiction, so this completes the proof of the fourth bullet point.
  
Cases~\ref{i}, \ref{i:i1} and~\ref{i:node} follow from Lemma~\ref{L:index} via 
similar arguments: in each case, somewhere injectivity follows from the
Riemann-Hurwitz formula \eqref{eqn:RiemannHurwitz} after observing that
$Z(d\varphi_i)$ and all the $k_z - 1$ must vanish.

  We now consider case~\ref{i:coverCase}.  
  By hypothesis, we have $L=1$, $\ind(v_1) + \#\Gamma_0(v_1) = 0$
  and $m_\zeta = 1$ for each 
  $\zeta \in \Gamma(v_1)$, and taking account of Lemma~\ref{lemma:upperBound},
  the identity in Lemma~\ref{L:index} now simplifies to
$$
2 \ge \ind(u^\infty) + \#\Gamma_0(u^\infty) + 2\#\Gamma^2(u^\infty) =
Z(d\varphi_1) + \sum_{z \in \Gamma(u_1)} (k_z - 1),
$$
which equals $2(k_1 - 1)$ 
by \eqref{eqn:RiemannHurwitz}.  The stated conclusions follow immediately.
\end{proof}

\subsubsection{Intersection numbers}

In the present setting, Lemma~\ref{lemma:u*v} provides an easy method
for computing Siefring intersection numbers since, according to
Lemma~\ref{lemma:winding},
all the orbits $\gamma$ appearing as positive asymptotic orbits satisfy
$\alpha_-(\gamma) = 0$ in the $S^1$-invariant trivialization.  This implies
that if $u$ and $u'$ each denote any of $u^\infty$, $u_i$ or $v_i$, we have
\begin{equation}
\label{eqn:relative}
u * u' = u \bullet_{\Phi_0} u',
\end{equation}
with $\Phi_0$ denoting the $S^1$-invariant trivialization and
$\bullet_{\Phi_0}$ denoting the relative intersection pairing described
in \S\ref{sec:Siefring}.

\begin{lemma} \label{L:mainLevelIntersection}
$\displaystyle 
\sum_{j,k = 1}^L u_j * u_k = u^\infty * u^\infty$.
\end{lemma}
\begin{proof}
By Lemma~\ref{lemma:up}, all the components in upper levels of $u^\infty$
are covers of trivial cylinders and gradient flow cylinders.  Any two
curves $u$ and $u'$ of this type satisfy $u \bullet_{\Phi_0} u' = 0$,
as one can define the trivialization $\Phi_0$ globally over $\widecheck{M}^+\spine$
and then make $u'$ disjoint from $u$ by a global perturbation of $u'$ in
the direction of~$\Phi_0$.  The formula thus follows by computing
$u^\infty \bullet_{\Phi_0} u^\infty$ as a double sum over all components
in all levels and applying \eqref{eqn:relative}.
\end{proof}

Let us assume from now on that
$$
(u^\infty,\tau_\infty) \in \overline{\mM}\nice(\{\widehat{J}_\tau\};H;\mathbf{m}).
$$
We can now give a complete description of $u^\infty$ in the case
$u^\infty * u^\infty = -1$, which by Lemma~\ref{lemma:upperBound} means
$\ind(u_\infty) \in \{-1,0\}$, $\#\Gamma_0(u^\infty) \in \{0,1\}$ and
all asymptotic orbits of $u^\infty$ are simply covered.

\begin{lemma}
\label{lemma:-1}
If $u^\infty * u^\infty = -1$, then $u^\infty$ is either a smooth nicely
embedded curve or a building with two nontrivial levels, where the main level
$u_1$ is a connected nicely embedded curve with $\ind(u_1) = u_1 * u_1 = -1$
and simply covered asymptotic orbits,
and the upper level is a disjoint union of trivial cylinders with a single
gradient flow cylinder from $\ModPlus{1}$.
\end{lemma}
\begin{proof}
Since Lemma~\ref{lemma:upperBound} implies
$\ind(u^\infty) + \#\Gamma_0(u^\infty) = \#\Gamma^m(u^\infty) = 0$ for all
$m \ge 2$, Lemma~\ref{L:index} then implies
$L=1$, $\ind(v_1) + \#\Gamma_0(v_1) = Z(d\varphi_1) = 0$ and
$k_z = m_\zeta = 1$ for all punctures $\zeta$ and~$z$.  Thus by
the Riemann-Hurwitz formula \eqref{eqn:RiemannHurwitz},
$u_1 = v_1$ and the main level is described by 
Case~\ref{i:coverCase} of Lemma~\ref{L:indexGames} with $\ell=0$.
Lemma~\ref{L:mainLevelIntersection} implies $u_1 * u_1 = -1$,
and $\ind(u_1) + \#\Gamma_0(u_1) = 0$ implies via \eqref{eqn:2cN}
that $c_N(u_1) = -1$, so by the adjunction inequality \eqref{eqn:adjunction},
$\delta(u_1) = \delta_\infty(u_1) = 0$, hence $u_1$ is nicely embedded.
The fact that all asymptotic orbits of both $u_1$ and $u^\infty$ are
simply covered and $u^\infty$ has arithmetic genus~$0$ implies moreover that
all components in upper levels are also somewhere injective.
Adding up the indices across levels, this eliminates all possibilities other 
than what was stated.
\end{proof}

Since $u^\infty * u^\infty$ is always either $-1$ or $0$ by
Lemma~\ref{lemma:upperBound}, we shall consider the case
$u^\infty * u^\infty = 0$ from now on.

\begin{lemma}  \label{L:simplyCovered}
If $u^\infty * u^\infty = 0$, then the main level consists of either a
single nicely embedded curve $u_1$ or two distinct nicely embedded curves
$u_1$ and $u_2$ that intersect each other transversely at a node
and nowhere else.  Moreover, if the main level is a single curve $u_1$,
then all its asymptotic orbits are simply covered if
$\ind(u_1) + \#\Gamma_0(u_1) = 2$, and exactly one of them is doubly covered
if $\ind(u_1) + \#\Gamma_0(u_1) = 0$.
\end{lemma}
\begin{proof}
Let us first rule out the possibility of a single doubly covered component
$u_1 = v_1 \circ \varphi_1$ from case~\ref{i:coverCase} of
Lemma~\ref{L:indexGames}.  If this scenario occurs, then
we know $\ind(v_1) + \#\Gamma_0(v_1) = 0$ and all the asymptotic orbits
of $v_1$ are simply covered.  Equation~\eqref{eqn:2cN} thus gives
$c_N(v_1) = -1$, and Lemma~\ref{lemma:spectralCovering} gives
$\bar{\sigma}(v) - \#\Gamma(v) = 0$, so by the adjunction formula
\eqref{eqn:adjunctionEq},
$$
v_1 * v_1 = 2 \left[ \delta(v) + \delta_\infty(v) \right] - 1.
$$
In particular, this is an odd integer.  But using \eqref{eqn:relative} and 
Lemma~\ref{L:mainLevelIntersection}, we also have
$$
0 = u^\infty * u^\infty = u_1 * u_1 = u_1 \bullet_{\Phi_0} u_1
= 4 \left( v_1 \bullet_{\Phi_0} v_1\right) = 4 \left( v_1 * v_1 \right)
$$
since $u_1$ is a double cover of $v_1$, so this implies that $0$ is an
odd number and thus rules out multiply covered components in the main level.

Exactly the same contradiction occurs if we consider
Case~\ref{i:node} of Lemma \ref{L:indexGames} assuming $u_1$ and $u_2$
are the same curve up to parametrization.  Indeed, $u_i * u_i$
is then an odd integer for $i=1,2$ due to the adjunction formula
\begin{equation}
\label{eqn:adjui}
u_i * u_i = 2 \left[ \delta(u_i) + \delta_\infty(u_i) \right] - 1,
\end{equation}
and Lemma~\ref{L:mainLevelIntersection} gives
\begin{equation}
\label{eqn:crossTerms}
0 = u^\infty * u^\infty = u_1 * u_1 + u_2 * u_2 + 2\left( u_1 * u_2 \right),
\end{equation}
which reduces to $0 = 4(u_1 * u_1)$ and again implies that $0$ is an
odd number.

Next consider case~\ref{i:node} when $u_1 \ne u_2$.  Combining
\eqref{eqn:adjui} and \eqref{eqn:crossTerms} in this case implies
$$
0 = 2 \sum_{i=1}^2 \left[ \delta(u_i) + \delta_\infty(u_i) \right] +
2 \left( u_1 * u_2 - 1 \right).
$$
Since $\ind(u_1)$ and $\ind(u_2)$ are both either $-1$ or~$0$,
genericity allows us via Lemma~\ref{lemma:noneHidden} to assume
$\delta_\infty(u_i) = 0$ for $i=1,2$ and moreover that $u_1 * u_2$ is the
(algebraic) count of \emph{actual} intersections between $u_1$ and~$u_2$,
with no additional asymptotic contributions.  Let us therefore rewrite
the above relation as
$$
1 = \delta(u_1) + \delta(u_2) + u_1 \bullet u_2,
$$
with $u_1 \bullet u_2 \ge 0$ denoting the count of actual intersections.
If $u_1$ and $u_2$ are connected at a node, then they necessarily intersect,
implying $u_1 \bullet u_2 = 1$ and $\delta(u_1) = \delta(u_2) = 0$,
hence both are embedded and they have only one intersection, which is
transverse and occurs at the node.  Equation~\eqref{eqn:adjui} then
implies that both satisfy $u_i * u_i = -1$, so they are nicely embedded.
If on the other hand there is no node, then the above relation between
$\delta(u_1)$, $\delta(u_2)$ and $u_1 \bullet u_2$ cannot hold, as all
three terms must be~$0$.  To see this, recall that the assumption
$(u^\infty,\tau_\infty) \in \overline{\mM}\nice(\{\widehat{J}_\tau\};H;\mathbf{m})$
means that there exist sequences
$$
\tau_\nu \to \tau_\infty \quad\text{ and }\quad u^\nu \to u^\infty 
\qquad \text{ as $\nu \to \infty$}
$$
where $\tau_\nu \in [0,1]$ and $u^\nu \in \mM\nice(\widehat{J}_{\tau_\nu};H;\mathbf{m})$,
so in particular all the $u^\nu$ are embedded.  But if any of the three
terms above were positive, then there would be at least one
isolated double point or critical point of $u_1$ or $u_2$, or an isolated
intersection between them, and any of these scenarios would give rise to
an isolated singularity of the curves $u^\nu$ for sufficiently large~$\nu$
due to local positivity of intersections.  This is a contradiction.

Finally, we show that in all remaining cases of Lemma~\ref{L:indexGames},
the single somewhere injective curve $u_1$ in the main level is
nicely embedded.  Lemma~\ref{L:mainLevelIntersection} implies
$u_1 * u_1 = 0$, so we just need to show $\delta(u_1) = \delta_\infty(u_1)
= 0$.  Since $u_1$ cannot have any nodal points in this case, local 
positivity of intersections implies $\delta(u_1) = 0$,
as a singularity in $u_1$ would again be seen by the embedded curves
$u^\nu$ for $\nu$ sufficiently large.  Thus we only still need to prove
$\delta_\infty(u_1) = 0$.  This follows from genericity
(Lemma~\ref{lemma:noneHidden}) if $\ind(u_1) \le 0$, which takes
care of cases~\ref{i:i1} and~\ref{i:coverCase} in Lemma~\ref{L:indexGames}.
These are the cases with $\ind(u_1) + \#\Gamma(u_1) = 0$, hence
$c_N(u_1) = -1$ by \eqref{eqn:2cN}, and the adjunction formula then gives
$$
0 = -1 + \left[ \bar{\sigma}(u_1) - \#\Gamma(u_1) \right],
$$
so by Lemma~\ref{lemma:spectralCovering}, $u_1$ has exactly one doubly
covered asymptotic orbit and the rest are simply covered.
We are now left only with case~\ref{i}, with $\ind(u_1) + \#\Gamma_0(u_1) = 2$.
Now \eqref{eqn:2cN} implies $c_N(u_1) = 0$, so the adjunction
formula \eqref{eqn:adjunctionEq} becomes
$$
0 = 2\delta_\infty(u_1) + \left[ \bar{\sigma}(u_1) - \#\Gamma(u_1) \right]
$$
and thus implies both $\delta_\infty(u_1) = 0$ and
$\bar{\sigma}(u_1) - \#\Gamma(u_1) = 0$.  By Lemma~\ref{lemma:spectralCovering},
the latter implies that all asymptotic orbits of $u_1$ are simply covered.
\end{proof}

\subsubsection{Conclusion of the compactness proof}

The preceding lemmas establish a complete picture of all the possible main 
levels of the building~$u^\infty$.  To finish the proof of
Propositions~\ref{T:compactness} and~\ref{T:deformationJ}, we only need
to describe the possible multiple covers of leaves of $\fF_+$ that can
occur in the upper levels.  These components are highly constrained for the
following reasons:
\begin{enumerate}
\item Most asymptotic orbits of the main level are simply covered, with
at most one exception which is doubly covered and occurs only if the
main level is a single curve $u_1$ with 
$u_1 * u_1 = \ind(u_1) + \#\Gamma_0(u_1) = 0$;
\item The building has arithmetic genus zero: since the possibly nodal
curve forming the main level is always connected, this implies that no
curve in any upper level can have more than one negative puncture;
\item All curves in upper levels are covers of cylinders.
\end{enumerate}
Let us first consider the case where
$u^\infty$ has nodes: then Lemmas~\ref{L:indexGames} 
and~\ref{L:simplyCovered} imply that there is only one node, which occurs
in the main level, where it connects two nicely embedded curves whose
asymptotic orbits are all simply covered.  As observed above, the genus
condition implies that no curve in any upper level can have more than
one negative puncture, and since they are all covers of cylinders,
the fact that orbits are simply covered means that no curves in upper levels
can be multiple
covers.  It follows that the upper levels consist entirely of trivial
cylinders or gradient flow cylinders, where each of the latter contributes~$1$
to the total index of~$u^\infty$.  The two curves in the main level each
have index either~$-1$ or~$0$, but since they are distinct, they cannot
both have index~$-1$ due to genericity (Lemmas~\ref{lemma:genericity}
and~\ref{lemma:index-1}).
This completes the description of all possible nodal buildings.

In the absence of nodes, the main level is a single nicely embedded curve~$u_1$,
and the above description of the upper levels still
applies whenever the asymptotic orbits of $u_1$ are all simple: outside
of the case $u^\infty * u^\infty = -1$, which was dealt with in
Lemma~\ref{lemma:-1}, this is true if and only if 
$\ind(u_1) + \#\Gamma_0(u_1) = 2$.  If $\ind(u_1) + \#\Gamma_0(u_1) = 0$,
then exactly one asymptotic orbit of $u_1$ is doubly covered, which allows
for a limited range of multiple covers to appear in the upper levels:
indeed, there can be doubly covered unbranched cylinders (which are either
trivial or cover gradient flow cylinders and thus have index~$1$),
and exactly one branched double cover with
two positive punctures and one negative puncture.  Suppose $u$ is such a
branched double cover, and the underlying simple curve is~$v$.  From
Lemma~\ref{lemma:wS}, the possible indices of $u$ are as follows:
\begin{itemize}
\item If $v = \RR \times \gamma$ with $\gamma$ elliptic,
then $\ind(u) = 2$;
\item If $v = \RR \times \gamma$ with $\gamma$ hyperbolic,
then $\ind(u) = 1$;
\item If $v$ is a gradient flow cylinder, then $\ind(u) = 3$.
\end{itemize}
The buildings enumerated in Propositions~\ref{T:compactness}
and~\ref{T:deformationJ} are thus found by putting together all possible
combinations of these ingredients that add up to the correct index.

As a particular consequence, the above arguments show that the only 
holomorphic buildings appearing in $\mM\nice_0(\widehat{J};H;\mathbf{m})$ and 
$\mM\nice_{-1}(\{\widehat{J}_\tau\};H;\mathbf{m})$ are smooth curves (i.e.~with no
nodes and only one level), hence these spaces are compact.
Note that by Lemma~\ref{lemma:index-1}, none of the curves in those spaces
are confined to the non-generic domain $\widehat{\nN}_-(\p E)$, hence
our genericity assumptions ensure that $\mM\nice_0(\widehat{J};H;\mathbf{m})$ and 
$\mM\nice_{-1}(\{\widehat{J}_\tau\};H;\mathbf{m})$ are also both $0$-dimensional
manifolds and therefore finite sets.  Lemma~\ref{lemma:genericity} implies
moreover that for any two distinct elements
$(u,\tau)$ and $(u',\tau') \in \mM\nice_{-1}(\{\widehat{J}_\tau\};H;\mathbf{m})$,
we have $\tau \ne \tau'$, and we define $I^\sing$ as the finite set
of values $\tau \in [0,1]$ for which such curves exist; this set cannot
include $0$ or~$1$ since both $\widehat{J}_0$ and $\widehat{J}_1$ are assumed generic.
For any $\tau \not\in I^\sing$, the non-existence of index~$-1$ curves
rules out all of the scenarios listed in Proposition~\ref{T:deformationJ},
leaving only the list in Proposition~\ref{T:compactness}.  The proof of
both propositions is now complete.

\subsection{Holomorphic foliations on the completed filling}
\label{sec:moduliSurface}

In this section we prove Propositions~\ref{T:moduliSurface} 
and~\ref{prop:deform}.  For both results, the main
step will be to show that the holomorphic pages living
in $\widehat{\nN}_-(\p E)$ extend to the rest of $\W$ as a foliation
with finitely many singularities at the nodal points, and that this
foliation varies smoothly with the parameter~$\tau$.  We will then use
the foliation to define suitable smooth structures on the interiors of
$\quotMod{}{}$ and $\quotModHomotop{}$.
It's worth recalling briefly the type of argument that was
used for this step in \cite{Wendl:fillable}: in that simpler setting, all main
level curves in the moduli space either have index~$2$ or are nodal curves 
with components of index~$0$, all of them satisfying the automatic 
transversality criterion of \cite{Wendl:automatic}.  The foliation then 
arises easily from a combination of the implicit function 
theorem and compactness, showing
that the index~$2$ curves fill an open and closed subset of $\W$ in the
complement of the images of finitely many nodal curves---the latter 
being a subset of codimension~$2$---and automatic transversality guarantees
that these families of curves always persist under changes in~$\tau$.
The crucial difference in the present setting is that in the compactness
statements of Propositions~\ref{T:compactness} and~\ref{T:deformationJ},
not all degenerations have codimension at least~$2$; in particular one
can imagine the above argument failing as the index~$2$ curves run into a
``wall'' of codimension~$1$ formed by index~$1$ curves.  Such walls exist
in $\quotMod{}{}$ and $\quotModHomotop{}$,
but it would be more accurate to call them \emph{seams}: since they always
include a gradient flow cylinder in an upper level, they come in
canceling pairs, with the consequence that every such degeneration can be
glued back together using a different gradient cylinder in order to
``cross the wall''.

For $i \in \{1,2\}$, define 
$$
\quotMod{i}{\tau} \subset \quotMod{}{\tau}
$$
to be the subset of all equivalence classes of buildings whose main levels
are connected smooth curves in $\mM\nice_i(\widehat{J}_\tau;H;\mathbf{m})$.
For $i \in \{-1,0\}$, we define $\quotMod{i}{\tau}$ similarly but allow it
additionally to contain equivalence classes of buildings whose main levels
are nodal curves with two connected components, one belonging to
$\mM\nice_0(\widehat{J}_\tau;H;\mathbf{m})$ and the other to 
$\mM\nice_{i}(\widehat{J}_\tau;H;\mathbf{m})$.
The subsets $\quotModHomotop{i} \subset \quotModHomotop{}$ and
$\quotMod{i}{} \subset \quotMod{}{}$ are defined similarly.

\begin{lemma}
\label{lemma:index2}
Suppose $([u_0],\tau_0) \in \quotModHomotop{2}$.  Then there exist 
neighborhoods $\uU \subset \quotModHomotop{}$ of $([u_0],\tau_0)$ and
$\vV \subset [0,1]$ of $\tau_0$ such that
$\uU \subset \quotModHomotop{2}$ and for every $\tau \in \vV$,
$$
\uU_\tau := \uU \cap \quotMod{}{\tau}
$$
is a contractible open subset of $\quotMod{}{\tau}$ 
in which the main levels define a smooth
$2$-parameter family of embedded curves with disjoint images that foliate
an open subset of~$\W$.
\end{lemma}
\begin{proof}
This is essentially a standard application of the implicit function theorem
for nicely embedded index~$2$ curves, see \cite{Wendl:thesis}*{Theorem~4.5.42}
or \cite{Wendl:Durham}*{Theorem~3.26}.
It derives mainly from two crucial facts: (1)~curves in 
$\mM\nice_2(\widehat{J}_\tau;H;\mathbf{m})$
satisfy the automatic transversality criterion of \cite{Wendl:automatic},
hence genericity is not required and the moduli space perturbs smoothly
with~$\tau$, and (2)~tangent spaces $T_u \mM\nice_2(\widehat{J}_\tau;H;\mathbf{m})$ are
equivalent to spaces of holomorphic sections of the normal bundle along~$u$,
and these sections are always nowhere zero.
\end{proof}

\begin{lemma}
\label{lemma:index1}
Suppose $([u_0],\tau_0) \in \quotModHomotop{1}$.  Then there exist
neighborhoods $\uU \subset \quotModHomotop{}$ of $([u_0],\tau_0)$ and
$\vV \subset [0,1]$ of $\tau_0$ and a homeomorphism
$$
\Psi : (-1,1)^2 \times \vV \to \uU
$$
such that for all $(x,y,\tau) \in (-1,1)^2 \times \vV$,
$$
\Psi(0,y,\tau) \in \quotMod{1}{\tau} \quad \text{ and } \quad
\Psi(x,y,\tau) \in \quotMod{2}{\tau} \text{ if $x \ne 0$}.
$$
Moreover, for each $\tau \in \vV$, the embedded curves
that constitute the main levels of $\Psi(x,y,\tau)$ for $(x,y) \in (-1,1)^2$
are disjoint from each other and form 
the leaves of a smooth foliation on an open subset of~$\W$.
\end{lemma}
\begin{proof}
The $2$-parameter family $(y,\tau) \mapsto \Psi(0,y,\tau) \in 
\quotMod{1}{\tau}$ arises for reasons similar to the proof of
Lemma~\ref{lemma:index2}: curves in $\mM\nice_1(\widehat{J}_\tau;H;\mathbf{m})$
satisfy the automatic transversality criterion of \cite{Wendl:automatic}
and are thus regular for every~$\tau$.  Indeed, the criterion is satisfied
because by Lemma~\ref{lemma:upperBound}, any 
$u \in \mM\nice_1(\widehat{J}_\tau;H;\mathbf{m})$ has
only simply covered asymptotic orbits and exactly one of them is hyperbolic.  
Moreover, $u * u = 0$, implying that for any fixed $\tau \in \vV$,
the main levels of the $1$-parameter family $y \mapsto \Psi(0,y,\tau)$
are all disjoint and thus foliate a smoothly embedded hypersurface in~$\W$.

We claim that gluing can be used to extend this foliation to a neighborhood 
of the hypersurface.  The crucial detail here is that each of the
equivalence classes $[u] := \Psi(0,y,\tau)$ is represented by exactly two
buildings $u_+$ and $u_-$ whose upper levels have non-identical images: indeed,
since all asymptotic orbits for the buildings representing elements of 
$\quotModPlus$ are simply covered and elliptic, the same is true for all
elements of $\quotModHomotop{}$, so that the upper levels of
$\Psi(0,y,\tau)$ must always be unions of trivial cylinders with a
gradient flow cylinder connecting the hyperbolic orbit to an elliptic orbit.
There are always exactly two choices of this gradient flow cylinder---they 
form a \emph{canceling pair} in the sense of Lemma~\ref{lemma:cancel}.
By the same argument as in Lemma~\ref{L:mainLevelIntersection}, both of
the buildings $u_\pm$ satisfy 
\begin{equation}
\label{eqn:uplusminus}
u_+ * u_+ = u_- * u_- = u_+ * u_- = 0,
\end{equation}
and observe that the
gradient flow cylinders in their upper levels are also automatically regular.
We can therefore glue both buildings to obtain a pair of $1$-parameter
families of smooth and nicely embedded index~$2$ curves, which we define
to be $\Psi(x,y,\tau)$ for $x > 0$ and $x < 0$ respectively.  Each of these
two families satisfies the same implicit function theorem that was used in
Lemma~\ref{lemma:index2}, hence they each foliate open subsets of~$\W$.
Moreover, the
homotopy invariance of the intersection pairing implies via 
\eqref{eqn:uplusminus} that if $u$ and $u'$ denote the main levels of 
$\Psi(x,y,\tau)$ and $\Psi(x',y',\tau)$ with $x$ and $x'$ both nonzero,
then $u * u' = 0$, hence the two open subsets foliated by the two families
are disjoint, and for similar reasons, both are disjoint from the main
levels of the curves $\Psi(0,y,\tau)$ but contain them in their closures.
This shows that the main levels of the
$2$-parameter family $(x,y) \mapsto \Psi(x,y,\tau)$ foliate an open 
subset of $\W$ for each $\tau$ sufficiently close to~$0$.
\end{proof}

The preceding pair of lemmas shows that $\quotModHomotop{2} \cup
\quotModHomotop{1}$ is an open subset of $\quotModHomotop{}$ and has the
topology of a $3$-dimensional manifold, and for each $\tau \in [0,1]$,
$\quotMod{2}{\tau} \cup \quotMod{1}{\tau} \subset \quotMod{}{\tau}$ 
is similarly open and is a $2$-dimensional manifold.  Denote the
closure of $\quotModHomotop{2}$ in 
$\quotModHomotop{} \setminus (\quotModPlus \times [0,1])$ by
$$
\quotModHomotop{\Nice} \subset \quotModHomotop{} \setminus
\left( \quotModPlus \times [0,1] \right),
$$
and define
$$
\quotMod{\Nice}{\tau} := \left\{ u \in \quotMod{}{\tau}\ \big|\ 
(u,\tau) \in \quotModHomotop{\Nice} \right\}
$$
for each $\tau \in [0,1]$.

\begin{lemma}
\label{lemma:lowerDimension}
The closure $\quotModHomotop{\Nice}$ is the union of the sets
$\quotModHomotop{i}$ for all $i \in \{-1,0,1,2\}$.  Moreover, the
$\quotModHomotop{i}$ are smooth manifolds of dimension
$i+1$, and for $i \in \{-1,0\}$ they decompose into the following subsets 
characterized by the main level $u_0$ of an equivalence class of buildings
$[u] \in \quotModHomotop{i}$:
\begin{itemize}
\item $[u] \in \quotModHomotop{\reg,i} := \quotModHomotop{i} \cap
\quotModHomotop{\reg}$ if and only if $u_0$ is a smooth nicely embedded
curve with $\ind(u_0) = i$ and all asymptotic orbits simply
covered, with $2 - i$ of them hyperbolic;
\item $[u] \in \quotModHomotop{\sing,i} := \quotModHomotop{i} \cap 
\quotModHomotop{\sing}$ if and only if $u_0$ is a nodal curve with two
nicely embedded connected components $v$ and $v'$, where $\ind(v) = 0$ with
all asymptotic orbits simply covered and elliptic, while
$\ind(v') = i$ with all asymptotic orbits simply covered and $-i$ of them
hyperbolic;
\item $[u] \in \quotModHomotop{\exot,i} := \quotModHomotop{i} \cap
\quotModHomotop{\exot}$ if and only if $u_0$ is a smooth nicely embedded
curve with $\ind(u_0) = i$ and one asymptotic orbit doubly covered, the
rest simply covered, and $-i$ of them hyperbolic.
\end{itemize}
\end{lemma}
\begin{proof}
This is essentially a repackaging of the main compactness results from
\S\ref{sec:compactnessProofs}, i.e.~Propositions~\ref{T:compactness} 
and~\ref{T:deformationJ}.  The statement about the dimension of
$\quotModHomotop{i}$ for $i \in \{-1,0\}$ follows directly from the implicit 
function theorem since $\{\widehat{J}_\tau\}$ is generic in the interior of $W$
and all of these curves must intersect that interior due to
Lemma~\ref{lemma:index-1}.
\end{proof}

\begin{lemma}
\label{lemma:openClosed}
We have
$$
\quotModHomotop{\Nice} = \quotModHomotop{} \setminus 
\left(\quotModPlus \times [0,1]\right),
$$
and for each $\tau \in [0,1]$,
any two buildings $u , u'$ representing equivalence classes in 
$\quotMod{}{\tau}$ satisfy $u * u' = 0$.
\end{lemma}
\begin{proof}
By construction, $\quotModHomotop{\Nice}$ is closed in
$\quotModHomotop{} \setminus (\quotModPlus \times [0,1])$; 
we claim that it is also open.  We've already seen
that $\quotModHomotop{2} \cup \quotModHomotop{1}$ is open due to
Lemmas~\ref{lemma:index2} and~\ref{lemma:index1}, so it suffices to show
that any $([u],\tau) \in \quotModHomotop{i}$ for $i \in \{-1,0\}$ has a
neighborhood in $\quotModHomotop{}$ contained in~$\quotModHomotop{\Nice}$.
Let $u$ denote a holomorphic building representing such an element.
A neighborhood of $([u],\tau)$ will consist of all nearby elements of the
$(i+1)$-dimensional moduli space described in Lemma~\ref{lemma:lowerDimension},
plus any other equivalence classes represented by buildings with fewer
levels (e.g.~smooth curves) that are close to converging to~$u$ or one of
its equivalent buildings in the SFT-topology.
Lemma~\ref{lemma:lowerDimension} describes the possible main levels of $u$,
and the upper levels are allowed to consist of anything that produces 
arithmetic genus zero and the right collection of asymptotic orbits (all of them
simply covered) at the positive ends.  This allows for exactly the same
range of possibilities as seen in Propositions~\ref{T:compactness}
and~\ref{T:deformationJ}: all components in the upper levels are 
covers of either trivial
cylinders or gradient flow cylinders, each with covering multiplicity at 
most~$2$.  Aside from the ordering of the punctures, the only ambiguity 
involved in the upper levels is therefore the option to replace each
gradient flow cylinder with its partner in a canceling pair.  But by the
argument in Lemma~\ref{L:mainLevelIntersection}, this
alteration does not change the value of $u * u$ or $u * v$ for any
other building $v$ with $([v],\tau) \in \quotModHomotop{}$.  In particular,
since $u$ is equivalent to some building arising as a limit of a sequence of
nicely embedded index~$2$ curves, $u * u = 0$, implying that
any other smooth curve obtained by gluing $u$ will also be nicely embedded
and therefore an element of $\quotModHomotop{\Nice}$.  This proves that the
set is open as claimed, hence it is a union of connected components
of $\quotModHomotop{} \setminus (\quotModPlus \times [0,1])$.  
Since $\quotModHomotop{\Nice}$ also contains
the holomorphic pages in $\widehat{\nN}_+(\p E)$ that form a neighborhood
of $\quotModPlus \times [0,1]$ in $\quotModHomotop{}$, it now follows from
the definition of $\quotModHomotop{}$ that
$\quotModHomotop{\Nice} = \quotModHomotop{} \setminus (\quotModPlus
\times [0,1])$.  The claim about intersection numbers then follows from
Proposition~\ref{prop:int0} via the homotopy invariance of the 
pairing $u * u'$.
\end{proof}

\begin{lemma}
\label{lemma:foliation}
For each $\tau \in [0,1]$, every point in $\W$ is in the image of the main
level of a unique element of $\quotMod{\Nice}{\tau}$.
\end{lemma}
\begin{proof}
Let $\Delta \subset \W \times [0,1]$ denote the set of all points $(x,\tau)$
such that $x$ is in the image of the main level for some equivalence class
of buildings in $\quotMod{-1}{\tau} \cup \quotMod{0}{\tau}$.
By Lemma~\ref{lemma:lowerDimension}, $\Delta$ is the smooth image of a
manifold with components of dimension at most~$3$, i.e.~it is a
``subset of codimension at least~$2$'' in $\W \times [0,1]$.  If follows
that $(\W \times [0,1]) \setminus \Delta$ is connected.

Now define
$\Theta \subset (\W \times [0,1]) \setminus \Delta$ as the set of all
$(x,\tau) \not\in \Delta$ such that $x$ is in the image of the main level
for some equivalence class of buildings in
$\quotMod{1}{\tau} \cup \quotMod{2}{\tau}$.
Lemmas~\ref{lemma:index2} and~\ref{lemma:index1} imply that
$\Theta$ is an open subset of $(\W \times [0,1]) \setminus \Delta$,
and Lemma~\ref{lemma:lowerDimension} implies that it is also closed.
We conclude that $\W = \Theta \cup \Delta$, meaning every 
$(x,\tau) \in \W \times [0,1]$ has the property that $x$ is in the
image of the main level for some element of $\quotMod{\Nice}{\tau}$.
Uniqueness then follows from the fact that any two such elements $u$ and $u'$
satisfy $u * u' = 0$, as Lemma~\ref{L:mainLevelIntersection} implies that
any isolated intersection of the main levels would make $u * u'$ positive.
\end{proof}

We've now shown that for every $\tau \in [0,1]$, the main levels of the
buildings representing elements of $\quotMod{\Nice}{\tau}$ define a 
smoothly $\tau$-dependent foliation $\fF_\tau$ of~$\W$, which is 
singular on the set
$$
\W\crit_\tau \subset \W
$$
consisting of images of nodes for main levels of elements 
in $\quotMod{\sing}{\tau}$.  We thus obtain a continuous map
$$
\Pi_\tau : \W \to \quotMod{\Nice}{\tau}
$$
sending each point $x \in \W$ to the unique $[u] \in \quotMod{\Nice}{\tau}$
whose main level contains~$x$, and the resulting map
$$
\Pi : \W \times [0,1] \to \quotModHomotop{\Nice} : (x,\tau) \mapsto
(\Pi_\tau(x),\tau)
$$
is also continuous.  The remaining steps toward the proof of
Propositions~\ref{T:moduliSurface} and~\ref{prop:deform} are to define a
suitable smooth structure on $\quotModHomotop{\Nice}$ and to understand
the topological relationship between the sets
$\quotMod{}{}$ and $\quotModHomotop{}$ and their various subsets of
regular, singular and exotic curves.  To obtain a smooth structure and
orientation on $\quotModHomotop{\Nice}$, we can conveniently make use
of the foliations $\fF_\tau$ and thus avoid talking about smoothness of
gluing maps or coherent orientations.

\begin{lemma}
\label{lemma:quotSmooth}
The spaces $\quotModHomotop{\Nice}$ and $\quotMod{\Nice}{\tau}$ for each
$\tau \in [0,1]$ admit unique smooth structures such that the maps
$\Pi_\tau$ and $\Pi$ are smooth, except possibly at $\W\crit_\tau$ and
$\bigcup_{\tau \in [0,1]} W\crit_\tau \times \{\tau\}$ respectively.
Moreover, $\quotModHomotop{\Nice}$ is orientable, and there exists a
diffeomorphism
$$
\Psi : \quotMod{\Nice}{} \times [0,1] \to \quotModHomotop{\Nice}
$$
which satisfies $\Psi(\quotMod{\Nice}{} \times \{\tau\}) = 
\quotMod{\Nice}{\tau}$ for every $\tau \in [0,1]$ and
is the identity map on $\quotMod{\Nice}{} \times \{0\}$ and
the neighborhood of $\quotModPlus \times [0,1]$ defined by the holomorphic pages
in~$\widehat{\nN}_+(\p E)$.
\end{lemma}
\begin{proof}
Given $([u],\tau) \in \quotModHomotop{\Nice}$, pick a point $p \in \W$
that is in the image of a non-nodal point of the main level of~$[u]$.
Choose also a small embedded open $2$-disk $\dD_p \subset \W$ that has $p$ in 
its interior and is positively transverse to the foliation~$\fF_\tau$.
Then after shrinking $\dD_p$ if necessary, we can assume that the main
level curve for every element in some neighborhood 
$\uU \subset \quotMod{\Nice}{\tau}$ of $[u]$ passes
through a unique point of~$\dD_p$, thus defining a homeomorphism
$\uU \to \dD_p$.  (A crucial detail behind this assertion
is that all of the main levels of elements in $\quotModHomotop{\Nice}$ are
embedded---in particular none of them are multiply covered, otherwise the
map $\uU \to \dD_p$ would be multi-valued.)  We shall identify $\dD_p$
smoothly with the open unit disk in $\CC$ and regard $\uU \to \dD_p$ as a
chart.  Since the foliation is smooth and has a canonical co-orientation
determined by the almost complex structure, all transition maps relating
two charts of this type are smooth and orientation preserving.  In the
same manner, the smooth dependence of $\fF_\tau$ on $\tau$ allows us to
define local charts on $\quotModHomotop{\Nice}$ since $\dD_p$ can still
be assumed positively transverse to $\fF_{\tau'}$ for all $\tau'$
sufficiently close to~$\tau$.  This makes $\quotModHomotop{\Nice}$ a smooth
oriented manifold in which the function 
$$
\quotModHomotop{\Nice} \to [0,1] : ([u],\tau) \mapsto \tau
$$
is a smooth function with no critical points.  The existence of the
diffeomorphism $\Psi$ follows since the foliations 
$\fF_\tau$ match $\fF_+$ and are thus $\tau$-independent in the region
of $\widehat{\nN}_+(\p E)$ foliated by holomorphic pages.
\end{proof}

\begin{lemma}
\label{lemma:singLeaves}
Using the diffeomorphism $\Psi$ from Lemma~\ref{lemma:quotSmooth},
the sets $\quotModHomotop{\sing}$ and $\quotModHomotop{\exot}$ are each
disjoint unions of finite collections of subsets of the form
$$
\left\{ \Psi(f(\tau),\tau) \in \quotModHomotop{\Nice}\ \big|\ \tau \in [0,1] \right\},
$$
for continuous maps $f : [0,1] \to \quotMod{\Nice}{}$ that are smooth except
at finitely many points in~$(0,1)$.
\end{lemma}
\begin{proof}
Recall from Proposition~\ref{T:deformationJ} the finite subset
$$
I^\sing \subset (0,1),
$$
which in the present context we can regard as the set of all parameter
values $\tau$ such that the foliation $\fF_\tau$ contains a leaf of
index~$-1$ (and only one such leaf).  When $\tau \not\in I^\sing$,
Lemma~\ref{lemma:lowerDimension} identifies $\quotMod{\sing}{\tau}$ and 
$\quotMod{\exot}{\tau}$ with $\quotMod{\sing,0}{\tau}$ and
$\quotMod{\exot,0}{\tau}$ respectively, which are each compact and are
made up of equivalence classes whose main levels contain nicely embedded
index~$0$ curves with simple elliptic asymptotic orbits.  These curves
satisfy the automatic transversality criterion of 
Proposition~\ref{prop:automatic},
hence there are finitely many of them for each $\tau$ and they can be
deformed smoothly under small perturbations in~$\tau$.  In particular,
transversality implies that the projection
\begin{equation}
\label{eqn:submersionParam}
\quotModHomotop{\sing,0} \cup \quotModHomotop{\exot,0} 
\to [0,1] : ([u],\tau) \mapsto \tau
\end{equation}
is a submersion.

We claim next that if $\tau_0 \in I^\sing$ and 
$$
\left\{ ([u_\tau],\tau) \in
\quotModHomotop{\sing} \ |\ \tau \in (\tau_0 - \epsilon,\tau_0) \right\}
$$
is a smooth $1$-parameter family, then this family admits a unique continuous
extension to $\tau \in (\tau_0 - \epsilon,\tau_0 + \epsilon)$ if
$\epsilon > 0$ is sufficiently small, and the extended family is also smooth
for $\tau \in (\tau_0,\tau_0 + \epsilon)$.  The compactness results of
\S\ref{sec:compactnessProofs} imply that as $\tau \to \tau_0$ from below,
$[u_\tau]$ converges to
an element $[u_{\tau_0}]$ of either $\quotMod{\sing,0}{\tau_0}$ or 
$\quotMod{\sing,-1}{\tau_0}$.
In the first case we can again use automatic transversality and apply the
implicit function theorem to continue the family.  In the second case,
we can assume that the smooth curves $u_\tau$ representing $[u_\tau]$
converge to a holomorphic building $u_{\tau_0}$ as in
case~\ref{item:deformation2nodalgradflow} of Prop.~\ref{T:deformationJ},
whose main level is a nodal
curve with one index~$0$ and one index~$-1$ component, and the upper levels
consist of a gradient flow cylinder connecting the unique hyperbolic orbit
of the index~$-1$ curve to an elliptic orbit.  All components in this
building are parametrically Fredholm regular, i.e.~they are points at which
the relevant parametric moduli space $\{ (v,\tau)\ |\ \tau \in [0,1],\ \text{$v$ is $J_\tau$-holomorphic} \}$
is cut out transversely by the nonlinear Cauchy-Riemann operator.  
It follows that $u_{\tau_0}$ can be glued in a unique way, proving that all pairs
$(u,\tau)$ close to $(u_{\tau_0},\tau_0)$ in the SFT-topology belong to the
given family $\{ (u_\tau,\tau)\ |\ \tau \in (\tau_0 - \epsilon,\tau_0) \}$.  
However, one can also replace the
gradient flow cylinder in the upper level of $u_{\tau_0}$
with its canceling partner from Lemma~\ref{lemma:cancel}, giving a
new building $u_{\tau_0}'$, which can be glued to obtain 
a new smooth $1$-parameter family of elements in $\quotModHomotop{\sing,0}$
with an end degenerating to~$u_{\tau_0}'$.  Since the index~$0$ curves forming
the new family have the same asymptotic orbits as the components of
$u_\tau$ and each have normal Chern number~$-1$, 
Proposition~\ref{prop:orientIndex0} implies that these two components of
$\quotModHomotop{\sing,0}$ receive the same orientation, for which the
submersion \eqref{eqn:submersionParam} either preserves or reverses 
orientation.  But these coherent orientations also assign opposite signs to
the two canceling gradient flow cylinders, and hence to the buildings
$u_{\tau_0}$ and $u_{\tau_0}'$ that form the boundaries of our two families
in $\quotModHomotop{\sing,0}$.  It follows that the family we obtained by
gluing consists of pairs $([u],\tau)$ with $\tau > \tau_0$, not
$\tau \le \tau_0$, thus continuing the family as claimed.  The uniqueness of
the continuation follows from the observation that
the two buildings $u_{\tau_0}$ and $u_{\tau_0}'$ are the \emph{only} possible
representatives of $[u_{\tau_0}] \in \quotMod{\sing}{\tau_0}$ up to ordering
of the punctures.

The proof of the same claim for $\quotModHomotop{\exot}$ is entirely analogous,
the main difference being that there are now two possible types of 
degenerations from $1$-parameter families in $\quotModHomotop{\exot,0}$ to 
buildings in $\quotModHomotop{\exot,-1}$, as the unique hyperbolic orbit for the
index~$-1$ curve may or may not be the same one that is doubly covered,
i.e.~this is the distinction between cases~\ref{item:deformation2thebirds} 
and \ref{item:deformation2unbranched} in Prop.~\ref{T:deformationJ}.
If it is not the same orbit, then the upper levels contain a gradient flow
cylinder, and the family is continued exactly as above by
gluing its canceling partner.  If the hyperbolic orbit is doubly covered, then 
we instead have
an unbranched double cover of a gradient flow cylinder in an upper level,
but as observed in Lemma~\ref{lemma:cancel}, this unbranched cover also
satisfies automatic transversality and is oriented opposite to the
unbranched double cover of its canceling partner.  Thus the same trick
works to continue the family, this time by replacing the unbranched cover
with its own canceling partner and then gluing the building.
Note that for this picture of $\quotModHomotop{\exot}$ to be complete, one
must also picture an index~$2$ branched double cover of a trivial cylinder
over an elliptic orbit in the top level of every building,
as appears in case~\ref{item:compactness2branched} of Proposition~\ref{T:compactness}
and cases~\ref{item:deformation2thebirds} and \ref{item:deformation2unbranched}
of Proposition~\ref{T:deformationJ}.  But this branched cover can be
treated as a constant object that plays no role in the deformation or
gluing arguments.

As a final remark, note that there are other types of buildings listed in
Proposition~\ref{T:deformationJ} that can represent elements of
$\quotMod{\exot,-1}{\tau_0}$ and may arise as degenerations of
index~$2$ curves, namely cases~\ref{item:deformation2vertigo} 
and~\ref{item:deformation2ihatethisone}.  However, since they do not
include any branched cover of a trivial cylinder in the top level,
these cannot arise as degenerations of curves in 
$\quotModHomotop{\exot,0}$, which are always represented by
case~\ref{item:compactness2branched} of Proposition~\ref{T:compactness}.  
Moreover, they are equivalent in
$\quotModHomotop{\exot,-1}$ to buildings from case~\ref{item:deformation2unbranched},
and can thus
be replaced by such buildings in order to obtain
two glued families\footnote{It is worth clarifying that no obstruction bundle
gluing (in the sense of \cite{HutchingsTaubes:gluing2}) is required here, as the
branched covers in our picture serve merely as a bit of extra data that is
not involved in the gluing construction.} in the way described above, so that they always uniquely 
fit into the same $1$-parameter families in $\quotModHomotop{\exot}$
moving $\tau$ both forward and backward.
\end{proof}

The lemmas proved so far establish the main topological properties of the
spaces $\quotMod{}{\tau}$ and $\quotModHomotop{}$ as described in
Propositions~\ref{T:moduliSurface} and~\ref{prop:deform}.  It remains only
to examine the restrictions of the maps $\Pi_\tau : \W \to \quotMod{\Nice}{\tau}$
to the holomorphic vertebrae $\dot{\Sigma}_i \subset \W$ for $i=1,\ldots,r$.
The resulting maps
$$
\dot{\Sigma}_i \stackrel{\Pi_\tau}{\longrightarrow} \quotMod{\Nice}{\tau}
$$
are local diffeomorphisms wherever the foliation $\fF_\tau$ is transverse
to~$\dot{\Sigma} := \dot{\Sigma}_1 \cup \ldots \cup \dot{\Sigma}_r$.
By Lemma~\ref{lemma:genVertebrae}, non-transverse intersections of main
level curves with $\dot{\Sigma}$ occur only for elements in a
$1$-dimensional submanifold of $\quotModHomotop{2}$ and a
discrete subset of $\quotModHomotop{1}$, and moreover, each individual
curve in these spaces has only one non-transverse intersection, with
local intersection index~$2$.  With this understood, the following local result
serves to characterize the maps 
$\dot{\Sigma} \stackrel{\Pi_\tau}{\longrightarrow} \quotMod{\Nice}{\tau}$
as generic branched covers.

\begin{lemma}
\label{lemma:branchPoints}
Suppose $J$ is a smooth almost complex structure on~$\CC^2$,
$$
u_\zeta : (\DD,i) \to (\CC^2,J) \quad\text{ for }\quad \zeta \in \DD
$$
is a smooth $2$-parameter family of $J$-holomorphic curves such that
$u_0(0) = 0$ and the map $\DD \times \DD \to \CC^2 : (z,\zeta) \mapsto u_\zeta(z)$
is an embedding, and $\Sigma \subset \CC^2$ is an embedded $J$-holomorphic
curve that has an isolated intersection of index $k \in \NN$ with $u_0$
at the origin.  Then the map
$$
\Sigma \to \DD : w \mapsto \zeta(w) \qquad\text{ such that }\qquad
w \in \im u_{\zeta(w)}
$$
has the local structure of a branched cover, with the origin as a branch
point of order~$k$.
\end{lemma}
\begin{proof}
The statement follows immediately from transversality if $k=1$, so
let us assume $k \ge 2$.  After a change of coordinates, we can assume 
without loss of generality that 
$$
u_\zeta(z) = (z,\zeta) \qquad\text{ and }\qquad J(z,\zeta) = \begin{pmatrix} i & \alpha(z,\zeta) \\
0 & j(z,\zeta) \end{pmatrix}
$$
for $\End_\RR(\CC)$-valued functions $\alpha,j$ that satisfy $j^2 = -\1$
and $i\alpha + \alpha j = 0$.  In these coordinates, we can write 
$\Sigma$ near the origin as the image of an embedded $J$-holomorphic disk 
$v = (\varphi,f) : (\DD,i) \to (\CC^2,J)$ that satisfies
$\varphi(0) = f(0) = 0$ by assumption, and the condition $k \ge 2$ implies
a tangential intersection with~$u_0$, thus $df(0) = 0$, so that
$\varphi : \DD \to \CC$ can be assumed an embedding.
Since the family of curves $u_\zeta$
is parametrized by the second complex coordinate, our goal is now to
show that the function $f : \DD \to \CC$ has the structure of a $k$-to-$1$
branch point at~$0$.  This follows easily from the similarity principle:
writing $v_0(z) := (\varphi(z),0)$, the equation $\p_s v + J(v) \p_t v = 0$ implies
\begin{equation*}
\begin{split}
\p_s v + J(v_0) \p_t v &= \begin{pmatrix}
\p_s \varphi \\
\p_s f
\end{pmatrix} +
\begin{pmatrix}
i & \alpha(\varphi,0) \\
0 & j(\varphi,0)
\end{pmatrix}
\begin{pmatrix}
\p_t \varphi \\
\p_t f
\end{pmatrix} = -\left[ J(v) - J(v_0) \right] \p_t v \\
&= - \left( \int_0^1 D_2 J(\varphi , t f) \cdot f \ dt \right) \p_t v =: -\widetilde{A} f,
\end{split}
\end{equation*}
where the linear dependence of the integral in the second line on $f$ is used
to define a smooth function $\widetilde{A} : \DD \to \Hom_\RR(\CC,\CC^2)$.
Projecting all of this to the second factor in $\CC \times \CC$ then produces
a linear Cauchy-Riemann type equation $\p_s f + j(\varphi,0) \p_t f + A f = 0$,
and since the intersection of $v$ with $u_0$ is isolated, the similarity
principle now implies that $f$ has a nontrivial Taylor series whose first
nonzero term is holomorphic.  That term must be a multiple of $z^k$, in
light of the intersection index, thus $f$ is given by
$$
f(z) = a z^k + |z|^k R(z)
$$
for some nonzero coefficient $a \in \CC$ and a continuous remainder function satisfying
$R(0)=0$.  On a small enough neighborhood of~$0$ so that $|R(z)| < |a|$, 
this can also be written as $f(w) = w^k$ in a new $C^1$-smooth complex coordinate defined by
$w := z \left( a + \frac{|z|^k}{z^k} R(z) \right)^{1/k}$.
\end{proof}

For the next statement, let $\overline{\Sigma}_i$ denote the compact
topological surface obtained by adding circles at infinity to each of the
cylindrical ends of~$\dot{\Sigma}_i$.

\begin{lemma}
\label{lemma:branchedCover}
For each $i=1,\ldots,r$ and every $\tau \in [0,1]$, the map 
$\dot{\Sigma}_i \stackrel{\Pi_\tau}{\longrightarrow} \quotMod{\Nice}{\tau}$
extends to a continuous map
$$
\left(\overline{\Sigma}_i, \p\overline{\Sigma}_i\right) \to
\left(\quotMod{}{\tau} , \quotModPlus\right)
$$
of degree~$m_i$ whose restriction to the boundary is a covering map.  
Moreover, it is a generic branched cover of surfaces with cylindrical ends
in the sense of Definition~\ref{defn:genericBC},
the images in $\quotMod{\Nice}{\tau}$ of the
branch points all lie in $\quotMod{\reg}{\tau}$, and
the nodes of curves in $\quotMod{\sing}{\tau}$ never intersect~$\dot{\Sigma}_i$.
\end{lemma}
\begin{proof}
The continuous extension and its degree are already clear from the fact
that the holomorphic pages (which form the cylindrical ends of
$\quotMod{\Nice}{\tau}$) each have exactly $m_i$ intersections with
$\dot{\Sigma}_i$, all of them transverse.  That $\Pi_\tau|_{\dot{\Sigma}_i}$
is a branched cover with only simple branch points follows from
Lemma~\ref{lemma:branchPoints}, together with the preceding remarks on
genericity and intersections.  The branch points are the tangential intersections
of $\dot{\Sigma}_i$ with leaves of the foliation, and
Lemma~\ref{lemma:genVertebrae} implies that such a
point $\zeta$ is necessarily the only point of tangency on a given leaf,
hence all images of branch points are distinct. Finally,
Lemma~\ref{lemma:genVertebrae} implies that the index~$0$ and $-1$
main level components in $\quotMod{\sing}{\tau}$ and $\quotMod{\exot}{\tau}$
always intersect $\dot{\Sigma}_i$ transversely, hence these are never
critical values of the branched cover.  Lemma~\ref{lemma:nodalBranching}
implies in turn that the two components of each nodal curve in
$\quotMod{\sing}{\tau}$ never intersect $\dot{\Sigma}_i$ in the same places,
hence their intersections with $\dot{\Sigma}_i$ are disjoint from the node.
\end{proof}

Our final lemma in this section concerns the Lefschetz-amenable case.

\begin{lemma}
\label{lemma:noBranchPoints}
The branched cover in Lemma~\ref{lemma:branchedCover} has no branch
points if and only if $\quotMod{\exot}{\tau} = \emptyset$ for all $\tau \in [0,1]$.
\end{lemma}
\begin{proof}
We show first that the absence of branch points rules out exotic fibers.
By Lemma~\ref{lemma:singLeaves}, it suffices to prove that 
$\quotMod{\exot}{}$ is empty whenever $\dot{\Sigma} \stackrel{\Pi_0}{\longrightarrow} \quotMod{\Nice}{}$
is an honest covering map.  Recall from Proposition~\ref{prop:vertebrae} that the
holomorphic vertebrae $\dot{\Sigma}_i$ are not isolated:
each can be shifted in the direction of the $\theta$-coordinate, producing a
smooth $S^1$-family of embedded $J$-holomorphic curves $\dot{\Sigma}_i^\theta \subset \W$
that foliate a smooth hypersurface
$$
Y_i := \bigcup_{\theta \in S^1} \dot{\Sigma}_i^\theta \subset \W.
$$
We cannot assume that the genericity conditions imposed in \S\ref{sec:genericity}
hold for intersections of leaves with \emph{every} curve in the family
$\dot{\Sigma}_i^\theta$, thus a leaf may have intersections of index greater
than $2$ with some of these curves, but the intersections are still isolated
and positive, thus Lemma~\ref{lemma:branchPoints} still applies and gives
each of the maps $\dot{\Sigma}_i^\theta \stackrel{\Pi_0}{\longrightarrow} \quotMod{\Nice}{}$
the structure of a (not necessarily generic) branched cover.  The rest of
the arguments in Lemma~\ref{lemma:branchedCover} also apply for every~$\theta$,
showing that the branch points of these covers are confined to a compact
subset, and they can be counted algebraically using the Riemann-Hurwitz
formula.  It follows that the condition of having no branch points is
independent of~$\theta$, hence this assumption implies that every
leaf of the foliation is transverse to the entire hypersurface
$Y := Y_1 \cup \ldots \cup Y_r$.

Recall that a neighborhood of infinity
in $\quotMod{\Nice}{}$ coincides with the foliation $\fF_+$ constructed
in \S\ref{sec:holPages}, and each leaf of the latter intersects the
region of $\widehat{E}$ above $Y$ in a disjoint union of cylindrical ends
whose boundary circles are in
bijective correspondence with the boundary components of the pages (see Figure~\ref{fig:holfol}), each of
them having degree $1$ under the projection $\widehat{\Sigma} \times S^1 \to S^1$.
In light of the transverse intersections with~$Y$, it follows that the same
is true for \emph{every} leaf of~$\fF_0$, implying that none can have an
end asymptotic to a doubly covered orbit (see Figure~\ref{fig:compactness2branched}),
which must occur if $\quotMod{\exot}{}$ were nonempty.

Conversely, if $\quotMod{\exot}{} = \emptyset$, then all main level curves
in $\quotMod{\Nice}{}$ have the same number of ends with the same asymptotic
orbits and multiplicities.  We claim that for any $\theta \in S^1$ and
$t > 0$ sufficiently large, all of them are transverse to the properly embedded surface
$\dot{\Sigma}^{(t,\theta)} := \{t\} \times \widehat{\Sigma} \times \{\theta\} \subset \widehat{\nN}(\p_h E)$.
Indeed, this is obvious for the holomorphic pages constructed in \S\ref{sec:holPages},
which form a neighborhood of infinity in the moduli space, and for everything else
the claim follows for $t \gg 0$ due to the asymptotic convergence of curves
to Reeb orbits, which are never tangent to $\dot{\Sigma}^{(t,\theta)}$.
The restriction of $\Pi$ to $\dot{\Sigma}^{(t,\theta)}$ is therefore a proper
covering map and thus satisfies the Riemann-Hurwitz formula, with $0$ for the
count of branch points.  Since $\dot{\Sigma}^{(t,\theta)}$ and
$\dot{\Sigma} = \dot{\Sigma}_1 \cup \ldots \cup \dot{\Sigma}_r$ are homeomorphic,
it now also follows from the Riemann-Hurwitz formula that
$\Pi|_{\dot{\Sigma}}$ cannot have branch points.
\end{proof}

The proof of Propositions~\ref{T:moduliSurface}
and~\ref{prop:deform} is now complete.

\subsection{The Lefschetz fibration on the filling}
\label{sec:Lefschetz}

We can now finish the proof of Theorems~\ref{thm:classification} 
and~\ref{thm:weak} by showing that if the spinal open book
$\boldsymbol{\pi}$ is {Lefschetz-amenable}, then the stable foliation from
Propositions~\ref{T:moduliSurface} and~\ref{prop:deform} gives rise to
a bordered Lefschetz fibration supporting the symplectic structure of~$W$.

Assume $\boldsymbol{\pi}$ is {Lefschetz-amenable}, so according to
Proposition~\ref{T:moduliSurface}, the set $\quotMod{\exot}{}$ is empty
and $\W$ is foliated (with finitely many singular points) by a mixture of
smoothly embedded $\widehat{J}$-holomorphic curves and finitely many nodal
curves that look like Lefschetz singular fibers.  Recall from 
\S\ref{sec:largeSubdomains} the bounded subdomains 
$\widehat{E}_R \subset \widehat{E}$ for $R > 0$, and let
$$
\W_R \subset \W
$$
denote the compact subdomain in $\W$ with $\p \W_R = \p \widehat{E}_R$;
its boundary (see Figure~\ref{fig:largeSubd}) is piecewise smooth and splits naturally into horizontal
and vertical faces
$$
\p \W_R = \p_v \W_R \cup \p_h \W_R.
$$
These subdomains with their symplectic and/or almost Stein data are 
deformation equivalent to $W$ by Lemma~\ref{lemma:bigContact}.
Now by taking $R > 0$ sufficiently large, we can assume near $\p \W_R$
that the foliation formed by the $\widehat{J}$-holomorphic curves in
$\quotMod{}{}$ is arbitrarily $C^\infty$-close to the $\RR$-invariant
foliation $\fF_+$; this follows from the fact that sequences of curves
in $\quotMod{}{}$ escaping to infinity necessarily converge to curves
in $\quotModPlus$.  In fact, these two foliations match precisely near
$\p_v \W_R$, since the curves in this region are contained fully in the
cylindrical end.  Near $\p_h \W_R$, we can now make a $C^\infty$-small 
modification ``by hand'' of the almost complex structure and the holomorphic 
curves so that the latter become precisely tangent to~$\fF_+$. 
After this modification, consider the restriction
$$
\Pi : \W_R \to \Sigma_0
$$
of the map in Proposition~\ref{T:moduliSurface}, where we define $\Sigma_0
\subset \quotMod{}{}$ as the image of this restricted map.  What we lose
by forgetting $\W \setminus \W_R$ is a collection of $1$-parameter families
of curves contained in $\widehat{\nN}_+(\p E)$; these form collar 
neighborhoods of the boundary in $\quotMod{}{}$, hence $\Sigma_0$ is a compact
surface with the same topological type as~$\quotMod{}{}$.
Since the nodal singularities can all be assumed to lie in the interior
of $\W_R$ for $R$ sufficiently large, 
$\Pi : \W_R \to \Sigma_0$ is now a bordered Lefschetz
fibration, and it is allowable if $(\W_R,\widehat{\omega})$ is minimal---which
is true if and only if $(W,\omega)$ is minimal---since the only closed
components allowed by the compactness results in \S\ref{sec:compactnessProofs}
are embedded spheres with self-intersection number~$-1$.
Lemma~\ref{lemma:bigFibrations} implies moreover that $\Pi : \W_R \to \Sigma_0$
supports the symplectic and/or Liouville
structure of $\W_R$, and in the almost Stein case, $(\W_R,\widehat{J},\widehat{f})$
is almost Stein deformation equivalent to the canonical structure for this
Lefschetz fibration by \cite{LisiVanhornWendl1}*{Theorem~C}.
With this, we've proved that the maps in Theorem~\ref{thm:classification}
sending equivalence classes of Lefschetz fibrations to equivalence classes
of fillings are surjectve.

To show that these maps are also injective, suppose we have two bordered
Lefschetz fibrations bounded by $\boldsymbol{\pi}$ that give rise to
deformation equivalent fillings.  These Lefschetz fibrations then admit
``double completions'' formed by gluing them into the model $\widehat{E}$
from \S\ref{sec:model} so that their vertical subbundles match~$V E$ near 
their boundaries, and we can choose tame almost complex structures on both
that match $J_+$ on the end and make all fibers holomorphic.  
We can therefore view them both as holomorphic foliations on the same
noncompact manifold $\W$, corresponding to two
distinct choices of almost complex structures $\widehat{J}_0$ and $\widehat{J}_1$
tamed by deformation-equivalent choices of symplectic data, all identical
on~$\widehat{\nN}_-(\p E)$.  Choosing a deformation of the symplectic
data and a corresponding deformation of tame almost complex structures,
Proposition~\ref{prop:deform} then connects the two foliations by a
smooth $1$-parameter family, producing an isotopy of bordered Lefschetz
fibrations which can be adjusted near $\p \W_R$ as in the previous
paragraph so that they support the family of symplectic structures.
The proof of Theorems~\ref{thm:classification} and~\ref{thm:weak} 
is now complete.

\subsection{Quasiflexible Stein structures}
\label{sec:quasi}

We now prove Theorem~\ref{thm:SteinDeformation}.

Assume $\Pi : W \to \Sigma_0$ is an allowable bordered Lefschetz fibration with fibers
of genus zero, and $(J_0,f_0)$ is an almost Stein structure on $W$ supported
by~$\Pi$.  Assume further that $(J_1,f_1)$ is a second almost Stein structure
on $W$ such that the symplectic structures $\omega_0 := -df_0 \circ J_0$
and $\omega_1 := -df_1 \circ J_1$ are homotopic through a smooth family of
symplectic structures $\{\omega_\tau\}_{\tau \in [0,1]}$ that are convex at
the boundary---recall that since $\p W$ has corners, the convexity condition
means that the associated Liouville vector fields are outwardly transverse
to both $\p_h W$ and~$\p_v W$.  The aim is to show that the Weinstein structures
induced by $(J_0,f_0)$ and $(J_1,f_1)$ are Weinstein homotopic.

The spinal open book $\boldsymbol{\pi} := \p\Pi$ on $M := \p W$ has spine
$M\spine = \p_h W$ and paper $M\paper = \p_v W$, with the fibration
$\pi\spine : M\spine \to \Sigma$ obtained by factoring
$\p_h E \stackrel{\Pi}{\longrightarrow} \Sigma_0$ through a suitable covering
map $\Sigma \to \Sigma_0$ to make its fibers connected, and
$\pi\paper : M\paper \to S^1$ defined from $\p_v E \stackrel{\Pi}{\longrightarrow} \p\Sigma_0$
by identifying each component of $\p\Sigma_0$ with~$S^1$.  We assume in the
following that this particular spinal open book is used for the construction
of the model $\widehat{E}$ in~\S\ref{sec:model}.  
Recall now from \cite{LisiVanhornWendl1}*{Theorem~1.24} that the space of
almost Stein structures supported by $\Pi : W \to \Sigma_0$ is contractible,
thus we are free after a deformation to assume that $(J_0,f_0)$ matches an
almost Stein structure constructed via the Thurston trick as in the proof of
that theorem.  Since the same application of the Thurston trick underlies
the almost Stein model constructed in \S\ref{sec:model}, one obtains the
following result:

\begin{lemma}
After a deformation of $(J_0,f_0)$ through supported almost Stein structures
on~$W$, the model $\widehat{E}$ in \S\ref{sec:model} can be constructed so that 
the bounded region $E \subset \widehat{E}$ with its almost Stein data
$(J_+,f_+)$ admits a diffeomorphism with a neighborhood
of $\p W$ in~$(W,J_0,f_0)$, identifying $\p_h W = \p_h E$, 
$\p_v W = \p_v E$, and the fibers of $\Pi$ in this neighborhood with
the fibers of $\nN(\p_h E) \stackrel{\Pi_h}{\longrightarrow} \Sigma$
and $\nN(\p_v E) \stackrel{\Pi_v}{\longrightarrow} (-1,0] \times S^1$.
\qed
\end{lemma}

Attaching $(W,J_0,f_0)$ to $(\widehat{E},J_+,f_+)$ via the lemma produces a
completed almost Stein domain $(\W,\widehat{J}_0,\widehat{f}_0)$ that is 
foliated by $\widehat{J}_0$-holomorphic curves matching the fibers of
$\Pi : W \to \Sigma_0$ in $W$ and the leaves of the foliation $\fF_+$
on $\W \setminus W$.  Note that $\widehat{J}_0$ in this construction cannot be
assumed generic. Nonetheless, the almost Stein condition is open, so
after a small perturbation of $\Pi$
away from $\p W$ and a corresponding perturbation of $(J_0,f_0)$ to
ensure that the perturbed fibers are still $J_0$-holomorphic, we can assume
without loss of generality that no curve in our $\widehat{J}_0$-holomorphic
foliation of $\W$ has more than one end asymptotic to a hyperbolic orbit,
and the finitely many curves that make up singular fibers have no ends
asymptotic to hyperbolic orbits.
Since these curves all have genus zero, it now follows that they all 
satisfy the criterion for automatic transversality from \cite{Wendl:automatic},
so they will survive a further perturbation of $\widehat{J}_0$ in the interior
of~$W$, which we now perform in order to assume the genericity conditions
of \S\ref{sec:genericity}.  Denote the resulting $\widehat{J}_0$-holomorphic
foliation of $\W$ by~$\fF_0$.

As in the proof of Theorem~\ref{thm:classification}, our original almost
Stein domain is Stein deformation equivalent to the enlarged compact
domain $(\W_R,\widehat{J}_0,\widehat{f}_0)$ in $\W$ for $R \gg 0$.
The symplectic deformation $\{\omega_\tau\}_{\tau \in [0,1]}$ can now also
be fit into this picture and gives rise to a smooth family
of symplectic structures $\widehat{\omega}_\tau$ on $\W$ that are independent
of $\tau$ on a neighborhood of infinity and match 
$-d\widehat{f}_\tau \circ \widehat{J}_\tau$ for $\tau \in \{0,1\}$,
where $(\widehat{J}_1,\widehat{f}_1)$ is a similar extension
of the almost Stein structure $(J_1,f_1)$ from $W$ to~$\W$, matching
$(J_+,f_+)$ near infinity.  By the constractibility of the space of
tame almost complex structures, we can choose a generic family
$\{\widehat{J}_\tau\}_{\tau \in [0,1]}$ of $\widehat{\omega}_\tau$-tame
almost complex structures that form a homotopy from $\widehat{J}_0$
to $\widehat{J}_1$ and match $J_+$ near infinity.  
Using Proposition~\ref{prop:deform}, the foliation $\fF_0$ now extends to a
smooth family of $\widehat{J}_\tau$-holomorphic foliations $\fF_\tau$,
which includes smooth deformations of finitely many nodal curves
(i.e.~the original singular fibers of~$\Pi$) from $\tau=0$ to $\tau=1$,
but does not include any exotic fibers since none were present in the
foliation~$\fF_0$.  Making the same modifications near infinity as in
the proof of Theorem~\ref{thm:classification}, the result is a smooth
family of allowable bordered Lefschetz fibrations 
$\Pi_\tau : \W_R \to \Sigma_0$ supporting the symplectic structures
$\widehat{\omega}_\tau$ for $R$ sufficiently large.
Theorem~C in \cite{LisiVanhornWendl1} now implies that the Weinstein
structure induced by $(\widehat{J}_1,\widehat{f}_1)$
lies in the canonical Weinstein homotopy class supported
by~$\Pi_1$, implying that it is also Weinstein homotopic to
the Weinstein structure induced by~$(\widehat{J}_0,\widehat{f}_0)$.
This completes the proof of the first statement in
Theorem~\ref{thm:SteinDeformation}.

If $\Sigma_0 = \DD^2$, then we can weaken the convexity hypothesis on
the family of symplectic structures $\{\omega_\tau\}_{\tau \in [0,1]}$ 
and assume instead that after smoothing the corners of~$\p W$,
each $(W,\omega_\tau)$ is a weak filling of $(M := \p W,\xi_\tau)$ for some
smooth family of contact structures $\xi_\tau$ on~$M$ matching
$\ker(-d f_\tau \circ J_\tau|_{TM})$ for $\tau \in \{0,1\}$.
The key observation here is that every component of the spine
$M\spine$ is a solid torus $\DD^2 \times S^1$, on which all closed $2$-forms
are exact, so the possible non-exactness of $\omega_\tau$ at $\p W$ can be
absorbed into the above construction by including in the symplectic data
near infinity a family of closed $2$-forms $\eta_\tau$ as in \S\ref{sec:Liouville},
which are assumed to vanish on~$\widehat{\nN}(\p_h E)$ and vanish
identically for $\tau \in \{0,1\}$.  Since $\eta_\tau$ changes the Reeb vector
field on the cylindrical end over the paper, it causes a change to
$\widehat{J}_\tau$ in this region, but the holomorphic pages here are
tangent to the fixed integrable distribution $\Xi_+$ and thus remain
holomorphic.  With this understood, the argument of the previous paragraph
now goes through with no further changes, and the proof
of Theorem~\ref{thm:SteinDeformation} is thus complete.

\end{document}